\documentclass[11pt]{article}

\usepackage{amsmath,amssymb,amscd,amsthm}
\usepackage{ascmac}
\usepackage[all]{xy}
\usepackage{enumerate}
\usepackage{cancel}
\usepackage[dvips]{graphicx}
\usepackage{authblk}
\usepackage{comment}
\usepackage{mathrsfs}
\usepackage{array,arydshln}
\usepackage{geometry}
\usepackage{url}
\usepackage{ulem}
\usepackage{bbm}
\theoremstyle{plain}
\newtheorem{thm}{Theorem}[section]
\newtheorem{pro}[thm]{Proposition}
\newtheorem{`thm'}[thm]{``Theorem''}
\newtheorem{cor}[thm]{Corollary}
\newtheorem{lem}[thm]{Lemma}
\newtheorem{conj}[thm]{Conjecture}
\newtheorem*{mainthm}{Main-Theorem}
\theoremstyle{definition}
\newtheorem{prob}[thm]{Problem}
\newtheorem{dfn}[thm]{Definition}
\newtheorem{fact}[thm]{Fact}

\newtheorem{asm}[thm]{Assumption}
\newtheorem{dfn-thm}[thm]{Definition-Theorem}
\newtheorem{dfn-pro}[thm]{Definition-Proposition}
\newtheorem{nota}[thm]{Notations}

\theoremstyle{remark}
\newtheorem{rmk}[thm]{Remark}
\newtheorem{rmks}[thm]{Remarks}
\newtheorem{ex}[thm]{Example}

\newcommand{\bb}[1]{\mathbb{#1}}
\newcommand{\fra}[1]{\mathfrak{#1}}
\newcommand{\ca}[1]{\mathcal{#1}}
\newcommand{\scr}[1]{\mathscr{#1}}
\newcommand{\Cl}{\mathit{Cliff}}
\newcommand{\innpro}[3]{\left<#1,#2\right>_{#3}}

\newcommand{\inpr}[3]{\left(#1\,\middle|\,#2\right)_{#3}}

\newcommand{\Sine}[2]{\frac{\sin#1\theta}{#1^#2}}
\newcommand{\Cosi}[2]{\frac{\cos#1\theta}{#1^#2}}
\newcommand{\rot}{{\rm rot}}

\newcommand{\od}{{\rm od}}
\newcommand{\alg}{{\rm alg}}
\newcommand{\ind}{{\rm ind}}
\newcommand{\id}{{\rm id}}
\newcommand{\PD}{{\rm PD}}
\newcommand{\fgt}{{\rm fgt}}

\newcommand{\ran}{{\rm ran}}
\newcommand{\HKT}{{\rm HKT}}

\newcommand{\Hom}{{\rm Hom}}
\newcommand{\End}{{\rm End}}
\newcommand{\dom}{{\rm dom}}

\newcommand{\Lie}{{\rm Lie}}
\newcommand{\fin}{{\rm fin}}

\newcommand{\hol}{{\rm hol}}

\newcommand{\Hol}{{\rm Hol}}
\newcommand{\Op}{{\rm Op}}

\newcommand{\vac}{{\rm vac}}

\newcommand{\Ad}{{\rm Ad}}

\newcommand{\bra}[1]{\left(#1\right)}
\newcommand{\bbra}[1]{\left\{#1\right\}}
\newcommand{\bbbra}[1]{\left[#1\right]}
\newcommand{\Dirac}{\cancel{\partial}}

\newcommand{\ud}[1]{\underline{#1}}

\newcommand{\flip}{{\rm flip}}
\newcommand{\vol}{{\rm vol}}
\newcommand{\Aut}{{\rm Aut}}
\newcommand{\ev}{{\rm ev}}
\newcommand{\lt}{{\rm lt}}
\newcommand{\rt}{{\rm rt}}
\newcommand{\vep}{\varepsilon}
\newcommand{\grotimes}{\widehat{\otimes}}

\begin{document}
\title{Topological Aspects of the Equivariant Index Theory of Infinite-Dimensional $LT$-Manifolds}
\author{Doman Takata \\
The University of Tokyo}

\date{\today}

\maketitle
\begin{abstract}
Let $T$ be the circle group and let $LT$ be its loop group. We formulate and investigate several topological aspects of the $LT$-equivariant index theory for proper $LT$-spaces, where proper $LT$-spaces are infinite-dimensional manifolds equipped with ``proper cocompact'' $LT$-actions.
Concretely, we introduce ``$\ca{R}KK$-theory for infinite-dimensional manifolds'', and by using it, we formulate an infinite-dimensional version of the $KK$-theoretical Poincar\'e duality homomorphism, and an infinite-dimensional version of the $\ca{R}KK$-theory counterpart of the assembly map, for proper $LT$-spaces.

The left hand side of the Poincar\'e duality homomorphism is formulated by the ``$C^*$-algebra of a Hilbert manifold'' introduced by Guoliang Yu. Thus, the result of this paper suggests that this construction carries some topological information of Hilbert manifolds. In order to formulate the assembly map in a classical way, we need crossed products, which require an invariant measure of a group. However, there is an alternative formula to define them using generalized fixed-point algebras. We will adopt it as the definition of ``crossed products by $LT$''.\\\\
{\it Mathematics Subject Classification $(2010)$.} 19K56; 19K35, 19L47,  22E67, 46L52, 58B34.
\end{abstract}
\tableofcontents

\section{Introduction}\label{section introduction}

The Atiyah-Singer index theorem states that the analytic index of the Dirac operator on a closed manifold is determined by topological data \cite{ASi1,ASi2,ASe1}. The overall goal of our research is to establish an infinite-dimensional version of this theorem. As a first step of this project, we have been studying infinite-dimensional manifolds equipped with $LT$-actions, where $T:=S^1$ and $LT$ is its loop group.
In the present paper, we study topological aspects of this problem. 

Let us begin with the $KK$-theoretical description of a non-compact version of the index theorem given by Kasparov \cite{Kas83,Kas15}. The details will be explained in Section \ref{section index theorem}.

\begin{fact}[\cite{Kas83,Kas15}]\label{Intro index theorem}
Let $X$ be a complete Riemannian manifold and let $G$ be a locally compact second countable Hausdorff group acting on $X$ in an isometric, proper and cocompact way. 
Let $W$ be a $G$-equivariant Clifford bundle, and let $D$ be a $G$-equivariant Dirac operator on $W$. 

$(1)$ $D$ has an {\bf analytic index} $\ind(D)$ as an element of $KK(\bb{C},\bb{C}\rtimes G )$.

$(2)$ The $KK$-element $[D]:=[(L^2(X,W),D)]\in KK_G(C_0(X),\bb{C})$ determines $\ind(D)$ with a $KK$-theoretical procedure 
$$\mu_G:KK_G(C_0(X),\bb{C})\to KK(\bb{C},\bb{C}\rtimes G),$$
that is to say, $\ind(D)=\mu_G([D])$. This procedure is called the {\bf analytic assembly map}.

$(3)$ $[D]$ is determined by the topological data $[\sigma_D^{Cl}]=[(C_0(X,W),0)]\in \ca{R}KK_G(X;C_0(X),Cl_\tau(X))$, where $Cl_\tau(X)$ is the section algebra of the Clifford algebra of the tangent bundle. More generally, there is a homomorphism called the {\bf Poincar\'e duality homomorphism}
$$\PD:KK_G(C_0(X),\bb{C})\to\ca{R}KK_G(X;C_0(X),Cl_\tau(X))$$
and it is isomorphic.

$(4)$ Consequently, $\ind(D)$ is completely determined by the topological data $[\sigma_D^{Cl}]$. The concrete formula giving the index is called the {\bf topological assembly map}, and it is denoted by
$$\nu_G:\ca{R}KK_G(X;C_0(X),Cl_\tau(X))\to KK(\bb{C},\bb{C}\rtimes G).$$
\end{fact}

Briefly speaking, what we will do in the present paper is to construct infinite-dimensional analogues of $\PD$ and $\nu_G$. For this aim, we need to overcome many problems including the following: {\it For an infinite-dimensional manifold, the $C_0$-algebra is trivial; For a non-locally compact space, we can not define $\ca{R}KK$-theory in a classical way.}

For the former problem, we will use a ``$C^*$-algebra of a Hilbert manifold $\ca{X}$'' denoted by $\ca{A}(\ca{X})$, introduced in \cite{Yu}, which is a wide generalization of a $C^*$-algebra of a Hilbert-Hadamard space defined in \cite{GWY}.
Strictly speaking, it is an infinite-dimensional analogue of, not  $C_0(X)$ but the ``graded suspension of $Cl_\tau(X)$''. However, $Cl_\tau(X)$ is $KK$-equivalent to $C_0(X)$ if $X$ is even-dimensional and $Spin^c$.
Thus, we can remove $C_0(X)$ from the index theorem for such manifolds, and we can reformulate it with the graded suspension of $Cl_\tau(X)$ which can be generalized to Hilbert manifolds.

For the latter one, by the same $KK$-equivalence, we have an isomorphism $\ca{R}KK(X;C_0(X),Cl_\tau(X))\cong \ca{R}KK(X;C_0(X),C_0(X))$. Moreover, it is isomorphic to $RK^0(X)$, which is the representable $K$-theory of Segal. While $\ca{R}KK(X;C_0(X),C_0(X))$ is defined by $C_0(X)$, $RK^0(X)$ is defined in a purely topological way. Thus, it might be possible to generalize $\ca{R}KK$-theory to an infinite-dimensional manifold in a topological way. With this observation, we will formulate $\ca{R}KK$-theory for infinite-dimensional spaces, via fields of $C^*$-algebras and Hilbert modules.

We can now state the main results of the present paper. A proper $LT$-space is, roughly speaking, an infinite-dimensional manifold equipped with a proper cocompact $LT$-action (see Section \ref{section problem spaces} for the details). We would like to study the index problem of a Spinor bundle twisted by a $\tau$-twisted $LT$-equivariant line bundle $\ca{L}$. 

\begin{thm}\label{Main theorem PD}
We can formulate a ``$KK$-theoretical Poincar\'e duality homomorphism'' for a Hilbert manifold satisfying a bounded geometry type condition equipped with a proper isometric group action. 
When we apply this construction to a proper $LT$-space, we obtain an appropriate answer: The Poincar\'e duality homomorphism assigns to the ``index element twisted by $\ca{L}$ constructed in \cite{T4}'' the $\ca{R}KK$-element corresponding to $\ca{L}$ .
\end{thm}

\begin{thm}\label{Main theorem ass}
We can formulate a ``topological assembly map'' for proper $LT$-spaces. This map assigns to the $\ca{R}KK$-element corresponding to $\ca{L}$ the ``analytic index of the Dirac operator twisted by $\ca{L}$'' constructed in \cite{Thesis}.
\end{thm}

\begin{cor}\label{Main corollary}
By the composition of the above two homomorphisms, the ``analytic index of the Dirac operator twisted by $\ca{L}$'' is assigned to the ``index element twisted by $\ca{L}$''.
\end{cor}

These results are obviously a part of the $LT$-equivariant index theorem. 
Moreover, they suggest that the $C^*$-algebra of a Hilbert manifold carries some topological information of the Hilbert manifold. For a locally compact Hausdorff space, such a result is in some sense automatic, because the category of locally compact Hausdorff spaces is equivalent to that of commutative $C^*$-algebras. On the other hand, in a non-locally compact setting, such results are highly non-trivial and quite interesting.

\vspace{0.5cm}

Let us explain possible applications of the $LT$-equivariant index theorem and the ideas used in the present paper (we have not completed the index theorem for proper $LT$-spaces, and we will summarize what we should do after the present paper in order to compete it in Section \ref{section unsolved}).

We begin with the quantization problem of Hamiltonian loop group spaces \cite{AMM,Mei,Son,LMS,LS}. A Hamiltonian loop group space is an infinite-dimensional symplectic manifold equipped with a loop group action and a proper moment map taking values in the dual Lie algebra of the loop group, where the loop group action on it is the gauge action. Since the moment map is equivariant and proper, and since the gauge action is proper and cocompact, the loop group action on a Hamiltonian loop group space is automatically proper and cocompact. 
Proper $LT$-spaces are obvious generalizations of Hamiltonian loop group spaces.

In general, a quantization is a procedure making a Hilbert space from a classical system. The quantization of a Hamiltonian $G$-space given by Bott is the equivariant index of the $Spin^c$-Dirac operator twisted by the pre-quantum line bundle. In \cite{Son}, this procedure was generalized to Hamiltonian $LG$-spaces by a formal and infinite-dimensional argument, using the one-to-one correspondence between Hamiltonian $LG$-spaces and quasi-Hamiltonian $G$-spaces given in \cite{AMM}. The loop group equivariant index theory will give a $KK$-theoretical description of this result, and the result of the present paper will enable us to compute the quantization in the topological language.

Another application is a noncommutative geometrical reformulation of the Freed-Hopkins-Telemen isomorphism (FHT isomorphism for short) \cite{FHTI,FHTII,FHTIII}. This result states that the Grothendieck completion of the semigroup consisting of isomorphism classes of positive energy representations of the loop group of $G$ (so called the Verlinde ring), is isomorphic to the twisted equivariant $K$-group of $G$ with respect to the conjugation action. This isomorphism is constructed as follows: For a given positive energy representation, the authors defined an $LG$-equivariant continuous field of Fredholm operators parameterized by $L\fra{g}^*$; Since the groupoid $L\fra{g}^*// LG$ is locally equivalent to $G// G$, they can pushforward the field to $G// G$; Since a positive energy representation is projective, the obtained field defines an element of the twisted $K$-group. Since the twisted $K$-theory of $L\fra{g}^*// LG$ should be isomorphic to the ``twisted $LG$-equivariant $K$-homology of $L\fra{g}^*$'' by the Poincar\'e duality, the FHT isomorphism might be related to the $LG$-equivariant index theory. See \cite{Loi} for a research on this line.



Finally, we add a comment on the Witten genus. The Witten genus of a compact manifold $M$ is defined ``using the fixed-point formula for the $S^1$-equivariant topological index of its free loop space $LM$'' \cite{Wit}. This theory is very interesting, but the definition is not quite satisfying, in that a Dirac operator whose index is the Witten genus has not been given. 

Although there are no direct connections between the Witten genus and the $LT$-equivariant index theory, the observations of the present paper can be quite useful for a noncommutative geometrical formulation of the Witten genus. 
In fact, if an $S^1$-equivariant index can be defined in an infinite-dimensional way, it should be realized as a homomorphism $\chi$ from the $S^1$-equivariant $K$-homology of $LM$ to $R(S^1)^*=\Hom(R(S^1),\bb{Z})$ or something, where $R(S^1)^*$ can be identified with the set of formal power series $\sum_{n\in\bb{Z}}a_nz^n$. Although $K$-homology theory does not work for non-locally compact spaces, we can define a substitute for it in the language of noncommutative geometry with the $C^*$-algebra of a Hilbert manifold. Moreover, by the Poincar\'e duality homomorphism constructed in the present paper, it might be possible to give a topological formula of $\chi$. If this formula coincides with the definition of the Witten genus, the index homomorphism $\chi$ will realize the Witten genus. This realization has an advantage compared with the original one, in that $\chi$ is defined as a homomorphism from an invariant of not $M$ but $LM$. It might be possible to use some global structure of $LM$.

\vspace{0.5cm}

From now on, we explain a few backgrounds of the present paper. Since we have explained possible applications, we focus on technical sides.

One of the main objects of the present paper is a $C^*$-algebra of a Hilbert manifold constructed by Guoliang Yu. It is a wide generalization of the construction of \cite{HKT}. In this paper, the authors defined a $C^*$-algebra for a Hilbert space using a finite-dimensional approximation and an inductive limit. In order to construct it, ``the suspension of Clifford algebra-valued function algebra'' was appropriate, as we have mentioned. This construction was used to prove the Baum-Connes conjecture for a-T-menable groups in \cite{HK}, and it was generalized to Hilbert bundles and Fredholm manifolds in \cite{Tro,DT}. Several years ago, the construction was reformulated and generalized to Hilbert-Hadamard spaces in \cite{GWY}. Guoliang Yu generalized this construction to Hilbert manifolds with positive injectivity radius, and we will use it in the present paper.

These two kinds of constructions have different advantages. The constructions of \cite{HKT,HK,Tro,DT} are rather algebraic, and hence they are convenient to compute several invariants. The construction of \cite{GWY,Yu} are rather geometric, and hence they look more natural and they are convenient to define group actions. In the present paper, we use both constructions.

There are many results on substitutes for ``$L^2$-spaces'' and ``Dirac operators'' on infinite-dimensional spaces related to loop groups, for example \cite{FHTII,Lan,Son,Was}. These results are based on representation theoretical ideas: The $L^2$-space on a compact group can be written using its representation theory by the Peter-Weyl theorem, and direction derivatives are written as infinitesimal representations. See also \cite{Kos} for algebraic Dirac operators.

In \cite{HKT,HK}, the authors constructed the ``$L^2$-space on a Hilbert space'' using the creation/annihilation operators. In \cite{T1}, we found out that it is strongly related to the above representation theoretical construction: The Peter-Weyl type construction for the infinite-dimensional Heisenberg group at fixed level gives a similar construction to \cite{HKT,HK}.
We will use it in the body of the present paper. 

There are probabilistic approaches using the Wiener measure. However, we would like to consider a Hilbert space like the ``$L^2$-space defined by the Lebesgue measure''. Thus, we do not discuss such a theory in the following.

An analytic index over a geometrical object $X$ should be a group homomorphism from the ``$K$-homology of $X$''. In this sense, the construction of an analytic index theory for $X$ is {\it formally} equivalent to the construction of a $C^*$-algebra substituting for a ``function algebra of $X$'' and a group homomorphism from the $K$-homology of this $C^*$-algebra. However, the {\it content} of this theory is the value of the analytic index homomorphism at a certain $K$-homology element. Thus, the explicit construction of an element of the $K$-homology group is quite important.
In order to define such a $K$-homology element for an infinite-dimensional space, one needs to define substitutes for an ``$L^2$-space'' and a ``Dirac operator'', {\it which is compatible with the ``function algebra''}. The Dirac element of \cite{HK}, and the index element and the Dirac element of \cite{T4} are such constructions. The constructions of \cite{T4} are less general, but much more direct, than that of \cite{HK}. We will reformulate the index element of \cite{T4} and we will use it in the present paper.

\vspace{0.5cm}

Let us explain the organization of the present paper. 

In Section \ref{section preliminaries}, we will review several standard facts on $KK$-theory. In particular, we will compare a classical formulation of the Bott periodicity and the reformulation using the Bott homomorphism explained in \cite[Section 2]{HKT}. This comparison will be necessary to generalize the local Bott element to infinite-dimensional manifolds.

In Section \ref{section index theorem}, we will describe the statement of the Kasparov index theorem in detail, and we will reformulate it for even-dimensional $Spin^c$-manifolds. Since $C_0(X)$ is $KK$-equivalent to $Cl_\tau(X)$ for an even-dimensional $Spin^c$-manifold $X$, we can remove $C_0(X)$ from the description of the index theory for such $X$. For the same reason, we can replace the right hand side of the Poincar\'e duality homomorphism with $\ca{R}KK_G(X;C_0(X),C_0(X))\cong RK^0_G(X)$, which is easy to generalize to proper $LT$-spaces. In addition, we will give an explicit definition of the topological assembly map, and the description of it under the above reformulation.

In Section \ref{section problem}, we will start the study of infinite-dimensional spaces. First, we will precisely set up the problem. The remainder of this section will be devoted to explain this setting in detail.
Second, we will review the representation theory of $LT$, and we will recall the substitute for the ``$\tau$-twisted group $C^*$-algebra of $LT$'' from \cite{T1}. Third, we will introduce $\ca{R}KK$-theory for non-locally compact action groupoids. The detailed study will be done in \cite{NT}, and we will just give definitions and a few necessary results. Finally, we will re-introduce loop group equivariant $KK$-theory, which was introduced in \cite{T4} in a naive way, and the detailed studies were put off. Thanks to the better definition given there, we will be able to prove desired properties on Kasparov products on this theory. With this collect $LT$-equivariant $KK$-theory, we will describe the results of the present paper.

In Section \ref{section PD}, we will give the definition of a $C^*$-algebra of a Hilbert manifold, following \cite{Yu}, and we will define an infinite-dimensional analogue of the Poincar\'e duality homomorphism. In order to indicate that this construction is appropriate, we apply it to proper $LT$-spaces. For this aim, we will reformulate the index element of \cite{T4} using the $C^*$-algebra of a Hilbert manifold given in \cite{Yu}. This reconstruction will contain a detailed study of this $C^*$-algebra. Then, we will compute the Poincar\'e duality homomorphism for this index element.

In Section \ref{section top ass map}, we will construct an infinite-dimensional analogue of the topological assembly map. The key of the construction will be the fixed-point algebra description of the crossed products and the descent homomorphisms obtained in Section \ref{section index theorem}. We will compute the topological assembly map for proper $LT$-spaces.
Finally, we will explain what we should do after the present paper, in order to complete the $LT$-equivariant index theory.

\subsection*{Notations}

\begin{itemize}
\setlength{\parskip}{0cm} 
  \setlength{\itemsep}{0cm} 
\item For a complex Hilbert space $V$, a complex bilinear form is denoted by angle brackets $\innpro{\bullet}{\bullet}{}$, and an inner product is denoted by parentheses $\inpr{\bullet}{\bullet}{}$. More generally, for a $C^*$-algebra $A$, an $A$-valued inner product on a Hilbert $A$-module $E$, is denoted by $\inpr{\bullet}{\bullet}{A}$ or $\inpr{\bullet}{\bullet}{E}$.
\item For a $C^*$-algebra $B$ and a Hilbert $B$-module $E$, the set of compact operators on $E$ is denoted by $\bb{K}_B(E)$ or simply by $\bb{K}(E)$, and that of bounded operators is denoted  by $\bb{L}_B(E)$ or simply by $\bb{L}(E)$. Associated to it, for a Hilbert space $V$, $\bb{K}(V)=\bb{K}_{\bb{C}}(V)$ denotes the set of all compact operators and $\bb{L}(V)=\bb{L}_{\bb{C}}(V)$ denotes the set of all bounded operators.
\item A group action on a set $X$ (e.g. a manifold, a vector space, a $C^*$-algebra) is often denoted by the common symbols $\alpha$. The automorphism corresponding to $g\in G$ is denoted by $\alpha_g$. When we need to emphasize $X$, we denote it by $\alpha^X_g$. 
\item A $G$-action induced by left translation on a function algebra on $G$, is often denoted by ``$\lt$'', namely $\lt_g(f)(x):=f(g^{-1}x)$. Similarly, $\rt_g(f)(x):=f(xg)$.
\item The left and right regular representations are denoted by $L$ and $R$, respectively.
\item $*$-homomorphisms to define left module structures on  Hilbert modules, are often denoted by the common symbol $\pi$ when it is not confusing. Associated to it, the $*$-homomorphism $A\rtimes G\to \bb{L}(E\rtimes G)$ induced by the descent homomorphism by $G$ from $\pi:A\to \bb{L}(E)$, is denoted by $\pi\rtimes \lt$.
\item For a Lie group $G$, its Lie algebra is denoted by $\Lie(G)$ or $\fra{g}$.
\item The symbol $\grotimes$ means the graded tensor product of Hilbert spaces, $C^*$-algebras or Hilbert modules. We often use this symbol even for trivially graded objects.
\item The algebraic tensor product is denoted by $\grotimes^\alg$, where for Hilbert spaces, $C^*$-algebras or Hilbert modules $V_1$ and $V_2$, $V_1\grotimes V_2$ stands for $\bbra{\sum_{\rm finite} v_1^i\grotimes v_2^i\mid v_j^i\in V_j}$. $V_1\grotimes V_2$ is the completion of $V_1\grotimes^\alg V_2$ with respect to an appropriate norm.
\item For a $\bb{Z}_2$-graded vector space $V=V_0\widehat{\oplus}V_1$, we denote the grading by $\partial$ and the graded homomorphism by $\epsilon$, that is to say, for $v\in V_i$ ($i=0$ or  $1$), $\partial v:=i$ and $\epsilon v=(-1)^iv$.
\item We always use graded commutators: For a $\bb{Z}_2$-graded vector space $V$ and two even or odd operators $F_1,F_2\in \End(V)$, we define $[F_1,F_2]:=F_1F_2-(-1)^{\partial F_1\partial F_2}F_2F_1$.
\item For a group $G$ equipped with a $U(1)$-central extension $1\to U(1)\xrightarrow{i}G^\tau\xrightarrow{p}G\to 1$, a $\tau$-twisted $G$-representation on a vector space $V$ is a group homomorphism $\rho:G^\tau\to U(V)$ so that $\rho\circ i(z)=z\id_V$ for $z\in U(1)$. Similarly, we define the concept of $\tau$-twisted $G$-actions on Hilbert modules or vector bundles.
\item We frequently use some objects which are substitutes for something that we can not define in a classical way. Such substitutes are denoted by the standard symbol with an underline.
For example, the Hilbert space substituting for the ``$L^2$-space of an infinite-dimensional manifold $\ca{M}$'' is denoted by $\ud{L^2(\ca{M})}$.
\item Infinite-dimensional objects are written in {\tt \textbackslash mathcal} (e.g. an infinite-dimensional manifolds are written as $\ca{X},\ca{M},\cdots$), and fields are written in {\tt \textbackslash mathscr} (e.g. continuous fields of Banach spaces are written as $\scr{A},\scr{B},\cdots$). The converse is not true: We use $\ca{S}$ to denote the $C^*$-algebra $C_0(\bb{R})$ equipped with a $\bb{Z}_2$-grading homomorphism $\epsilon f(t):=f(-t)$.
\item $\vep$ and $\delta$ are used to describe Assumption \ref{assumption related to bounded geometry}. Thus, we use $\epsilon$ and $\Delta$ as small positive real numbers. $\epsilon$ is often used as a grading homomorphism, but we believe it is not confusing.
\item All the locally compact group appearing in the present paper are assumed to be amenable for simplicity. Thus, the crossed products are uniquely determined.
\end{itemize}



\section{Preliminaries}\label{section preliminaries}


\subsection{Several concrete $C^*$-algebras}

In this subsection, we define several basic $C^*$-algebras appearing in the present paper.

\begin{dfn}
$(1)$ For a Euclidean space $V$, we define the {\bf Clifford algebra of $V$} by $\Cl_\pm(V):=T_{\bb{C}}(V)/\left< v^2\mp |v|^2\cdot 1\right>$, where $T_{\bb{C}}(V)$ is the tensor algebra over $\bb{C}$. They inherit the $\bb{Z}_2$-grading from $T_{\bb{C}}(V)$, because $\left< v^2\mp |v|^2\cdot 1\right>$ is contained in the even-part.

$(2)$ The metric of $V$ induces metrics on $\Cl_\pm(V)$. By left multiplications $\Cl_\pm(V)\to\End(\Cl_\pm(V))$, and the $\bb{Z}_2$-graded $C^*$-algebra structures on $\End(\Cl_\pm(V))$, we define a $\bb{Z}_2$-graded $C^*$-algebra structures on both of $\Cl_\pm(V)$. Note that $v^*=\pm v$ in $\Cl_\pm(V)$ for each $v\in V$.

$(3)$ A $\bb{Z}_2$-graded Hermite vector space $S$ is called a {\bf Clifford module of $V$}, if it is equipped with an even $*$-homomorphism $c:\Cl_-(V)\to \End(S)$. A Clifford module of $V$ is called a {\bf Spinor} if it is an irreducible representation of $\Cl_-(V)$. 
\end{dfn}

It is known that an even-dimensional vector space $V$ has two different Spinors up to equivalence \cite{Fur}. We fix one of them, and we denote it by $(S,c)$. The Clifford multiplication $c$ gives an isomorphism $\Cl_-(V)\cong \End(S)$ as $\bb{Z}_2$-graded $C^*$-algebras. Associated to it, we obtain the following.

\begin{lem}\label{Definitions in preliminaries, Clifford algebras}
$(1)$ The dual space $S^*=\Hom(S,\bb{C})$ is a $*$-representation space of $\Cl_+(V)$ by $c^*(v)\cdot f:=(-1)^{\partial f}f\circ c(v)$. It gives an isomorphism $c^*:\Cl_+(V)\cong \End(S^*)$.

$(2)$ We can define a right Hilbert $\Cl_+(V)$-module structure on $S$, by identifying $S$ with $S^{**}$. In fact, the operations are given by $s\cdot v:=s\circ c^*(v)$ and $\inpr{s_1}{s_2}{\Cl_+}:=s_1^{*}\grotimes s_2\in \End(S^*)\cong \Cl_+(V)$ for $s,s_1,s_2\in S$ and $v\in V$. 
\end{lem}
\begin{ex}
Let $V$ be a $2$-dimensional Euclidean space equipped with a complex structure $J$. It defines a complex base $\{z,\overline{z}\}$ of $V\otimes \bb{C}$. Let $S$ be $\bb{C}\widehat{\oplus}\bb{C}$ and we define $c:V\otimes \bb{C}\to \End(S)$ by
$$c(z):=\begin{pmatrix} 0 & -\sqrt{2} \\ 0 & 0\end{pmatrix} ,\ 
c(\overline{z}):=\begin{pmatrix} 0 & 0 \\ \sqrt{2} & 0\end{pmatrix} .$$
One can easily check that the restriction of $c$ to $V$ extends to $\Cl_-(V)$. The tensor product of copies of this example plays an important role in the present paper.
\end{ex}

There is a family version of these constructions.
Let $X$ be a Riemannian manifold. A pair of a Hermitian vector bundle $S$ and a bundle homomorphism $c:\Cl_+(TX)\to \End(S)$ preserving the $\bb{Z}_2$-grading is called a {\bf Spinor bundle $S$ over $X$}.
If $X$ admits such a bundle, $X$ is called a $Spin^c$-manifold.
For example, an almost complex manifold is $Spin^c$, whose Spinor bundle is given by the same construction of the above example.

We define a $C^*$-algebra $Cl_\tau(X):=C_0(X,\Cl_+(TX))$. 
The following is a family version of Lemma \ref{Definitions in preliminaries, Clifford algebras}.

\begin{lem}\label{Lemma S is a Hilbert Cltau module}
Let $(S,c)$ be a Spinor bundle over $X$. By the family version of Lemma \ref{Definitions in preliminaries, Clifford algebras}, $C_0(X,S)$ admits a Hilbert $Cl_\tau(X)$-module structure. Moreover, a Hilbert $C_0(X)$-module $C_0(X,S^*)$ admits a left $Cl_\tau(X)$-module structure $c^*:Cl_\tau(X)\to \bb{L}_{C_0(X)}(C_0(X,S^*))$.
\end{lem}

The following easy $\bb{Z}_2$-graded $C^*$-algebras are related to almost all objects appearing in the present paper.

\begin{dfn}
$(1)$ We define a $\bb{Z}_2$-graded $C^*$-algebra $\ca{S}$ by $C_0(\bb{R})$ with the grading homomorphism $\epsilon f(t):=f(-t)$. In other words, an even and odd function are an even and odd element, respectively.
We define an unbounded multiplier $X$ on $\ca{S}$ by $Xf(t):=tf(t)$.

$(2)$ Similarly, we define a $\bb{Z}_2$-graded $C^*$-algebra $\ca{S}_\vep$ for $\vep>0$ by $C_0(-\vep,\vep)$ with the grading homomorphism $\epsilon f(t):=f(-t)$. 
We define a multiplier $X$ on $\ca{S}_\vep$ by $Xf(t):=tf(t)$.
\end{dfn}

\subsection{Equivariant $KK$-theory and equivariant $\ca{R}KK$-theory}\label{section KK and RKK for finite dimension}

We use the unbounded picture of equivariant $KK$-theory. We refer to \cite{Bla,Kas88} for the bounded picture.
For the following definition, see \cite{T4}.

\begin{dfn}\label{dfn in preparation unbounded picture of KK}
Let $G$ be a locally compact second countable Hausdorff group, and let $A$, $B$ be separable $G$-$C^*$-algebras. An {\bf unbounded $G$-equivariant Kasparov $(A,B)$-module} is a triple $(E,\pi,D)$ satisfying the following condition:
\begin{itemize}
\item $E$ is a countably generated $\bb{Z}_2$-graded $G$-equivariant Hilbert $B$-module;
\item $\pi:A\to \bb{L}_B(E)$ is an even $G$-equivariant $*$-homomorphism;
\item $D:E\supseteq \dom(D)\to E$ is a densely defined odd regular self adjoint operator satisfying the following conditions:
\begin{itemize}
\item There exists a $*$-subalgebra $A'\subseteq A$ such that every $a\in A'$ satisfies the following conditions: $\pi(a)$ preserves $\dom(D)$ and $[\pi(a),D]$ is bounded;
\item $\pi(a)(1+D^2)^{-1}\in \bb{K}_B(E)$;
\item $G$ preserves $\dom(D)$. The map $G\ni g\mapsto g(D)-D$  is $\bb{L}_B(E)$-valued, and it is norm-continuous.
\end{itemize}
\end{itemize}
When $D$ satisfies $g(D)=D$, the Kasparov module $(E,\pi,D)$ is said to be {\bf actually equivariant}.
\end{dfn}

The central tool of $KK$-theory is the Kasparov product. There is a criterion to be a Kasparov product in the unbounded picture \cite{Kuc}. For bimodules $E_1,E_2$ and $e\in E_1$, we define $T_{e_1}:E_2\to E_1\grotimes E_2$ by $e_2\mapsto e_1\grotimes e_2$.

\begin{pro}[\cite{Kuc}]\label{Kucs criterion}
Let $G$ be a locally compact second countable Hausdorff group, and let $A$, $A_1$, $B$ be separable $G$-$C^*$-algebras. 
Let $(E_1,\pi_1,D_1)$ be an unbounded $G$-equivariant Kasparov $(A,A_1)$-module, and $(E_2,\pi_2,D_2)$ be an unbounded $G$-equivariant Kasparov $(A_1,B)$-module. $x$ and $y$ are the corresponding $KK$-elements of $(E_1,\pi_1,D_1)$ and $(E_2,\pi_2,D_2)$, respectively.
A $G$-equivariant Kasparov $(A,B)$-module $(E_1\grotimes E_2,D)$ is a representative of the Kasparov product of $x$ and $y$, if the following conditions are fulfilled:
\begin{itemize}
\item For all $v$ in some dense subset of $AE_1$, the graded commutator
$$\bbbra{
\begin{pmatrix} D & 0 \\ 0 & D_2 \end{pmatrix},
\begin{pmatrix} 0 & T_v \\ T_v^* & 0 \end{pmatrix}
}$$
is bounded on $\dom(D\oplus D_2)$. We call it the {\bf connection condition}.
\item $\dom(D)$ is contained in $\dom(D_1\grotimes \id)$.
\item There exists $\kappa\in\bb{R}$ such that $\inpr{\id\grotimes D_1(e)}{D(e)}{}+\inpr{D(e)}{\id\grotimes D_1(e)}{}\geq \kappa\inpr{e}{e}{}$ for any $e\in \dom(D)$. We call it the {\bf positivity condition}.
\end{itemize}
\end{pro}

We recall $\ca{R}KK$-theory which is a generalization of $KK$-theory. A generalization of it to infinite-dimensional manifolds is the central object of the present paper.

\begin{dfn}
Let $G$ be a locally compact second countable Hausdorff group and let $X$ be a $\sigma$-compact locally compact Hausdorff space equipped with a continuous $G$-action. 

$(1)$ A $C^*$-algebra $A$ is a {\bf $C_0(X)$-$C^*$-algebra} if it admits a $*$-homomorphism $\phi:C_0(X)\to Z(M(A))$ satisfying $\phi(C_0(X))A=A$. It is also called an {\bf $X\rtimes \{e\}$-$C^*$-algebra}. In the following, we regard this homomorphism as a left module structure and $\phi(f)a$ is denoted by $f\cdot a$.

$(2)$ In addition, if $A$ is a $G$-$C^*$-algebra and $\phi:C_0(X)\to Z(M(A))$ is $G$-equivariant, $A$ is called a {\bf $G$-$C_0(X)$-$C^*$-algebra}. It is also called an {\bf $X\rtimes G$-$C^*$-algebra} or an {\bf $X\rtimes G$-equivariant $C^*$-algebra}.
\end{dfn}
\begin{rmks}
$(1)$ The above concept is generalized to {\bf groupoid equivariant $C^*$-algebras} in \cite{LG}. Our terminology ``$X\rtimes G$-equivariant $C^*$-algebra'' comes from it. In fact, a continuous $G$-action on $X$ defines a topological groupoid $X\rtimes G$.

$(2)$ When we need to emphasize a $C^*$-algebra $A$ does not have a $C_0(X)$-structure (or we ignore it), we call $A$ a {\bf single $C^*$-algebra}. This is because we regard a $C_0(X)$-algebra as a field of $C^*$-algebras parameterized by $X$ in the present paper.
\end{rmks}

\begin{dfn}[{\cite[Definition 2.19]{Kas88}}]
Let $G$ be a locally compact second countable Hausdorff group and let $X$ be a $\sigma$-compact locally compact Hausdorff space equipped with a continuous $G$-action. Let $A,B$ be $X\rtimes G$-$C^*$-algebras. A $G$-equivariant Kasparov $(A,B)$-module $(E,\pi,F)$ is said to be {\bf $X\rtimes G$-equivariant} if it satisfies the condition
$$\pi(f\cdot a)(e\cdot b)=\pi(a)(e\cdot (f\cdot b)).$$
A homotopy between two $X\rtimes G$-equivariant Kasparov $(A,B)$-modules is an $X\rtimes G$-equivariant Kasparov $(A,BI)$-module by the obvious $X\rtimes G$-algebra structure on $BI$. The set of homotopy classes of $X\rtimes G$-equivariant Kasparov $(A,BI)$-modules is denoted by $\ca{R}KK_G(X;A,B)$.
\end{dfn}

We can define the concept of Kasparov products on $\ca{R}KK$-theory \cite[Proposition 2.21]{Kas88}. We can deal with $\ca{R}KK$-theory in the unbounded picture in an obvious way.
We omit the details. 

We recall several fundamental and necessary constructions, emphasizing the relationship between $KK$-theory and $\ca{R}KK$-theory.

\begin{dfn}
Let $G$ be a locally compact second countable Hausdorff group and let $X$ be a $\sigma$-compact locally compact Hausdorff space equipped with a continuous $G$-action. Let $A,B$ be $X\rtimes G$-equivariant $C^*$-algebras.

$(1)$ We define the {\bf forgetful homomorphism}
$$\fgt:\ca{R}KK_G(X;A,B)\to KK_G(A,B)$$
by forgetting the $C_0(X)$-module structure.

$(2)$ Let $D,D_1,D_2$ be $G$-equivariant $C^*$-algebras. We define
$$\sigma_D:\ca{R}KK_G(X;A,B)\to \ca{R}KK_G(X;D\grotimes {A},D\grotimes {B})$$
by $(E,\pi,F)\mapsto (D\grotimes E,\id\grotimes \pi, \id\grotimes F)$, and we define
$$\sigma_{X,A}:KK_G(D_1,D_2)\to \ca{R}KK_G(X;D_1\grotimes A,D_2\grotimes A)$$
by $(E,\pi,F)\mapsto ( E\grotimes A,\pi\grotimes \id,  F\grotimes \id)$.
\end{dfn}
\begin{rmks}
$(1)$ When we define a $C_0(X)$-algebra structure on the tensor product of two $C_0(X)$-algebras $A$ and $B$ by that of $A$, we denote it by $\uwave{A}\grotimes  B$ with a wave underline. For example, we will encounter both of $C_0(X)\grotimes \uwave{Cl_\tau(X)}$ and $\uwave{C_0(X)}\grotimes Cl_\tau(X)$ to define local Bott elements.

$(2)$ We can also define $\sigma_D:\ca{R}KK_G(X;A,B)\to \ca{R}KK_G(X;A\grotimes {D},B\grotimes {D})$ in an obvious way. When we distinguish these two homomorphisms, we use the symbol $\sigma_D^i$, where $D$ is tensored at the $i$-th factor, although there are no differences between them.

$(3)$ All the single $C^*$-algebras which we will encounter in the present paper are nuclear, and the tensor products are uniquely determined.
\end{rmks}

At first sight, $\ca{R}KK$-theory is almost the same with $KK$-theory. However, $\ca{R}KK$-theory has a strong topological flavor. Roughly speaking, an element of $\ca{R}KK_G(X;A,B)$ is a ``$G$-equivariant family of Kasparov modules parameterized by $X$''. We briefly explain this idea. With this observation, we will formulate a non-locally compact version of $\ca{R}KK$-theory in Section \ref{section RKK}.

Let $B$ be a $C_0(X)$-$C^*$-algebra. Then, we can define a $C^*$-algebra $B_x$ by $B/C_xB$, where $C_x$ is the ideal of $C_0(X)$ consisting of $f\in C_0(X)$ such that $f(x)=0$. Thus, $B$ is, roughly speaking, the ``continuous section algebra'' of this field. The field given in this say is called an {\bf upper semi-continuous field of $C^*$-algebras over $X$} (u.s.c. field for short). Conversely, a u.s.c. field of $C^*$-algebras over $X$ defines a $C_0(X)$-$C^*$-algebra, and these two operations are mutually inverse.
See \cite{Nil,Blan} for details on this result.

If $E$ is a Hilbert $B$-module, we can define the fiber of $E$ at $x$ by $E_x:=E/EB_x$. Then, a $B$-module homomorphism $F$ on $E$ defines a $B_x$-module homomorphism on $E_x$, because $F$ preserves $EB_x$. It is denoted by $F_x$.

Let $A$ be a $C_0(X)$-$C^*$-algebra. Suppose that a homomorphism $\pi:A\to \bb{L}_B(E)$ satisfying $\pi(fa)(eb)=\pi(a)(e(fb))$ is given. Then, the induced operator $\pi(a)_x$ is determined by $a_x$. In fact, if $a_x=0$, we can approximate it by finite sums of $fa'$ for $f\in C_x$ and $a'\in A$. Since $\pi(fa')(eb)=\pi(a')(e)\cdot fb\in EB_x$, $\pi(a)(E)\subseteq EB_x$, and hence $\pi(a)_x=0$. Consequently, $\pi$ determines a family of $*$-homomorphisms $\{\pi_x:A_x\to \bb{L}_{B_x}(E_x)\}_{x\in X}$. 

Combining the preceding three paragraphs, we find that an $X\rtimes \{e\}$-equivariant Kasparov module $(E,\pi,F)$ defines a ``continuous family'' $\{(E_x,\pi_x,F_x)\}_{x\in X}$, where $E_x$ is a Hilbert $B_x$-module, $\pi_x:A_x\to \bb{L}_{B_x}(E_x)$ is a $*$-homomorphism, and $F_x\in \bb{L}_{B_x}(E_x)$ satisfies the condition to be a Kasparov module. It is obvious that we can do the same constructions for equivariant situations.

In order to describe $\ca{R}KK$-element purely in the language of fields of Kasparov modules, that is to say, in order to obtain an $\ca{R}KK$-element from a field of Kasparov modules, we need to define the concept of (u.s.c.) fields of adjointable operators and compact operators in the language of fields of modules (strictly speaking, if we define the field of adjointable operators in a natural way, it is not upper semi-continuous). We postpone this problem until Section \ref{section RKK}.  Although we will use the terminology ``continuous field of Kasparov modules'' in Section \ref{section index theorem} to define several basic $\ca{R}KK$-elements, this terminology stands for ``the field defined from a Kasparov module'' in the above way.

\subsection{Bott periodicity and the $C^*$-algebra $\ca{S}$}\label{appendix S}

We prepare the Bott periodicity theorem and its reformulation using a $*$-homomorphism given in  \cite{HKT}. We will use it in order to generalize the ``local Bott element'' to infinite-dimensional manifolds.

\begin{dfn}\label{dfn Bott period by Dirac and dual Dirac}
Let $V$ be a finite-dimensional Euclidean space. $O(V)$ denotes the compact Lie group consisting of orthogonal transformation on $V$.

$(1)$ We define the {\bf Bott element} $[b_V]\in KK_{O(V)}(\bb{C},Cl_\tau(V))$ by the bounded transformation of
$$\bra{Cl_\tau(V),1,C},$$
where $1:\bb{C}\to \scr{C}(E)\subseteq \bb{L}_{\scr{C}(E)}(\scr{C}(E))$ is the $*$-homomorphism given by $z\mapsto z\id$, and 
$C$ is an unbounded multiplier on $Cl_\tau(V)$ defined by the pointwise Clifford multiplication: For $h\in Cl_\tau(V)$ and $v\in V$, we define $Ch(v):=v\cdot h(v)$. Note that $C$ is $O(V)$-invariant and the above $KK$-element is $O(V)$-equivariant.

$(2)$ We define the {\bf Dirac element} $[d_V]\in KK_{O(V)}(Cl_\tau(V),\bb{C})$ by the bounded transformation of
$$\bra{L^2(V,\Cl_+(TV)),\pi,D},$$
where $\pi:Cl_\tau(V) \to \bb{L}_{\bb{C}}(L^2(V,\Cl_+(TV)))$ is the $*$-homomorphism given by left multiplication, and $D$ is the Dirac operator given by 
$$\sum_i\widehat{e_i}\frac{\partial}{\partial x_i}$$
for an orthonormal base $\{e_i\}$ and the corresponding coordinate $\{x_i\}$, where $\widehat{e_i}v:=(-1)^{\partial v}v\cdot e_i$. Note that $D$ is $O(V)$-invariant and the above $KK$-element is $O(V)$-equivariant.
\end{dfn}

\begin{lem}\label{lem proof of bott period classical}
$[b_V]\grotimes_{Cl_\tau(V)}[d_V]={\bf 1}_{\bb{C}}$ and $[d_V]\grotimes_{\bb{C}}[b_V]={\bf 1}_{Cl_\tau(V)}$. Consequently, $Cl_\tau(V)$ is $KK_{O(V)}$-equivalent to $\bb{C}$. This result is called the {\bf Bott periodicity theorem}.
\end{lem}
\begin{proof}
One can check that the unbounded Kasparov module 
$$(L^2(V,\Cl_+(TV)),1,D+C)$$
satisfies the conditions to be a Kasparov product of $[b_V]$ and $[d_V]$. By the spectral theory of the harmonic oscillator, we notice that $\ker(D+C)$ is spanned by $e^{-\frac{1}{2}\|v\|^2}\cdot 1$, where ``$1$'' is the multiplicative identity of the Clifford algebra. Since it is rotation invariant, the Kasparov product is ${\bf 1}_{\bb{C}}$ in the sense of equivariant $KK$-theory.

In order to prove the opposite direction, we use the well-known rotation trick. We refer to \cite{Ati68,HKT}.
Let ``$\flip$'' be the automorphism on $Cl_\tau(V)\grotimes Cl_\tau(V)$ defined by $f_1\grotimes f_2\mapsto (-1)^{\partial f_1\partial f_2}f_2\grotimes f_1$. It is homotopic to $j^*\grotimes \id$ by the rotation homotopy, where $j:V\to V$ is defined by $x\mapsto -x$. Thus,
\begin{align*}
[d_V]\grotimes_\bb{C}[b_V]&=
\bra{{\bf 1}_{Cl_\tau(V)}\grotimes_\bb{C}[b_V]}\grotimes_{Cl_\tau(V)\grotimes Cl_\tau(V)}\bbra{[\flip]\grotimes_{Cl_\tau(V)\grotimes Cl_\tau(V)}\bra{{\bf 1}_{Cl_\tau(V)}\grotimes_\bb{C}[d_V]}}\\
&= \bra{{\bf 1}_{Cl_\tau(V)}\grotimes_\bb{C}[b_V]}\grotimes_{Cl_\tau(V)\grotimes Cl_\tau(V)}\bra{[j^*]\grotimes_\bb{C}[d_V]}\\
&= [j^*]\grotimes_\bb{C} \bra{[b_V]\grotimes_{Cl_\tau(V)}[d_V]}=[j^*].
\end{align*}
Since $[j^*]\grotimes_{Cl_\tau(V)} [j^*]={\bf 1}_{Cl_\tau(V)}$, we have the formula $[d_V]\grotimes_\bb{C}[b_V]\grotimes_{Cl_\tau(V)} [j^*]={\bf 1}_{Cl_\tau(V)}$. Since the left inverse coincides with the right inverse, we obtain the result.
\end{proof}

We now explain another proof of the Bott periodicity following \cite{HKT}, and we prove the relationship between these two proofs. We define $\ca{A}(V)$ by the graded tensor product $\ca{S}\grotimes Cl_\tau(V)$.

\begin{dfn}\label{dfn of the Bott hom}
The Bott homomorphism $\beta:\ca{S}\to \ca{A}(V)$ is defined by the functional calculus
$$f\mapsto f(X\grotimes \id+\id\grotimes C).$$
It defines a $KK$-element $[\beta]\in KK_{O(V)}(\ca{S},\ca{A}(V))$.
\end{dfn}

We will prove that $[\beta]=\sigma_{\ca{S}}([b_V])$ for even-dimensional $V$. Before that, we prove that $\ca{S}$ is $KK$-equivalent to the trivially graded $C^*$-algebra $\bb{C}^2$. More concretely, we will construct four $KK$-elements $[d_\pm]\in KK(\ca{S}, \bb{C})$ and $[b_\pm]\in KK(\bb{C},\ca{S})$, and we will prove that $[d_{\ca{S}}]:=[d_+]\oplus [d_-]\in KK(\ca{S},\bb{C}\oplus \bb{C})$ and $[b_{\ca{S}}]:=[b_+]\oplus [b_-]\in KK(\bb{C}\oplus \bb{C},\ca{S})$ are mutually inverse, where these $\oplus$'s are defined by the additivity of $KK$-theory with respect to both variables. For this aim, it suffices to prove that $[b_\pm]\grotimes [d_\pm]={\bf 1}_{\bb{C}}$, and $[b_\pm]\grotimes [d_\mp]=0$ (double signs are in the same order) and $[d_{\ca{S}}]\grotimes_{\bb{C}^2}[b_{\ca{S}}]={\bf 1}_{\ca{S}}$. We will prove the former two statements directly, and prove the last one by an easy algebraic argument.

Let us introduce several notations. On function spaces  ($C_0$ or $L^2$) on $\bb{R}$, we define $\bb{Z}_2$-grading homomorphisms $\epsilon$ by $\epsilon(f)(t):=f(-t)$.
$L^2(\bb{R})_{\rm gr}$ 
denotes the $\bb{Z}_2$-graded Hilbert space $L^2(\bb{R})$ by $\epsilon$.

\begin{dfn}
We define $[b_\pm]\in KK(\bb{C},\ca{S})$ by the bounded transformations of
$$(\ca{S},1,\pm X)$$
and $[d_\pm]\in KK(\ca{S},\bb{C})$ by the bounded transformations of
$$\bra{L^2(\bb{R})_{\rm gr},\mu,\pm\frac{d}{dx}\circ \epsilon},$$
where $\mu:\ca{S}\to \bb{L}_{\bb{C}}(L^2(\bb{R})_{\rm gr})$ is given by left multiplication.
\end{dfn}

\begin{lem}\label{Lemma Kasparov product of S}
$[b_\pm]\grotimes [d_\pm]={\bf 1}_{\bb{C}}$, and $[b_\pm]\grotimes [d_\mp]=0$.
\end{lem}
\begin{proof}
It is easy to prove that the representatives of $[b_\pm]\grotimes [d_\pm]$ and $[b_\pm]\grotimes [d_\mp]$ are given by
$$\bra{L^2(\bb{R})_{\rm gr},1,\pm\bra{\frac{d}{dx}\circ \epsilon+X}}\text{ and }
\bra{L^2(\bb{R})_{\rm gr},1,\mp\bra{\frac{d}{dx}\circ \epsilon-X}}$$
respectively. It suffices to prove that $\ind(\frac{d}{dx}\circ \epsilon+X)=1$ and $\ind(\frac{d}{dx}\circ \epsilon-X)=0$. 
Under the decomposition with respect to the $\bb{Z}_2$-grading $L^2(\bb{R})_{\rm gr}=L^2(\bb{R})_\ev\widehat{\oplus} L^2(\bb{R})_\od$, the operators can be written as
$$\frac{d}{dx}\circ \epsilon+X=\begin{pmatrix}
0 & \frac{d}{dx}-X \\ \frac{d}{dx} +X & 0\end{pmatrix}\text{ and }\ \frac{d}{dx}\circ \epsilon-X=\begin{pmatrix}
0 & \frac{d}{dx}+X \\ \frac{d}{dx} -X & 0\end{pmatrix}.
$$
Thanks to the spectral theory of the harmonic oscillator, the index of the former operator is $1$, and that of the latter one is $0$, because the creation operator is injective on the whole of $L^2(\bb{R})$, and the annihilation operator has one-dimensional kernel which is contained in $L^2(\bb{R})_\ev$.
\end{proof}

It is clear that $KK(\bb{C}^2,\bb{C}^2)\cong M_2(\bb{Z})$ as a ring. 
Therefore,
$$[b_{\ca{S}}]\grotimes_{\ca{S}}[d_{\ca{S}}]=
\begin{pmatrix}
[b_+]\grotimes_{\ca{S}}[d_+] &[b_+]\grotimes_{\ca{S}}[d_-]\\
[b_-]\grotimes_{\ca{S}}[d_+] & [b_-]\grotimes_{\ca{S}}[d_-]
\end{pmatrix}=\begin{pmatrix} 1 & 0 \\ 0 & 1 \end{pmatrix}={\bf 1}_{\bb{C}^2}\in KK(\bb{C}^2,\bb{C}^2).$$

It is clear that $[d_{\ca{S}}]\grotimes_{\bb{C}^2}[b_{\ca{S}}]=[d_+]\grotimes_{\bb{C}}[b_+]+[d_-]\grotimes_{\bb{C}}[b_-]$. 
We will prove that it is ${\bf 1}_{\ca{S}}$ by computing $KK(\ca{S},\ca{S})$ in detail.

\begin{lem}\label{Lemma computation of KK(S,S)}
$KK(\ca{S},\ca{S})$ is isomorphic to $\bb{Z}^4$ and the base is given by
$$\{[d_+]\grotimes_{\bb{C}}[b_+],[d_+]\grotimes_{\bb{C}}[b_-],[d_-]\grotimes_{\bb{C}}[b_+],[d_-]\grotimes_{\bb{C}}[b_-]\}.$$
${\bf 1}_{\ca{S}}$ is given by $[d_+]\grotimes_{\bb{C}}[b_+]+[d_-]\grotimes_{\bb{C}}[b_-]=[d_{\ca{S}}]\grotimes_{\bb{C}^2}[b_{\ca{S}}]$.
\end{lem}
\begin{proof}
We prove that $\ca{S}$ can be written as an extension of easy $C^*$-algebras. We define $\ev_0:\ca{S}\to \bb{C}$ by $f\mapsto f(0)$. It preserves the grading, because any odd function vanishes at $0$. Next, we define $\iota:Cl_\tau(0,\infty)\to \ca{S}$ as follows. First, we notice that $Cl_\tau(0,\infty)$ is isomorphic to $C_0(0,\infty)\widehat{\oplus}C_0(0,\infty)$ with the product 
$$(f_0\widehat{\oplus} f_1)\cdot (g_0\widehat{\oplus} g_1):=(f_0g_0+f_1g_1)\widehat{\oplus} (f_0g_1+f_1g_0).$$
Under this identification, $\iota $ is defined by $f_0\widehat{\oplus}f_1\mapsto[t\mapsto f_0(|t|)+{\rm sign}(t)f_1(|t|)]$. Then,
$$0\to Cl_\tau(0,\infty)\xrightarrow{\iota} \ca{S}\xrightarrow{\ev_0} \bb{C}\to 0$$
is obviously a short exact sequence.

In order to compute $KK(\ca{S},\ca{S})$, we note that $KK(\bb{C},\ca{S})$ and $ KK(\ca{S},\bb{C})$ are isomorphic to $ \bb{Z}^2$, and bases of these groups are given  by $\{[b_+],[b_-]\}$, and $\{[d_+],[d_-]\}$, respectively. In fact, the former statement is obvious from the above short exact sequence and  the six-term exact sequence for $K$-theory and $K$-homology. The latter one is obvious from the fact that $[b_\pm]\grotimes_{\ca{S}}[d_\pm]=1$ and $[b_\pm]\grotimes_{\ca{S}}[d_\mp]=0$ and the Kasparov product gives a bilinear form 
$KK(\bb{C},\ca{S})\times KK(\ca{S},\bb{C})\to \bb{Z}.$

By using the six-term exact sequence again, we notice that $KK(\ca{S},\ca{S})\cong \bb{Z}^4$.
 we have four group homomorphisms from $KK(\ca{S},\ca{S})$ to $\bb{Z}$ given by
$f_{\pm,\pm}:x\mapsto [b_\pm]\grotimes_{\ca{S}} x\grotimes_{\ca{S}} [d_\pm]$
(double signs are in arbitrary order). It is easy to see that $\{[d_{\pm}]\grotimes_{\bb{C}}[b_{\pm}]\}$ and $\{f_{\pm,\pm}\}$ are dual bases (double signs are in arbitrary order).

We must prove that ${\bf 1}_{\ca{S}}=[d_+]\grotimes_{\bb{C}}[b_+]+[d_-]\grotimes_{\bb{C}}[b_-]$. It is equivalent to the four equations
$$f_{\pm,\pm}({\bf 1}_{\ca{S}})=f_{\pm,\pm}([d_+]\grotimes_{\bb{C}}[b_+]+[d_-]\grotimes_{\bb{C}}[b_-])$$
(double signs are in arbitrary order).
They are satisfied by Lemma \ref{Lemma Kasparov product of S}.
\end{proof}

\begin{pro}\label{new Bott is classical Bott}
When $V$ is even-dimensional, $\sigma_{\ca{S}}([b_V])=[\beta]$ as elements of $KK_{O(V)}(\ca{S},\ca{A}(V))$.
\end{pro}
\begin{proof}
Thanks to the results so far, it suffices to prove the following four equalities:
$$\bbra{[b_{\pm}]\grotimes_{\ca{S}}\sigma_{\ca{S}}([b_V])}\grotimes_{\ca{A}(V)}\bbra{[d_{\pm}]\grotimes_{\bb{C}}[d_V]}=\bbra{[b_{\pm}]\grotimes_{\ca{S}}[\beta]}\grotimes_{\ca{A}(V)}\bbra{[d_{\pm}]\grotimes_{\bb{C}}[d_V]}\text{ and}$$
$$\bbra{[b_{\pm}]\grotimes_{\ca{S}}\sigma_{\ca{S}}([b_V])}\grotimes_{\ca{A}(V)}\bbra{[d_{\mp}]\grotimes_{\bb{C}}[d_V]}=\bbra{[b_{\pm}]\grotimes_{\ca{S}}[\beta]}\grotimes_{\ca{A}(V)}\bbra{[d_{\mp}]\grotimes_{\bb{C}}[d_V]},$$
(double signs are in the same order). Obviously, 
$$\bbra{[b_{\pm}]\grotimes_{\ca{S}}\sigma_{\ca{S}}([b_V])}\grotimes_{\ca{A}(V)}\bbra{[d_{\pm}]\grotimes_{\bb{C}}[d_V]}={\bf 1}_{\bb{C}} \text{ and}$$
$$\bbra{[b_{\pm}]\grotimes_{\ca{S}}\sigma_{\ca{S}}([b_V])}\grotimes_{\ca{A}(V)}\bbra{[d_{\mp}]\grotimes_{\bb{C}}[d_V]}=0.$$
We need to compute the right hand sides.
Let us compute $\bbra{[b_{\pm}]\grotimes_{\ca{S}}[\beta]}$. In the unbounded picture, they are given by
$$\bra{\ca{A}(V),1,\pm(X\grotimes \id+\id\grotimes C)},$$
respectively. Thus, we have the formulas $[b_{+}]\grotimes_{\ca{S}}[\beta]=[b_{+}]\grotimes_{\bb{C}}[b_V]$ and $[b_{-}]\grotimes_{\ca{S}}[\beta]=[b_{-}]\grotimes_{\bb{C}}(Cl_\tau(V),1,-C)$. 
This means that 
$$\bbra{[b_{\pm}]\grotimes_{\ca{S}}[\beta]}\grotimes_{\ca{A}(V)}\bbra{[d_{\mp}]\grotimes_{\bb{C}}[d_V]}=0.$$
Moreover, 
$$\bbra{[b_{\pm}]\grotimes_{\ca{S}}[\beta]}\grotimes_{\ca{A}(V)}\bbra{[d_{\pm}]\grotimes_{\bb{C}}[d_V]}=[(Cl_\tau(V),1,\pm C)]\grotimes_{Cl_\tau(V)}[d_V]=\ind(D\pm C).$$
Thanks to the spectral theory of the harmonic oscillator, we can compute the kernel of these operators: $\ker(D+C)=\bb{C} e^{-\frac{1}{2}\|v\|^2}\cdot 1$ and $\ker(D-C)=\bb{C} e^{-\frac{1}{2}\|v\|^2}\cdot \vol$. The index of the former one is $1$ as we have mentioned, and that of the latter one is $(-1)^{\dim(V)}$. Thus, the statement holds when $V$ is even-dimensional.
\end{proof}

\begin{rmk}\label{rmk excuse on the factor Svep}
The most natural homomorphism between $KK_G(A,B)$ and $KK_G(\ca{S}\grotimes A,\ca{S}\grotimes B)$ is $\sigma_{\ca{S}}$, which will appear in the following section.
Unfortunately, it can be not an isomorphism, because $KK_G(A,B)$ and $KK_G(\ca{S}\grotimes A,\ca{S}\grotimes B)$ can be not isomorphic.
Roughly speaking, this operation is given by
$$KK_G(A,B)\ni x\mapsto\begin{pmatrix}x & 0 \\ 0 & x \end{pmatrix}\in M_2(KK_G(A,B))\cong KK_G(\ca{S}\grotimes A,\ca{S}\grotimes B)$$
for $G$-$C^*$-algebras $A$ and $B$.

There is a trick to solve this non-triviality using the crossed product by $\bb{Z}_2$ (although we do not consider this kind of tricks after this remark). By the grading homomorphism $\epsilon_{\ca{S}}$, we can define a $\bb{Z}_2$-action on $\ca{S}$. We can also define a $\bb{Z}_2$-action on $\bb{C}^2$ by $(a,b)\mapsto (b,a)$. Then, $[b_{\ca{S}}]$ and $[d_{\ca{S}}]$ are $\bb{Z}_2$-equivariant $KK$-elements. In fact, the $\bb{Z}_2$-action on $\ca{S}\oplus \ca{S}$ is given by $(f_1,f_2)\mapsto (\epsilon_{\ca{S}}(f_2),\epsilon_{\ca{S}}(f_1))$, and similarly for $L^2(\bb{R})_{\rm gr}\oplus L^2(\bb{R})_{\rm gr}$.
Thus, by the operation $j_{\bb{Z}_2}$ and a $KK$-equivalence $\bb{C}^2\rtimes\bb{Z}_2\cong \bb{C}$, we have a $KK$-equivalence $\ca{S}\rtimes\bb{Z}_2\cong \bb{C}$. Thus, the operation
$$\sigma_{\ca{S}_\vep\rtimes\bb{Z}_2}:KK_G(A,B) \to KK_G(\ca{S}_\vep\rtimes\bb{Z}_2\grotimes A,\ca{S}_\vep\rtimes\bb{Z}_2\grotimes B)$$
is isomorphic.

Suppose that a $C^*$-algebra $B$ can be regarded as a ``graded suspension of $\ca{A}$''. The $C^*$-algebra of a Hilbert space \cite{HKT} is an example of such a $C^*$-algebra. If we can define a $\bb{Z}_2$-action on $B$ which can be regarded as ``$\epsilon_{\ca{S}}\grotimes \id$ on $\ca{S}\grotimes \ca{A}$'' in some sense, it will be possible define the ``$K$-homology of $\ca{A}$'' by
$KK(B\rtimes \bb{Z}_2,\ca{S}\rtimes \bb{Z}_2)$, by ``using'' the preceding paragraph.
\end{rmk}

\section{Equivariant index theorem for finite-dimensional manifolds}\label{section index theorem}

In this section, we recall the index theorem for complete Riemannian manifolds with proper cocompact actions \cite{Kas83,Kas15}. We will also prove several necessary results.
Then, we will reformulate it for even-dimensional $Spin^c$-manifolds. In particular, we will rewrite the $KK$-theoretical topological index following the idea explained at the end of Section \ref{section KK and RKK for finite dimension}. Finally, we will give a twisted equivariant versions of them.

\subsection{The precise statement of the index theorem}

We recall the precise statement of the index theorem \cite{Kas83,Kas15} and we prepare necessary $KK$-elements.
Let $X$ be a complete Riemannian manifold and let $G$ be a locally compact second countable Hausdorff group acting on $X$ in an isometric, proper, cocompact way. 
Let $W$ be a $G$-equivariant Clifford bundle: A $G$-equivariant Hermitian bundle over $X$ equipped with a $G$-equivariant bundle homomorphism $c:TX\to \End(W)$ such that $c(v)$ is skew-adjoint and $c(v)^2=-\|v\|^2\id_W$. Take a $G$-equivariant Dirac operator $D$ on $W$. In this setting, we have the following three objects.
$\mu$ denotes the modular function of $G$.


\begin{dfn}\label{dfn of an ind}
$(1)$ A dense subspace $C_c(X,W)$ of $L^2(X,W)$ admits a pre-Hilbert $C_c(G)$-module structure
$$\inpr{s_1}{s_2}{\bb{C}\rtimes G }(g):=\mu(g)^{-1/2}\int_X\inpr{s_1(x)}{g\cdot [s_2(g^{-1}\cdot x)]}{W_x}dx,$$
$$s\cdot b:=\int_G\mu(g)^{-1/2}(g\cdot s)b(g^{-1})dg$$
for $s_1,s_2,s\in C_c(X,W)$ and $b\in C_c(G)$.
The pair of the completion of this pre-Hilbert module and the extension of the operator $D:C_c^\infty(X,W)\circlearrowright$ defines an element of $KK(\bb{C},\bb{C}\rtimes G)$ and it is called the {\bf analytic index} of $D$, and it is denoted by $\ind_{\bb{C}\rtimes G}(D)$. See the beginning of \cite[Section 5]{Kas15}.

$(2)$ The triple $(L^2(X,W),\pi,D)$ is an unbounded $G$-equivariant Kasparov $(C_0(X),\bb{C})$-module, where $\pi$ is given by left multiplication. The corresponding $KK$-element is denoted by $[D]\in KK_G(C_0(X),\bb{C})$, and it is called the {\bf index element} of $D$. See \cite[Lemma 3.7]{Kas15}.

$(3)$ The triple $(C_0(X,W),\pi,0)$ is an unbounded $X\rtimes G$-equivariant Kasparov $(C_0(X),Cl_\tau(X))$-module, where $\pi$ is given by left multiplication. Note that $C_0(X,W)$ admits a left $\Cl_-(TX)$-module structure, and thus it admits a right Hilbert $Cl_\tau(X)$-module structure, thanks to Lemma \ref{Lemma S is a Hilbert Cltau module}. The corresponding $\ca{R}KK$-element is denoted by $[\sigma_D^{Cl}]\in\ca{R}KK_G(X;C_0(X),Cl_\tau(X))$, and it is called the {\bf Clifford symbol element} of $D$. See \cite[Definition 3.8]{Kas15}.
\end{dfn}

The following is the precise statement of the index theorem.
The task of the remainder of this subsection is to explain the  details on this result, and to prove several related equalities.

\begin{thm}[\cite{Kas83,Kas15}]\label{theorem index theorem in section 3}
$(1)$ $\ind_{\bb{C}\rtimes G}(D)$ is determined by $[D]$ by a $KK$-theoretical procedure called the {\bf analytic assembly map}.

$(2)$ $KK_G(C_0(X),\bb{C})$ and $\ca{R}KK_G(X;C_0(X),Cl_\tau(X))$ are isomorphic ({\bf Poincar\'e duality}). Under this isomorphism, $[D]$ corresponds to $[\sigma_D^{Cl}]$.

$(3)$ Consequently, $\ind_{\bb{C}\rtimes G}(D)$ is completely determined by $[\sigma_D^{Cl}]$. The explicit formula to compute $\ind_{\bb{C}\rtimes G}(D)$ from $[\sigma_D^{Cl}]$ is called the {\bf topological assembly map} in the present paper.
\end{thm}

Let us begin with Poincar\'e duality.
The geometrical setting defines the following $KK$-elements. See also Definition \ref{dfn Bott period by Dirac and dual Dirac} $(2)$.
\begin{dfn}\label{dfn section 3 Dirac and local dual Dirac}
$(1)$ The Dirac element of $X$ is defined by 
$$[d_X]:=[(L^2(X,\Cl_+(TX)),\pi,D)]\in KK_G(Cl_\tau(X),\bb{C}),$$
where $D=\sum_i \widehat{e_i}\nabla_{e_i}^{\rm LC}$ for an orthonomal base $\{e_i\}$ of each tangent space and the Levi-Civita connection $\nabla^{\rm LC}$, and $\pi$ is given by the Clifford multiplication by the left. See \cite[Definition 2.2]{Kas15}.

$(2)$ There is a positive real $\vep$ such that the injectivity radius at any $x\in X$ is greater than $2\vep$, because the group action is cocompact. 
Let $U_x$ be the $\vep$-ball centered at $x$ in $X$ and let $\Theta_x$ be the vector field on $U_x$ defined by
$$\Theta_x(y):=-\log_y(x)=
\bra{d\exp_x}_{\log_x(y)}{\rm Eul}_x(\log_x(y))=\text{``}\overrightarrow{xy}\text{''},$$
where $\log_x:U_x\to T_xX$ is the local inverse of $\exp_x:T_xX\to X$, and ${\rm Eul}_x$ is the Euler vector field on $T_xX$ centered at the origin. The {\bf local Bott element} $[\Theta_{X,1}]$ is defined by the element of $\ca{R}KK_G(X;C_0(X),\uwave{C_0(X)}\grotimes Cl_\tau(X) )$ represented by the family of Kasparov modules
$$\bra{\bbra{Cl_\tau(U_x)}_{x\in X},\{1_x\}_{x\in X},\{\vep^{-1}\Theta_x\}_{x\in X}},$$
where $1_x$ denotes the homomorphism $C_0(X)_x\cong \bb{C}\to \bb{L}_{Cl_\tau(X)}(Cl_\tau(U_x))$ given by $z\mapsto z\id$, and $\Theta_x$ denotes left multiplication by $\Theta_x$, namely $[\Theta_xf](y):=\Theta_x(y)f(y).$
This family is organized into a single Kasparov module $(C_0(U,\Cl_+(T^{\rm fiber} U),\pi,\Theta)$, where $U:=\bbra{(x,y)\ \middle|\  x\in X, y\in U_x}$ is a fiber bundle over $X$ equipped with the diagonal $G$-action $g\cdot (x,y)=(g\cdot x,g\cdot y)$, $T^{\rm fiber} U= \coprod_{x\in X}\coprod_{u\in U_x} T_uU_x$, and $\pi$ is left multiplication. When we define a $C_0(X)$-algebra structure on $C_0(X)\grotimes Cl_\tau(X)$ by the product over the second tensor multiple, we obtain another {\bf local Bott element} $[\Theta_{X,2}]\in \ca{R}KK_G(X;C_0(X),C_0(X)\grotimes \uwave{Cl_\tau(X)} )$. 
See \cite[Definition 2.3]{Kas15}.

\end{dfn}
\begin{rmk}
Local Bott elements come from other $\ca{R}KK$-elements of the ``narrower'' $C^*$-algebra:
$$[\Theta_{X,1}]\in \ca{R}KK_G\bra{X;C_0(X),C_0(U)\cdot \bra{\uwave{C_0(X)}\grotimes Cl_\tau(X) }},$$
$$[\Theta_{X,2}]\in \ca{R}KK_G\bra{X;C_0(X),C_0(U)\cdot \bra{C_0(X)\grotimes \uwave{Cl_\tau(X)} }}.$$
Strictly speaking, the local Bott elements given in the Definition should be denoted by $[\Theta_{X,1}]\grotimes[\iota_{U,X\times X}]$, and similarly for $[\Theta_{X,2}]$, where the canonical inclusion $C_0(U)\cdot \bra{C_0(X)\grotimes \uwave{Cl_\tau(X)} }\hookrightarrow C_0(X)\grotimes \uwave{Cl_\tau(X)}$  is denoted by $\iota_{U,X\times X}$.
In the present paper, by an abuse of notation, we use the same symbols to represent them unless we need to emphasize the difference, namely ``$[\Theta_{X,1}]\grotimes[\iota_{U,X\times X}]=[\Theta_{X,1}]$'' and ``$[\Theta_{X,2}]\grotimes[\iota_{U,X\times X}]=[\Theta_{X,2}]$''.
\end{rmk}

\begin{fact}[Theorem 4.1 and Theorem 4.6 of \cite{Kas15}]\label{finite-dimensional index theorem}
We define a homomorphism $\PD:KK_G(C_0(X),\bb{C})\to \ca{R}KK_G(X;C_0(X),Cl_\tau(X))$ by
$$\PD(x):=[\Theta_{X,2}]\grotimes_{X,C_0(X)\grotimes \uwave{Cl_\tau(X)}}\bbra{\sigma_{X,Cl_\tau(X)}\bra{x}},$$
and we call it the {\bf Poincar\'e duality homomorphism}.

$(1)$ $\PD$ is isomorphic and its inverse is given by the following composition:
$$\ca{R}KK_G(X;C_0(X),Cl_\tau(X))\xrightarrow{\fgt} KK_G(C_0(X),Cl_\tau(X))\xrightarrow{-\grotimes[d_X]} KK_G(C_0(X),\bb{C}).$$

$(2)$ $\PD([D])$ coincides with $[\sigma_D^{Cl}]$.
\end{fact}

Let us explain the assembly maps.

\begin{dfn}[Theorem 3.14 of \cite{Kas88}]
$(1)$ Let $A$ and $B$ be $X\rtimes G$-$C^*$-algebras, and let $(E,\pi,F)$ be an $X\rtimes G$-equivariant Kasparov $(A,B)$-module. $C_c(G,E)$ admits a pre-Hilbert $C_c(G,B)$-module structure and a $*$-homomorphism $\pi\rtimes \lt:C_c(G,A)\to \bb{L}_{C_c(G,B)}(C_c(G,E))$ given by the following operations: For $e,e_1,e_2\in C_c(G,E)$, $a\in A\rtimes G$, $b\in B\rtimes G$ and $g\in G$,
$$\inpr{e_1}{e_2}{C_c(G,B)}(g):=\int_Gg'^{-1}.\inpr{e_1(g')}{e_2(g'g)}{E}dg',$$
$$[e\cdot b](g):=\int_G e(g')ga'.[b(g'^{-1}g)]dg',$$
$$\pi\rtimes \lt(a)(e)(g):=\int_G \pi(a(g'))g'.[e(g'^{-1}g)]dg'.$$
By the completion with respect to the above inner product, we obtain a Hilbert $B\rtimes G$-module $E\rtimes G$ and a $*$-homomorphism $\pi\rtimes \lt:A\rtimes G\to \bb{L}_{B\rtimes G}(E\rtimes G)$.
These operations are compatible with the $C(X/G)$-algebra structures on $A\rtimes G$ and $B\rtimes G$.
We define an operator $\widetilde{F}$ on $E\rtimes G$ by $[\widetilde{F}e](g):=F[e(g)]$. Then, the triple $(E\rtimes G,\pi\rtimes \lt,\widetilde{T})$ is an $(X/G)\rtimes\{e\}$-equivariant Kasparov $(A\rtimes G,B\rtimes G)$-module.
The correspondence $(E,\pi,F)\mapsto (E\rtimes G,\pi\rtimes \lt,\widetilde{F})$ induces a homomorphism
$$j_G:\ca{R}KK_G(X;A,B)\to \ca{R}KK(X/G,A\rtimes G,B\rtimes G)$$
and it is called the {\bf descent homomorphism}.

$(2)$ By the same construction, we can define a homomorphism
$$j_G:KK_G(A,B)\to KK(A\rtimes G,B\rtimes G)$$
for arbitrary $G$-$C^*$-algebras $A,B$. It is called the {\bf descent homomorphism}.
\end{dfn}

The following is obvious by definition.
\begin{lem}\label{Lemma j commutes with fgt}
The following diagram commutes:
$$\begin{CD}
\ca{R}KK_G(X;A,B) @>j_G>> \ca{R}KK(X/G,A\rtimes G,B\rtimes G)\\
@V\fgt VV @VV\fgt V \\
KK_G(A,B) @>j_G>> KK(A\rtimes G,B\rtimes G).
\end{CD}$$
\end{lem}

Since the $G$-action on $X$ is proper and cocompact, there is a compactly supported continuous function $\fra{c}:X\to \bb{R}_{\geq 0}$ satisfying
$\int_G\fra{c}(g^{-1}\cdot x)dg=1$
for every $x\in X$. Such a function is called a {\bf cut-off function}. It is easy to see that all such functions are homotopic by the linear homotopy.

\begin{dfn}
$(1)$ A cut-off function $\fra{c}$ defines a projection $c$ of $C_0(X)\rtimes G$ by $\{c(g)\}(x):=\mu(g)^{-1/2}\sqrt{\fra{c}(x)\fra{c}(g^{-1}\cdot x)}$. The corresponding $KK$-element is denoted by $[c_X]\in KK(\bb{C},C_0(X)\rtimes G)$. 

$(2)$ The {\bf analytic assembly map} 
$\mu_G:KK_G(C_0(X),\bb{C})\to KK(\bb{C},\bb{C}\rtimes G )$
is defined by 
$$\mu_G(x):=[c_X]\grotimes_{C_0(X)\rtimes G}\bbra{j_G(x)}.$$

$(3)$ The {\bf topological assembly map} 
$\nu_G:\ca{R}KK_G(X;C_0(X),Cl_\tau(X))\to KK(\bb{C},\bb{C}\rtimes G )$
is defined by 
$$\nu_G(y):=[c_X]\grotimes_{C_0(X)\rtimes G}\bbra{\fgt\circ j_G(y)}\grotimes_{Cl_\tau(X)\rtimes G}\{j_G([d_X])\}.$$
\end{dfn}

In Theorem \ref{theorem index theorem in section 3} $(3)$, we stated that there is an explicit formula to compute the analytic index of $D$ from its symbol. It is the topological assembly map. The proof is quite formal.

\begin{pro}[{\cite[Theorem 5.6]{Kas15}}]\label{pro index thm big diagram}
$\mu_G(x)=\nu_G\circ \PD(x)$.
\end{pro}
\begin{proof}
we have the following commutative diagram:
{\small
$$\begin{CD}
KK_G(C_0(X),\bb{C}) @<-\grotimes[d_X]<< KK_G(C_0(X),Cl_\tau(X)) @<\fgt<< \ca{R}KK_G(X;C_0(X),Cl_\tau(X)) \\
@Vj_GVV @VVj_GV @VVj_GV \\
KK(C_0(X)\rtimes G,\bb{C}\rtimes G) @<-\grotimes j_G[d_X]<< KK(C_0(X)\rtimes G,Cl_\tau(X)\rtimes G) @<\fgt << \ca{R}KK(X/G;C_0(X)\rtimes G,Cl_\tau(X)\rtimes G) \\
@V[c_X]\grotimes-VV @V[c_X]\grotimes-VV \\
KK(\bb{C},\bb{C}\rtimes G) @<-\grotimes j_G[d_X]<< KK(\bb{C},Cl_\tau(X)\rtimes G).
\end{CD}$$}
The top right square commutes by Lemma \ref{Lemma j commutes with fgt}. The top left square commutes thanks to \cite[Theorem 3.11]{Kas88}. The bottom left square commutes thanks to the associativity of the Kasparov product. Thus, $\nu_G(y)=\mu_G(\fgt(y)\grotimes[d_X])$ for every $y\in \ca{R}KK_G(X;C_0(X),Cl_\tau(X))$. Thanks to Fact \ref{finite-dimensional index theorem} $(2)$, we obtain the result.
\end{proof}

Related to Fact \ref{finite-dimensional index theorem} $(1)$, we have the following result. We will use a corollary of it later.
For the details on differential geometrical facts used in the following, see \cite{KN1,KN2,Sak} for example.

\begin{lem}[{See \cite[Theorem 2.4]{Kas15}}]\label{Lemma family version of the Bott periodicity}
$[\Theta_{X,1}]\in \ca{R}KK_G\bra{X;C_0(X),C_0(U)\cdot \bra{\uwave{C_0(X)}\grotimes Cl_\tau(X) }}$ and $[d_X]$ are mutually inverse in the following sense: 
$$[\Theta_{X,1}]\grotimes_{X,C_0(U)[\uwave{C_0(X)}\grotimes Cl_\tau(X)]}\bbra{[\iota_{U,X\times X}]\grotimes \sigma_{X,C_0(X)}[d_X]}={\bf 1}_{X,C_0(X)} $$
$$\bbra{[\iota_{U,X\times X}]\grotimes \sigma_{X,C_0(X)}[d_X]}\grotimes_{X,C_0(X)}[\Theta_{X,1}]={\bf 1}_{X,C_0(U)\cdot(Cl_\tau(X)\grotimes \uwave{C_0(X)})}.$$
\end{lem}
\begin{proof}
The former is proved in \cite[Theorem 4.8]{Kas88}. We deduce the latter one from the former one by the rotation trick. 

Let $U\times_X U:=\bbra{(u,v,x)\in X^3\mid u,v\in U_x}$. On the $C_0(X)$-algebra 
$$C_0(U\times_X U)\cdot\bbra{Cl_\tau(X)\grotimes Cl_\tau(X)\grotimes \uwave{C_0(X)}},$$
we can define two automorphisms: $\flip^*$ is defined by  the pullback of the diffeomorphism $\flip:(u,v,x)\mapsto (v,u,x)$, and $j^*$ is defined by that of $j:(u,v,x)\mapsto (-u,v,x)$, where $-u:=\exp_x(-\log_x(u))$.
For the moment, we suppose that $[\flip^*]=[j^*]$ in the $\ca{R}KK$-group, which will be proved in the following  paragraphs.
We use a similar argument of the proof of Lemma \ref{lem proof of bott period classical}.
\begin{align*}
&\bbra{[\iota_{U,X\times X}]\grotimes \sigma_{X,C_0(X)}[d_X]}\grotimes_{X,C_0(X)}[\Theta_{X,1}]\\
&= \bbra{{\bf 1}_{X,C_0(U)\cdot(Cl_\tau(X)\grotimes \uwave{C_0(X)})}\grotimes_{X,C_0(X)}[\Theta_{X,1}]}\grotimes_{X,C_0(U\times_X U)\cdot [Cl_\tau(X)\grotimes Cl_\tau(X)\grotimes \uwave{C_0(X)}]}\\
&\ \ \ \bbra{[\flip^*]\grotimes_{X,C_0(U\times_X U)\cdot [Cl_\tau(X)\grotimes Cl_\tau(X)\grotimes \uwave{C_0(X)}]}{\bf 1}_{X,C_0(U)\cdot(Cl_\tau(X)\grotimes \uwave{C_0(X)})}\grotimes_{X,C_0(X)}[\iota_{U,X\times X}]\grotimes \bbra{\sigma_{X,C_0(X)}[d_X]}}\\
&=\bbra{{\bf 1}_{X,C_0(U)\cdot(Cl_\tau(X)\grotimes \uwave{C_0(X)})}\grotimes_{X,C_0(X)}[\Theta_{X,1}]}\grotimes_{X,C_0(U\times_X U)\cdot [Cl_\tau(X)\grotimes Cl_\tau(X)\grotimes \uwave{C_0(X)}]}\\
&\ \ \ \bbra{[j^*]\grotimes_{X,C_0(U\times_X U)\cdot [Cl_\tau(X)\grotimes Cl_\tau(X)\grotimes \uwave{C_0(X)}]}[\iota_{U,X\times X}]\grotimes \bbra{\sigma_{X,C_0(X)}[d_X]}}\\
&=[j^*]\grotimes_{X,C_0(U\times_X U)\cdot [Cl_\tau(X)\grotimes Cl_\tau(X)\grotimes \uwave{C_0(X)}]}\bbra{[\Theta_{X,1}]\grotimes_{X,\uwave{C_0(X)}\grotimes Cl_\tau(X)}\bbra{[\iota_{U,X\times X}]\grotimes \sigma_{X,C_0(X)}[d_X]}}=[j^*].
\end{align*}
Thus, $\bbra{[\iota_{U,X\times X}]\grotimes \sigma_{X,C_0(X)}[d_X]}\grotimes_{X,C_0(X)}[\Theta_{X,1}]$ is an isomorphism in $\ca{R}KK$-theory. Since the left and right inverse are the same, we obtain the result.

Let us prove the equality $[\flip^*]=[j^*]$. We would like to use the ``rotation homotopy''.
Take other neighborhoods of the diagonal
$$\bb{U}_{\sqrt{2}\vep}:=\bbra{(u,v,x)\in X^3\mid \rho(u,x)^2+\rho(v,x)^2< 2\vep^2}\text{ and}$$
$$\bb{U}_{\vep}:=\bbra{(u,v,x)\in X^3\mid \rho(u,x)^2+\rho(v,x)^2< \vep^2}.$$
Then, we can define continuous maps $\flip_{\bb{U}_{\sqrt{2}\vep}}$, $j_{\bb{U}_{\sqrt{2}\vep}}$,  $\flip_{\bb{U}_{\vep}}$ and $j_{\bb{U}_{\vep}}$ in an obvious way.
Each space admits a fiber bundle structure over $X$ by the projection onto the third factor, and it is obvious that $\bb{U}_{\vep}\subseteq U\times_X U\subseteq \bb{U}_{\sqrt{2}\vep}$.
On $\bb{U}_{\sqrt{2}\vep}$, we can consider the rotation homotopy
$$\bb{U}_{\sqrt{2}\vep}\times[0,\pi/2]\ni ((u,v,x),t)\mapsto (u\cos t-v\sin t,u\sin t+v\cos t,x)\in \bb{U}_{\sqrt{2}\vep},$$
between $\flip_{\bb{U}_{\sqrt{2}\vep}}$ and $j_{\bb{U}_{\sqrt{2}\vep}}$,
where $u\cos t-v\sin t$ is defined by the exponential map on the fiber $\bb{U}_{\sqrt{2}\vep}$ at $x$: $\exp_{(x,x)}^{\bb{U}_{\sqrt{2}\vep}|_x}(\log_x(u)\cos t-\log_x(v)\sin t)$, and similarly for $u\sin t+v\cos t$. 
Thus, $[\flip_{\bb{U}_{\sqrt{2}\vep}}^*]=[j_{\bb{U}_{\sqrt{2}\vep}}^*]$ in the $\ca{R}KK$-group.

The new neighborhoods of the diagonal are related to $U\times_XU$ by $\bb{U}_{\vep}\xrightarrow{i_1}U\times_X U\xrightarrow{i_2}\bb{U}_{\sqrt{2}\vep}$. Thus, we have $G$-equivariant commutative diagrams
\begin{center}
$\begin{CD}
U\times_X U @>i_2>> \bb{U}_{\sqrt{2}\vep}\\
@V\flip VV @VV\flip_{\bb{U}_{\sqrt{2}\vep}} V \\
U\times_X U @>i_2>> \bb{U}_{\sqrt{2}\vep},\end{CD}$\ \ \ \ \  \ \ \ and\ \ \ \ \ \ \ \ 
$\begin{CD}U\times_X U @>i_2>> \bb{U}_{\sqrt{2}\vep}\\
@Vj VV @VVj_{\bb{U}_{\sqrt{2}\vep}} V \\
U\times_X U @>i_2>> \bb{U}_{\sqrt{2}\vep}.\end{CD}$
\end{center}
These induce a $G$-equivariant homotopy equivalence
$$i_2\circ \flip=\flip_{\bb{U}_{\sqrt{2}\vep}} \circ i_2\sim j_{\bb{U}_{\sqrt{2}\vep}}\circ i_2=i_2\circ j,$$
and hence $[\flip^*]\grotimes [i_2^*]=[j^*]\grotimes [i_2^*].$
Thus, it suffices to prove that $[i_2^*]$ is invertible. It suffices to find a homotopy inverse. Let $r:\bb{U}_{\sqrt{2}\vep}\to \bb{U}_{\vep}$ be the map $(u,v,x)\mapsto (2^{-1/2}u,2^{-1/2}v,x)$. It is equivariant and continuous because the exponential maps are equivariant with respect to the isometric group action. We define $k:=i_1\circ r:\bb{U}_{\sqrt{2}\vep}\to U\times_X U$. It is a homotopy inverse of $i_2$. In fact, both of $i_2\circ k$ and $k\circ i_2$ are given by the same formula $(u,v,x)\mapsto (2^{-1/2}u,2^{-1/2}v,x)$. These are homotopic to the identity by $H_t(u,v,x):=((1-t+2^{-1/2}t)u,(1-t+2^{-1/2}t)v,x)$.
\end{proof}

\begin{cor}\label{cor -theta is theta}
Let $[-\Theta_{X,1}]$ be the $\ca{R}KK$-element defined by
$$\bra{\bbra{Cl_\tau(U_x)}_{x\in X},\{1_x\}_{x\in X},\{-\vep^{-1}\Theta_x\}_{x\in X}}.$$
If $X$ is even-dimensional and orientable, $[-\Theta_{X,1}]=[\Theta_{X,1}]$.
\end{cor}
\begin{proof}
By the above proposition, it suffices to prove that
$$[-\Theta_{X,1}]\grotimes_{X,\uwave{C_0(X)}\grotimes Cl_\tau(X)}\bbra{[\iota_{U,X\times X}]\grotimes \sigma_{X,C_0(X)}[d_X]}=[\Theta_{X,1}]\grotimes_{X,\uwave{C_0(X)}\grotimes Cl_\tau(X)}\bbra{[\iota_{U,X\times X}]\grotimes \sigma_{X,C_0(X)}[d_X]}.$$
The right hand side is given by the bounded transformation of
$$\bra{\bbra{L^2(U_x,\Cl_+(TU_x))}_{x\in X},\{1_x\}_{x\in X},\left\{D-
\frac{\vep^{-1}\Theta_x}{\sqrt{1-\vep^{-2}\Theta_x^2}}\right\}_{x\in X}},$$
where $D$ is the Dirac operator given by the same formula of Definition \ref{dfn section 3 Dirac and local dual Dirac}.

Let us compare it with $[\Theta_{X,1}]\grotimes_{X,\uwave{C_0(X)}\grotimes Cl_\tau(X)}\bbra{[\iota_{U,X\times X}]\grotimes \sigma_{X,C_0(X)}[d_X]}$.
Fix an orientation of $X$.
On the Clifford algebra of an oriented vector space, we can define the volume element by $e_1e_2\cdots e_{\dim(V)}$ for an oriented orthonormal base $\{e_1,e_2,\cdots, e_{\dim(V)}\}$. The same construction on each tangent space defines a global section on $\Cl_+(TU_x)$, and it is denoted by $\{\vol_x\}_{x\in X}$. A simple computation on the Clifford algebra gives the following formulas:
$$\vol\circ e_i=(-1)^{\dim(X)-1}e_i\vol\ \text{ and }\ \vol\circ \widehat{e_i}=(-1)^{\dim(X)}\widehat{e_i}\circ \vol.$$
Note that $\vol_x\in Cl_\tau(U_x)$ commutes with $\nabla^{\rm LC}_v$ for any tangent vector $v$. 
Therefore, since $\dim(X)$ is even,
$$\bra{D-
\frac{\vep^{-1}\Theta_x}{\sqrt{1-\vep^{-2}\Theta_x^2}}}\circ \vol_x=\vol_x\circ \bra{D+
\frac{\vep^{-1}\Theta_x}{\sqrt{1-\vep^{-2}\Theta_x^2}}}.$$
Thus, the family $\{\vol_x\}_{x\in X}$ gives an isomorphism 
between $[-\Theta_{X,1}]\grotimes_{X,\uwave{C_0(X)}\grotimes Cl_\tau(X)}\bbra{[\iota_{U,X\times X}]\grotimes \sigma_{X,C_0(X)}[d_X]}$ and $[\Theta_{X,1}]\grotimes_{X,\uwave{C_0(X)}\grotimes Cl_\tau(X)}\bbra{[\iota_{U,X\times X}]\grotimes \sigma_{X,C_0(X)}[d_X]}.$
\end{proof}

\subsection{A reformulation of the index theorem for even-dimensional $Spin^c$-manifolds}\label{section index theorem Spinc 2n}

There are no $C^*$-algebras which play roles of $C_0$-algebras for infinite-dimensional manifolds. However, there are constructions of $C^*$-algebras which play roles of ``the suspension of the Clifford algebra-valued function algebra of infinite-dimensional manifolds'' \cite{HKT,HK,Tro,DT,GWY,Yu}. 

Although $C_0$ and $Cl_\tau$ are different, they are $\ca{R}KK$-equivalent for even-dimensional $Spin^c$-manifolds. Thus, there is a possibility to generalize $KK$-theoretical results on an even-dimensional $Spin^c$-manifolds formulated in the language of $C_0$-algebra, to some infinite-dimensional manifolds by using the $C^*$-algebras of Hilbert manifolds. 
This is the fundamental idea of \cite{T4}. In this paper, the author rewrote the latter half of the inverse of the Poincar\'e duality homomorphism, and formulate a substitute for it for proper $LT$-spaces. Moreover, he computed it for a special case.

In this subsection, we follow the same idea. Basically, we need to replace a single $C^*$-algebra $C_0(X)$ with $Cl_\tau(X)$.
Moreover, since it is easy to introduce the concept of an infinite-dimensional version of the ``$C_0(X)$-algebra $C_0(X)$'', we will remove the ``$C_0(X)$-algebra $Cl_\tau(X)$'' from the theory.

Let us begin with the $\ca{R}KK$-elements to identify $C_0$ and $Cl_\tau$ at the $KK$-theory level.
As mentioned in Fact \ref{Definitions in preliminaries, Clifford algebras}, an irreducible left $\Cl_-(V)$-module $S$ admits a right Hilbert $\Cl_+(V)$-module structure, and $S^*$ admits a left $\Cl_+(V)$-module structure. Thus, the following two $\ca{R}KK$-elements make sense.

\begin{dfn}
Let $X$ be an even-dimensional Riemannian $Spin^c$-manifold equipped with an isometric proper cocompact $G$-action. 
We suppose that the $G$-action lifts to a Spinor bundle $(S,c)$. Then its dual $(S^*,c^*)$ is also $G$-equivariant.

$(1)$ $C_0(X,S)$ has a Hilbert $Cl_\tau(X)$-module structure by the above observation, and it admits a left $C_0(X)$-module structure given by $[\pi(f)s](x):=f(x)s(x)$ for $f\in C_0(X)$, $s\in C_0(X,S)$ and $x\in X$. We define an $\ca{R}KK_G$-element $[S]$ by
$$[S]:=(C_0(X,S),\pi,0)\in \ca{R}KK_G(X;C_0(X),Cl_\tau(X)).$$

$(2)$ $C_0(X,S^*)$ has a Hilbert $C_0(X)$-module structure, and it admits a left $Cl_\tau(X)$-module structure given by $[c^*(f)s](x):=c^*(f(x))s(x)$ for $f\in Cl_\tau(X)$, $s\in C_0(X,S^*)$ and $x\in X$. We define an $\ca{R}KK_G$-element $[S^*]$ by
$$[S^*]:=(C_0(X,S^*),c^*,0)\in \ca{R}KK_G(X;Cl_\tau(X),C_0(X)).$$
\end{dfn}

Fact \ref{Definitions in preliminaries, Clifford algebras} shows that $Cl_\tau(X)\cong \End(S^*)$. Using it, one can prove the following result.

\begin{lem}
$(1)$ These two $\ca{R}KK_G$-elements give an $\ca{R}KK_G$-equivalence:
\begin{center}
$[S]\grotimes_{X,Cl_\tau(X)}[S^*]=1_{X,C_0(X)}$ and $[S^*]\grotimes_{X,C_0(X)}[S]=1_{X,Cl_\tau(X)}$.
\end{center}

$(2)$ Consequently, $\fgt([S])$ and $\fgt([S^*])$ gives a $KK_G$-equivalence between $C_0(X)$ and $Cl_\tau(X)$.
\end{lem}

Using these $\ca{R}KK_G$- and $KK_G$- equivalences, we reformulate the Kasparov index theorem. We suppose that the injectivity radius is greater than $2\vep$ everywhere (it is always possible if $X$ admits an isometric cocompact group action), and we put $\ca{A}(X):=\ca{S}_\vep\grotimes Cl_\tau(X)$. Although $\ca{S}_\vep$ is not essential for the following definition (and it is possible to replace $\vep$ with another number), it plays an essential role to reformulate the theory further using the Bott homomorphisms of Definition \ref{dfn of the Bott hom}.

\begin{dfn}[See also Section 2.5 of \cite{T4}]\label{dfn reformulated KK elements}
Let $X$ be an even-dimensional complete Riemannian 
$Spin^c$-manifold and let $G$ be a locally compact second countable Hausdorff group acting on $X$ in an isometric, proper and cocompact way. We assume that the $G$-action lifts to the Spinor bundle.
Suppose that the injectivity radius is greater than $2\vep$ everywhere.
Let $W$ be a $G$-equivariant Clifford bundle and let $D$ be a $G$-equivariant Dirac operator on $W$. 

$(1)$ We reformulate $KK$-elements appearing in the index theorem as follows:
\begin{itemize}
\item $[\widetilde{D}]:=\sigma_{\ca{S}_\vep}\bra{\fgt[S^*]\grotimes [D]}\in KK_G(\ca{A}(X),\ca{S}_\vep)$;
\item $[\widetilde{\sigma_D^{Cl}}]:=\sigma_{\ca{S}_\vep}\bra{[\sigma_D^{Cl}]\grotimes [S^*]}\in \ca{R}KK_G(X;\ca{S}_\vep\grotimes C_0(X),\ca{S}_\vep\grotimes C_0(X))$;
\item $[\widetilde{d_X}]:=\sigma_{\ca{S}_\vep}\bra{[d_X]}\in KK_G(\ca{A}(X),\ca{S}_\vep)$;
\item $[\widetilde{c_X}]:=\sigma_{\ca{S}_\vep}\bra{[c_X]}\in KK(\ca{S}_\vep,\ca{S}_\vep\grotimes [C_0(X)\rtimes G])$;
\item $[\widetilde{\Theta_{X,1}}]:=\sigma_{\ca{S}_\vep}\bra{[\Theta_{X,1}]}\in \ca{R}KK_G(X;\ca{S}_\vep\grotimes C_0(X),\uwave{C_0(X)}\grotimes \ca{A}(X) )$; and
\item $[\widetilde{\Theta_{X,2}}]:=\sigma_{\ca{S}_\vep}\bra{[\Theta_{X,2}]\grotimes \sigma_{C_0(X)}^1\bra{[S^*]}\grotimes\sigma_{X,C_0(X)}^2\bra{\fgt[S]}}\in \ca{R}KK_G(X;\ca{S}_\vep\grotimes C_0(X),\ca{A}(X) \grotimes \uwave{C_0(X)})$.
\end{itemize}

$(2)$ The corresponding homomorphisms to $\PD$, $\mu_G$ and $\nu_G$ are denoted with tilde:
\begin{itemize}
\item $\widetilde{\PD}(\widetilde{x}):=[\widetilde{\Theta_{X,2}}]\grotimes\bra{\sigma_{X,C_0(X)}^2(\widetilde{x})}: KK(\ca{A}(X),\ca{S}_\vep)\to \ca{R}KK_G(X;\ca{S}_\vep\grotimes C_0(X),\ca{S}_\vep\grotimes C_0(X))$;
\item $\widetilde{\mu_G}(\widetilde{x}):=[\widetilde{c_X}]\grotimes \sigma_{\ca{S}_\vep}\bra{\fgt\bra{ j_G([S])}}\grotimes j_G(\widetilde{x}): KK(\ca{A}(X),\ca{S}_\vep)\to KK(\ca{S}_\vep,\ca{S}_\vep\rtimes G)$; and
\item $\widetilde{\nu_G}(\widetilde{y}):=[\widetilde{c_X}]
\grotimes\bra{\fgt\circ j_G([\widetilde{y}])}\grotimes 
\sigma_{\ca{S}_\vep}\bbra{j_G(\fgt[S]\grotimes [{d_X}])}: \ca{R}KK_G(X;\ca{S}_\vep\grotimes C_0(X),\ca{S}_\vep\grotimes C_0(X))\to KK(\ca{S}_\vep,\ca{S}_\vep\rtimes G)$.
\end{itemize}
\end{dfn}

\begin{pro}\label{prop index theorem for Spinc}
These reformulated objects satisfy the following index theorem type equalities:
\begin{center}
$[\widetilde{\sigma_D^{Cl}}]=\widetilde{\PD}[\widetilde{D}]$, 
$[\widetilde{D}]= \fgt\bra{[S^*]\grotimes [\widetilde{\sigma_D^{Cl}}]\grotimes [S]}\grotimes [\widetilde{d_X}]$ and
$\widetilde{\mu_G}([\widetilde{D}])=\widetilde{\nu_G}([\widetilde{\sigma_D^{Cl}}]).$

\end{center}
\end{pro}
\begin{proof}
The following facts show the statement: The Kasparov product is associative; $j_G$ commutes with the Kasparov product; $\fgt$ commutes with $j_G$; Fact \ref{finite-dimensional index theorem}; $\sigma_{\ca{S}_\vep}$ commutes with the Kasparov product; $\sigma_{\ca{S}_\vep}$ commutes with $\sigma_{X,C_0(X)}^2$'s; $[S]\grotimes [S^*]=1_{C_0(X)}$; $[S^*]\grotimes [S]=1_{Cl_\tau(X)}$; $\sigma_{\ca{S}_\vep}\bra{\mu_G([D])}=\widetilde{\mu_G}([\widetilde{D}])$ and
$\sigma_{\ca{S}_\vep}\bra{\nu_G([\sigma_D^{Cl}])}=\widetilde{\nu_G}([\widetilde{\sigma_D^{Cl}}])$. We leave the details to the reader.
\end{proof}

\begin{rmk}\label{Section 3 unsolved points}
The formulas on the inverse of the Poincar\'e duality and the analytic assembly map are not satisfying, because the factor $[S]$ or $[S^*]$ appears. Since there is no $C_0$ in the theory of  infinite-dimensional manifolds, what we can formulate is not ``$\fgt$'' itself but the composition of it and $[S^*]\grotimes-\grotimes[S]$, and similarly for $j$. We do not study infinite-dimensional versions of them in the present paper.
We will mention this problem in Section \ref{section unsolved} in order to explain what the next problem is.
\end{rmk}

\begin{nota}
From now on, the reformulated Poincar\'e duality homomorphims  and the assembly maps are denoted without tildes: 
$\PD([\widetilde{D}]):=\widetilde{\PD}([\widetilde{D}])$, 
$\mu_G([\widetilde{D}]):=\widetilde{\mu_G}([\widetilde{D}])$ and $\nu_G([\widetilde{\sigma_D^{Cl}}]):=\widetilde{\nu_G}([\widetilde{\sigma_D^{Cl}}]).$
\end{nota}

In order to generalize the reformulated $KK$-elements to infinite-dimensional manifolds, we will rewrite the local Bott elements using Proposition \ref{new Bott is classical Bott} and Corollary \ref{cor -theta is theta}.
We begin with the explicit description of $[\widetilde{\Theta_{X,2}}]$ in the language of fields of Kasparov modules. By an abuse of notation, $\ca{A}(U_x)$ denotes $\ca{S}_\vep\grotimes Cl_\tau(U_x)$ (note that we define $\ca{A}(X)$ only for manifolds whose injectivity radius is bounded below).

\begin{lem}\label{lemma reformulated Theta 2 is in fact Theta 1}
$[\widetilde{\Theta_{X,2}}]$ is represented by
$$\sigma_{\ca{S}_\vep}\bra{\{Cl_\tau(U_x)\}_{x\in X},\{1\}_{x\in X},\{\vep^{-1}\Theta_x\}_{x\in X}},$$
or equivalently $\sigma_{\ca{S}_\vep}[\Theta_{X,1}]$ after the isomorphism $\ca{A}(X) \grotimes \uwave{C_0(X)}\cong \uwave{C_0(X)}\grotimes \ca{A}(X)$.
\end{lem}
\begin{proof}
We first explicitly describe the collection terms $\sigma_{C_0(X)}^1([S^*])$ and $\sigma_{X,C_0(X)}^2\circ \fgt([S])$ in the language of fields.
Noticing that the $C_0(X)$-algebra $C_0(X)\grotimes \uwave{Cl_\tau(X)}$ is given by the family\\
$\{C_0(X)\grotimes \Cl_+(T_xX)\}_{x\in X}$, we find that $\sigma_{C_0(X)}^1([S^*])\in \ca{R}KK(X;C_0(X)\grotimes\uwave{Cl_\tau(X)},C_0(X)\grotimes \uwave{C_0(X)})$ 
is  given by the family 
$$\bra{\{C_0(X)\grotimes S^*_x\}_{x\in X},\{\id\grotimes c^*_x\}_{x\in X},\{0\}_{x\in X}}$$
as a family of Kasparov modules.
Similarly, 
$\sigma_{X,C_0(X)}^2\circ \fgt([S])\in \ca{R}KK(X;C_0(X)\grotimes \uwave{C_0(X)},Cl_\tau(X)\grotimes \uwave{C_0(X)})$ is given by
$$\bra{\{C_0(X,S)\grotimes \bb{C}_x\}_{x\in X},\{\mu\grotimes 1\}_{x\in X},\{ 0\}_{x\in X}}$$
as a family of Kasparov modules, where $\mu:C_0(X)\to \bb{L}(C_0(X,S))$ is given by the pointwise multiplication.

Therefore, at the level of fields of modules, thanks to $\Cl_+(T_xX)\grotimes_{\Cl_+(T_xX)} S_x^*\cong S_x^*$,
\begin{align*}
&\bbbra{C_0(U_x)\grotimes \Cl_{+}(T_xX)}
\grotimes_{C_0(X)\grotimes \Cl_+(T_xX)} 
\bbbra{C_0(X)\grotimes S_x^*}
\grotimes_{C_0(X)\grotimes \bb{C}_x} 
\bbbra{C_0(X,S)\grotimes \bb{C}_x} 
 \\
&\cong\bbbra{C_0(U_x)\grotimes S_x^*}
\grotimes_{C_0(X)\grotimes \bb{C}_x} 
\bbbra{C_0(X,S)\grotimes \bb{C}_x} 
 \\
 &\cong C_0(U_x,S)\grotimes S_x^*.
\end{align*}
The isomorphism is given by the formula
$$[f\grotimes u]\grotimes [g\grotimes v]\grotimes [h\grotimes w]\mapsto (-1)^{\partial h(\partial u+\partial v)}fgh\grotimes c^*(u)vw
$$
for $f\in C_0(U_x)$, $u\in \Cl_+(T_xX)$, $g\in C_0(X)$, $v\in S_x^*$, $h\in C_0(X,S)$ and $w\in \bb{C}_x$.

One can easily check that the family of operators $\{\vep^{-1}\Theta_x(\bullet)\grotimes \id\grotimes \id\}_{x\in X}$ satisfies the conditions to be the triple Kasparov product. Under the above identification, this operator is transformed into $1\grotimes \vep^{-1}\Theta_{x}(\bullet)$. More concretely
$$1\grotimes \vep^{-1}\Theta_{x}(\bullet)[f\grotimes s](y):=
(-1)^{\partial f}f(y)\grotimes c^*(\vep^{-1}\Theta_x(y))(s).$$

Next, we identify the trivial bundle $U_x\times S^*_x$ with $S^*|_{U_x}$ by the parallel transformations along the geodesics starting from $x$. We choose a connection $\nabla^{S^*}$ on $S^*$ such that it is compatible with the Levi-Civita connection $\nabla^{\rm LC}$ and the Clifford multiplication $c^*$ in the following sense
: For a section $s\in C^\infty(X,S^*)$ and vector fields $v,w\in \scr{X}(X)$, we have $\nabla^{S^*}_v[c^*(w)s]=c^*(\nabla^{\rm LC}_vw)s+c^*(w)(\nabla^{S^*}_vs)$. There always exists such a connection because there exists a Spin structure on each small open set and we can use the patchwork argument in this context \cite[Corollary 3.41]{BGV}. 
Let $P^x_y$ be the parallel transformations of $TX$, along the unique minimal geodesic traveling from $x$ to $y$, with respect to $\nabla^{\rm LC}$. We use the same symbol for that of $S^*$ with respect to $\nabla^{S^*}$. Thanks to the condition imposed on $\nabla^{S^*}$, we have $P^x_y[c^*(v)s]=c^*(P^x_y(v))P^x_y(s)$ for $v\in T_xX$ and $s\in S_x^*$. 

We must compute $\Theta_{y}(x)$ under this identification. Let $\gamma:[0,1]\to X$ be the minimal geodesic traveling from $x$ to $y$.
By definition, $\Theta_y(x)=-\dot{\gamma}(0)$, $\Theta_x(y)=\dot{\gamma}(1)$ and $P^x_y(\dot{\gamma}(0))=\dot{\gamma}(1)$. Thus, we have $P^x_y(\Theta_y(x))=-\Theta_x(y)$. Consequently, we have $P^x_y(c^*(\Theta_y(x))s)=-c^*(\Theta_x(y))P^x_y(s)$.
Since $c^*$ corresponds to the left Clifford multiplication under the identification $S^*\grotimes S\cong \Cl_+(TX)$, the following diagram commutes:
$$\begin{CD}
C_0(U_x,S)\grotimes S^*_x @>P_x>> C_0(U_x,S^*\grotimes S) \\
@V1\grotimes c^*(\Theta_{\bullet}(x))VV @VV-\Theta_x(\bullet)V \\
C_0(U_x,S)\grotimes S^*_x @>P_x>> C_0(U_x,S^*\grotimes S),
\end{CD}$$
where $P_x$ is defined by 
$$P_x[f\grotimes s](y):= (-1)^{\partial f(y) \partial s}P^x_y(s)\grotimes f(y)$$ and the right vertical arrow $-\Theta_x(\bullet)$ means left multiplication by $-\Theta_x(\bullet)$ under the identification $C_0(U_x,S^*\grotimes S)\cong Cl_\tau(U_x)$. Thus, $[\widetilde{\Theta_{X,2}}]$ is represented by
$$\sigma_{\ca{S}_\vep}\bra{\{Cl_\tau(U_x)\}_{x\in X},\{1\}_{x\in X},\{-\vep^{-1}\Theta_x\}_{x\in X}}.$$
Since $X$ is orientable (this is because $X$ is $Spin^c$), we can use Corollary \ref{cor -theta is theta}, and we finish the proof.
\end{proof}

We can further rewrite the above by a family of the Bott homomorphisms of Definition \ref{dfn of the Bott hom}.
We define a $*$-homomorphism $\beta_x:\ca{S}_\vep\to \ca{A}(X)$ by
$$
\beta_x(f)\phi(y):=
\begin{cases}
f(X\grotimes 1+1\grotimes \Theta_x(y))\phi(y) & y\in U_{x} \\
0 & y\notin U_x.
\end{cases}
$$
In order to prove that this is a continuous section, let us study $f(X\grotimes 1+1\grotimes \Theta_x(y))$.
We can divide $f$ into two parts: $f(t)=f_0(t^2)+tf_1(t^2)$ for $f_0,f_1\in C_0[0,\vep^2)$.
Then, the value of $f(X\grotimes 1+1\grotimes \Theta_x(\bullet))$ at $(s,y)\in (-\vep,\vep)\times X$ is given by
$$f_0(s^2+\rho(x,y)^2)+(s\cdot 1+\Theta_x(y))f_1(s^2+\rho(x,y)^2),$$
where $\rho$ is the distance function on $X$ and ``$1$'' in the second term is the multiplicative identity of the Clifford algebra.
Since $f_0$ and $f_1$ vanishes on $t^2\geq \vep^2$, $\beta_x(f)\phi(y)$ vanishes on
the boundary of $U_x$, and consequently $\beta_x(f)\phi$ is continuous on $X$.

\begin{pro}\label{Prop reformulated local Bott element}
$[\widetilde{\Theta_{X,2}}]$
is represented by the field of Kasparov modules
$$\bra{\{\ca{A}(X)\}_{x\in X},\{\beta_x\}_{x\in X},\{0\}_{x\in X}}.$$
\end{pro}
\begin{proof}
By the same argument of Proposition \ref{new Bott is classical Bott} and the above result, we notice that $[\widetilde{\Theta_{X,2}}]$ is represented by
$\bra{\{\ca{A}(U_x)\}_{x\in X},\{\beta_x\}_{x\in X},\{0\}_{x\in X}}$. 
It is homotopic to $\bra{\{\ca{A}(X)\}_{x\in X},\{\beta_x\}_{x\in X},\{0\}_{x\in X}}$ by $\beta_x(\ca{S}_\vep)\ca{A}(X)\subseteq \ca{A}(U_x)$.

\end{proof}

The following example explain the geometrical meaning of the reformulated index theorem.

\begin{ex}\label{ex computation of PD}
We have reformulated the Poincar\'e duality homomorphism, the index element and the Clifford symbol element.
Let us compute them for the following case. {\it Let $E$ be a $G$-equivariant Hermitian vector bundle with a $G$-invariant metric connection $\nabla^E$ over $X$. Let $D_E$ be the Dirac operator defined by
$$D_E:=\sum \id_E\grotimes c(v_n)\circ \nabla^{E\grotimes S}_{v_n}: C_c^\infty(X,E\grotimes S)\circlearrowright,$$
where $\{v_n\}$ is an orthonomal base of the tangent space.}

In this situation, the reformulated $KK$-elements are given by the  following:
\begin{itemize}
\item $[\widetilde{D_E}]=\sigma_{\ca{S}_\vep}\bra{L^2(X,E\grotimes \Cl_+(TX)),\pi,D_E'}$, 
where $D_E'=\sum \id\grotimes c(v_n)\circ \nabla^{E\grotimes \Cl_+(TX)}_{v_n}$ and $\pi$ denotes the Clifford multiplication on the left $c^*$; and
\item $[\widetilde{\sigma_{D_E}^{Cl}}]=\sigma_{\ca{S}_\vep}\bra{C_0(X,E),\pi,0}$.
\end{itemize}
Therefore, the reformulated Poincar\'e duality homomorphism is a map which assigns to the $KK$-element corresponding to the Dirac operator twisted by $E$ the $\ca{R}KK$-element corresponding to its coefficient $E$. 
\end{ex}


\subsection{A convenient description of the topological assembly map}\label{section index theorem top ass map}

The goal of this subsection is to describe all the ingredients of the topological assembly map in more convenient form. This is essential to define a substitute for the topological assembly map for proper $LT$-spaces.

Let us begin with the descent homomorphism for $\ca{R}KK$-theory.
The key ingredient is the formula to describe crossed products of $X\rtimes G$-$C^*$-algebras using the generalized fixed-point algebras \cite[Theorem 2.14]{EE}. It is easy to give the field description of this result.
Using it and its Hilbert module version, we can describe the descent homomorphism in the language of fields.

We first give a review of the proof of \cite[Section 2]{EE}. 
In order to explain it, we need the concept of {\bf generalized fixed-point algebras}. Let $G$ be a locally compact second countable Hausdorff group and let $X$ be a $\sigma$-compact locally compact Hausdorff space equipped with a proper $G$-action.
Let $A$ be an $X\rtimes G$-$C^*$-algebra, whose action is denoted by $\alpha^A:G\to \Aut(A)$.
With the $G$-action on $X\times X$ given by $g:(x,y)\mapsto (gx,gy)$, we can define a $C_0(X\times_G X)$-$C^*$-algebra 
$C_0(X\times_GA),$
by the set of all continuous functions $f:X\to A$ satisfying the following conditions: it is $G$-invariant $f(x)=\alpha_g[f(g^{-1}\cdot x)]$; and the norm function $X\times_G X\ni [(x,y)]\mapsto \|f(x)(y)\|\in\bb{R}_{\geq 0}$ vanishes at infinity, where we regard $A$ as the section algebra over $X$. The generalized fixed-point algebra of $A$ 
is defined by the restriction of the $C_0(X\times_G X)$-$C^*$-algebra $C_0(X\times_GA)$ to the orbit space of the diagonal set $\Delta(X)/G\cong X/G$:
$$A^{G,\alpha}:=C_0(X\times_GA)|_{\Delta(X)/G}.$$

\begin{ex}
When $X=G$ and $A=C_0(G)$, whose action is defined by left translation ``$\lt$'', $C_0(G\times_GC_0(G))$ is  $C_0(G)$. The isomorphism is induced by the homeomorphism $G\times G\ni [(g,h)] \mapsto h^{-1}g\in G$. The restriction to $\Delta(G)/G$ is the fiber at $[(e,e)]$, which is $\bb{C}$. Thus, the generalized fixed-point algebra $C_0(G)^{G,\lt}$ is, as everyone expects, given by $\bb{C}=C(G/G)$.
\end{ex}

Using the concept of generalized fixed-point algebras, we can describe crossed products.
We need the following ingredients:
the unitary representation $R$ of $G$ on $L^2(G)$ given by $R_g\phi(x):=\phi(xg)\sqrt{\mu(g)}$; the action $\Ad R:G\to \Aut(\bb{K}(L^2(G)))$ given by $\Ad R_g(k):=R_g\circ k\circ R_g^{-1}$ for  $k\in \bb{K}(L^2(G))$. By using it, we can define a $G$-action on $\bb{K}(L^2(G))\grotimes A$ by $(\Ad R\grotimes \alpha)_g(k\grotimes a):=\Ad R_g(k)\grotimes \alpha^A_g(a)$ for $a\in A$, $k\in \bb{K}(L^2(G))$ and $g\in G$.  

\begin{pro}[\cite{EE}]\label{crossed product and generalized fixed-point algebra}
Let $G$ be a locally compact second countable Hausdorff group and let $X$ be a $\sigma$-compact locally compact Hausdorff space equipped with a proper $G$-action. We assume that $G$ is amenable, for simplicity. 
Let $A$ be an $X\rtimes G$-$C^*$-algebra, whose action is denoted by $\alpha^A:G\to \Aut(A)$.
Then, the crossed product $A\rtimes_{\alpha^A} G$ is isomorphic to the generalized fixed-point algebra
$$(\bb{K}(L^2(G))\grotimes A)^{G,\Ad R\grotimes \alpha}$$
as $C_0(X/G)$-algebras.
\end{pro}

In the language of u.s.c. fields of $C^*$-algebras and integral kernels, this result is described as follows. 

\begin{pro}\label{cp and fpa}
Let $\scr{A}:=(\{A_x\}_{x\in X},\Gamma_{\scr{A}})$ be the u.s.c. field associated to $A$. 
Then,
$$A\rtimes G\cong C_0\bra{X\times_{G,\Ad R\grotimes \alpha^A}\bbra{\bb{K}(L^2(G))\grotimes A_x}_{x\in X}},$$
where a section of the field $\bbra{\bb{K}(L^2(G))\grotimes A_x}_{x\in X}$ is continuous if and only if it can be approximated by finite sums of $k\grotimes a$ for $k\in \bb{K}(L^2(G))$ and $a\in \Gamma_{\scr{A}}$ (See also Definition \ref{dfn of tensor product of usc field of modules} for details).

Let $f\in C_c(G,A)$. The integral kernel of the corresponding equivariant section $X\to \bbra{\bb{K}(L^2(G))\grotimes A_x}_{x\in X}$ is given by
$$x\mapsto\bbbra{(g,h)\mapsto \mu(h)^{-1}\alpha^A_{g^{-1}}[f(gh^{-1})(gx)]},$$
where $\mu$ is the modular function of $G$. When we regard it as a function on $X\times G\times G$, we denote it by $k_f(g,h;x)$.
\end{pro}
\begin{rmk}
Note that $f(gh^{-1})(gx)$ is the evaluation of $f(gh^{-1})\in A$ at $g\cdot x$. Thanks to ``$\alpha^A_{g^{-1}}$'', $k_f(g,h;x)\in A_x$.
\end{rmk}

Under the above isomorphism, let us compute $[c_X]\in KK(\bb{C},C_0(X)\rtimes G)$.
Recall that $[c_X]$ is defined by the projection element $c(g)(x)=\mu(g)^{-1/2}\sqrt{\fra{c}(x)\fra{c}(g^{-1}\cdot x)}$ of $C_0(X)\rtimes G$ for a cut-off function $\fra{c}:X\to \bb{R}_{\geq 0}$. Thus, the integral kernel corresponding to $c$ is given by
$$k_{c}(g,h;x)=\mu(h)^{-1}c(gh^{-1})(g\cdot x)=\mu(g)^{-1/2}\sqrt{\fra{c}(g\cdot x)}\cdot \mu(h)^{-1/2}\sqrt{\fra{c}(h\cdot x)}.$$
Put $\fra{c}_x(g):=\mu(g)^{-1/2}\sqrt{\fra{c}(g\cdot x)}$. It is an $L^2$-unit vector; in fact, $\|\fra{c}_x\|^2_{L^2}=\int \fra{c}_x(g)\mu(g)^{-1}dg =\int \fra{c}(g\cdot x)\mu(g)^{-1}dg =\int \fra{c}(g^{-1}\cdot x)dg =1$.
Thus, the operator given by the above integral kernel is the rank one projection to $\bb{C}\sqrt{\fra{c}_x}$. This projection is denoted by $P_{\sqrt{\fra{c}_x}}$. Since the Hilbert $\bb{K}(L^2(G))$-module corresponding to a rank one projection is always isomorphic to $L^2(G)^*$ by $k\mapsto \sqrt{\fra{c}_x}^*\circ k$, we obtain the following.

\begin{lem}\label{lemma operator description of cX}
For a cut-off function $\fra{c}:X\to \bb{R}_{\geq 0}$, we put $\fra{c}_x(g):=\mu(g)^{-1/2}\sqrt{\fra{c}(g\cdot x)}$. Then, the corresponding $KK$-element $[c_X]$ is given by the family of rank one projections $x\mapsto P_{\sqrt{\fra{c}_x}}$.
In the language of modules, it is represented by
$$\bra{C_0\bra{X\times_{G,R}L^2(G)^*},1,0}.$$
\end{lem}

We would like to describe the descent homomorphism in the language of ``generalized fixed-point modules'' which is defined in an obvious way. Before that, we need to recall that Proposition \ref{crossed product and generalized fixed-point algebra} follows from the following lemma and the fact that $C_0(G)\rtimes G$ is isomorphic to  $\bb{K}(L^2(G))$.

\begin{lem}[{\cite[Lemma 2.9 and Lemma 2.8]{EE}}]\label{2.9}
$(1)$ Let $G$ be a locally compact second countable Hausdorff group and let $K$ be a compact subgroup of $G$. Let $A$ be a $G$-$C^*$-algebra, whose action is denoted by $\alpha$. Suppose that an action $\beta:K\to \Aut(A)$ commuting with $\alpha$ is given. Then, the inclusion $\iota:A^K\to A$ induces an isomorphism
$$A^K\rtimes G\to (A\rtimes G)^K,$$
where $B^K$ for a $K$-$C^*$-algebra $B$ is the subalgebra consisting of all $K$-invariant elements of $B$.

$(2)$ Let $\alpha$ and $\beta$ be commuting actions of $G$ on $A$. We define the following two $C^*$-algebras: Using the $G$-action on $A\rtimes_\alpha G$ given by $[\widetilde{\beta}_gF](h):=\beta_g[F(h)],$ we define $C_0(X\times_{G,\widetilde{\beta}}(A\rtimes_\alpha G))$; Using the $G$-action $\tau\grotimes\alpha$ on $C_0(X\times_{G,\beta}A)$ given by $[(\tau\grotimes \alpha)_gF](x):=\alpha_g[F(g^{-1}x)]$, we define the crossed product $C_0(X\times_{G,\beta}A)\rtimes_{\tau\grotimes \alpha} G$.
Then, they are isomorphic:
$$C_0(X\times_{G,\beta}A)\rtimes_{\tau\grotimes \alpha} G\cong C_0(X\times_{G,\widetilde{\beta}}(A\rtimes_\alpha G)).$$
\end{lem}
\begin{rmk}
$(2)$ follows from $(1)$ by the following argument: The fiber of the left hand side at $x\in X$ is $A^{G_x}\rtimes_{\alpha}G$; The fiber of the right hand side at $x$ is $[A\rtimes_{\alpha}G]^{G_x}$; They are isomorphic thanks to $(1)$. Proposition \ref{crossed product and generalized fixed-point algebra} follows from $(2)$ by a formal argument.
\end{rmk}


Let us describe the descent homomorphism. For simplicity, we deal with only actually equivariant Kasparov modules. Note that every $\ca{R}KK_G$-element has such a representative because we can use the averaging procedure thanks to the properness of the $G$-action on $X$. See \cite[Section 5]{Kas15}. See also Corollary \ref{Cor positivity of the commutator implies the homotopy invariance}.

\begin{pro}\label{crossed product and generalized fixed-point module} 
$(1)$ Let $G$ be a locally compact second countable Hausdorff group and let $X$ be a $\sigma$-compact locally compact Hausdorff space equipped with a proper $G$-action. We assume that $G$ is amenable, for simplicity.
Let $A$ and $B$ be $X\rtimes G$-$C^*$-algebras, and let $(E,\pi,F)$ be an $X\rtimes G$-equivariant Kasparov $(A,B)$-module. The $G$-action on $E$ is denoted by $\alpha^E$.
We assume that $F$ is actually equivariant: $g(F)=F$ for every $g\in G$. Then, the descent homomorphism is given by
$$j_G(E,\pi,F)=\bra{[\bb{K}(L^2(G))\grotimes E]^{G,\Ad R\grotimes \alpha^E},[\id\grotimes \pi]^G,[\id\grotimes F]^G},$$
where $[\id\grotimes \pi]^G(a)$ and $[\id\grotimes F]^G$ are induced maps from $\id\grotimes \pi(a)$ and $\id\grotimes F$ to the fixed-point module, respectively.

$(2)$ In the language of u.s.c. fields, $j_G(E,\pi,F)$ can be written as follows. Let $\scr{A}:=(\{A_x\}_{x\in X},\Gamma_{\scr{A}})$, $\scr{B}:=(\{B_x\}_{x\in X},\Gamma_{\scr{B}})$, $\scr{E}:=(\{E_x\}_{x\in X},\Gamma_{\scr{E}})$ be the u.s.c. fields associated to $A$, $B$ and $E$, respectively. Let $\pi=\{\pi_x\}_{x\in X}$ and $F=\{F_x\}_{x\in X}$ be the field description of $\pi $ and $F$. Then, $j_G\bra{\scr{E},\{\pi_x\}_{x\in X},\{ F_x\}_{x\in X}}$ is given by
$$\bra{
C_0\bra{X\times_{G,\alpha^E\grotimes \Ad R }\bbra{E_x\grotimes\bb{K}(L^2(G))}_{x\in X}},
\{\pi_x\grotimes\id\}_{x\in X},\{ F_x\grotimes\id\}_{x\in X}}.$$
\end{pro}
\begin{proof}
$(2)$ is obvious from $(1)$.

We first prove that $E\rtimes G$ is isomorphic to the fixed-point module $[\bb{K}(L^2(G))\grotimes E]^{G,\Ad R\grotimes \alpha^E}$ as bimodules. By the ``same'' argument of the proof of Proposition \ref{crossed product and generalized fixed-point algebra}, we have an isometric homomorphism from the left hand side to the right hand side as bimodules. This is possible because all the algebraic operations on $E\rtimes G$ (right $B\rtimes G$-action, left $A\rtimes G$-action and the $B\rtimes G$-valued inner product) are given by the parallel formulas of the corresponding operations on crossed products. For example, the right action of $b\in C_c(G,B)$ on $e\in C_c(G,E)$ looks like the ``multiplication of $e$ and $b$ in a crossed product algebra'', although $e$ and $b$ live in different places. 

This argument does not guarantee that the isometry is surjective.
Thus, what we need to prove is the analogue of Lemma \ref{2.9} $(1)$ for Hilbert modules:
{\it Let $K$ be a compact subgroup of $G$. We suppose that $K$ acts on $B$ and $E$, whose actions are denoted by $\beta^B:K\to \Aut(B)$ and $\beta^E:K\to \Aut(E)$, and we also suppose that they are compatible in the following sense $\beta^E_k(eb)=\beta^E_k(e)\beta^B_k(b)$ and $\beta^B_k(\inpr{e_1}{e_2}{})=\inpr{\beta^E_k(e_1)}{\beta^E_k(e_2)}{}$.
We assume that the $K$-actions $\beta$'s commute with the $G$-actions $\alpha$'s. Then, the inclusion $\iota:E^K\to E$ induces an isomorphism
$$E^{K,\beta^E}\rtimes_\alpha G\to (E\rtimes_\alpha G)^{K,\beta^E}$$
as Hilbert modules.}

We borrow an idea from \cite[Lemma 4.1]{LG}. We prove that this statement is derived from Lemma \ref{2.9} $(1)$. Let us consider the $C^*$-algebra $\bb{K}_B(E\oplus B)$. It is decomposed into four components:
$$\bb{K}_B(E\oplus B)=\begin{pmatrix} \bb{K}_B(E) & E \\ E^* & B\end{pmatrix}.$$
Note that the product on the matrix algebra $\bb{K}_B(E\oplus B)$ contains all the information on the Hilbert module $E$ as follows: For $k\in\bb{K}_B(E)$, $e,e_1,e_2\in E$ and $b\in B$, 
{\small 
$$
\begin{pmatrix} 0 & e \\ 0 & 0\end{pmatrix}\begin{pmatrix} 0 & 0 \\ 0 & b\end{pmatrix}
=\begin{pmatrix} 0 & eb \\ 0 & 0\end{pmatrix}; \ 
\begin{pmatrix} k & 0 \\ 0 & 0\end{pmatrix}\begin{pmatrix} 0 & e \\ 0 & 0\end{pmatrix}
=\begin{pmatrix} 0 & ke \\ 0 & 0\end{pmatrix}; \text{ and }
\begin{pmatrix} 0 & 0 \\ e_1^* & 0\end{pmatrix}\begin{pmatrix} 0 & e_2 \\ 0 & 0\end{pmatrix}
=\begin{pmatrix} 0 & 0 \\ 0 & \inpr{e_1}{e_2}{B}\end{pmatrix},
$$}
where $e_1^*$ is the map $x\mapsto \inpr{e_1}{x}{B}$.

Moreover, the group action preserves the decomposition into $2\times 2$-matrices:
$$g\cdot \begin{pmatrix} k & e \\ f & b\end{pmatrix}=\begin{pmatrix} g(k) & g(e) \\ g(f) & g(b)\end{pmatrix}$$
for $g\in G$, $k\in\bb{K}_B(E)$, $e\in E$, $f\in E^*$ and $b\in B$, and similarly for the $K$-action.
Thus, we have
$$\bb{K}_B(E\oplus B)^K=\begin{pmatrix} \bb{K}_B(E)^K & E^K \\ [E^{*}]^K & B^K\end{pmatrix}\text{ and}$$
$$\bb{K}_B(E\oplus B)\rtimes G=\begin{pmatrix} \bb{K}_B(E)\rtimes G & E\rtimes G \\ E^*\rtimes G & B\rtimes G\end{pmatrix}.$$

Applying Lemma \ref{2.9} $(1)$ to $\bb{K}_B(E\oplus B)$, we obtain the isomorphism
$$\begin{pmatrix} \bb{K}_B(E)^K\rtimes G & E^K\rtimes G \\ [E^*]^K\rtimes G & B^K\rtimes G\end{pmatrix}\cong \begin{pmatrix} [\bb{K}_B(E)\rtimes G]^K & [E\rtimes G]^K \\ [E^*\rtimes G]^K & [B\rtimes G]^K\end{pmatrix}$$
as {\it $C^*$-algebras}. Therefore, we have an isomorphism between $E^K\rtimes G$ and $[E\rtimes G]^K$ as {\it Hilbert modules}.

Finally, we prove the property on $\widetilde{F}$. We use the language of fields and integral kernels. Let $F=\{F_x\}_{x\in X}$. For $e\in C_c(G,E)$, the integral kernel of the corresponding element of $[\bb{K}(L^2(G))\grotimes E]^G$ is given by
$k_e(g_1,g_2;x)=\mu(g_2)^{-1}g_1^{-1}[e(g_1g_2^{-1})(g_1x)].$
Thus,
\begin{align*}
k_{\widetilde{F}e}(g_1,g_2;x) 
&=\mu(g_2)^{-1}g_1^{-1}\bbbra{\bbra{\widetilde{F}e}(g_1g_2^{-1})(g_1x)} \\
&=\mu(g_2)^{-1}g_1^{-1}\bbbra{F_{g_1x}\bbra{e(g_1g_2^{-1})(g_1x)}} \\
&=F_x\bbra{\mu(g_2)^{-1}g_1^{-1}[e(g_1g_2^{-1})(g_1x)] } \\
&=F_x[k_{e}(g_1,g_2;x) ],
\end{align*}
where we have used the actual $G$-equivariance of $F$ at the third equality.
\end{proof}

The final tool to define the original topological assembly map is the ``descent of the Dirac element'' $j_G([d_X])$. 
According to Definition \ref{dfn reformulated KK elements}, for the reformulated one, we need 
$$\sigma_{\ca{S}_\vep}\circ j_G(\fgt[S]\grotimes [d_X])\in KK(\ca{S}_\vep\grotimes (C_0(X)\rtimes G),\ca{S}_\vep\grotimes (\bb{C}\rtimes G)).$$
Thus, we compute $j_G(\fgt[S]\grotimes [d_X])$.

The $KK$-element $\fgt[S]\grotimes [d_X]$ is given by the index element $[D]$ of the following $Spin^c$-Dirac operator $D$. In fact, since $\Cl_+(TX)\cong S^*\grotimes S$, we have $C_0(X,S)\grotimes L^2(X,\Cl_+(TX))=L^2(X,S)$. Put $D:=\sum_nc(v_n)\circ \nabla^S_{v_n}$ for an orthonormal base $\{v_n\}$. $\pi$ denotes left multiplication of $C_0(X)$ on $L^2(X,S)$. Then, $\fgt[S]\grotimes [d_X]$ can be represented by
$\bra{L^2(X,S),\pi,D}$. 

We explicitly compute $j_G([D])$ with the same spirit of Proposition \ref{crossed product and generalized fixed-point module}. We do not need to assume that $S$ is a Spinor bundle, and so we compute it in the following (slightly more general) situation: {\it On a Clifford bundle $E$ equipped with a $G$-equivariant Clifford multiplication $c:TX\to \End(E)$ and a $G$-equivariant Clifford connection $\nabla^E$, we have an equivariant Dirac operator $D= \sum_nc(v_n)\circ \nabla^E_{v_n}$.}
Note that 
$G $ has a unitary representation on $L^2(G )$ by $R _g \phi(h):=\sqrt{\mu(g)}\phi(h g)$, and $G$ acts on $(\bb{C}\rtimes G)  $ by $\rt_g (b)(h):=b(hg)$. The latter action gives Hilbert $\bb{C}\rtimes G$-module automorphisms with respect to the following Hilbert module structure: $b\cdot a:=a^\vee* b$ and $\inpr{b_1}{b_2}{\bb{C}\rtimes G}:=(b_2*b_1^*)^\vee$ for $a,b,b_1,b_2\in \bb{C}\rtimes G$, where $b^\vee(g):=\sqrt{\mu(g)}^{-1}b(g^{-1})$.
See Definition \ref{dfn of an ind} for the origin of this Hilbert module structure.

The following definition will be justified in the following proposition.


\begin{dfn}\label{def of descented Dirac element bimodule}
Let $X$ be a complete Riemannian manifold equipped with an isometric proper action of a locally compact second countable Hausdorff group $G $. For simplicity, we suppose that $G$ is amenable.
Let $E$ be a $G $-equivariant Clifford bundle over $X$ equipped with a $G $-equivariant Dirac operator $D$.

$(0)$ Let $C_c\bra{X\times_{G ,\alpha^E\grotimes R \grotimes \rt} \{E\grotimes L^2(G )\grotimes (\bb{C}\rtimes G)\}}$ be the set of compactly supported continuous sections $k:X\times G \times G \to E$  satisfying the equivariance condition 
$$k(g_1,g_2;x)=\sqrt{\mu(g)}\alpha^E_gk(g_1g,g_2g;g^{-1}x),$$
where ``compactly supported'' means that the closure of the set $\bbra{x\in X\,\middle|\, k(\bullet,\bullet,x)\neq 0}/G$ is compact.
It has a pre-Hilbert $C_c(G)$-module structure by the following operations: For $k,k_1,k_2\in C_c\bra{X\times_{G ,\alpha^E\grotimes R \grotimes \rt} \{E\grotimes L^2(G )\grotimes (\bb{C}\rtimes G) \}}$, $b\in C_c(G )$, $x\in X$ and $g,g_1,g_2\in G $.
$$\inpr{k_1}{k_2}{(\bb{C}\rtimes G)  }(g):=\sqrt{\mu(g)}^{-1}\int_Xc(x)\int_G \int_G  \inpr{k_1(\eta,\xi;x)}{k_2(\eta,g^{-1}\xi ;x)}{E}d\eta d\xi dx;$$
$$k\cdot b(g_1,g_2;x):=\int_G  k(g_1,\eta^{-1}g_2;x)b(\eta^{-1})\sqrt{\mu(\eta)}^{-1}d\eta.$$

$(1)$ We define a Hilbert $(\bb{C}\rtimes G)  $-module
$$L^2\bra{X\times_{G ,\alpha^E\grotimes R \grotimes \rt} \{E\grotimes L^2(G )\grotimes (\bb{C}\rtimes G) \}}$$
by the completion of $C_c\bra{X\times_{G ,\alpha^E\grotimes R \grotimes \rt} \{E\grotimes L^2(G )\grotimes (\bb{C}\rtimes G)\}}$ with respect to the above inner product.

$(2)$ We define a $*$-homomorphism 
$$\pi\rtimes \lt:C_0(X)\rtimes G \to \bb{L}_{\bb{C}\rtimes G}\bra{L^2\bra{X\times_{G ,\alpha^E\grotimes R \grotimes \rt} \{E\grotimes L^2(G )\grotimes (\bb{C}\rtimes G) \}}}$$
under the identification $C_0(X)\rtimes G \cong C_0(X\times_G  \bb{K}(L^2(G )))$, by the formula
$$[\pi\rtimes \lt  (a)k](g_1,g_2;x):=\int_G  k_a(g_1,\eta;x)k(\eta,g_2;x)d\eta .$$

$(3)$ For $k\in C_c\bra{X\times_{G ,\alpha^E\grotimes R \grotimes \rt} \{E\grotimes L^2(G )\grotimes (\bb{C}\rtimes G)\}}$, we define an element $k(g_1,g_2)\in C_c(X,E)$ by $k(g_1,g_2)(x):=k(g_1,g_2;x)$. We define $C_c^\infty\bra{X\times_{G ,\alpha^E\grotimes R \grotimes \rt} \{E\grotimes L^2(G )\grotimes (\bb{C}\rtimes G)\}}$ by the subset consisting of all smooth sections.
Associated to $D$, we define an unbounded operator $\overline{D}$ on $C_c^\infty\bra{X\times_{G ,\alpha^E\grotimes R \grotimes \rt} \{E\grotimes L^2(G )\grotimes (\bb{C}\rtimes G)\}}$ by
$$\overline{D}(k)(g_1,g_2;x):=D[k(g_1,g_2)](x).$$
Since $\overline{D}(k)$ is smooth and $G$-invariant, $\overline{D}$ is well-defined as a map on $C_c^\infty\bra{X\times_{G ,\alpha^E\grotimes R \grotimes \rt} \{E\grotimes L^2(G )\grotimes (\bb{C}\rtimes G)\}}$. Its extension to an appropriate domain is denoted by the same symbol.
\end{dfn}

\begin{pro}\label{formula on descent of Dirac element}
$(0)$ $C_c\bra{X\times_{G ,\alpha^E\grotimes R \grotimes \rt} \{E\grotimes L^2(G )\grotimes (\bb{C}\rtimes G)\}}$ is a pre-Hilbert $C_c(G)$-module and $\pi\rtimes \lt$ is actually a $*$-homomorphism. 

$(1)$ This $(C_0(X)\rtimes G, \bb{C}\rtimes G)  $-bimodule is isomorphic to $L^2(X,E)\rtimes G $ by the following correspondence: For $e\in C_c(G , C_c(X,E))\subseteq L^2(X,E)\rtimes G $, we define $k_e$ 
by
$$k_e(x)(g_1,g_2):=\sqrt{\mu(g_2)}^{-1}g_1^{-1}\bbbra{e(g_1g_2^{-1},g_1x)}.$$
We regard it as ``a function on $G\times G$ $\grotimes$ an element of $E_x$''.
This family defines an element of 
$C_c\bra{X\times_{G ,\alpha^E\grotimes R \grotimes \rt} \{E\grotimes L^2(G )\grotimes (\bb{C}\rtimes G)  \}}$.

$(2)$ For $e\in C_c(G,C_c(X,E))$, 
$$\overline{D}k_e=k_{\widetilde{D}e}.$$
Thus, we denote $\overline{D}$ by $\widetilde{D}$ from now on.
Consequently, $j_G ([D])$ is represented by
$$\bra{L^2\bra{X\times_{G ,\alpha^E\grotimes R \grotimes \rt} \{E\grotimes L^2(G )\grotimes (\bb{C}\rtimes G) \}},\pi\rtimes \lt  ,\widetilde{D}}.$$
\end{pro}

\begin{proof}
$(0)$ Simple computations show the statement. We leave it to the reader.

$(1)$ We denote the correspondence $e\mapsto k_e$ by $\Psi$. We need to check the following things: $(a)$ $\Psi$ is a left module homomorphism; $(b)$ $\Psi$ is a right module homomorphism; $(c)$ $\Psi$ is isometric; and $(d)$ The image of $\Psi$ is dense.

$(a)$ For $e\in C_c\bra{X\times_{G } \{E\grotimes L^2(G )\grotimes (\bb{C}\rtimes G)\}}$ and $a\in C_c(X\times_G \bb{K}(L^2(G )))$,
\begin{align*}
&k_{\pi\rtimes \lt  (a)(e)}(g_1,g_2;x)\\
&\ \ \ =\sqrt{\mu(g_2)}^{-1}\alpha^E_{g_1^{-1}}\bbbra{\pi\rtimes \lt  (a)(e)(g_1g_2^{-1};g_1x)} \\
&\ \ \ =\sqrt{\mu(g_2)}^{-1}\alpha^E_{g_1^{-1}}\bbbra{\int_G  a(\eta;g_1x)\alpha^E_\eta[e(\eta^{-1}g_1g_2^{-1};\eta^{-1}g_1x)]}d\eta \\
&\ \ \ = \sqrt{\mu(g_2)}^{-1}\alpha^E_{g_1^{-1}}\bbbra{\int_G 
\mu(\eta^{-1}g_1)\alpha^E_{g_1}[k_a(g_1,\eta^{-1}g_1;x)]\alpha^E_\eta\bra{
\sqrt{\mu(g_2)}\alpha^E_{\eta^{-1}}\alpha^E_{g_1}[k_e(\eta^{-1}g_1,g_2;x)]}d\eta} \\
&\ \ \ = \int_G  k_a(g_1,\eta^{-1}g_1;x)k_e(\eta^{-1}g_1,g_2;x)\mu(\eta^{-1}g_1)d\eta \\
&\ \ \ =\int_G  k_a(g_1,\eta^{-1};x)k_e(\eta^{-1},g_2;x)\mu(\eta^{-1})d\eta \\
&\ \ \ =\int_G  k_a(g_1,\eta;x)k_e(\eta,g_2;x)d\eta \\
&\ \ \ =[\pi\rtimes \lt  (a)k](x)(g_1,g_2),
\end{align*}
where we used the definitions of $k_a$ and $k_e$ at the second equality, the left invariance of the measure at the fifth equality, and the property of the modular function at the sixth equality.

$(b)$ It is obtained by a similar calculation of $(a)$.

$(c)$ 
For $e_1,e_2\in C_c(G ,L^2(X,E))$,
\begin{align*}
&\inpr{k_{e_1}}{k_{e_2}}{(\bb{C}\rtimes G)  }(g ) \\
&\ \ \ =\sqrt{\mu(g)}^{-1}\int_Xc(x)\int_G \int_G  \inpr{k_{e_1}(g_1,g_2;x)}{k_{e_2}(g_1,g^{-1}g_2 ;x)}{E_x}dg_1 dg_2 dx\\
&\ \ \ =\sqrt{\mu(g)}^{-1}\int_Xc(x)\int_G \int_G  \inpr{\sqrt{\mu(g_2)}^{-1}\alpha^E_{g_1^{-1}}e_1(g_1g_2^{-1};g_1 x)}{\sqrt{\mu(g^{-1}g_2)}^{-1}\alpha^E_{g_1^{-1}}e_2(g_1g_2^{-1}g ;g_1 x)}{E_x} dg_1 dg_2dx \\
&\ \ \ =\int_G \int_G \int_X c(x)\inpr{e_1(g_1g_2^{-1};g_1 x)}{e_2(g_1g_2^{-1}g ;g_1 x)}{E_{g_1x}} \mu(g_2)^{-1}dxdg_2 dg_1  \\
&\ \ \ =\int_G \int_G \int_X c(g_1^{-1}x)\inpr{e_1(g_2; x)}{e_2(g_2g ;x)}{E_x} dxdg_2 dg_1  \\
&\ \ \ =\int_G \inpr{e_1(g_2)}{e_2(g_2g )}{L^2(X,E)}dg_2 \\
&\ \ \ =\inpr{e_1}{e_2}{(\bb{C}\rtimes G)  }(g ),
\end{align*}
where we have used the definition of $k_e$ at the second equality, Fubini's theorem at the third one, the $G $-invariance of the measure on $X$ and the left invariance of the measure on $G $ at the fourth one, and the definition of the cut-off function at the fifth one.

$(d)$ For $k\in C_c\bra{X\times_{G ,\alpha^E\grotimes R \grotimes \rt} \{E\grotimes L^2(G )\grotimes (\bb{C}\rtimes G)\}}$, we prove that there exists $e\in C_c(G ,C_c(X,E))\subseteq L^2(X,E)\rtimes G$ such that $k_e=k$. In fact, if we put $e(\eta;x):=\sqrt{\mu(\eta)}^{-1}k(1,\eta^{-1};x)$, we obtain
\begin{align*}
k_e(g_1,g_2;x)
&=\sqrt{\mu(g_2)}^{-1}\alpha^E_{g_1^{-1}}e(g_1g_2^{-1};g_1x) \\
&=\sqrt{\mu(g_2)}^{-1}\alpha^E_{g_1^{-1}}\bbbra{\sqrt{\mu(g_2g_1^{-1})}k(1,g_2g_1^{-1};g_1x)}\\
&=\sqrt{\mu(g_2)}^{-1}\alpha^E_{g_1^{-1}}\bbbra{\sqrt{\mu(g_2g_1^{-1})}\sqrt{\mu(g_1)}\alpha^E_{g_1}k(1\cdot g_1,g_2g_1^{-1}g_1;g_1^{-1}g_1x)}\\
&=k(g_1,g_2;x),
\end{align*}
where we used the $G $-invariance of $k$ at the third equality.

$(2)$ One can easily prove this statement by using the definition of $k_e$ and the $G $-invariance of $D$. We leave it to the reader.

Moreover we can define descent homomorphism in terms of only unbounded Kasparov modules for actually equivariant unbounded Kasparov modules \cite[Definition-Proposition 2.10]{T4}, by the same formula. Thus, this result immediately follows from $(1)$ and $(2)$.
\end{proof}
\begin{rmks}\label{rmk exposition of bimodule structure on decent family descr}
$(1)$ On the fiber $E_x\grotimes L^2(G )\grotimes (\bb{C}\rtimes G)$ at $x\in X$, we can define a Hilbert $\bb{C}\rtimes G$-module structure by the following operations:
$$(e\grotimes \phi\grotimes b')\cdot b:=e\grotimes \phi\grotimes [b^\vee *b'],$$
$$\inpr{e_1\grotimes \phi_1\grotimes b_1}{e_2\grotimes \phi_2\grotimes b_2}{\bb{C}\rtimes G}:=\inpr{e_1}{e_2}{E_x}\inpr{\phi_1}{\phi_2}{L^2(G)}(b_2*b_1^*)^\vee,$$
where $b^\vee(g):=\sqrt{\mu(g)}^{-1}b(g^{-1})$. Moreover, $(\pi\rtimes \lt)_x(a_x)$ is given by $\id\grotimes a_x\grotimes \id$ for $a_x\in \bb{K}(L^2(G))$.
These operations vary continuously on $X$, and hence we obtain a locally trivial bundle of Hilbert $\bb{C}\rtimes G$-module. This structure naturally induces the above bimodule structure to the section space.

$(2)$ This Hilbert module bundle structure is given by the tensor product of the locally trivial Hilbert bundle $E\grotimes L^2(G)$ and the trivial Hilbert $\bb{C}\rtimes G$-module bundle $X\times \bb{C}\rtimes G$. Thus, in order to use the formulas, we will often denote  {\bf symbolically} a section of this bundle as $\phi\grotimes \psi$, where $\phi$ is regarded as a map $X\to E\grotimes L^2(G)$ and $\psi$ is regarded as a map $X\to \bb{C}\rtimes G$. In this notation, the above bimodule structure can be described as follows:
$$[\pi\rtimes \lt(a)\phi\grotimes \psi](x)=[\id_{E_x}\grotimes a(x)]\phi(x)\grotimes \psi(x),$$
$$[\phi\grotimes \psi\cdot b](x) =\phi(x)\grotimes [b^\vee*\psi(x)],$$
$$\inpr{\phi_1\grotimes \psi_1}{\phi_2\grotimes \psi_2}{(\bb{C}\rtimes G) }=\int_X c(x)\inpr{\phi_1(x)}{\phi_2(x)}{E_x\grotimes L^2(G)}[\psi_{2}(x)*\psi_{1}(x)^*]^\vee dx,$$
where $b^\vee(g):=\sqrt{\mu(g)}^{-1}b(g^{-1})$. These formulas will be adopted as the definition of the descent homomorphism for proper $LT$-spaces.
\end{rmks}



We can further rewrite these simple formulas in much more algebraic way, under the following assumptions.

\begin{asm}\label{principal bundle situation}
Until the next subsection, we suppose the following conditions on $X$ and $G$:
\begin{itemize}
\item $G$ is a finite-dimensional unimodular Lie group, and $H$ is a closed Lie subgroup of the center of $G$;
\item The $H$-action given by the restriction of that of $G$, is free and smooth; and
\item The orbit map $H\ni h\mapsto h\cdot x\in X$ is isometric for each $x\in X$.
\end{itemize}
\end{asm}

For $v\in \fra{h}=\Lie(H)$, we denote the infinitesimal action of $v$ on $L^2(X,E)$ by $d\alpha^E_v$: 
$$d\alpha^E_vs(x)=\left. \frac{d}{dt}\right|_{t=0}\alpha^E_{\exp(tv)}s(\exp(-tv)\cdot x).$$
The associated operator on $L^2(X,E)\rtimes G$ is denoted by $\widetilde{d\alpha^E_v}$ as usual.

\begin{lem}\label{Lemma base differential is fiber differential}
For $e\in C_c^\infty(G,C_c^\infty(X,E))$ and $v\in \fra{h}$, $k_{\widetilde{d\alpha^E_v}e}$ is computed as follows:
$$k_{{\widetilde{d\alpha^E_v}}e}(x)=
dR_{-v}k_e(x)+d\rt_{-v}k_e(x).$$
\end{lem}

\begin{proof}
The following computation shows it:
\begin{align*}
k_{\widetilde{d\alpha^E_v}e}(g_1,g_2;x)
&=\alpha^E_{g_1^{-1}}{d\alpha^E_v}e(g_1g_2^{-1};g_1x) \\
&=\alpha^E_{g_1^{-1}}\left. \frac{d}{dt}\right|_{t=0}{\alpha^E_{\exp(tv)}}e(g_1g_2^{-1};\exp(-tv)g_1x) \\
&=\alpha^E_{g_1^{-1}}\left. \frac{d}{dt}\right|_{t=0}{\alpha^E_{\exp(tv)}}e(g_1\exp(-tv)[g_2\exp(-tv)]^{-1};[g_1\exp(-tv)]x) \\
&=\alpha^E_{g_1^{-1}}\left. \frac{d}{dt}\right|_{t=0}{\alpha^E_{\exp(tv)}}\alpha^E_{g_1\exp(-tv)}k_e(g_1\exp(-tv),g_2\exp(-tv);x) \\
&= \left. \frac{d}{dt}\right|_{t=0}k_e(g_1\exp(-tv),g_2\exp(-tv);x) \\
&=dR_{-v}k_e(g_1,g_2;x)+d\rt_{-v}k_e(g_1,g_2;x).
\end{align*}
\end{proof}

\begin{ex}
Suppose that $H= G=\bb{R}$, $X=\bb{R}^2$ and $E=X\times\bb{C}$. Let us consider the $G$-action on $X$ given by $g:(x,y)\mapsto (x+g,y)$, and its lift on $E$ given by $g:((x,y),z)\mapsto ((x+g,y),z)$. We denote the infinitesimal generator $1\in\fra{g}$ of $G$ by $v$.
Then, $d\alpha^E_v=-\frac{\partial}{\partial x}$.

Let $e:G\to C_c^\infty(X,E)$ be a smooth function. Then, $k_e$ is a smooth function on $G\times G\times X=\bb{R}\times\bb{R}\times\bb{R}^2$. The coordinate is denoted by $(g_1,g_2;x,y)$. Then, $dR_v=\frac{\partial}{\partial g_1}$, $d\rt_v=\frac{\partial}{\partial g_2}$. Since $k_e(g_1,g_2;x,y)= e(g_1-g_2;x+g_1,y)$, we have
\begin{align*}
k_{\widetilde{d\alpha^E_v}e}(g_1,g_2;x,y)
&=-\frac{\partial}{\partial x}e(g_1-g_2;x+g_1,y)\\
&=-\frac{\partial}{\partial g_1}e(g_1-g_2;x+g_1,y)-\frac{\partial}{\partial g_2}e(g_1-g_2;x+g_1,y)\\
&=dR_{-v}k_e(g_1,g_2;x,y)+d\rt_{-v}k_e(g_1,g_2;x,y).
\end{align*}
This is the most fundamental case of the above formula.
\end{ex}
\begin{nota}
In order to emphasize how $\widetilde{d\alpha^E_v}$ acts on each fiber, we denote the above formula by $k_{\widetilde{d\alpha^E_v}e}
=\id_{E}\grotimes dR_{-v}\grotimes \id (k_e)+\id_{E}\grotimes \id\grotimes d\rt_{-v}(k_e)$.
\end{nota}

Let $D$ be the Dirac operator given by
$$\sum_ic(e_i)d\alpha^E_{-e_i}+D_{\text{base}},$$
where $c(e_i)$ is the Clifford multiplication of the vector field induced by $e_i\in\fra{h}$, and $D_{\text{base}}=\sum_{n}c(v_n)\circ \nabla^E_{v_n}$ is the ``Dirac operator for the base direction'', where $\{v_n\}$ is an orthonomal base of the normal bundle of the $H$-orbit.

The following is fundamental.

\begin{lem}\label{Lemma invariant section vs restr to quot space}
Let $V\to X$ be a $G$-equivariant vector bundle (it can be a $C^*$-algebra bundle, a Hilbert space bundle, or a Hilbert module bundle). We say two elements $v,w\in V$ are $H$-equivalent if $w=h\cdot v$ for some $h\in H$. Consequently, if $v\in V_x$, $w=h\cdot v\in V_{hx}$. Then, the quotient space under the $H$-equivalence admits a $G/H$-equivariant vector bundle structure over $X/H$. 
A $G$-invariant section on $V$ defines a $G/H$-invariant section on $V/H$ and vice versa.
\end{lem}
\begin{nota}
By regarding $V/H$ as the ``restriction of $V$ to local slices'', and taking into account the above lemma, we often denote $C_c(X\times_G V)$ by $C_c(X/H\times_{G/H}V)$.
Following this notation, the function spaces appearing in Definition \ref{def of descented Dirac element bimodule} are also denoted by
$C_0\bra{X/H\times_{G/H}\bb{K}(L^2(G))}$, \\
$C_0\bra{X/H\times_{G/H}\{E\grotimes \bb{K}(L^2(G))\}}$ and $L^2\bra{X/H\times_{G/H}\{E\grotimes L^2(G)\grotimes (\bb{C}\rtimes G) \}}.$
\end{nota}


Let us rewrite $j_G([D])$ in an algebraic way.
A cut-off function $\fra{c}:X\to\bb{R}_{\geq 0}$ induces a cut-off function $\overline{\fra{c}}$ on $X/H$ by the orbit integral $\overline{\fra{c}}(x):=\int_H\fra{c}(h^{-1}\widetilde{x})dh$, where $x\in X/H$ and $\widetilde{x}$ is a lift of $x$. The following is obvious by the arguments so far.

\begin{pro}\label{descented Dirac for principal bundle}
$(1)$ $L^2\bra{X/H\times_{G/H}\{E\grotimes L^2(G)\grotimes (\bb{C}\rtimes G) \}}$ admits a Hilbert $(\bb{C}\rtimes G)$-module structure given by the following operations: For 
$$\phi\grotimes \psi,\phi_1\grotimes \psi_1,\phi_2\grotimes \psi_2,\in L^2\bra{X/H\times_{G/H}\{E\grotimes L^2(G)\grotimes (\bb{C}\rtimes G) \}}$$
 (for the remark on this notation, see Remark \ref{rmk exposition of bimodule structure on decent family descr} (2)) and $b\in \bb{C}\rtimes G$,
\begin{itemize}
\item $[\phi\grotimes\psi]\cdot b(x)=\phi(x)\grotimes [b^\vee*\psi(x)]$; and
\item $\inpr{\phi_1\grotimes\psi_1}{\phi_2\grotimes\psi_2}{(\bb{C}\rtimes G)  }=\int_{X/H}\overline{\fra{c}}(x)\inpr{\phi_1(x)}{\phi_2(x)}{E_x\grotimes L^2(G)}[\psi_2(x)*\psi_1(x)^*]^\vee dx$.
\end{itemize}

$(2)$ This Hilbert module admits a left module structure
$$\pi\rtimes \lt :{C_0\bra{X/H\times_{G/H}\bb{K}(L^2(G))}}\to\bb{L}_{\bb{C}\rtimes G}\bra{L^2\bra{X/G\times_{G/H}\{E\grotimes L^2(G)\grotimes (\bb{C}\rtimes G) \}}}$$
given by $[\pi\rtimes \lt (a)\phi\grotimes\psi](x)=[\id_{E_x}\grotimes a(x)\phi(x)]\grotimes\psi(x)$. 

$(3)$ $\overline{D}$ denotes the operator
$$\sum_i[\id_{E}\grotimes dR_{v_i}\grotimes \id+ \id_{E}\grotimes \id\grotimes d\rt_{v_i}]\grotimes c(v_i)+D_{\rm base}.$$
Then, $j_G([D])$ is represented by
$$\bra{L^2\bra{X/H\times_{G/H}\{E\grotimes L^2(G)\grotimes (\bb{C}\rtimes G) \}},\pi\rtimes \lt ,\overline{D}}.$$
\end{pro}

\subsection{Twisted equivariant version}\label{section index theorem twisted equivariant}

In the previous three subsections, we supposed that the $G$-action on $X$ lifts to the Clifford bundle. In this subsection, we will study the case of a Clifford bundle over $X$ which admits an action of a {\it $U(1)$-central extension of $G$} which is compatible with the $G$-action on $X$. It is a natural situation for Hamiltonian loop group spaces. 
The formulas proved this subsection will be adopted as definitions in Section \ref{section top ass map}.

More concretely, we will describe parallel results for the following setting.
For a continuous $U(1)$-central extension of $G$
$$1\to U(1)\xrightarrow{i} G^\tau \xrightarrow{p} G\to 1,$$
we suppose that $G$ (not $G^\tau$) acts on a complete Riemannian manifold $X$ in an isometric, proper and cocompact way, and that $X$ has a $G$-equivariant Spinor bundle $(S,c)$.
We consider a Hermitian vector bundle $\pi:F\to X$ equipped with a $G^\tau$-action satisfying $\alpha^F_{i(z)}=z\id_F$ and $\pi(g\cdot f)=p(g)\cdot \pi(f)$ for $z\in U(1)$, $g\in \widetilde{G}$ and $f\in F$. 
Such a vector bundle is said to be {\bf $\tau$-twisted $G$-equivariant}.
{\it We would like to study the $\tau$-twisted $G$-equivariant index of $S\grotimes F$.}
Needless to say, it is possible to deal with this problem as a $G^\tau$-twisted problem. However, we need more economical formulas in order to generalize them to infinite-dimensional manifolds.
We refer to \cite[Section 3.3]{Thesis} for the detailed arguments.

\begin{nota}
Let 
$1\to U(1)\xrightarrow{i} G^\tau \xrightarrow{p} G\to 1$ be a $U(1)$-central extension of $G$.

$(1)$ A function $f$ from $G^\tau$ to a vector space (Hilbert spaces, $C^*$-algebras, Hilbert modules, and so on) is said to be {\bf at level $q$} if $f(i(z)g)=z^qf(g)$. 

$(2)$ The set of compactly supported continuous function on $G^\tau$ at level $q$, is denoted by $C_c(G,q\tau)$. Other types of function spaces are denoted in the same way: $C_c^\infty(G,q\tau)$, $L^2(G,q\tau)$, and so on.

$(3)$ Let $A$ be a $G$-$C^*$-algebra. It is automatically a $G^\tau$-$C^*$-algebra through the homomorphism $p:G^\tau\to G$. The subalgebra $A\rtimes_{q\tau}G$ of $A\rtimes G^\tau$ is defined by the completion of $C_c(G,q\tau)$ in $A\rtimes G^\tau$. It is called the {\bf $q\tau$-twisted crossed product of $A$ by $G$}. For a Hilbert $A$-module $E$, we define $E\rtimes_{q\tau}G$ in the same way.
\end{nota}

\begin{dfn}\label{twisted equivariant KK theory}
We suppose the same conditions.

$(1)$ Let $A$ and $B$ be separable $G$-$C^*$-algebras. They are  automatically equipped with the $G^\tau$-actions.
A $G^\tau$-equivariant Kasparov module $(E,\pi,F)$ satisfying the following  is said to be {\bf $q\tau$-twisted $G$-equivariant}:
\begin{center}
$\alpha^E_{i(z)}(e)=z^qe$ for all $z\in U(1)$.
\end{center}
We define $KK_G^{q\tau}(A,B)$ by the set of homotopy classes of $q\tau$-twisted $G$-equivariant Kasparov $(A,B)$-modules for $q\in \bb{Z}$.
These are direct summands of $KK_{G^\tau}(A,B)$, that is to say, 
$KK_{G^\tau}(A,B)=\bigoplus_{q\in\bb{Z}} KK_G^{q\tau}(A,B).$

$(2)$ Let $(E,\pi,F)$ be a $q\tau$-twisted $G$-equivariant Kasparov $(A,B)$-module. Then, we can define a Kasparov $(A\rtimes_{(p-q)\tau}G,B\rtimes_{p\tau}G)$-module $(E\rtimes_{p\tau}G,\pi\rtimes \lt ,\widetilde{F}|_{E\rtimes_{p\tau}G})$. By the correspondence $(E,\pi,F)\mapsto (E\rtimes_{p\tau}G,\pi\rtimes \lt ,\widetilde{F}|_{E\rtimes_{p\tau}G})$, we define a homomorphism
$$j_G^{p\tau}:KK_G^{q\tau}(A,B)\to KK(A\rtimes_{(p-q)\tau}G,B\rtimes_{p\tau}G).$$
It is called the {\bf partial descent homomorphism}.
When we should emphasize ``$q$'' in the partial descent homomorphism, $j_G^{p\tau}$ is denoted by $j_{G,q}^{p\tau}$.

$(3)$ In the same situation, we define the analytic assembly map
\begin{align*}
\mu_G^{p\tau}&:KK_G^{p\tau}(C_0(X),\bb{C})\to 
KK(\bb{C},\bb{C}\rtimes_{p\tau}G)
\end{align*}
by $\mu_G^{p\tau}(x):=[c_X]\grotimes_{C_0(X)\rtimes G}j_{G,p}^{p\tau}(x)$ (see the following remark).

$(4)$ We do the parallel constructions for $\ca{R}KK$-theory: 
$$\ca{R}KK_G^{q\tau}(X;A,B)\subseteq \ca{R}KK_{G^\tau}(X;A,B),$$ 
$$j_G^{p\tau}:\ca{R}KK_G^{q\tau}(X;A,B)\to \ca{R}KK(X;A\rtimes_{(p-q)\tau}G,B\rtimes_{p\tau}G)\text{ and}$$
$$\nu_G^{p\tau}:\ca{R}KK_G^{p\tau}(X;C_0(X),Cl_\tau(X)) \to
KK(\bb{C},\bb{C}\rtimes_{p\tau}G).$$

\end{dfn}
\begin{rmks}
$(1)$ We use the notation of the twisted $K$-theory, because a central extension of a group acting on a space is a special example of a twisting. See \cite{FHTI} for the details.

$(2)$ The Kasparov product takes the following form:
$$KK_G^{p\tau}(A,A_1)\times KK_G^{q\tau}(A_1,B)\to KK_G^{(p+q)\tau}(A,B).$$
In particular, we can still define the Poincar\'e duality homomorphism by the same formula:
$$\PD:KK_G^{p\tau}(C_0(X),\bb{C})\cong \ca{R}KK_G^{p\tau}(X;C_0(X),Cl_\tau(X)).$$
Moreover, $\mu_G^{p\tau}=\nu_G^{p\tau}\circ \PD$.
The reformulated version 
$$\PD:KK_G^{p\tau}(\ca{A}(X),\ca{S}_\vep)\cong 
\ca{R}KK_G^{p\tau}(X;\ca{S}_\vep\grotimes C_0(X),\ca{S}_\vep\grotimes C_0(X)),$$
$$\mu_G^{p\tau}=\nu_G^{p\tau}\circ \PD:KK_G^{p\tau}(\ca{A}(X),\ca{S}_\vep)\to KK(\ca{S}_\vep,\ca{S}_\vep\rtimes_{p\tau}G)$$
can be easily proved.

$(3)$ Let us briefly explain the reason why $j^{p\tau}_{G,p}$ (which is defined on $KK_G^{p\tau}(C_0(X),\bb{C})$) takes values in $KK(C_0(X)\rtimes G,\bb{C}\rtimes_{p\tau}G)$. Let $a:G^\tau\to C_0(X)$ be at level $q$, and let $e:G^\tau \to E$ be at level $p$. Then
\begin{align*}
\pi\rtimes \lt (a)(e)(g)& =\int_{G^\tau} \pi(a(h))\alpha^E_h(e(h^{-1}g))dh\\
& =\int_G \int_{U(1)} \pi(a(z\overline{h}))\alpha^E_{z\overline{h}}(e([z\overline{h}]^{-1}g))dh \\
&= \int_G \int_{U(1)} z^q\pi(a(\overline{h}))z^p\bbra{\alpha^E_{\overline{h}}(z^{-p}e(\overline{h}^{-1}g))}dh.
\end{align*}
This integral vanishes unless $q=0$.

$(4)$ By definition, $[c_X]$ is an element of $KK(\bb{C},C_0(X)\rtimes G^\tau)$. However, since $i(U(1))$ acts on $C_0(X)$ trivially, the projection $c\in C_0(X)\rtimes G^\tau$ is at level $0$. Therefore, $[c_X]$ belongs to the direct summand $KK(\bb{C},C_0(X)\rtimes G)$. We used this fact to define $\mu_G^{p\tau}$.

\end{rmks}
The above assembly maps contain all the information of the assembly maps for $G^\tau$ in the following sense.
Note that $\bb{C}\rtimes_{p\tau}G$ is contained in $\bb{C}\rtimes G^\tau$ as a direct summand. Thus, we have an injection $KK(\bb{C},\bb{C}\rtimes_{p\tau}G) \hookrightarrow KK(\bb{C},\bb{C}\rtimes G^\tau)$. 

\begin{pro}
The following two diagrams commute:
$$\xymatrix{
KK_G^{p\tau}(C_0(X),\bb{C}) \ar@{^{(}->}[r] \ar_{\mu_G^{p\tau}}[d] & KK_{G^\tau}(C_0(X),\bb{C}) \ar_{\mu_{G^\tau}}[d] \\
KK(\bb{C},\bb{C}\rtimes_{p\tau}G) \ar@{^{(}->}[r] & KK(\bb{C},\bb{C}\rtimes G^\tau)}\ \ \xymatrix{
\ca{R}KK_G^{p\tau}(X;C_0(X),Cl_\tau(X)) \ar@{^{(}->}[r] \ar_{\nu_G^{p\tau}}[d] & \ca{R}KK_{G^\tau}(X;C_0(X),Cl_\tau(X)) \ar_{\nu_{G^\tau}}[d] \\
KK(\bb{C},\bb{C}\rtimes_{p\tau}G) \ar@{^{(}->}[r] & KK(\bb{C},\bb{C}\rtimes G^\tau).}$$
\end{pro}

Let us describe the twisted versions of Proposition \ref{cp and fpa}, Proposition \ref{crossed product and generalized fixed-point module} and Proposition \ref{descented Dirac for principal bundle}.

Note that $G$ acts on $\bb{K}\bra{L^2(G,q\tau)}$ by the formula
$$\Ad R_g(k):=R_{\widetilde{g}}\circ k\circ R_{\widetilde{g}^{-1}},$$
where $\widetilde{g}$ is a chosen lift of $g\in G$. The ambiguity of the choice of a lift is cancelled out. The same argument of  \cite[Proposition 4.8]{Loi} shows the following.

\begin{pro}\label{twisted twisted cp and fpa}
Let $A$ be an $X\rtimes G$-$C^*$-algebra and let $\scr{A}:=(\{A_x\}_{x\in X},\Gamma_{\scr{A}})$ be the u.s.c. field associated to $A$.
Then, we have an isomorphism
$$A\rtimes_{q\tau} G\cong C_0\bra{X\times_{G,\Ad R\grotimes \alpha^A}\bbra{\bb{K}(L^2(G,q\tau))\grotimes A_x}_{x\in X}}.$$

Let $a\in C_c(G,A)$. The integral kernel of the corresponding equivariant section is given by
$$x\mapsto\bbbra{(g,h)\mapsto \mu(h)^{-1}\alpha^A_{g^{-1}}[a(gh^{-1})(gx)]},$$
where $a(gh^{-1})(gx)$ is the evaluation of $a(gh^{-1})\in A$ at $gx$.
The above integral kernel is also denoted by
$k_a(g,h;x).$
\end{pro}

Note that 
$\bb{K}\bra{L^2\bra{G,p\tau},L^2\bra{G,(p-q)\tau}}$ is a Hilbert $\bb{K}\bra{L^2\bra{G,p\tau}}$-module by the following operations: For $k,k_1,k_2\in \bb{K}\bra{L^2\bra{G,p\tau},L^2\bra{G,(p-q)\tau}}$ and $b\in \bb{K}\bra{L^2\bra{G,p\tau}}$, 
\begin{center}
$k\cdot b:=k\circ b$, and $\inpr{k_1}{k_2}{\bb{K}\bra{L^2\bra{G,p\tau}}}:=k_1^*\circ k_2$.
\end{center}
This Hilbert module admits a left $\bb{K}\bra{L^2\bra{G,(p-q)\tau}}$-module structure
$$\pi:\bb{K}\bra{L^2\bra{G,(p-q)\tau}}\to
\bb{L}_{\bb{K}\bra{L^2\bra{G,p\tau}}}\bra{\bb{K}\bra{L^2\bra{G,p\tau},L^2\bra{G,(p-q)\tau}}}$$
given by $\pi(a)(k):=a\circ k$.

Let $\{V_x\}$ be a $q\tau$-twisted $X\rtimes G$-equivariant u.s.c. field of vector spaces.
Then, $G$ acts on neither $\bb{K}(L^2(G,p\tau),L^2(G,(p-q)\tau))$ nor $\{V_x\}_{x\in X}$, but it does act on $\{\bb{K}(L^2(G,p\tau),L^2(G,(p-q)\tau))\grotimes V_x\}_{x\in X}$ by the following: For $F\in \bb{K}(L^2(G,p\tau),L^2(G,(p-q)\tau))$, $v\in V_x$, $\widetilde{g}\in G^\tau$ and $z\in U(1)$,
$$R_{z\widetilde{g}}\circ F \circ R_{(z\widetilde{g})^{-1}}\grotimes (z\widetilde{g})\cdot v=[z^{p-q}R_{\widetilde{g}}]\circ F \circ [z^{-p}R_{\widetilde{g}^{-1}}]\grotimes z^{q}(\widetilde{g}\cdot v)
=R_{\widetilde{g}}\circ F \circ R_{\widetilde{g}^{-1}}\grotimes \widetilde{g}\cdot v.$$
With this observation, we can describe the twisted descent homomorphism for $\ca{R}KK$-theory in the language of fields.

\begin{pro}\label{twisted cp and fpa and cp and fpm}
Let $A$ and $B$ be $X\rtimes G$-$C^*$-algebras and let $(E,\pi,F)$ be a $q\tau$-twisted $X\rtimes G$-equivariant Kasparov $(A,B)$-module. We suppose that $F$ is actually equivariant.
Let $\scr{A}:=(\{A_x\}_{x\in X},\Gamma_{\scr{A}})$, $\scr{B}:=(\{B_x\}_{x\in X},\Gamma_{\scr{B}})$ and  $\scr{E}:=(\{E_x\}_{x\in X},\Gamma_{\scr{E}})$ be the u.s.c. fields associated to $A$, $B$ and $E$, respectively. 
Then, we have an isomorphism
$$E\rtimes_{p\tau} G\cong C_0\bra{X\times_{G,\alpha^E\grotimes \Ad R}\bbra{E_x\grotimes \bb{K}(L^2(G,p\tau),L^2(G,(p-q)\tau))}_{x\in X}}$$
as bimodules. Moreover, $\pi\rtimes_{(p-q)\tau}\lt$ corresponds to $\{\pi_x\grotimes \id\}_{x\in X}$ and $F$ corresponds to $\{F_x\grotimes \id\}_{x\in X}$. Therefore, $j_G^{p\tau}\bra{\scr{E},\{\pi_x\}_{x\in X},\{ F_x\}_{x\in X}}$ is represented by
$$\bra{
C_0\bra{X\times_{G,\alpha^E\grotimes \Ad R }\bbra{E_x\grotimes\bb{K}(L^2(G,p\tau),L^2(G,(p-q)\tau))}_{x\in X}},
\{\pi_x\grotimes\id\}_{x\in X},\{ F_x\grotimes\id\}_{x\in X}}.$$
\end{pro}

The following is the twisted version of Proposition \ref{formula on descent of Dirac element}. Let $[D]=(L^2(X,E),\pi,D)$ be an index element of a $G$-equivariant Dirac operator.

\begin{pro}\label{prop descent of Dirac element for twisted version}
We define a $\bra{C_0(X)\rtimes_{p\tau} G,\bb{C}\rtimes_{p\tau}G}$-bimodule
$$\bra{L^2\bra{X\times_{G ,\alpha^E\grotimes R \grotimes \rt} \{E\grotimes L^2(G,p\tau) \grotimes (\bb{C}\rtimes_{-p\tau} G) \}},\pi\rtimes \lt ,\widetilde{D}}$$
in the same way of Definition \ref{def of descented Dirac element bimodule}. Then, it is isomorphic to $j_{G,0}^{p\tau}([D])$ by the correspondence $e\mapsto k_e$ given by
$$k_e(g_1,g_2;x):=\sqrt{\mu(g_2)}^{-1}g_1^{-1}\bbbra{e(g_1g_2^{-1},g_1x)}.$$
\end{pro}
\begin{rmks}\label{Remark on the vee}
$(1)$ Since $e\in E\rtimes_{q\tau }G$ is at level $q$, $k_e$ is at level $q$ with respect to $g_1$, and it is at level $-q$ with respect to $g_2$.

$(2)$ The Hilbert $(\bb{C}\rtimes G)$-module structure on $\bb{C}\rtimes G$ used in Definition \ref{def of descented Dirac element bimodule} is defined by $e\cdot b:=b^\vee*e$ and $\inpr{e_1}{e_2}{\bb{C}\rtimes G}=[e_2*e_1^*]^\vee$ for $e,e_1,e_2,b\in \bb{C}\rtimes G$.
The correspondence $b\mapsto b^\vee$ exchanges $\bb{C}\rtimes_{p\tau}G$ and $\bb{C}\rtimes_{-p\tau}G$. Thus, $\bb{C} \rtimes_{-p\tau}G$ has a Hilbert $\bb{C}\rtimes_{p\tau}G$-module structure by these formulas.
\end{rmks}

Let us give the twisted version of Proposition \ref{descented Dirac for principal bundle}. We work on the same situation. 
We can compute $k_{\widetilde{d\alpha^E_v}e}$ by the same argument of Lemma \ref{Lemma base differential is fiber differential}.
Take a linear splitting $\fra{s}:\fra{g}\hookrightarrow \fra{g}^\tau$. For $v\in \fra{h}$ and $e\in C_c^\infty(G^\tau,C_c^\infty(X,E))$ at level $p$,
$$k_{\widetilde{d\alpha^E_{v}}e}(x)=
dR_{-\fra{s}(v)}k_e(x)+d\rt_{-\fra{s}(v)}k_e(x).$$
Note that the right hand side is independent of the choice of $\fra{s}$, because the infinitesimal generator of $i(U(1))$ acts on $L^2(G,p\tau)$ as $p\sqrt{-1}\id$, and on $\bb{C}\rtimes_{-p\tau}G$ as $-p\sqrt{-1}\id$, respectively.

\begin{pro}\label{twisted descented Dirac for principal bundle}
$(1)$ $L^2\bra{X/H\times_{G/H}\{E\grotimes L^2(G,p\tau)\grotimes (\bb{C}\rtimes_{-p\tau} G) \}}$ admits a Hilbert $\bb{C}\rtimes_{p\tau} G$-module structure, and it admits a $*$-homomorphism
$$\pi\rtimes \lt:C_0\bra{X/H\times_{G/H}\bb{K}(L^2(G))}\to
\bb{L}_{\bb{C}\rtimes_{p\tau} G}\bra{L^2\bra{X/H\times_{G/H}\{E\grotimes L^2(G,p\tau)\grotimes (\bb{C}\rtimes_{-p\tau} G) \}}}$$
by the same formulas of the untwisted cases Proposition \ref{descented Dirac for principal bundle}.

$(2)$ We define an operator $\overline{D}$ on $L^2\bra{X/H\times_{G/H}\{E\grotimes L^2(G,p\tau)\grotimes (\bb{C}\rtimes_{-p\tau} G) \}}$ by
$$\sum_i[\id_{E}\grotimes dR_{\fra{s}(v_i)}\grotimes \id+ \id_{E}\grotimes \id\grotimes d\rt_{\fra{s}(v_i)}]\grotimes c(v_i)+D_{\rm base}.$$
Then, $j_{G,0}^{p\tau}([D])$ is represented by
$$\bra{L^2\bra{X/H\times_{G/H}\{E\grotimes L^2(G,p\tau)\grotimes (\bb{C}\rtimes_{-p\tau} G) \}},\pi\rtimes \lt ,\overline{D}}.$$
\end{pro}


\section{Index problem on proper $LT$-spaces}\label{section problem}

From this section, we start the study of infinite-dimensional spaces. The aim of this section is to explain the precise setting of the problem and formulate the main result. First, we will define proper $LT$-spaces and set up the concrete problem. We will clarify what is necessary to formulate. Second, we will review representation theory of $LT$ from \cite{FHTII}. In addition, we will  recall the substitute for the ``$\tau$-twisted group $C^*$-algebra of $LT$'' from \cite{Thesis}. 
Third, we will introduce $\ca{R}KK$-theory for non-locally compact spaces. The detailed study on this subject will be done in \cite{NT}. 
Finally, we will study $LT$-equivariant $KK$-theory, which was introduced in \cite{T4} but the study on general theory was left. In this subsection, we will prove that the Kasparov product is well-defined and associative for reformulated $LT$-equivariant $KK$-theory. Then, we will formulate the main results.

\subsection{The geometrical setting}\label{section problem spaces}

We define the loop group of the circle group $T=S^1$ as a Hilbert Lie group as follows. In this setting, the paper is full of circle groups, and hence we must distinguish all of them: The target of the loop group is denoted by $T$; The source of loops is denoted by $\bb{T}_\rot$; When we consider a central extension by a circle group, we denote it by $U(1)$.

\begin{dfn}
$(1)$ Let $U_{L^2_m}$ be the completion of 
$$\bbra{f\in C^\infty(\bb{T}_\rot,\bb{R})\ \middle|\  \int f(\theta)d\theta=0}$$
with respect to the ``$L^2_m$-metric''
$$\|f\|_{L^2_m}^2:=\frac{1}{\pi}\int_0^{2\pi} \left|  \left|\frac{d}{d\theta}
\right|^mf(\theta)\right|^2d\theta$$
for $m\geq\frac{1}{2}$, 
where $\left|\frac{d}{d\theta}\right|^m$ is given by the functional calculus of the self-adjoint operator $\frac{d}{d\theta}$ on $C^\infty(\bb{T}_\rot,\bb{R})$.

$(2)$ Let $\Pi_T\subseteq \fra{t}$ be the kernel of $\exp:\fra{t}\to T$, which can be identified with the set of homotopy classes of free loops on $T$: $\Pi_T\cong \pi_1(T)\cong \bb{Z}$.
We define the Hilbert Lie group $LT_{L^2_m}$ by
$$LT_{L^2_m}:= T\times \Pi_T\times U_{L^2_m}.$$
It is canonically identified with the set of ``$L^2_m$-loops'' by the following correspondence: For an $L^2_m$-loop $l$, there is an element  $n\in \Pi_T$ and an $L^2_m$-map $f:S^1\to \bb{R}$ satisfying $l(\theta)=\exp(\theta n)\exp(f(\theta))$; The $T$-component is given by $\exp\bra{\int_{S^1}f(\theta)d\theta}$; The $U_{L^2_m}$-component is given by $u(\theta):=f(\theta)-\int_{\bb{T}_\rot}f(\theta)d\theta$.

$(3)$ Each of $U_{L^2_m}$ and $LT_{L^2_m}$ has a rotation symmetry given by the same formula $\theta_0\cdot u(\theta):=u(\theta+\theta_0)$.

$(4)$ The Hilbert space $L\fra{t}_{L^2_{m-1}}^*:=\Omega^1_{L^2_{m-1}}(\bb{T}_\rot,\fra{t})$ consisting of $\fra{t}$-valued $1$-forms over $\bb{T}_\rot$ can be regarded as the set of connections on the trivial $T$-bundle over $\bb{T}_\rot$. Thus, it admits an $LT_{L^2_m}$-action defined by the gauge transformation: $l\cdot A(\theta):=A(\theta)+dl(\theta)l(\theta)^{-1}.$ We denote the holonomy map by $\hol:L\fra{t}_{L^2_{m-1}}^*\to T$.

$(5)$ If $m'>m$, there exist continuous homomorphisms $U_{L^2_{m'}}\to U_{L^2_{m}}$ and $LT_{L^2_{m'}}\to LT_{L^2_{m}}$, and these make two inverse systems. The inverse limits of them are denoted by $U$ and $LT$, and they are identified with the set of $C^\infty$-loops, thanks to the Sobolev embedding theorem\footnote{The resulting Lie groups are ILH-Lie groups. ``ILH'' stands for ``Inverse Limit of Hilbert''. See \cite{Omo} for details.}. We can do the same thing of $(4)$ for this case: We define $L\fra{t}^*:=\Omega^1_{C^\infty}(\bb{T}_\rot,\fra{t})$; It is equipped with the holonomy map $\hol:L\fra{t}^*\to T$; It is equipped with an $LT$-action given by the gauge transformation.
\end{dfn}

\begin{rmks}
$(1)$ Although the standard norm of $U_{L^2_m}$ is given by $\|f\|_{L^2_m}^2:=\frac{1}{\pi}\int_0^{2\pi} \left|  (1+\triangle)^{m/2}f(\theta)\right|^2d\theta$, we use the above, in order to simplify the formulas.
The same kind of simplification is mentioned in \cite{Fre}.

$(2)$ The gauge action given in $(4)$ is isometric, proper and cocompact. Thus, the $LT_{L^2_m}$-manifold $L\fra{t}_{L^2_{m-1}}^*$ satisfies the setting of the index theorem of \cite{Kas15}, except that the manifold and the group are infinite-dimensional. In fact,  manifolds which we study in the present paper, are different from it by  compact sets.

$(3)$ The Lie algebra $\Lie(U_{L^2_m})$ of $U_{L^2_m}$ can be identified with itself. That of $LT_{L^2_m}$ is isomorphic to $\fra{t}\oplus \Lie(U_{L^2_m})$.

$(4)$ The holonomy map is given by the following composition:
$L\fra{t}_{L^2_{m-1}}^*\xrightarrow{\int} \fra{t}\xrightarrow{\exp}T.$
\end{rmks}

\begin{dfn}[See \cite{Thesis}]
Let $M$ be an even-dimensional compact Riemannian $T$-equivariant $Spin^c$-manifold equipped with a $T$-invariant smooth map $\phi:M\to T$. We define an $LT_{L^2_m}$-manifold $\ca{M}_{L^2_m}$ for $1/2\leq m$ by the fiber product
$$\begin{CD}
\ca{M}_{L^2_m} @>\Phi>> L\fra{t}_{L^2_{m-1}}^* \\
@V\Hol VV @VV\hol V \\
M@>\phi>> T
\end{CD}$$
and we call a {\bf proper $LT$-space}. The induced maps are denoted by $\Hol$ and $\Phi$ as above.
\end{dfn}

\begin{rmks}
$(1)$ A proper $LT$-space is a Hilbert manifold equipped with an isometric, proper and cocompact $LT_{L^2_m}$-action, and it has a proper smooth equivariant map $\ca{M}_{L^2_m}\to L\fra{t}_{L^2_{m-1}}^*$ induced by $\phi$. Apply this construction for $m=\infty$, we obtain an ILH-manifold $\ca{M}$, where we omit the subscript $L^2_\infty$.

$(2)$ We can also say that $\phi$ is $T$-equivariant with respect to the adjoint action on $T$: $\phi(t_0.x)=t_0\phi(x)t_0^{-1}$, since $T$ is commutative. If one wants to replace $T$ with some noncommutative group, one needs to assume that $\phi:M\to G$ is equivariant with respect to the adjoint action. See \cite{AMM,Mei} for details.
\end{rmks}

\begin{ex}
Let $M$ be an $S^1$-symplectic manifold. Then, by perturbing the symplectic form, $M$ admits an $S^1$-valued moment map $\phi$ \cite{McD}. Then, the induced proper $LT$-space is a {\bf Hamiltonian $LT$-space} \cite{AMM}. In this case, the induced map $\ca{M}_{L^2_m}\to L\fra{t}^*_{L^2_{m-1}}$ is called the {\bf moment map}. In this sense, proper $LT$-spaces are obvious generalizations of Hamiltonian $LT$-spaces.
\end{ex}

A proper $LT$-space has the following simple topological type. Although there are no directly corresponding results for noncommutative loop groups, the following is related to the abelianization of \cite{LMS}.

\begin{lem}[\cite{Thesis}]\label{decomposition lemma on proper LT space}
Let $\ca{M}_{L^2_k}=M_\phi\times_\hol L\fra{t}^*_{L^2_{k-1}}$ be a proper $LT$-space.
Let $\fra{t}\subseteq L\fra{t}_{L^2_{k-1}}^*$ be the subspace consisting of constant sections. Then, $\widetilde{M}:=\Phi^{-1}(\fra{t})\subseteq \ca{M}_{L^2_k}$ is a smooth manifold and it is a principal $\Pi_T$-bundle over $M$. Moreover, it is a global slice of the $\fra{t}\times U_{L^2_k}$-action, and hence we have $\ca{M}_{L^2_k}\cong \widetilde{M}\times U_{L^2_k}$.
\end{lem}

The goal of our project is to formulate and prove an $LT$-equivariant index theorem for proper $LT$-spaces. We investigate the topological aspects of this project in the present paper. In order to explain the problem precisely, we need to introduce a $U(1)$-central extension and a Spinor of $LT$. In the following, we identify $\fra{t}$ and $\bb{R}$ to simplify the notations. Inner products are represented like products of real numbers.

\begin{dfn}[See \cite{PS,FHTII,Thesis}]\label{def of LTtau}
$(1)$ For $f_1,f_2\in U_{L^2_m}$, we define a two-cocycle 
$$\tau(f_1,f_2):=\exp\bra{\frac{i}{\pi}\int_{\bb{T}_\rot}f_1(\theta)f_2'(\theta)d\theta}\in U(1).$$
The central extension $U_{L^2_m}^\tau$ is given by $U_{L^2_m}\times U(1)$ equipped with the multiplication given by $(f_1,z_1)\cdot (f_2,z_2):=(f_1+ f_2,z_1z_2\tau(f_1,f_2))$ for $f_1,f_2\in U_{L^2_m}$. The lift of the $\bb{T}_\rot$-action is given by $\theta \cdot (f,z)=(\theta \cdot f,z)$,

$(2)$ Let $\kappa^\tau:\Pi_T\to \Hom(\Pi_T,\bb{Z})$ be an injective homomorphism such that $\kappa^\tau(n,m)\in 2\bb{Z}$. Then, we can define a $1$-dimensional representation $\kappa^\tau(n/2):T\to U(1)$ for $n\in \Pi_T$ in an obvious way. This representation is also denoted by $\kappa^\tau_{n/2}$.
We define a $U(1)$-central extension $(T\times\Pi_T)^\tau$ of $T\times\Pi_T$ by $(T\times\Pi_T)\times U(1)$ whose multiplication is defined by $((t_1,n_1),z_1)\cdot((t_2,n_2),z_2)=((t_1t_2,n_1+n_2),z_1z_2\kappa^\tau_{n_1/2}(t_2)\kappa^\tau_{n_2/2}(t_1)^{-1})$ for $(t_1,n_1),(t_2,n_2)\in T\times \Pi_T$. 

$(3)$ The $U(1)$-central extension of $LT_{L^2_m}$ is defined by the exterior tensor product: $LT_{L^2_m}^\tau:=(T\times\Pi_T)^\tau\boxtimes U_{L^2_m}^\tau$.
\end{dfn}

In order to define a Spinor space, we need to specify a base of $L\fra{t}_{L^2_{m}}$. 

\begin{nota}
$(1)$ On $\Lie(U_{L^2_m})$, we take a complete orthonomal system
$$\cos\theta,\sin\theta,\Cosi{2}{m},\Sine{2}{m},\Cosi{3}{m},\Sine{3}{m},\cdots.$$
The corresponding tangent vectors are denoted by $(e_1,f_1,e_2,f_2,\cdots)$ and the coordinate with respect to this base is denoted by $(x_1,y_1,x_2,y_2,\cdots)$.

$(2)$ We define an unbounded operator $d$ on $\Lie(U_{L^2_m})$ by $du:=du/d\theta$, which is the infinitesimal generator of the $\bb{T}_\rot$-action.
We introduce the complex structure $J$ there by $d/|d|$: Concretely, $J(e_n)=-f_n$ and $J(f_n)=e_n$. We introduce the following complex orthonormal base:
$$z_n:=\frac{1}{\sqrt{2}}\bra{e_n+\sqrt{-1}f_n} \text{ and }\ \overline{z_n}:=\frac{1}{\sqrt{2}}\bra{e_{n}-\sqrt{-1}f_{n}}.$$

$(3)$ The finite-dimensional approximation $U_{L^2_m,N}$ is the subgroup whose Lie algebra is given by the linear span of 
$$e_1,f_1,e_2,f_2,\cdots,e_N,f_N.$$
The algebraic inductive limit with respect to the natural inclusion $U_{L^2_m,N}\hookrightarrow U_{L^2_m,N+1}$ is denoted by $U_\fin$. We define $d$, $J$ and the complex base on $U_{L^2_m,N}$ in the same way of $(1)$ and $(2)$.
\end{nota}

For convenience of the reader, we explicitly describe the Lie bracket on $\Lie(U_{L^2_m})$. One can prove the following by a simple calculation.

\begin{lem}
We denote the infinitesimal generator of $i(U(1))$ by $K$. Then,
$[e_n,f_n]=n^{1-2m}K$ and $[z_n,\overline{z_n}]=-\sqrt{-1}n^{1-2m}K$.
\end{lem}

Let us introduce the Spinor space. In addition, we introduce its dual and other two Clifford multiplications on the tensor product of the Spinor and its dual, which will be used to define a $KK$-element substituting for the index element in the next section.

\begin{dfn}
$(1)$ The Spinor space $S_U$ of $\Lie(U_{L^2_m})$ is defined by the exterior algebra of the negative part of the complexfication of $\Lie(U_{L^2_m})$:
\begin{itemize}
\item $S_{U,\fin}:=\bigwedge^\alg \bigoplus_{n>0}\bb{C}\overline{z_n}$.
\item $S_{U}$ is the completion of $S_{U,\fin}$ with respect to the metric given by 
$$\inpr{\overline{z_{i_1}}\wedge \overline{z_{i_2}}\wedge \cdots \wedge \overline{z_{i_n}}}{\overline{z_{j_1}}\wedge \overline{z_{j_2}}\wedge \cdots \wedge \overline{z_{j_m}}}{S_U}=\begin{cases}
1 & (n=m \text{ and }i_1=j_1,i_2=j_2,\cdots,i_n=j_m) \\
0 & (\text{otherwise}),\end{cases}$$
where we have assumed that $i_1<i_2<\cdots<i_n$ and $j_1<j_2<\cdots<j_m$.
\item The Clifford multiplication is defined by
\begin{align*}
\gamma(\overline{z_n})&:= \sqrt{2}\overline{z_n}\wedge \\
\gamma(z_n)&:= -\sqrt{2}\overline{z_n}\rfloor,
\end{align*}
where $\rfloor$ is the interior product $(\overline{z_n}\wedge)^*$.
\end{itemize}
The unit vector corresponding  to ``$1$'' in the exterior algebra $S_U$ is denoted by ${\bf 1}_f$, where ``$f$'' comes from ``fermion''.

$(2)$ The dual of $(S_U,\gamma)$ is naturally a left $\Cl_+(\Lie(U_{L^2_m}))$-module. We explicitly construct it as follows:
\begin{itemize}
\item $S_{U,\fin}^*:=\bigwedge^\alg \bigoplus_{n>0}\bb{C}z_n$.
\item $S_{U}^*$ is the completion of $S_{U,\fin}^*$ with respect to the parallel way of $S_U$.
\item $S_U$ and $S_U^*$ are mutually dual by the {\it bilinear} pairing
$$\innpro{\overline{z_{i_1}}\wedge \overline{z_{i_2}}\wedge \cdots \wedge \overline{z_{i_n}}}{{z_{j_1}}\wedge {z_{j_2}}\wedge \cdots \wedge {z_{j_m}}}{}:=\begin{cases}
1 & (n=m \text{ and }i_1=j_1,i_2=j_2,\cdots,i_n=j_m) \\
0 & (\text{otherwise}),\end{cases}$$
where we have assumed that $i_1<i_2<\cdots<i_n$ and $j_1<j_2<\cdots<j_m$.
\item The Clifford multiplication is given by
\begin{align*}
\gamma^*(\overline{z_n})&= -\sqrt{2}z_n\rfloor \circ \epsilon_{S^*_U}\\
\gamma^*(z_n)&= -\sqrt{2}z_n\wedge\circ \epsilon_{S^*_U}.
\end{align*}
\end{itemize}
The unit vector corresponding  to ``$1$'' in the exterior algebra $S_U^*$ is denoted by ${\bf 1}^*_f$.

$(3)$ We define two other Clifford multiplications on $S_U^*\grotimes S_U$ by
\begin{align*}
c(v)&:= \frac{1}{\sqrt{2}}\bra{\id\grotimes\gamma(v)-\sqrt{-1}\gamma^*(v)\grotimes \id}\\
c^*(v)&:=\frac{\sqrt{-1}}{\sqrt{2}}\bra{\id\grotimes\gamma(v)+\sqrt{-1}\gamma^*(v)\grotimes \id}
\end{align*}
for $v\in \Lie(U_\fin)$. Note that $c(v)^2=-\|v\|^2\id$ and $\{c^*(v)\}^2=\|v\|^2\id$.
See \cite{FHTII,T4} for details.
\end{dfn}

The geometrical situation of the problem of the present paper is the following.

\begin{prob}
Let $M$ be an even-dimensional compact Riemannian $T$-equivariant $Spin^c$-manifold equipped with a $T$-invariant smooth map $\phi:M\to T$ and let $\ca{M}=M_\phi\times_\hol L\fra{t}^*$ be the corresponding proper $LT$-space.
Suppose that a $T$-equivariant Spinor bundle $S_M$ and a $\tau$-twisted $LT_{L^2_k}$-equivariant line bundle $\ca{L}$ over $\ca{M}$, are given. The pullback of $S_M$ to $\widetilde{M}$ is denoted by $S_{\widetilde{M}}$.
We define an $LT$-equivariant Spinor bundle $ S_{\ca{M}}$ over $\ca{M}$ by the exterior tensor product of $S_{\widetilde{M}}$ and the trivial bundle $U\times S_U$. Then, {\it formulate and prove the index theory for the $\tau$-twisted $LT$-equivariant Clifford module bundle $\ca{L}\grotimes  S_{\ca{M}}$.}
\end{prob}

\begin{rmk}
The assumption that $M$ is even-dimensional and $Spin^c$ is not essential. One can remove this assumption by considering ``$K$-theoretical orientation sheaf'' $\Cl_+(\Hol^*TM)$.
\end{rmk}

The main result of the present paper is the topological side of the above problem: {\it a construction of two homomorphisms substituting for the Poincar\'e duality homomorphism and the topological assembly map}. In order to explain them, we need to introduce the substitute for the $\tau$-twisted group $C^*$-algebra of $LT$, a non-locally compact version of $\ca{R}KK$-theory, and $LT$-equivariant $KK$-theory. The following three subsections are devoted to them.

\subsection{A review of the representation theory of $LT$}\label{section problem rep of LT}

In this subsection, we give a review of the representation theory of $LT$ and recall the construction of a substitute for the $\tau$-twisted group $C^*$-algebra of $LT$. 

We have defined the central extension 
$$1\to U(1)\xrightarrow{i} LT_{L^2_m}^\tau \xrightarrow{p} LT_{L^2_m}\to 1$$
for any $m\geq 1/2$ in Definition \ref{def of LTtau}. Recall that $\bb{T}_\rot$-action of $LT_{L^2_m}$ lifts to the central extension $LT_{L^2_m}^\tau$. 
With this lift, we define the concept of positive energy representation. 
The unitary group $U(V)$ of a Hilbert space $V$ is equipped with the compact-open topology

\begin{dfn}[\cite{PS}]
A {\bf positive energy representation} (PER for short) of $LT_{L^2_m}$ at level $\tau$ on a separable Hilbert space $H$, is a continuous homomorphism $\rho:LT_{L^2_m}^\tau\to U(H)$ satisfying the following:
\begin{itemize}
\item $\rho$ is at level $1$, that is to say, $\rho(e_{LT},z)=z\id_H$ for $(e_{LT},z)\in i(U(1))\subseteq LT_{L^2_m}^\tau$;
\item $\rho$ lifts to $LT_{L^2_m}^\tau\rtimes \bb{T}_\rot$; and
\item The orthogonal decomposition $H=\oplus H_n$ given by the weight of the circle action defined by the restriction of $\rho$ to $\bb{T}_\rot$, satisfies the following: $\dim H_n<\infty$ for all $n$ and $H_n=0$ for all sufficiently small $n$'s. We may impose that $H_n=0$ for all $n<0$, by retaking the lift to $ LT_{L^2_m}^\tau\rtimes\bb{T}_\rot$.
\end{itemize}
\end{dfn}

We introduce several standard notions of representation theory (the irreducibility, the direct sum and so on) in an obvious way. In particular, the dual representation of a PER is at level $-1$ and of ``negative energy''.

We can also define the concept of PERs of $T\times\Pi_T$, $U_{L^2_m}$ and its subgroups in an obvious way.
It is known that $U_{L^2_m}^\tau$ has the unique PER up to isomorphism as an infinite-dimensional version of the Stone-von Neumann theorem (See \cite[Theorem 2.4]{Kir}). See also \cite{PS,FHTII,T1} for details.
We will construct it by the following recipe: First,we define a PER of the finite-dimensional approximation $U_{L^2_m,N}^\tau$; Second, we define a homomorphism from the PER of $U_{L^2_m,N}^\tau$ to that of $U_{L^2_m,M}^\tau$ for $N<M$; Finally, we take the Hilbert space inductive limit of this system and it is the desired PER. We define it at the infinitesimal level.

\begin{dfn}
$(1)$ On $L^2(\bb{R}^N)$, we define an infinitesimal representation $d\rho$ of $U_{L^2_m,N}^\tau$ and an action of the infinitesimal generator $d$ of $\bb{T}_\rot$, by the following operators:
\begin{align*}
d\rho\bra{e_n}&:=n^{\frac{1}{2}-m}\frac{\partial}{\partial x_n};\\
d\rho\bra{f_n}&:=\sqrt{-1}n^{\frac{1}{2}-m}x_n\times; \\
d\rho(d)&:=\frac{i}{2}\sum_{n=1}^Nn\bbbra{
\bra{-\frac{\partial^2}{\partial x_n^2}+x_n^2-1}}.
\end{align*}

$(2)$ Let $L^2(\bb{R}^N)_\fin$ be the subspace which is algebraically spanned by functions of the form ``polynomial $\times$ $e^{-\frac{\|x\|^2}{2}}$''. Obviously, the operators $d\rho\bra{e_n}$'s, $d\rho\bra{f_n}$'s and $d\rho(d)$ preserve it.
\end{dfn}

\begin{rmk}
By using the complex base of the Lie algebra, we can rewrite $d\rho(d)$ as
$$d\rho(d)=-\sqrt{-1}\sum_nn^{2m}d\rho(z_n)d\rho(\overline{z_n}).$$
This expression is more appropriate than the above for the infinite-dimensional case.
\end{rmk}

With the above representation of $U_{L^2_m,N}$, we define a PER of $U_{L^2_m}$ as follows:

\begin{dfn}
$(1)$ We define an isometric embedding 
$$I_N:L^2(\bb{R}^N)\ni f\mapsto f\grotimes \frac{1}{\pi^{1/4}}e^{-\frac{x_{N+1}^2}{2}}\in L^2(\bb{R}^{N+1})$$
for each $N$.
Note that $I_N$'s are equivariant: $d\rho(d)\circ I_N=I_N\circ d\rho(d)$ and $I_N\circ d\rho(v)=d\rho(v)\circ I_N$ for $v\in U_{L^2_m,N}$. The former equivariance is because $\frac{1}{\pi^{1/4}}e^{-\frac{x_{N+1}^2}{2}}$ belongs to the kernel of $d\rho(z_{N+1})d\rho(\overline{z_{N+1}})$. Note that $I_N$'s preserve $L^2(\bb{R}^N)_\fin$'s.

$(2)$ We define $\ud{L^2(\bb{R}^\infty)}$ by the {\it Hilbert space} inductive limit $\varinjlim L^2(\bb{R}^N)$, and $\ud{L^2(\bb{R}^\infty)_\fin}$ by the {\it algebraic} inductive limit $\varinjlim^\alg L^2(\bb{R}^N)_\fin$. 
On $\ud{L^2(\bb{R}^\infty)_\fin}$, we can define the operators
$$d\rho\bra{e_n},\ d\rho\bra{f_n} \text{ and } \ d\rho(d)$$
since the embeddings $I_N$'s are equivariant. The extension of these operators are denoted by the same symbols.
\end{dfn}

\begin{nota}
The ``infinite tensor product'' $\frac{1}{\pi^{1/4}}e^{-\frac{x_1^2}{2}}\grotimes \frac{1}{\pi^{1/4}}e^{-\frac{x_2^2}{2}}\grotimes \cdots$ defines a unit vector denoted by ${\bf 1}_b$, where ``$b$'' comes from ``boson''.
\end{nota}

We can do the same things for the dual representation, that is to say, we can define a continuous homomorphism $\rho^*:U_{L^2_m}^\tau\to U(\ud{L^2(\bb{R}^\infty)^*})$ in the standard way, which is at level $-1$, and $\ud{L^2(\bb{R}^\infty)^*}$ has the ``highest weight vector'' ${\bf 1}_b^*$.
Then, $\rho^*$ induces a $*$-isomorphism $\Op:\bb{C}\rtimes_\tau U_{L^2_m,N}\to \bb{K}(L^2(\bb{R}^N)^*)$.\footnote{A function $f$ on $U_{L^2_m,N}^\tau$ at level $p$ defines the trivial operator on a representation space $(V,\sigma)$ at level $q$, unless $p+q=0$. This is because the integral
$$\int_{U_{L^2_m,N}^\tau} f(g)\sigma(g) dg=\int_{U_{L^2_m,N}}\int_{U(1)} z^pf(\overline{g})z^q\sigma(\overline{g}) dzd\overline{g}$$
vanishes unless $p+q=0$.}
 On the right hand side, we have a natural connecting homomorphism $\bb{K}(L^2(\bb{R}^N)^*)\to \bb{K}(L^2(\bb{R}^{N+1})^*)$ given by $k\mapsto k\grotimes P$, where $P$ is the one-dimensional projection onto $\bb{C}e^{-\frac{1}{2}x_{N+1}^2}$. With this idea, a substituting $C^*$-algebra for the $\tau$-twisted group $C^*$-algebra of $U_{L^2_m}$ has been defined in \cite{T1}.

\begin{dfn}\label{definition of twisted group Cstar algebra of LT}
$(1)$ We define a $C^*$-algebra $\ud{\bb{C}\rtimes_\tau U_{L^2_m}}$ by the $C^*$-algebra inductive limit:
$$\ud{\bb{C}\rtimes_\tau U_{L^2_m}}:=\varinjlim_{N\to \infty} \bb{K}(L^2(\bb{R}^N)^*).$$
It is naturally isomorphic to $\bb{K}\bra{\ud{L^2(\bb{R}^\infty)^*}}$. When we regard an element $b\in \ud{\bb{C}\rtimes_\tau U_{L^2_m}}$ as an element of $\bb{K}\bra{\ud{L^2(\bb{R}^\infty)^*}}$ by the natural identification, it is denoted by $\Op(b)$. This symbol comes from ``Operator''. The rank one projection onto $\bb{C}{\bf 1}_b^*$ is denoted by $P_{\bb{C}{\bf 1}_b^*}$.

$(2)$ We define a $C^*$-algebra $\ud{\bb{C}\rtimes_\tau LT_{L^2_m}}$ by
$$ \ud{\bb{C}\rtimes_\tau LT_{L^2_m}}:=\bb{C}\rtimes_\tau (T\times \Pi_T)\grotimes \ud{\bb{C}\rtimes_\tau U_{L^2_m}}.$$

$(3)$ For every $N\in \bb{N}$, we define $ [\bb{C}\rtimes_\tau U_{L^2_m,N}]_\fin\subseteq \bb{K}(L^2(\bb{R}^N)^*)$ by the set of finite-rank operators preserving $L^2(\bb{R}^N)_\fin^*$. These subspaces are preserved by the connecting homomorphisms. Hence the algebraic inductive limit $\varinjlim^\alg [\bb{C}\rtimes_\tau U_{L^2_m,N}]_\fin$ makes sense, and it is denoted by $[\bb{C}\rtimes_\tau U_{L^2_m}]_\fin$.
We define a dense subalgebra $ \ud{[\bb{C}\rtimes_\tau LT_{L^2_m}]_\fin}$  of $ \ud{\bb{C}\rtimes_\tau LT_{L^2_m}}$ by the algebraic tensor product
$$\ud{[\bb{C}\rtimes_\tau LT_{L^2_m}]_\fin}:= C_c^\infty(T\times\Pi_T,\tau)\grotimes^\alg
\ud{[\bb{C}\rtimes_\tau U_{L^2_m}]_\fin}.$$

\end{dfn}
\begin{nota}

A PER of $LT_{L^2_m}$ is given by the tensor product of those of $T\times\Pi_T$'s and $U_{L^2_m}$'s. Thus, we can define a $*$-homomorphism $\ud{\bb{C}\rtimes_\tau LT_{L^2_m}}\to \bb{L}(H)$ for a PER $H$. This homomorphism is denoted by ``$\Op$'', following the notation of Definition \ref{definition of twisted group Cstar algebra of LT} $(1)$.
\end{nota}

\begin{rmks}
$(1)$ The rank one operator appearing in the definition of the connecting homomorphism $\bb{K}(L^2(\bb{R}^N)^*)\to \bb{K}(L^2(\bb{R}^{N+1})^*)$ is given by the Gaussian 
$$\vac(a_{N+1},b_{N+1}):=\frac{N^{1-2m}}{2\pi}e^{-\frac{(N+1)^{1-2m}}{4}(a_{N+1}^2+b_{N+1}^2)}$$
on $U_{L^2_m,N+1}\ominus U_{L^2_m,N}$. This is proved by a simple calculation using the definition of $\rho$ and $\pi_\rho$, and we leave it to the reader.

$(2)$ A finite rank operator preserving $L^2(\bb{R}^N)_\fin$ is given by a finite linear combination of 
$$[d\rho(z_{1})]^{\alpha_1}[d\rho(z_{2})]^{\alpha_2}\cdots [d\rho(z_{N})]^{\alpha_N}\circ P_{\bb{C}{\bf 1}_b^*} \circ [d\rho(\overline{z_{1}})]^{\beta_1}[d\rho(\overline{z_{2}})]^{\beta_2}\cdots [d\rho(\overline{z_{N}})]^{\beta_N}.$$
One can prove the following formulas by integration by parts and the Leibniz rule:
$$d\rho(z_{k})\circ\Op(f)=\Op\bra{
\bbbra{-k^{-\frac{1}{2}+m}\bra{\frac{\partial}{\partial a_k}+\frac{1}{2}k^{1-2m}a_k}+\sqrt{-1}k^{-\frac{1}{2}+m}\bra{\frac{\partial}{\partial b_k}+\frac{1}{2}k^{1-2m}b_k}}f},$$
$$\Op(f)\circ d\rho(\overline{z_{k}})=\Op\bra{\bbbra{
-k^{-\frac{1}{2}+m}\bra{\frac{\partial}{\partial a_k}+\frac{1}{2}k^{1-2m}a_k}-\sqrt{-1}k^{-\frac{1}{2}+m}\bra{\frac{\partial}{\partial b_k}+\frac{1}{2}k^{1-2m}b_k}}f}.$$
By this formula, one finds that $[\bb{C}\rtimes U_{L^2_m,N}]_\fin$ is the set of functions of the form ``polynomial $\times$ Gaussian''.
Therefore, $\ud{[\bb{C}\rtimes U_{L^2_m}]_\fin}$ is regarded as the set of ``functions'' of the form ``polynomial $\times$ Gaussian''.
\end{rmks}

We have used the $G$-action ``$\rt$'' on $\bb{C}\rtimes G$ to compute the descent homomorphism in Definition \ref{def of descented Dirac element bimodule}, which is given by $\rt_gf(h):=f(hg)$ for $f\in C_c(G)$ and $g,h\in G$. In addition, we define another action ``$\lt$'' by $\lt_gf(h):=f(g^{-1}h)$.
For a unitary representation $\rho:G\to U(V)$, we can define a $*$-homomorphism $\pi_\rho:\bb{C}\rtimes G\to \bb{L}(V)$. By simple calculations, we have $\pi_\rho(\lt_gf)=\rho_g\circ \pi_\rho(f)$ and $ \pi_\rho(\rt_gf)=\pi_\rho(f)\circ\rho_{g^{-1}}$. If $G$ is a Lie group, we can also define ``$d\lt$'' and  ``$d\rt$'' and they satisfy $\pi_\rho(d\lt_Xf)=d\rho_X\circ \pi_\rho(f)$ and $ \pi_\rho(d\rt_Xf)=\pi_\rho(f)\circ d\rho_{-X}$ for $X\in \fra{g}$.

Since $\ud{\bb{C}\rtimes_{\tau} LT_{L^2_m}}$ is not defined by the completion of $C_c(LT_{L^2_m},\tau)$, the original definitions of ``$\lt$'' and  ``$\rt$'' do not work. However, the right hand sides of $\pi_\rho(\lt_gf)=\rho_g\circ \pi_\rho(f)$ and $ \pi_\rho(\rt_gf)=\pi_\rho(f)\circ\rho_{g^{-1}}$ still make sense. With this observation, we introduce the $LT$-counterparts of ``$\lt$'' and  ``$\rt$''.

\begin{dfn}\label{dfn of roup Cstar algebra of LT}
$(1)$ For $b\in \ud{\bb{C}\rtimes_\tau U_{L^2_m}}$ and $g\in U_{L^2_m}^\tau$, we define $\lt_gb:=\Op^{-1}\bra{ \rho_g\circ \Op(b)}$ and $\rt_gb:=\Op^{-1}\bra{ \Op(b)\circ \rho_{g^{-1}}}$.
We define the infinitesimal versions of ``$\lt$'' and ``$\rt$'' in an obvious way.

$(2)$ They and the corresponding actions on $T\times\Pi_T$, induce two $LT_{L^2_m}^\tau$-actions on $\ud{\bb{C}\rtimes_\tau LT_{L^2_m}}$. We use the same symbols ``$\lt$'' and ``$\rt$'' to denote these new actions.
\end{dfn}


\subsection{$\ca{R}KK$-theory for non-locally compact action groupoids}\label{section RKK}

The goal of this subsection is to define ``$\ca{R}KK$-theory for non-locally compact action groupoids'' based on the idea which we have explained in the last several paragraphs of Section \ref{section KK and RKK for finite dimension}. In the present paper, we deal with only action groupoids. We do not explain the whole story and we will merely define necessary concepts and prove several easy results. For more general cases and details, see \cite{NT}. I emphasize that the primitive idea of this theory is due to Shintaro Nishikawa. I thank him for allowing me to introduce this theory before the collaboration paper \cite{NT}.


For simplicity, in this subsection, we work under the following assumption. It is satisfied for proper $LT$-spaces.

\begin{asm}
Throughout this subsection, $\ca{X}$ is assumed to be a metrizable space. In particular, we have Urysohn's lemma: {\it For any closed subsets $C,C'\subseteq \ca{X}$ with $C\cap C'=\emptyset$, there exists a continuous function $f:\ca{X}\to [0,1]$ such that $f|_C=0$ and $f|_{C'}=1$.}
We also suppose that $\ca{G}$ is a topological metrizable group and it acts on $\ca{X}$ continuously. Then, $\ca{G}\times \ca{X}$ is also metrizable.
\end{asm}

Let us introduce the substitute for the concept of ``$\ca{X}\rtimes \ca{G}$-$C^*$-algebras''. We begin with the definition of upper semi-continuous fields (u.s.c. fields, for short) of Banach spaces.

\begin{dfn}[{\cite[Definition 10.1.2]{Dix}}]\label{former definition of usc fields}
$(1)$ A {\bf u.s.c. field of Banach spaces over $\ca{X}$} is a pair $\scr{E}=(\{E_x\}_{x\in \ca{X}},\Gamma_{\scr{E}})$ of a family of Banach spaces $\{E_x\}_{x\in \ca{X}}$ parameterized by $\ca{X}$, and a subset $\Gamma_{\scr{E}}$ of $\bbra{s:\ca{X}\to \coprod_{x\in \ca{X}}E_x\mid s(x)\in E_x}$ satisfying the following conditions:
\begin{itemize}
\item $\Gamma_{\scr{E}}$ is a complex vector space by the pointwise operations;
\item The evaluation homomorphism $\Gamma_{\scr{E}}\ni s\mapsto s(x)\in E_x$ is surjective;
\item The function $\ca{X}\ni x\mapsto \|s(x)\|\in \bb{R}_{\geq 0}$ is upper semi-continuous for each $s\in \Gamma_{\scr{E}}$; and
\item A section $s:\ca{X}\to \coprod_{x\in \ca{X}}E_x$ is an element of $\Gamma_{\scr{E}}$ if the following is satisfied: For every $x\in \ca{X}$ and every $\epsilon>0$, these exists a neighborhood $U_{x,\epsilon}$ of $x$ and an $s'\in \Gamma_{\scr{E}}$ such that $\|s(y)-s'(y)\|\leq \epsilon$ for $y\in U_{x,\epsilon}$.
\end{itemize}

$(2)$ A u.s.c. field $\scr{E}=(\{E_x\}_{x\in \ca{X}},\Gamma_{\scr{E}})$ is {\bf $\bb{Z}_2$-graded} if every $E_x$ is $\bb{Z}_2$-graded: $E_x= E_{0,x}\widehat{\oplus}E_{1,x}$, and both of $\scr{E}_i=(\{E_{i,x}\}_{x\in \ca{X}},\Gamma_{\scr{E}_i})$ ($i=0,1$) are u.s.c. fields, where 
$$\Gamma_{\scr{E}_i}:=\bbra{s\in\Gamma_{\scr{E}}\mid s(x)\in E_{i,x}\text{ for every }x\in \ca{X}}.$$

\end{dfn}
\begin{rmk}
If the function $\ca{X}\ni x\mapsto \|s(x)\|$ is continuous for every $s\in\Gamma_{\scr{E}}$, the field is called a {\bf continuous field of Banach spaces}.
\end{rmk}

\begin{dfn}\label{dfn of homomorphism of usc field of B-spaces}
Let $\scr{E}_1=(\{E_{1,x}\}_{x\in \ca{X}},\Gamma_{\scr{E}_1})$ and $\scr{E}_2=(\{E_{2,x}\}_{x\in \ca{X}},\Gamma_{\scr{E}_2})$ be two u.s.c. fields of Banach spaces over $\ca{X}$. 

$(1)$ A family of bounded linear maps $\phi=\{\phi_x:E_{1,x}\to E_{2,x}\}$ is called a {\bf homomorphism between $\scr{E}_1$ and $\scr{E}_2$ over $\ca{X}$} if it satisfies the following: For every $s\in \Gamma_{\scr{E}_1}$, the resulting section $\phi(s):x\mapsto \phi_x(s_x)$ also belongs to $\Gamma_{\scr{E}_2}$. It is said to be injective/surjective/isomorphic/isometric if every $\phi_x$ is injective/surjective/isomorphic/isometric.

$(2)$ A homomorphism $\phi:\scr{E}_1\to \scr{E}_2$ over $\ca{X}$ between $\bb{Z}_2$-graded u.s.c. fields of Banach spaces is {\bf even} if it preserves the $\bb{Z}_2$-grading, and is {\bf odd} if it reveres the $\bb{Z}_2$-grading.
\end{dfn}
\begin{rmk}
A $\bb{Z}_2$-graded u.s.c. field of Banach space $\scr{E}$ admits an isometric isomorphism $\epsilon=\{\epsilon_x\}:\scr{E}\to \scr{E}$ such that $\epsilon_x(e_{0,x}+e_{1,x})=e_{0,x}-e_{1,x}$, where $e_{i,x}\in E_{i,x}$ for $i\in \{0,1\}$ and $x\in \ca{X}$. It is called the grading homomorphism. A homomorphism is even and odd if and only if it commutes and anti-commutes with $\epsilon$, respectively.
\end{rmk}

As mentioned in \cite[Appendix]{TXLG}, the ``total space'' $\widetilde{\scr{E}}:=\coprod_{x\in \ca{X}}E_x$ admits the unique topology so that $\Gamma_{\scr{E}}$ is the set of all continuous sections of the bundle $\widetilde{\scr{E}}\to \ca{X}$. Then, the above homomorphism $\{\phi_x\}$ defines a continuous bundle map $\phi:\widetilde{\scr{E}_1}\to \widetilde{\scr{E}_2}$.
We can construct the pullback of a u.s.c. field of Banach spaces by a continuous map $F:\ca{Y}\to \ca{X}$. This field over $\ca{Y}$ is denoted by $F^*\scr{E}$. 

Using this construction, we can define the concept of continuous group actions. See also \cite[Definition 3.5]{LG}. By the $\ca{G}$-action on $\ca{X}$, we define two continuous maps $s,r: \ca{G}\times \ca{X}\to  \ca{X}$ by $s(g,x):=x$ and $r(g,x):=g\cdot x$. 

\begin{dfn}
$(1)$ Let $\scr{E}=(\{E_x\}_{x\in \ca{X}},\Gamma_{\scr{E}})$ be a u.s.c. field of Banach spaces over $\ca{X}$.
A {\bf continuous $\ca{G}$-action on $\scr{E}$} is an isomorphism $\alpha:s^*\scr{E}\to r^*\scr{E}$ over $\ca{G}\times\ca{X}$ satisfying the following commutative diagram for every $x\in \ca{X}$ and $g\in \ca{G}$:
$$\xymatrix{
E_x \ar^-{\alpha_{(g,x)}}[rr] \ar_-{\alpha_{(hg,x)}}[rd]&& E_{g\cdot x}\ar^-{\alpha_{(h,g\cdot x)}}[ld] \\
& E_{(hg)\cdot x}&}$$
An isomorphism $\alpha_{(g,x)}$ is also denoted by $(\alpha_g)_x$, by regarding $\alpha_g$ as an automorphism on $\scr{E}$. If every isomorphism is isometric, this action is said to be isometric. A u.s.c. field of Banach spaces over $\ca{X}$ equipped with a $\ca{G}$-action is said to be {\bf $\ca{X}\rtimes \ca{G}$-equivariant}.

$(2)$ For $\ca{X}\rtimes \ca{G}$-equivariant u.s.c. fields of Banach spaces $\scr{E}_1$ and $\scr{E}_2$, a homomorphism $\phi:\scr{E}_1\to \scr{E}_2$ over $\ca{X}$ is said to be $\ca{X}\rtimes \ca{G}$-equivariant if it satisfies $(\alpha_g)_x\circ \phi_x=\phi_{g\cdot x}\circ (\alpha_g)_x$.
\end{dfn}

More interesting types of fields, including fields of $C^*$-algebras or those of Hilbert modules, can be defined.

\begin{dfn}
$(1)$ Let $\scr{A}=(\{A_x\},\Gamma_{\scr{A}})$ be a u.s.c. field of Banach spaces. It is a {\bf u.s.c. field of $C^*$-algebras over $\ca{X}$} if the following conditions are satisfied:
\begin{itemize}
\item Each fiber $A_x$ is a $C^*$-algebra;
\item $\Gamma_{\scr{A}}$ is closed under the multiplication and the adjoint; and
\item The norm satisfies the $C^*$-condition $\|a^*a\|=\|a\|^2$ for each $x\in \ca{X}$ and $a\in A_x$.
\end{itemize}

For u.s.c. fields of $C^*$-algebras $\scr{A}_1$ and $\scr{A}_2$ over $\ca{X}$, a homomorphism $\phi:\scr{A}_1\to\scr{A}_2$ over $\ca{X}$ is called a {\bf $*$-homomorphism} if it preserves the multiplication and the adjoint. We define the concepts of $\bb{Z}_2$-graded u.s.c. fields of $C^*$-algebras, even/odd $*$-homomorphisms, in the same way for fields of Banach spaces.

$(2)$ A u.s.c. field of $C^*$-algebras $\scr{A}=(\{A_x\},\Gamma_{\scr{A}})$ over $\ca{X}$ is said to be {\bf locally separable} if the following condition is satisfied: For every $x\in \ca{X}$, there exists an open neighborhood $U_x$ and a countable set of continuous sections $\{a_n\}\subseteq \Gamma_{\scr{A}}$ such that $\{a_n(y)\}\in A_y$ is dense for every $y\in U_x$.

$(3)$ Let $\scr{A}=(\{A_x\},\Gamma_{\scr{A}})$ be a u.s.c. field of $C^*$-algebras.
A u.s.c. field of Banach spaces $\scr{E}=(\{E_x\},\Gamma_{\scr{E}})$ is a {\bf u.s.c. field of Hilbert $\scr{A}$-modules over $\ca{X}$} if the following conditions are satisfied:
\begin{itemize}
\item Each fiber $E_x$ is a Hilbert $A_x$-module;
\item $\Gamma_{\scr{E}}$ is closed under the multiplication by $\Gamma_{\scr{A}}$, and the inner product of two elements $e_1,e_2\in \Gamma_{\scr{E}}$ defines an element of $\Gamma_{\scr{A}}$; and
\item The Banach space norm of $e\in E_x$ is given by $\|e\|_{E_x}=\sqrt{\|\inpr{e}{e}{E_x}\|_{A_x}}$.
\end{itemize}

$(4)$ A u.s.c. field $\scr{E}$ of Hilbert $\scr{A}$-modules over $\ca{X}$ is said to be {\bf locally countably generated} if the following condition is satisfied: For any $x\in X$, there exist a countable set $\{s_n\}_{n\in \bb{N}}\subseteq \Gamma_{\scr{E}}$ and an open neighborhood $U_x$ such that the $A_y$-linear span of $\{s_n(y)\}_{n\in \bb{N}}\subseteq E_y$ is dense in $E_y$ for every $y\in U_x$.

$(5)$ A u.s.c. field of $C^*$-algebras $\scr{A}$ over $\ca{X}$ is said to be {\bf $\ca{X}\rtimes \ca{G}$-equivariant} if an isometric $\ca{G}$-action on $\scr{A}$ is given and each $(\alpha_g)_x$ preserves the multiplication and the adjoint. We define the concept of $\ca{X}\rtimes \ca{G}$-equivariant Hilbert $\scr{A}$-modules in a parallel way.
\end{dfn}

\begin{rmk}
Let $\scr{E}$ be a locally countably generated u.s.c. field of Hilbert $\scr{A}$-modules over $\ca{X}$. Then, thanks to the proof of \cite[Lemma 5.9]{MS} and a partition of unity, we can find a countable set $\{s_n\}\subseteq \Gamma_{\scr{E}}$ so that $\{s_n(x)\}$ spans $E_x$ over $A_x$ for every $x\in \ca{X}$. Moreover, $\scr{E}$ can be written as an orthogonally complementable submodule of the trivial bundle $l^2(\bb{N})\grotimes \scr{A}$. This result will be proved in \cite{NT}.
\end{rmk}

Let us give typical examples.

\begin{ex}
$(1)$ Let $C(\ca{X})$ be the $*$-algebra consisting of $\bb{C}$-valued continuous functions on $\ca{X}$. Then,
$$\scr{C}(\ca{X})=(\{\bb{C}_x\}_{x\in \ca{X}},C(\ca{X}))$$
is an $\ca{X}\rtimes \ca{G}$-equivariant continuous field of $C^*$-algebras by the ``trivial'' action $(\alpha_g)_x:\bb{C}_x\ni z\mapsto z\in\bb{C}_{g\cdot x}$.
Note that the single $C^*$-algebra $C_0(\ca{X})$ is trivial if $\ca{X}$ is not locally compact. In this sense, the concept of u.s.c. fields of $C^*$-algebras is more general than single $C^*$-algebras. 

$(2)$ Let $\ca{E}=\coprod_{x\in\ca{X}}E_x\to \ca{X}$ be a $\ca{G}$-equivariant locally trivial Hilbert space bundle whose fibers are separable, and let $C(\ca{X},\ca{E})$ be the set of continuous sections. Then, 
$$\scr{E}=(\{E_x\}_{x\in \ca{X}},C(\ca{X},\ca{E}))$$
is an $\ca{X}\rtimes \ca{G}$-equivariant continuous field of Hilbert $\scr{C}(\ca{X})$-modules.

$(3)$ For a $\ca{G}$-$C^*$-algebra $A$ and a $\ca{G}$-equivariant Hilbert $A$-module $E$, $A\grotimes \scr{C}(\ca{X})=(\{A\},C(\ca{X},A))$ is an $\ca{X}\rtimes \ca{G}$-equivariant continuous field of $C^*$-algebras and $E\grotimes\scr{C}(\ca{X})=(\{E\},C(\ca{X}, E))$ is an $\ca{X}\rtimes \ca{G}$-equivariant continuous field of Hilbert $A\grotimes \scr{C}(\ca{X})$-module by the diagonal action. These fields are said to be {\bf trivial}.
\end{ex}

Let us define the filed of adjointable operators and that of compact operators for a u.s.c. field of $C^*$-algebras $\scr{B}$ and a u.s.c. field of Hilbert $\scr{B}$-modules $\scr{E}$.

\begin{dfn}\label{dfn:  filed of scrB-compact operators}
Let $\scr{B}$ be an $\ca{X}\rtimes\ca{G}$-equivariant u.s.c. field of $C^*$-algebras and let $\scr{E}$ be a locally countably generated $\ca{X}\rtimes\ca{G}$-equivariant u.s.c. field of Hilbert $\scr{B}$-modules.

$(1)$ The filed of adjointable operators $\scr{L}(\scr{E})=\scr{L}_{\scr{B}}(\scr{E})$ over $\scr{E}$ is defined by the family of $C^*$-algebras $\{\bb{L}(E_x)\}_{x\in \ca{X}}$ and the set of sections $\Gamma_{\scr{L}(\scr{E})}$ given by the following: A section $x\mapsto T_x$ is continuous if and only if, for any continuous section $\xi=\{\xi_x\}_{x\in \ca{X}}\in \Gamma_{\scr{E}}$, the new sections $x\mapsto T_x\xi_x$ and $x\mapsto (T_x)^*\xi_x$ are again continuous.

$(2)$ We define a single $C^*$-algebra $\bb{L}_{\scr{B}}(\scr{E})$ by the set of bounded sections
$$\bbra{\{T_x\}\in \Gamma_{\scr{L}(\scr{E})}\mid \|T_x\|\text{ is bounded}}$$
by the pointwise operations and the norm $\|\{T_x\}\|:=\sup_{x\in \ca{X}}\|T_x\|$.

$(3)$ A continuous section $\{ F_x\}\in \Gamma_{\scr{L}(\scr{E})}$ is a {\bf locally finite rank operator} if the following condition is satisfied: For each point $x\in \ca{X}$, there exists an open neighborhood $U_x$ and continuous sections $e_1,f_1,\cdots,e_n,f_n$ of $\scr{E}$ such that $F_y=\sum e_i(y)\grotimes [f_i(y)]^*$ on $U_x$, where the operator $e_i(y)\grotimes [f_i(y)]^*$ is defined by $E_y\ni v\mapsto e_i(y)\inpr{f_i(y)}{v}{B_y}\in E_y$.

$(4)$ The filed of $\scr{B}$-compact operators $\scr{K}(\scr{E})=\scr{K}_{\scr{B}}(\scr{E})$ over $\scr{E}$ is defined by the family of $C^*$-algebras $\{\bb{K}(E_x)\}_{x\in \ca{X}}$ and the set of sections $\Gamma_{\scr{K}(\scr{E})}$ given by the following. A section $x\mapsto T_x\in \bb{K}(E_x)$ is continuous if and only if there exists a net of locally finite rank operators $\{\{(F_\lambda)_x\}_{x\in \ca{X}}\}_{\lambda\in \Lambda}$ satisfying the following condition: For any $\epsilon>0$ and $x\in X$, there exists an open neighborhood $U_{\epsilon,x}$ and $\lambda_0$ such that $\lambda\geq \lambda_0$ implies $\|T_y-(F_\lambda)_y\|<\epsilon$ on $y\in U_{\epsilon,x}$.
\end{dfn}

The field $\scr{K}(\scr{E})$ is an $\ca{X}\rtimes\ca{G}$-equivariant u.s.c. field of $C^*$-algebras, where the $\ca{G}$-action is defined by $(\alpha_g^{\scr{K}(\scr{E})})_x(k_x):=(\alpha_g^{\scr{E}})_x\circ k_x\circ(\alpha_{g^{-1}}^{\scr{E}})_{g\cdot x}$. This will be proved in \cite{NT}

\begin{dfn}\label{dfn of tensor product of usc field of modules}
Let $\scr{A}=(\{A_x\},\Gamma_{\scr{A}})$ and $\scr{B}=(\{B_x\},\Gamma_{\scr{B}})$ be $\ca{X}\rtimes \ca{G}$-equivariant u.s.c. fields of $C^*$-algebras.

$(1)$ Let $\scr{E}_1=(\{E_{1,x}\},\Gamma_{\scr{E}_1})$ be an $\ca{X}\rtimes \ca{G}$-equivariant u.s.c. field of Hilbert $\scr{A}$-modules, and let $\scr{E}_2=(\{E_{2,x}\},\Gamma_{\scr{E}_2})$ be an $\ca{X}\rtimes \ca{G}$-equivariant u.s.c. field of Hilbert $\scr{B}$-modules. Suppose that an $\ca{X}\rtimes \ca{G}$-equivariant $*$-homomorphism $\pi=\{\pi_x\}_{x\in \ca{X}}:\scr{A}\to \scr{L}_{\scr{B}}(\scr{E})$ is given. Then, we define an $\ca{X}\rtimes \ca{G}$-equivariant u.s.c. field of Hilbert $\scr{B}$-modules $\scr{E}_1\grotimes_{\scr{A}}\scr{E}_2$ by $(\{E_{1,x}\grotimes_{A_x}E_{2,x}\}_{x\in \ca{X}},\Gamma_{\scr{E}_1\grotimes_{\scr{A}}\scr{E}_2})$, where a section $s$ is an element of $\Gamma_{\scr{E}_1\grotimes_{\scr{A}}\scr{E}_2}$, if for every $x\in \ca{X}$ and for every $\epsilon>0$ there exists a neighborhood $U_{x,\epsilon}$ of $x$ and finitely many sections $\{s_{1,i}\}\subseteq \Gamma_{\scr{E}_1}$ and $\{s_{2,i}\}\subseteq \Gamma_{\scr{E}_2}$ such that $\|s(y)-\sum_i s_{1,i}(y)\grotimes s_{2,i}(y)\|<\epsilon$ on $y\in U_{x,\epsilon}$. For an adjointable operator $T$ on $\scr{E}_1$, we can define $T\grotimes \id$ in an obvious way.

$(2)$ In particular, we can define the pushout as follows. Let $\sigma:\scr{A}\to \scr{B}$ be an $\ca{X}\rtimes \ca{G}$-equivariant $*$-homomorphism. It gives a Hilbert $\scr{B}$-module $\scr{B}$ equipped with a $*$-homomorphism $\sigma:\scr{A}\to \scr{B}\subseteq \scr{L}_{\scr{B}}(\scr{B})$. We define the pushout $\sigma_*(\scr{E})$ by $\scr{E}\grotimes_{\scr{A}}\scr{B}$.
\end{dfn}
\begin{rmk}
Let $I$ be the closed interval $[0,1]$.
For an $\ca{X}\rtimes \ca{G}$-equivariant u.s.c. field of $C^*$-algebras $\scr{A}$, we define another equivariant u.s.c. field $\scr{A}I$ by $(\{A_x\grotimes C(I)\}_{x\in \ca{X}},\Gamma_{\scr{A}I})$, where $\Gamma_{\scr{A}I}$ is defined by the same way of $(1)$ of the above definition.
Then, we have a $*$-homomorphism 
$$\ev_t=\{\ev_{t,x}:A_xI=A_x\grotimes C(I)\ni a\grotimes f\mapsto a\cdot f(t)\in A_x\}.$$
Therefore, for an $\ca{X}\rtimes \ca{G}$-equivariant u.s.c. field of Hilbert $\scr{A}I$-module $\scr{E}$, we can define the ``evaluation at $t$ of $\scr{E}$'' by $\ev_{t,*}\scr{E}$. We will use this construction to define the homotopy equivalence.

For this tensor product construction and $T\in \bb{L}_{\scr{A}I}(\scr{E})$, $T\grotimes \id$ is denoted by $\ev_{t,*}(T)\in \bb{L}_{\scr{A}}(\ev_{t,*}(\scr{E}))$.
\end{rmk}

Now, we can extend $\ca{R}KK$-theory to our infinite-dimensional setting.

\begin{dfn}\label{dfn homotopy of RKK element}
Let $\scr{A}=(\{A_x\},\Gamma_{\scr{A}})$ and $\scr{B}=(\{B_x\},\Gamma_{\scr{B}})$ be locally separable $\ca{X}\rtimes \ca{G}$-equivariant u.s.c. fields of $C^*$-algebras. 

$(1)$ An {\bf $\ca{X}\rtimes \ca{G}$-equivariant Kasparov $(\scr{A},\scr{B})$-module} is a triple $(\scr{E},\pi,F)$ satisfying the following conditions:
\begin{itemize}
\item $\scr{E}=(\{E_x\},\Gamma_{\scr{E}})$ is a locally countably generated $\ca{X}\rtimes \ca{G}$-equivariant u.s.c. field of $\bb{Z}_2$-graded Hilbert $\scr{B}$-modules.
\item $\pi=\{\pi_x\}$ is an $\ca{X}\rtimes \ca{G}$-equivariant even $*$-homomorphism $\scr{A}\to \scr{L}_{\scr{B}}(\scr{E})$.
\item $F=\{F_x\}\in \bb{L}_{\scr{B}}(\scr{E})$ is an odd element and it satisfies the following conditions: For any $a\in \Gamma_{\scr{A}}$, all the sections 
\begin{center}
$x\mapsto [\pi_x(a_x),F_x]$, $x\mapsto \pi_x(a_x)(T_x-T_x^*)$ and $x\mapsto\pi_x(a_x)\{1-(T_x)^2\}$
\end{center}
belong to $\Gamma_{\scr{K}_{\scr{B}}(\scr{E})}$, and the section 
$$\ca{G}\times \ca{X} \ni (g,x)\in a_{g\cdot x}\{F_{g\cdot x}-(\alpha_g)_x\circ F_{x}\circ (\alpha_{g^{-1}})_{g\cdot x}\}\in \bb{K}_{B_{g\cdot x}}(E_{g\cdot x})$$
belongs to $\Gamma_{r^*\scr{K}_{\scr{B}}(\scr{E})}$.
\end{itemize}

$(2)$ The set of isomorphism classes of $\ca{G}$-equivariant Kasparov $(\scr{A},\scr{B})$-modules is denoted by $\bb{E}_{\ca{X}\rtimes\ca{G}}(\scr{A},\scr{B})$. It is an abelian semigroup, whose addition is defined by $(\scr{E}_1,\pi_1,F_1)+(\scr{E}_2,\pi_2,F_2):=(\scr{E}_1\oplus \scr{E}_2,\pi_1\oplus\pi_2,F_1\oplus F_2)$.

$(3)$ Two elements $(\scr{E}_0,\pi_0,F_0), (\scr{E}_1,\pi_1,F_1)\in \bb{E}_{\ca{X}\rtimes\ca{G}}(\scr{A},\scr{B})$ are said to be {\bf homotopic} if there exists $(\scr{E},\pi,F)\in \bb{E}_{\ca{X}\rtimes\ca{G}}(\scr{A},\scr{B}I)$ such that $(\scr{E}_i,\pi_i,F_i)$ is isomorphic to $\ev_{i,*}(\scr{E},\pi,F)$ for $i=0,1$. 
This is obviously an equivalence relation. 
Such a triple is called a homotopy between  $(\scr{E}_0,\pi_0,F_0)$ and $(\scr{E}_1,\pi_1,F_1)$. 

$(4)$ The set of homotopy classes of $\ca{X}\rtimes \ca{G}$-equivariant Kasparov $(\scr{A},\scr{B})$-modules is denoted by $\ca{R}KK_{\ca{G}}(\ca{X};\scr{A},\scr{B})$. The homotopy class of $(\scr{E}_0,\pi_0,F_0)$ is denoted by $[(\scr{E}_0,\pi_0,F_0)] \in \ca{R}KK_{\ca{G}}(\ca{X};\scr{A},\scr{B})$.

$(5)$ If $\ca{G}$ is equipped with a $U(1)$-central extension, we define $\ca{R}KK_{\ca{G}}^\tau(\ca{X};\scr{A},\scr{B})$ in the same way of Definition \ref{twisted equivariant KK theory}.
\end{dfn}

$\ca{R}KK_{\ca{G}}(\ca{X};\scr{A},\scr{B})$ is an abelian group, whose addition is induced by the addition in $\bb{E}_{\ca{X}\rtimes \ca{G}}(\scr{A},\scr{B})$. In fact, the direct sum operation is homotopy invariant, because the direct sum of the homotopies gives a homotopy between the direct sums. 

We can also define the concept of Kasparov products.
\begin{dfn}
Let $\scr{A}$, $\scr{A}_1$ and $\scr{B}$ be locally separable $\ca{X}\rtimes\ca{G}$-equivariant u.s.c. fields of $C^*$-algebras. Let $(\scr{E}_1,\pi_1,F_1)\in \bb{E}_{\ca{X}\rtimes \ca{G}}(\scr{A},\scr{A}_1)$ and let $(\scr{E}_2,\pi_2,F_2)\in \bb{E}_{\ca{X}\rtimes \ca{G}}(\scr{A}_1,\scr{B})$. 

$(1)$ For $s_1\in \Gamma_{\scr{E}_1}$, we define $T_{s_1}: \scr{E}_2\to \scr{E}_1\grotimes_{\scr{A}_1} \scr{E}_2$ by $\Gamma_{\scr{E}_2}\ni e_2\mapsto [x\mapsto s_{1,x}\grotimes s_{2,x}\in E_{1,x}\grotimes_{A_{1,x}} E_{2,x}] \in\Gamma_{\scr{E}_1\grotimes_{\scr{A}_1}\scr{E}_2}$.
An element $F\in \bb{L}_\scr{B}(\scr{E}_1\grotimes \scr{E}_2)$ is a {\bf $F_2$-connection} if
$$\bbbra{\begin{pmatrix}F & 0 \\ 0 & F_2\end{pmatrix}, \begin{pmatrix}0& T_{s_{1}} \\ T_{s_{1}}^* & 0\end{pmatrix}}\in \Gamma_{\scr{K}_{\scr{B}}(\scr{E}\oplus \scr{E}_2)}$$
for any $s_1\in \Gamma_{\scr{E}_1}$

$(2)$ $(\scr{E}_1\grotimes \scr{E}_2,\pi,F)\in \bb{E}_{\ca{X}\rtimes \ca{G}}(\scr{A},\scr{B})$ is a {\bf Kasparov product of $(\scr{E}_1,\pi_1,F_1)$ and $(\scr{E}_2,\pi_2,F_2)$} if $F$ is an $F_2$-connection satisfying the following: For arbitrary $a\in \Gamma_{\scr{A}}$, there exists a continuous section $P\in \scr{L}_{\scr{B}}(\scr{E})$ whose value at each point is positive, such that the section $x\mapsto \pi_x(a_x)[F_x,F_{1,x}\grotimes \id]\pi_x(a_x)^*-P_x$ belongs to $\Gamma_{\scr{K}_{\scr{B}}(\scr{E}_1\grotimes \scr{E}_2)}$.
\end{dfn}


\begin{thm}[\cite{NT}]
We suppose that the $\ca{G}$-action on $\ca{X}$ is proper and there exists a local slice at every point of $\ca{X}$.
We further suppose that $\scr{A},\scr{A}_1$ and $\scr{B}$ are locally separable. Then, for arbitrary $(\scr{E}_1,\pi_1,F_1)\in \bb{E}_{\ca{X}\rtimes \ca{G}}(\scr{A},\scr{A}_1)$ and $(\scr{E}_2,\pi_2,F_2)\in \bb{E}_{\ca{X}\rtimes \ca{G}}(\scr{A}_1,\scr{B})$, there exists a Kasparov product $(\scr{E}_1\grotimes \scr{E}_2,\pi,F)\in \bb{E}_{\ca{X}\rtimes \ca{G}}(\scr{A},\scr{B})$. It is well-defined at the $\ca{R}KK$-level, and it is associative. 
\end{thm}

The Kasparov product of $x\in \ca{R}KK_{\ca{G}}(\ca{X};\scr{A},\scr{A}_1)$ and $y\in\ca{R}KK_{\ca{G}}(\ca{X};\scr{A}_1,\scr{B})$ is denoted by $x\grotimes_{\scr{A}_1}y\in \ca{R}KK_{\ca{G}}(\ca{X};\scr{A},\scr{B})$.

We have used the fact that an equivariant $\ca{R}KK$-element is represented by an actually equivariant $\ca{R}KK$-cycle in Proposition \ref{crossed product and generalized fixed-point module}. An analogous result for non-locally compact groupoids which satisfies a certain condition, will be proved in \cite{NT}. For this time, we prove it for a special case, which will be necessary for the main theorem.

We begin with a useful criterion to be operator homotopic.

\begin{lem}[{See \cite[Lemma 2.1.18]{JT}}]\label{Lemma positivity of the commutator implies the homotopy invariance}
Let $\scr{A}$ and $\scr{B}$ be $\ca{X}\rtimes\ca{G}$-equivariant u.s.c. fields of $C^*$-algebras. Suppose that $(\scr{E},\pi,F),(\scr{E},\pi,F')\in \bb{E}_{\ca{X}\rtimes \ca{G}}(\scr{A},\scr{B})$ satisfy the following: For any $a\in\Gamma_{\scr{A}}$, there exists a continuous section $P\in \scr{L}_{\scr{B}}(\scr{E})$ whose value at each point is positive such that the section $x\mapsto \pi_x(a_x)[F_x,F'_x]\pi_x(a_x)^*-P_x$ belongs to $\Gamma_{\scr{K}_{\scr{B}}(\scr{E})}$. Then, $(\scr{E},\pi,F)$ is homotopic to $(\scr{E},\pi,F')$ as equivariant Kasparov modules.
\end{lem}
\begin{proof}
We define two $C^*$-subalgebras $Q_{\scr{A}}(\scr{E})$ and $I_{\scr{A}}(\scr{E})$ of $\bb{L}_{\scr{B}}(\scr{E})$ by
$$Q_{\scr{A}}(\scr{E}):=\bbra{T\in \bb{L}_{\scr{B}}(\scr{E})\mid \text{The section }x\mapsto [\pi_x(a(x)),T_x]\text{ belongs to }\Gamma_{\scr{K}(\scr{E})}\text{ for any }a\in \Gamma_{\scr{A}}}$$
$$I_{\scr{A}}(\scr{E}):=\bbra{T\in Q_{\scr{A}}(\scr{E})\mid \text{The section }x\mapsto \pi_x(a(x))T_x\text{ belongs to }\Gamma_{\scr{K}(\scr{E})}\text{ for any }a\in \Gamma_{\scr{A}}}$$
Note that $I_{\scr{A}}(\scr{E})$ is a two-sided ideal of $Q_{\scr{A}}(\scr{E})$.
Since $F$ and $F'$ belongs to $Q_{\scr{A}}(\scr{E})$, so does $[F,F']$. We check that $[F,F']\geq 0 \mod I_{\scr{A}}(\scr{E})$.
For each state $\omega$ of $\bb{L}_{B_x}(E_x)/\bb{K}_{B_x}(E_x)$ and each $a\in A_x$ for $x\in X$, let us consider the state of $Q_{\scr{A}}(\scr{E})/I_{\scr{A}}(\scr{E})$ defined by $[T]\mapsto \omega(q_x(\pi_x(a)T_x\pi_x(a)^*))$, where $[T]$ is the equivalence class of $T$ in $Q_{\scr{A}}(\scr{E})/I_{\scr{A}}(\scr{E})$ and $q_x$ is the natural projection $\bb{L}_{B_x}(E_x)\to \bb{L}_{B_x}(E_x)/\bb{K}_{B_x}(E_x)$. This new state is denoted by $\mu(\omega,a,x)$. Then, the set of all such states $\{\mu(\omega,a,x)\}$ is a faithful family of states of $Q_{\scr{A}}(\scr{E})/I_{\scr{A}}(\scr{E})$. This is because $\prod_{x\in\ca{X}}\ev_x:Q_{\scr{A}}(\scr{E})/I_{\scr{A}}(\scr{E})\to \prod_{x\in\ca{X}}Q_{A_x}(E_x)/I_{A_x}(E_x)$ is injective and the set $\{\mu(\omega,a,x)\}_{\omega,a}$ is a faithful family of $Q_{A_x}(E_x)/I_{A_x}(E_x)$ for each $x\in\ca{X}$.
Since $\mu(\omega,a,x)([F,F'])=\omega(\pi_x(a)[F,F']\pi_x(a))=\omega(q_x(P_x))\geq 0$ for each $a,\omega,x$, the commutator $[F,F'] \geq 0 \mod I_{\scr{A}}(\scr{E})$. The remainder of the proof is the same with that of \cite[Lemma 2.1.18]{JT} and left to the reader.
\end{proof}

\begin{dfn}
Let $\scr{A},\scr{B}$ be locally separable $\ca{X}\rtimes\ca{G}$-equivariant u.s.c. fields of $C^*$-algebras and let $C$ and $D$ be separable $\ca{G}$-$C^*$-algebras. Suppose that all of $C\grotimes \scr{A}$, $D\grotimes \scr{A}$, $C\grotimes \scr{B}$ and $D\grotimes \scr{B}$ are $\ca{X}\rtimes\ca{G}$-equivariant u.s.c. fields of $C^*$-algebras (see the following remark).
We define the following two operations:
$$\sigma_{\ca{X},\scr{A}}:KK_{\ca{G}}(C,D)\to \ca{R}KK_{\ca{G}}(\ca{X};C\grotimes \scr{A},D\grotimes \scr{A})$$
by $(E,\pi,F)\mapsto (\{E\grotimes_{\bb{C}}A_x\}_{x\in\ca{X}},\{\pi\grotimes \id_{A_x}\}_{x\in\ca{X}},\{T\grotimes \id_{A_x}\}_{x\in\ca{X}})$, and
$$\sigma_{C}:\ca{R}KK_{\ca{G}}(\ca{X};\scr{A}, \scr{B})\to \ca{R}KK_{\ca{G}}(\ca{X};C\grotimes \scr{A},C\grotimes \scr{B})$$
by $(\{E_x\},\{\pi_x\},\{F_x\})\mapsto (\{C\grotimes_{\bb{C}} E_x\},\{\id_C\grotimes \pi_x\},\{\id_C\grotimes F_x\})$.
\end{dfn}
\begin{rmk}
Even for locally compact cases, it is non-trivial whether tensor product of a u.s.c. field of $C^*$-algebras and a single $C^*$-algebra is again a u.s.c. field \cite{KW}. However, even for infinite-dimensional cases,
it is obvious that the tensor product of a single nuclear $C^*$-algebra and a locally trivial field of $C^*$-algebras is again u.s.c. This is the case which we will deal with in the following sections.
\end{rmk}

\begin{cor}\label{Cor positivity of the commutator implies the homotopy invariance}
Let $C$ and $D$ be separable $\ca{G}$-$C^*$-algebras, and let  $(E,\pi,F)$ be a $\ca{G}$-equivariant Kasparov $(C,D)$-module. Then,  $\sigma_{\ca{G},\scr{C}(\ca{G})}(E,\pi,F)$ is operator homotopic to the family $(\{E\}_{g\in \ca{G}},\{\pi\}_{g\in \ca{G}},\{g(F)\}_{g\in \ca{G}})$.
\end{cor}
\begin{proof}
Let $c\in \Gamma_{C\grotimes \scr{C}(\ca{G})}$.  By an easy computation, $\pi(c_g)[F,g(F)]\pi(c_g)^*$ is given by
$$2\pi(c_g)F^2\pi(c_g)^*+\pi(c_g)F\cdot \{\pi(c_g^*)(g(F)-F)\}^*+\pi(c_g)(g(F)-F)\cdot F\pi(c_g)^*.$$
The first term is a positive element of $\bb{L}_{C\grotimes \scr{C}(\ca{G})}(E\grotimes \scr{C}(\ca{G}))$. The second and the third ones are continuous sections of $\scr{K}(C\grotimes \scr{C}(\ca{G}))$.  By Lemma \ref{Lemma positivity of the commutator implies the homotopy invariance}, we obtain the result.
\end{proof}


\subsection{Loop group equivariant $KK$-theory and the main result}\label{subsection statement of the main result}

Loop groups are non-locally compact. Although the $LT$-equivariant Kasparov modules and $LT$-equivariant $KK$-groups make sense, and although the phrase ``a Kasparov module is a Kasparov product of other two Kasparov modules'' makes sense, we encounter a serious trouble when we try to study the Kasparov product at the level of $KK$-theory. 
In the previous paper \cite{T4}, such problems are put off.

In this subsection, we re-introduce ``$LT$-equivariant $KK$-theory'' using the inductive limit of $LT_{L^2_m}$-equivariant $KK$-theory, and we will prove the desired properties on the Kasparov product. This definition shares the same spirit with the concept of ILH-Lie groups of \cite{Omo}. 

Let us begin with a fundamental construction on equivariant $KK$-theory. Let $\frac{1}{2}\leq m\leq m'$. Suppose that $A$ and $B$ are $LT_{L^2_m}$-$C^*$-algebras. They are at the same time $LT_{L^2_{m'}}$-$C^*$-algebras by the inclusion $i_{m',m}:LT_{L^2_{m'}}\hookrightarrow LT_{L^2_m}$. Thus, we can define a  homomorphism
$$i_{m',m}^*:KK_{LT_{L^2_{m}}}(A,B)\to KK_{LT_{L^2_{m'}}}(A,B)$$
by the pullback of the group action by $i_{m',m}$. 
The family of homomorphisms $\{i_{m',m}\}$ satisfies the condition to be a directed system $i_{m',m}^*\circ i_{m'',m'}^*= i_{m'',m}^*$ for $m''\geq m'\geq m$.
Using it, we re-define ``$LT$-equivariant $KK$-theory'' as follows.

\begin{dfn}\label{definition of LT KK theory}
$(1)$ Let $A$ and $B$ be $LT$-$C^*$-algebras. Suppose that the $LT$-action on them extends to $LT_{L^2_m}$ for some $m\geq 1/2$.
By the inductive limit of this system, we define
$$\bb{KK}_{LT}(A,B):=\varinjlim_{m'\to \infty} KK_{LT_{L^2_{m'}}}(A,B).$$
We call it the {\bf $LT$-equivariant $\bb{KK}$-theory}.

$(2)$ We define {\bf $\ca{M}\rtimes LT$-equivariant $\bb{KK}$-group} by the same construction. Let $\scr{A}$ and $\scr{B}$ be $\ca{M}\rtimes LT$-equivariant u.s.c. fields of $C^*$-algebras, which are obtained by $\ca{M}_{L^2_{m}}\rtimes LT_{L^2_{m}}$-equivariant u.s.c. fields of $C^*$-algebras for some $m\geq 1/2$.
Then, we define 
$$\ca{R}\bb{KK}_{LT}(\scr{A},\scr{B}):=\varinjlim_{m'\to \infty} \ca{R}KK_{LT_{L^2_{m'}}}(\scr{A},\scr{B}).$$
\end{dfn}

We can define the concept of Kasparov products on $\bb{KK}_{LT}$-theory and on $\ca{R}\bb{KK}_{LT}(\ca{M})$-theory. For the latter one, we can do that at the level of $\ca{R}KK_{LT_{L^2_m}}(\ca{M}_{L^2_m})$-theory \cite{NT}. For the former one, we need an appropriate version of the Kasparov technical theorem. 

The following statement is the copy of Theorem 1.4 of \cite{Kas88} except that the two non-locally compact groups $LT_{L^2_{m'}}$ and $LT_{L^2_{m}}$ appear. 
Note that the assumption of the statement is about the $LT_{L^2_{m}}$-equivariance and the conclusion is about the  $LT_{L^2_{m'}}$-equivariance {\it which is weaker than the $LT_{L^2_{m}}$-equivariance} since $m'> m$. This result is sufficient for $\bb{KK}_{LT}$-theory.

\begin{thm}
Let $m'>m\geq 1/2$.

Let $J$ be a $\sigma$-unital $LT_{L^2_m}$-algebra, $A_1$ and $A_2$ $\sigma$-unital subalgebras in $\ca{M}(J)$ such that $A_1$ is an $LT_{L^2_m}$-algebra. Let $\Delta$ be a subset in $\ca{M}(J)$ which is separable in the norm topology and derives\footnote{On a $\bb{Z}_2$-graded algebra $B$, a subset $S$ derives a subalgebra $B'$ if the graded commutator $[d,b]\in B'$ for every  $d\in \Delta$ and $b\in B'$.} $A_1$ and $\phi:LT_{m'}\to \ca{M}(J)$ a bounded function. Assume that $A_1\cdot A_2\subseteq J$, $A_1\cdot \phi(LT_{L^2_m})\subseteq J$, $\phi(LT_{L^2_m})\cdot A_1\subseteq J$, and the functions $g\mapsto a\phi(g)$ and $g\mapsto \phi(g)a$ are norm-continuous on $LT_{L^2_m}$ for any $a\in A_1+J$. Then, there are $LT_{L^2_{m'}}$-continuous positive even elements $M_1,M_2\in \ca{M}(J)$ such that $M_1+M_2=1$, all elements $M_ia_i,[M_i,d],M_2\phi(g),\phi(g)M_2,g(M_i)-M_i$ belong to $J$ for any $a_i\in A_i$, $d\in \Delta$, $g\in LT_{L^2_{m'}}$ $(i=1,2)$, and the functions $g\mapsto M_2\phi(g)$ and $g\mapsto \phi(g)M_2$ are norm-continuous on $LT_{L^2_{m'}}$.
\end{thm}
\begin{proof}
In order to prove it, we follow the argument of \cite[Theorem 1.4]{Kas88}. We explain only the necessary changes from it.
See also \cite{JT,Bla} for the detailed expositions on the technical theorem.

The proof of \cite[Theorem 1.4]{Kas88}  is outlined as follows: $(1)$ A lemma related to the quasi-central and quasi-invariant approximate unit is verified; $(2)$ Using this result, the technical theorem is verified. 

The corresponding result to $(1)$ is the following (it is again the copy of Lemma of Theorem 1.4 of \cite{Kas88} except that the two groups $LT_{L^2_{m'}}$ and $LT_{L^2_{m}}$ appear): {\it Let $B$ be a $C^*$-algebra with an $LT_{L^2_m}$-action, $A$ a $\sigma$-unital $LT_{L^2_m}$-$C^*$-algebra which is an $LT_{L^2_m}$-subalgebra of $B$, $Y$ a $\sigma$-compact, locally compact space, $\phi:Y\to B$ a function satisfying the following condition: $[\phi(y),a]\in A$ for all $a\in A$ and $y\in Y$, and all the functions $y\mapsto [\phi(y),a]$ are norm-continuous for  all $a\in A$. Then, there is a countable approximate unit $\{u_i\}$ for $A$ which has the following properties: $\lim_{i\to \infty }\|[u_i,\phi(y)]\|=0$ for all $y\in Y$ and $\lim_{i\to \infty}\|g(u_i)-u_i\|=0$ for all $g\in LT_{L^2_{m'}}$. Both limits are uniform on compact subsets of $Y$ and bounded subsets of $LT_{L^2_{m'}}$}. Note that $A$ is automatically $LT_{L^2_{m'}}$-subalgebra because the inclusion $LT_{L^2_{m'}}\to LT_{L^2_{m}}$ is continuous.

In the proof of the lemma in \cite{Kas88}, an exhaustive sequence consisting of relatively compact open subsets of $G$, was chosen. Then, the set $Z$ in the proof of the lemma is compact, and hence the proof works.
In our case, the same argument does not work, because $LT_{L^2_m}$ is non-locally compact. However, there is an exhaustive sequence $X_1\subseteq X_2\subseteq \cdots LT_{L^2_{m'}}$ consisting of {\bf bounded sets}. Then, thanks to the Rellich lemma, $i_{m',m}(X_i)$'s are relatively compact. Take an exhaustive sequence $Y_1\subseteq Y_2\subseteq \cdots \subseteq Y$ such that $\overline{Y_i}$ is compact and $\bigcup Y_i=Y$, just like the proof of \cite[Lemma 1.4]{Kas88}. Using these sequences and the same argument in that proof, we obtain an approximate unit $\{u_i\}$ of $A$ such that $\|[u_i,\phi(y)]\|<i^{-1}$ for all $y\in \overline{Y_i}$ and $\|g(u_i)-u_i\|<i^{-1}$ for all $g\in \overline{i_{m',m}(X_i)}$. Note that the closure $\overline{i_{m',m}(X_i)}$ is taken in the $L^2_m$-topology. Since $i_{m',m}(\overline{X_i})\subseteq \overline{i_{m',m}(X_i)}$, the approximate unit $\{u_i\}$ satisfies the desired properties.

The argument to prove $(2)$ from $(1)$ still works for our case. We leave the details to the reader.
\end{proof}
\begin{rmk}\label{Remark Montel group}
The above argument works for a topological group $\ca{G}$ which has an ``approximation from the outside by the compact homomorphism''. This concept is defined as follows: $(1)$ For each $i\in \bb{N}$, suppose that a topological group $\ca{G}_i$ with a metric function which contains $\ca{G}$ as a dense subgroup, is given; $(2)$ Suppose that a continuous homomorphism $j_{i',i}:\ca{G}_{i'}\to \ca{G}_i$ such that all bounded sets are mapped to relatively compact sets, is given for every $i<i'\in \bb{N}$; $(3)$ Suppose that $j_{i',i}\circ j_{i'',i'}=j_{i'',i}$ for all $i<i'<i''\in \bb{N}$; and $(4)$ $\ca{G}=\cap_{i\in \bb{N}} \ca{G}_i$. For such a group, we can define $\bb{KK}_{\ca{G}}$-theory by the inductive limit of $KK_{\ca{G}_i}$-theory, and this theory has the Kasparov product. Needless to say, we can replace $\bb{N}$ with a more general directed system including $\bb{R}_{\geq 1/2}$.

Thanks to the Rellich lemma, any mapping group $C^\infty(M,G)$ for a compact manifold $M$ and a compact Lie group $G$ satisfies the above conditions. In particular, $\bb{KK}_{C^\infty(M,G)}$-theory can be defined.
\end{rmk}

By using this theorem, we can prove the desired properties on the Kasparov product in an appropriate sense. For example, for $x\in KK_{LT_{L^2_m}}(A,A_1)$ and $y\in KK_{LT_{L^2_m}}(A_1,B)$, there exists a $KK$-element $z\in KK_{LT_{L^2_{m'}}}(A,B)$ which is a Kasparov product of $i_{m',m}^*(x)$ and $i_{m',m}^*(y)$ for $m'>m$. Such desired properties and standard arguments on inductive limits guarantee the following.

\begin{cor}
Let $A,B,A_1$ and $A_2$ be separable $LT$-$C^*$-algebras whose actions continuously extend to $LT_{L^2_m}$ for some $m<\infty$. Then, the Kasparov product
$$\bb{KK}_{LT}(A,A_1)\times \bb{KK}_{LT}(A_1,B)\to \bb{KK}_{LT}(A,B)$$
is well-defined. This product is denoted by $(x,y)\mapsto x\grotimes_{A_1}y$. This Kasparov product is associative, that is to say, for $x\in \bb{KK}_{LT}(A,A_1)$, $y\in \bb{KK}_{LT}(A_1,A_2)$ and $z\in \bb{KK}_{LT}(A_2,B)$, we have
$(x\grotimes_{A_1} y)\grotimes_{A_2}z=x\grotimes_{A_1} (y\grotimes_{A_2}z)$.
\end{cor}

Since $LT^\tau$ satisfies the condition explained in Remark \ref{Remark Montel group}, the same arguments work for $LT^\tau$-equivariant theory, and hence $\tau$-twisted $LT$-equivariant $KK$-theory can be defined. 

\begin{cor}
Let $A,B$ and $A_1$ be separable $LT$-$C^*$-algebras whose actions continuously extend to $LT_{L^2_m}$ for some $m<\infty$. We can define $\bb{KK}_{LT}^{p\tau}$-theory for $p\in \bb{Z}$, and we can prove that the Kasparov product
$$\bb{KK}_{LT}^{p\tau}(A,A_1)\times \bb{KK}_{LT}^{q\tau}(A_1,B)\to \bb{KK}_{LT}^{(p+q)\tau}(A,B)$$
is well-defined and associative.
\end{cor}

With this theory, we can state the main result of the present paper.
$\ca{A}(\ca{M}_{L^2_k})$ is a $C^*$-algebra playing the role of ``$\ca{S}_\vep\grotimes Cl_\tau(\ca{M}_{L^2_k})$'', which will be defined in Section \ref{section PD}.

\begin{mainthm}
Let $\ca{M}$ be a proper $LT$-space.

$(1)$ We can define a homomorphism substituting for the Poincar\'e duality homomorphism
$$\PD :\bb{KK}_{LT}^\tau(\ca{A}(\ca{M}_{L^2_k}),\ca{S}_\vep)\to \ca{R}\bb{KK}_{LT}^\tau(\ca{M};\ca{S}_\vep\grotimes\scr{C}(\ca{M}),\ca{S}_\vep\grotimes\scr{C}(\ca{M})).$$
This homomorphism assigns $\sigma_{\ca{S}_\vep}([\ca{L}])$ to the ``index element $[\widetilde{\ca{D}}]$ of the Dirac operator twisted by a $\tau$-twisted $LT$-equivariant line bundle $\ca{L}$'' constructed in \cite{T4}.

$(2)$ We can define a homomorphism substituting for the topological assembly map 
$$\ud{\nu_{LT_{L^2_m}}^\tau}:\ca{R}KK_{LT_{L^2_m}}^\tau(\ca{M}_{L^2_m};\scr{C}(\ca{M}_{L^2_m}),\scr{C}(\ca{M}_{L^2_m}))\to  KK(\bb{C},\ud{\bb{C}\rtimes_\tau LT_{L^2_m}}).$$
This homomorphism assigns to $[\ca{L}]$ the ``analytic index of the Dirac operator twisted by $\ca{L}$'' constructed in \cite{Thesis}, which is denoted by $\ud{\ind(\ca{D}_{\ca{L}})}$.

$(3)$ Consequently, we have $\ud{\nu_{LT_{L^2_m}}^\tau}(\PD([\widetilde{\ca{D}}]))=\ud{\ind(\ca{D}_{\ca{L}})}$.
\end{mainthm}

\begin{rmk}
In fact, what we will define for $(1)$ is a homomorphism
$$KK_{LT_{L^2_m}}^\tau(\ca{A(M}_{L^2_k}),\ca{S})\to \ca{R}KK_{LT_{L^2_m}}^\tau(\ca{M}_{L^2_m};\ca{S}\grotimes \scr{C}(\ca{M}_{L^2_m}),\ca{S}\grotimes \scr{C}(\ca{M}_{L^2_m}))$$
for $k\geq 1/2$ and $m$ which is sufficiently larger than $k$. We will reformulate the ``index element of the Dirac operator twisted by a $\tau$-twisted $LT$-equivariant line bundle $\ca{L}$'' in the next section in order to fit the new construction of $\ca{A}(\ca{M}_{L^2_k})$.

\end{rmk}


\section{The Poincar\'e duality homomorphism for infinite-dimensional manifolds}\label{section PD}

In this section, we formulate the Poincar\'e duality homomorphism for infinite-dimensional Hilbert manifolds, and we compute it for proper $LT$-spaces. Strictly speaking, a modification of the Poincar\'e duality homomorphism for the case of proper $LT$-spaces will be formulated and computed (see Section \ref{section unsolved} for this modification). 

For this aim, we begin with an explanation of the concept of ``$C^*$-algebras of Hilbert manifolds'' which was announced in \cite{Yu}. Since we think this $C^*$-algebra is quite important, we will study it for proper $LT$-spaces in details.


\subsection{A $C^*$-algebra of a Hilbert manifold and Poincar\'e duality homomorphism}\label{section PD PD}

The goal of this subsection is to formulate an infinite-dimensional version of the reformulated version of the Poincar\'e duality (see also Definition \ref{dfn reformulated KK elements} and Proposition \ref{Prop reformulated local Bott element}). In the present paper, we do not prove that it is isomorphic. Instead, we will show that the homomorphism constructed in this section, gives an appropriate result for proper $LT$-spaces in the sense that the answer is parallel to Example \ref{ex computation of PD}. It will be done in the following two subsections.

Let us begin with the concept of ``$C^*$-algebras of Hilbert manifolds'' following \cite{GWY,Yu}.

\begin{dfn}[See \cite{GWY}]\label{def field of Clifford algebras}
Let $\ca{X}$ be a Hilbert manifold and let $\vep>0$.
First, we consider the space
$$\Pi(\ca{X}):=\prod_{(x,t)\in \ca{X}\times [0,\vep)}\Cl_+(T_x\ca{X}\oplus t\bb{R}),$$
where 
$$t\bb{R}:=\begin{cases}
\bb{R} & (t\neq 0) \\
0 & (t=0).\end{cases}
$$
This is a space of possibly non-continuous Clifford algebra-valued functions.
Then, we consider the $C^*$-algebra
$$\Pi_b(\ca{X}):=\bbra{s\in \prod(\ca{X})\ \middle|\  \|s(x,t)\|\text{ is bounded.}}$$
equipped with pointwise algebraic operations (addition, multiplications and the adjoint) and the uniform norm.
\end{dfn}

\begin{dfn}[\cite{GWY,Yu}]
$(1)$ Let $\ca{X}$ be a Hilbert manifold. {\it Suppose that its injectivity radius is greater than $2\varepsilon$ everywhere for some $0<\vep\leq \infty$}. Let $x_0,x\in \ca{X}$, and suppose that $d(x,x_0)<2\vep$. Then, $x_0$ is contained in the image of $\exp_{x}:B_{2\vep}(T_{x}\ca{X})\to \ca{X}$, and hence it is contained in the domain of $\log_x:\exp_x(B_{2\vep}(T_{x}\ca{X}))\to T_{x}\ca{X}$.
The {\bf local Clifford operator at $x_0$} is defined by
$$C_{x_0}(x,t):=(-\log_{x}(x_0),t)\in T_x\ca{X}\oplus t\bb{R},$$
or equivalently $C_{x_0}(x,t)=\bra{(d\exp_{x_0})_x(\log_{x_0}(x)),t}=$``$(\overrightarrow{x_0x},t)$''.

$(2)$ The {\bf Bott homomorphism $\beta_{x_0}:\ca{S}_\vep\to \Pi_b(\ca{X})$ centered at $x_0\in \ca{X}$} is defined by
$$\beta_{x_0}(f)(x,t):=
\begin{cases}
f(C_{x_0}(x,t)) & (d(x,x_0)<2\vep) \\
0 & (d(x,x_0)\geq\vep),
\end{cases}$$
for $f\in \ca{S}_\vep$, where $f(C_{x_0}(x,t))$ is the functional calculus in the $C^*$-algebra $\Cl_+(T_x\ca{X}\oplus t\bb{R})$.
\end{dfn}

\begin{dfn}[\cite{Yu}]
The $C^*$-algebra $\ca{A(X)}$ is defined by the $C^*$-subalgebra of $\Pi_b(\ca{X})$ generated by the image of the Bott homomorphisms:
$$\ca{A(X)}:=C^*\bra{\bbra{
\beta_{x_0}(f)\ \middle|\  x_0\in\ca{X}, f\in \ca{S}_\vep}}.$$
\end{dfn}
\begin{rmk}
This $C^*$-algebra depends on the choice of $\vep$.
If $\ca{X}$ is at the same time a Hilbert-Hadamard space, and if we chose $\infty$ as $\vep$, we obtain the same $C^*$-algebra constructed in \cite{GWY}.
\end{rmk}

Since this definition is parallel to \cite[Definition 5.14]{GWY}, $\ca{A(X)}$ has similar properties. 
The same arguments work, except for \cite[Lemma 5.8]{GWY}. It is due to the fact that Hilbert-Hadamard spaces are ``non-positively curved''.
We  prove a corresponding result by the infinite-dimensional version of the Rauch comparison theorem \cite[Theorem 19]{Bil}. See also  \cite{KN2,Sak}.


\begin{lem}
Let $\ca{Y}$ and $\widetilde{\ca{Y}}$ be Hilbert manifolds. Fix $p\in \ca{Y}$ and $\widetilde{p}\in \widetilde{\ca{Y}}$, and take an isometric embedding $I:T_{\widetilde{p}} \widetilde{\ca{Y}}\hookrightarrow T_p \ca{Y}$. Suppose that both injectivity radii of $\ca{Y}$ and $\widetilde{\ca{Y}}$ are greater than $2\vep$, and suppose that all the sectional curvatures of $\widetilde{\ca{Y}}$ are not greater than those of $\ca{Y}$. For a smooth curve $\widetilde{c}:[0,1]\to B_{\vep}(\widetilde{p})\subseteq \widetilde{\ca{Y}}$, we define ${c}:=\exp_{{p}}\circ I \circ \exp_{\widetilde{p}}^{-1}\circ \widetilde{c}:[0,1]\to B_{\vep}(p)\subseteq\ca{Y}$. Then, we have an inequality
$$\int_0^1\|\dot{c}(t)\|dt\leq \int_0^1\|\dot{\widetilde{c}}(t)\|dt.$$
\end{lem}

For simplicity, we impose the following ``bounded geometry type'' assumption. Note that it is automatically satisfied if $\ca{X}$ admits an isometric cocompact group action. In particular, a proper $LT$-space satisfies this assumption.

\begin{asm}\label{assumption related to bounded geometry}
We suppose that $\ca{X}$ is a Hilbert manifold whose injectivity radius is bounded below by $2\vep$. We also suppose that all the sectional curvatures of $\ca{X}$ are bounded above by $ \delta$. When $\delta>0$, by retaking $\vep$ if necessary, we suppose that $\vep<{\pi}/{2\sqrt{\delta}}$ from the beginning.
\end{asm}

\begin{lem}[{Compare with \cite[Lemma 5.8]{GWY}}]\label{Lemma 5.8 of this paper}
Let $\ca{X}$ be a Hilbert manifold satisfying Assumption \ref{assumption related to bounded geometry}.
For $x,x_0,x_1\in \ca{X}$ satisfying $d(x,x_0)$, $d(x,x_1)<\vep$, we have an inequality
$$\|C_{x_0}(x,t)-C_{x_1}(x,t)\|\leq 2d(x_0,x_1).$$
\end{lem}
\begin{proof}
We would like to apply the Rauch comparison theorem to the case when $\widetilde{M}=\ca{X}$ and $M=S^\infty(\delta^{-1/2})$, where $S^\infty(\delta^{-1/2})$ is the infinite-dimensional sphere with radius $\delta^{-1/2}$.
Pick up $y\in S^\infty(\delta^{-1/2})$ and take an isometric embedding $I:T_x\ca{X}\to T_yS^\infty(\delta^{-1/2})$. We may assume that $y$ is the north pole.
By the composition $\exp_y\circ I\circ \log_x$, we can define a local embedding from the $\vep$-neighborhood of $x$ in $\ca{X}$ to that of $y$ in $S^\infty(\delta^{-1/2})$. Note that this neighborhood of $y$ in $S^\infty(\delta^{-1/2})$ is contained in the northern hemisphere, thanks to the condition $\vep<\frac{\pi}{2\sqrt{\delta}}$.
Let $y_i:=\exp_y\circ I\circ \log_x(x_i)$ for $i=0,1$. 
Let $\widetilde{c}$ be the unique minimal geodesic connecting $x_0$ and $x_1$, and let $c:= \exp_y\circ I\circ \log_x\circ \widetilde{c}$. Then, thanks to the comparison theorem, we have the inequality 
$L(c)\leq L(\widetilde{c}).$
Since $\widetilde{c}$ is a geodesic, $L(\widetilde{c})=d_{\ca{X}}(x_0,x_1)$. Clearly, $d_{S^\infty(\delta^{-1/2})}(y_0,y_1)\leq L(c)$. Combining them, we obtain the inequality $d_{S^\infty(\delta^{-1/2})}(y_0,y_1)\leq d_{\ca{X}}(x_0,x_1)$.

By definition of the Clifford operator and $y_i$'s, $\exp_yI[-C_{x_i}(x,t)]=y_i$.
Since $I$ is isometric, we have $\|C_{x_0}(x,t)-C_{x_1}(x,t)\|=\|I[C_{x_0}(x,t)]-I[C_{x_1}(x,t)]\|$. Therefore, in order to obtain the result, it is sufficient to prove that $d_{S^\infty(\delta^{-1/2})}(y_0,y_1)\geq \frac{1}{2}\|I[C_{x_0}(x,t)]-I[C_{x_1}(x,t)]\|$.

Take the minimal geodesic $c_1$ connecting $y_0$ and $y_1$ in $S^\infty(\delta^{-1/2})$. Put $c_2:=\log_y\circ c_1$. Then, $L(c_1)=d_{S^\infty(\delta^{-1/2})}(y_0,y_1)$ and $L(c_2)\geq \|I[C_{x_0}(x,t)]-I[C_{x_1}(x,t)]\|$. Thus, it is sufficient to prove that $L(c_1)\geq \frac{1}{2}L(c_2)$.
By definition of $c_1$ and $c_2$, we have $\dot{c}_1(t)=(d\exp_y)_{c_2(t)}[\dot{c}_2(t)]$. 
Since $L(c_i)=\int_{0}^1\|\dot{c}_i(t)\|dt$, it suffices to prove that $\|(d\exp_y)_z(v)\|\geq \frac{1}{2}\|v\|$ for all $z$ which belongs to the northern hemisphere and $v\in T_zS^\infty(\delta^{-1/2})$.
Now, this is clear by the following arguments:
Consider the two dimensional sphere $S^2(\delta^{-1/2})\subseteq S^\infty(\delta^{-1/2})$ which contains $y,z$ and whose tangent space contains $v$; Then, calculate the differential of $\exp_y$ at $x$ using the polar coordinate system. We leave the details to the reader.
\end{proof}

In the same way of Corollary 5.12 of \cite{GWY}, we can prove the following.

\begin{cor}
If a net $\{x_i\}_{i\in I}\subseteq \ca{X}$ converges to $x_0\in \ca{X}$, we have
$$\lim_{i\to \infty}\|\beta_{x_0}(f)-\beta_{x_i}(f)\|=0$$
for every $f\in \ca{S}_\vep$.
\end{cor}

By this corollary, we can prove several important properties.
Let $\Gamma$ be a Hausdorff (possibly non-locally compact) group. Suppose that $\Gamma$ acts on $\ca{X}$ in an  isometric and continuous way. Then, $\Gamma$ acts on $\prod_b(\ca{X})$ as follows: For an isometry $\phi:\ca{X}\to \ca{X}$ and $f\in \prod_b(\ca{X})$, we define
$$\phi_*(f)(x,t):=\Cl_+\{(d\phi)_{\phi^{-1}(x)}\oplus\id_{t\bb{R}}\}{f(\phi^{-1}(x),t)}\in \Cl_+(T_x\ca{X}\oplus t\bb{R}),$$
where $\Cl_+\{F\}$ for a linear map $F$ on a Hilbert space $V$ is the induced homomorphism on $\Cl_+(V)$.
This $\Gamma$-action on $\prod_b(\ca{X})$ gives a continuous action on $\ca{A(X)}$ as follows. 

\begin{pro}\label{YusCstar algebra admits a group action}
$(1)$ For an isometry $\phi:\ca{X}\to \ca{X}$ and $x_0\in \ca{X}$,
$$\phi_*\circ \beta_{x_0}=\beta_{\phi(x_0)}.$$
Consequently, $\phi_*$ preserves $\ca{A(X)}$. 

$(2)$ Therefore, $\Gamma$ acts on $\ca{A}(\ca{X})$. This action is continuous in the point-norm topology. 
\end{pro}
\begin{proof}
See  \cite[Section 6]{GWY}.
\end{proof}

We summarize other necessary properties of $\ca{A}(\ca{X})$. Thanks to Lemma \ref{Lemma 5.8 of this paper} , one can prove them in the same way of \cite{GWY}.

\begin{pro}\label{properties of GWY algebra}
Let $\ca{X}$ be a Hilbert manifold satisfying Assumption \ref{assumption related to bounded geometry}.

$(1)$ $\ca{A(X)}$ is separable whenever $\ca{X}$ is separable.


$(2)$ When $\ca{X}$ is finite-dimensional, $\ca{A(X)}$ coincides with the set of continuous sections vanishing at infinity
$$\bbra{f\in C_0(\ca{X}\times[0,\vep),\Cl_+(TX\oplus t\bb{R}))}.$$
This $C^*$-algebra is isomorphic to $C_0([0,\vep),\Cl_+(t\bb{R}))\grotimes Cl_\tau(\ca{X})$.

$(3)$ For a subset $\ca{Y}$ of $\ca{X}$, we define $\ca{A}(\ca{X},\ca{Y})$ by the $C^*$-subalgebra of $\ca{A}(\ca{X})$ generated by 
$$\bbra{\beta_x(f)\, \middle|\, x\in \ca{Y},f\in \ca{S}_\vep},$$ following Lemma 7.2 of \cite{GWY}. Then,
\begin{itemize}
\item $\ca{A(X,X)}=\ca{A(X)}$;
\item If $\ca{Y}\subseteq \ca{Z}$, we have $\ca{A(X,Y)}\subseteq \ca{A(X,Z)}$;
\item $\ca{A(X,Y)}=\ca{A(X,\overline{Y})}$; and
\item If $\ca{Y}_1\subseteq \ca{Y}_2\subseteq \cdots$,  we have $\ca{A(X},\overline{\cup_i\ca{Y}_i})= \varinjlim_i \ca{A(X},\ca{Y}_i)$.
\end{itemize}

Consequently, if $\overline{\cup_i\ca{Y}_i}=\ca{X}$, we have $\ca{A(X)}=\varinjlim_i \ca{A(X},\ca{Y}_i)$.
\end{pro}
\begin{rmks}\label{Cstar algebra S and Cl}
$(1)$ $\ca{S}_\vep$ is isomorphic to
$$C_0([0,\vep),\Cl_+(t\bb{R}))$$
by the following homomorphism. For $f\in \ca{S}_\vep$, we have the even-odd decomposition $f=f_0+f_1$, where $f_0(t)=\frac{f(t)+f(-t)}{2}$ and $f_1=f-f_0$. Then, we can define an element of $C_0([0,\vep),\Cl_+(t\bb{R}))$ by $s\mapsto f_0(s){\bf 1}_f+f_1(s)v$, where ${\bf 1}_f$ is ``$1\in \Cl_+(t\bb{R}))$'' and $v$ is the unit vector of ``$t\bb{R}$'' for $s>0$. Note that $f_1(0)=0$, and hence $f_1(s)v$ makes sense for any $s\geq 0$. One can prove that the correspondence $f\mapsto [s\mapsto f_0(s){\bf 1}_f+f_1(s)v]$ is a $*$-isomorphism. By this description, we have a short exact sequence
$$0\to Cl_\tau(0,\vep)\to \ca{S}_\vep\to \bb{C}\to 0.$$
We have used a similar exact sequence in the proof of \ref{Lemma computation of KK(S,S)}.

$(2)$ Therefore, $(2)$ above gives an isomorphism $\ca{A(X)}\cong \ca{S}_\vep\grotimes Cl_\tau(\ca{X})$ for finite-dimensional $\ca{X}$. This isomorphism has been used in Section \ref{section index theorem Spinc 2n} as a convention.
Moreover, by $(3)$, we have $\ca{A}(l^2(\bb{N}))\cong\varinjlim_N \ca{S}\grotimes Cl_\tau(\bb{R}^N)$, where we choose $\vep=\infty$. The right hand side is the $C^*$-algebra defined in \cite{HKT}.

$(3)$ Since all of $\Cl_+(\bb{R}^N)$, $C_0(\bb{R}^N)$, $Cl_\tau(\bb{R})$ and $\bb{C}$ are nuclear, so is $\ca{S}\grotimes Cl_\tau(\bb{R}^N)$ (see \cite[Section 6.3--6.5]{Mur}). Thus, its limit $\ca{A}(l^2(\bb{N}))$ is again nuclear. This fact will be used in Section \ref{section PD LT spaces}.
\end{rmks}

We can formulate an infinite-dimensional version of the Poincar\'e duality homomorphism, by imitating Proposition \ref{Prop reformulated local Bott element}. 

\begin{dfn}\label{dfn reformulated local boot element}
Let $\ca{X}$ be a Hilbert manifold satisfying Assumption \ref{assumption related to bounded geometry}. Suppose that a Hausdorff group $\Gamma$ acts on $\ca{X}$ in a proper and isometric way.
Then, the {\bf reformulated local Bott element} is defined by
$$[\widetilde{\Theta_{\ca{X},2}}]:=(\{\ca{A(X)}\}_{x\in \ca{X}},\{\beta_x\}_{x\in \ca{X}},\{0\}_{x\in \ca{X}})\in \ca{R}KK_\Gamma(\ca{X},\ca{S}_\vep\grotimes \scr{C}(\ca{X}),\ca{A(X)}\grotimes \scr{C}(\ca{X})),$$
where the $\Gamma$-action on the filed $\ca{A(X)}\grotimes \scr{C}(\ca{X})$ is given by the diagonal action
$$\ca{A(X)}_x\ni a\mapsto \gamma_*(a)\in \ca{A(X)}_{\gamma.x}.$$
\end{dfn}

\begin{dfn}
We define the Poincar\'e duality homomorphism by the Kasparov product
$$\PD:KK_\Gamma(\ca{A(X)},\ca{S}_\vep)\ni [D]\mapsto 
[\widetilde{\Theta_{\ca{X},2}}]\grotimes \sigma_{\ca{X},\scr{C}(\ca{X})}^2([D])\in \ca{R}KK_\Gamma(\ca{X},\ca{S}_\vep\grotimes \scr{C}(\ca{X}),\ca{S}_\vep\grotimes\scr{C}(\ca{X})).$$
We call the image of this map $\PD ([D])$ the {\bf symbol element} of $[D]$.
\end{dfn}

\subsection{The index element of \cite{T4} and a study on $C^*$-algebras of Hilbert manifolds}\label{section PD Cstar algebra HKT Yu}

We would like to apply the above construction for proper $LT$-spaces. The value of ``$\PD$'' at the index element of the ``Dirac operator twisted by the $\tau$-twisted $LT$-equivariant line bundle $\ca{L}$'' should be $\sigma_{\ca{S}_\vep}([\ca{L}])$, according to Example \ref{ex computation of PD}.

The task of this subsection is to define a $\bb{KK}_{LT}^\tau$-element substituting for the reformulated index element studied in Section \ref{section index theorem Spinc 2n}. 
Although we have defined a similar $KK_{LT}^\tau$-element in \cite{T4}, we will reconstruct it for the following two reasons. First, we adopted the $C^*$-algebra of a Hilbert space as a substitute for the function algebra in \cite{T4}, which does not coincides with $\ca{A}(\ca{M}_{L^2_k})$ in the present paper. Second, the definition of the $\ca{A(M})$-module structure of the previous paper was not quite natural, although the Hilbert space is natural. We would like to reconstruct it so that the module structure looks natural. Instead, we must modify the Hilbert space. 

Let us outline this subsection.
Fix $k>1/2$, $l>4$ and $m\geq k+l$.
We first study the $C^*$-algebra $\ca{A}(\ca{M}_{L^2_k})$. We will introduce a possibly bigger $C^*$-algebra $\ca{A}_\HKT(\ca{M}_{L^2_k})$ which is more convenient to study the $K$-homological element. Thanks to it, we can deal with the $U_{L^2_k}$-part and the $\widetilde{M}$-part separately. After that, we will reconstruct an equivariant unbounded Kasparov module which plays a role of the reformulated index element, by the following steps: $(1)$ We will define a Hilbert space substituting for ``$L^2(U_{L^2_k},\ca{L})$''; $(2)$ We will prove that it admits a continuous $U_{L^2_m}$-action $L$ which looks like the left regular representation; $(3)$ We will define the $\ca{A}_\HKT(U_{L^2_k})$-module structure $\pi$ on the Hilbert $\ca{S}$-module $\ca{S}\grotimes \ca{H}$ substituting for ``$\ca{S}\grotimes L^2(U_{L^2_k},\ca{L}\grotimes S_U^*\grotimes S_U)$''; $(4)$ We will define an operator substituting for the Dirac operator $\Dirac$ on $\ca{H}$; $(5)$ We will prove that the triple $(\ca{S}\grotimes \ca{H},\pi,\id\grotimes \Dirac)$ is an unbounded $U_{L^2_m}$-equivariant Kasparov $(\ca{A}_\HKT(U_{L^2_k}),\ca{S})$-module; $(6)$ We will prove that this unbounded Kasparov module can be restricted to $\ca{A}(U_{L^2_k})$; and $(7)$ We will finally prove that the constructed element is independent of $l$ as an element of $\bb{KK}_{LT}^\tau(\ca{A}(\ca{M}_{L^2_k}),\ca{S}_{\vep})$. $(1)$--$(5)$ have been essentially done in \cite{T4}, but we will clarify the proofs. We will also give several different arguments in order to investigate the $C^*$-algebra of a Hilbert space from the viewpoint of global analysis more. When we use the same arguments or estimates of \cite{T4}, we will refer to the corresponding results and we will omit the details.

At the $KK$-theory level, the index element we are going to construct is, roughly speaking, the pullback of the one constructed in \cite{T4} via the inclusion $\ca{A}(\ca{M}_{L^2_k})\to \ca{A}_\HKT(\ca{M}_{L^2_k})$. 
However, the fact that $\ca{A}(\ca{M}_{L^2_k})$ has a natural unbounded $KK$-element, is itself interesting. In addition, it is important to know ``how an element of $\ca{A}(\ca{M}_{L^2_k})$ can be seen as a function on $\ca{M}_{L^2_k}$''. Our result is useful to consider this problem.

Let us begin with the construction of $\ca{A}_\HKT(\ca{M}_{L^2_k})$ inspired by \cite{HKT}. See also \cite{Tro}.
We will also introduce the alternative description of $\ca{A}(\ca{M}_{L^2_k})$ related to $\ca{A}_\HKT(\ca{M}_{L^2_k})$ with the idea of Proposition \ref{properties of GWY algebra} $(3)$ and \cite[Remark 7.7]{GWY}.

Recall that each $C^*$-algebra $\ca{S}$ and $\ca{S}_\vep$ has an (un)bounded multiplier $X$ given by $Xf(t):=tf(t)$.
They correspond to one another under the zero-extension $\iota:\ca{S}_\vep\hookrightarrow \ca{S}$.
We denote the coordinate of $U_{L^2_k,N}$ by $(x_1,y_1,\cdots,x_N,y_N)$, and we denote the corresponding base of the Lie algebra by $\{e_1,f_1,\cdots,e_N,f_N\}$. The Clifford multiplication by $v\in \Lie(U_{L^2_k,N})$ from the left is denoted by $v:\Cl_+(\Lie(U_{L^2_k,N}))\to\Cl_+(\Lie(U_{L^2_k,N}))$.

\begin{dfn}
$(1)$ We define an unbounded multiplier on $Cl_\tau(U_{L^2_k,M}\ominus U_{L^2_k,N})$ by $C_{N+1}^M:=\sum_{j=N+1}^M(x_j\grotimes e_j+y_j\grotimes f_j)$. Using it, we define a Bott-homomorphism
$\beta_N^M:\ca{S}\grotimes Cl_\tau(U_{L^2_k,N}\times \widetilde{M})\to \ca{S}\grotimes Cl_\tau(U_{L^2_k,M}\times \widetilde{M})$
for $N\leq M$ by 
\begin{align*}
\beta_N^M:f\grotimes h\mapsto f(X\grotimes 1+1\grotimes C_{N+1}^M)\grotimes h\in &\ca{S}\grotimes Cl_\tau(U_{L^2_k,M}\ominus U_{L^2_k,N})\grotimes Cl_\tau(U_{L^2_k,N}\times \widetilde{M})\\
&\cong \ca{S}\grotimes Cl_\tau(U_{L^2_k,M}\times \widetilde{M}).
\end{align*}
These make a directed system, namely $\beta_{N'}^M\circ \beta_N^{N'}=\beta_N^M$ for $N\leq N'\leq M$.
We define a $C^*$-algebra $\ca{A}_\HKT(\ca{M}_{L^2_k})$ by the inductive limit of this system:
$$\ca{A}_\HKT(\ca{M}_{L^2_k}):=\varinjlim_N \ca{S}\grotimes Cl_\tau(U_{L^2_k,N}\times \widetilde{M}).$$
The canonical homomorphisms are denoted by $\beta_N^\infty:\ca{S}\grotimes Cl_\tau(U_{L^2_k,N}\times \widetilde{M})
\to \ca{A}_\HKT(\ca{M}_{L^2_k})$.

$(2)$ In the same way, we define $\ca{A}_\HKT(U_{L^2_k}):=\varinjlim_N \ca{S}\grotimes Cl_\tau(U_{L^2_k,N}).$

$(3)$ Let $\ca{S}_\fin$ be the subalgebra generated by the following two functions: 
$$f_{\rm ev}(t):=\frac{1}{t^2+1}\  \text{ and }\ f_{\rm od}(t):=\frac{t}{t^2+1}.$$
Note that $\ca{S}_\fin$ is a dense subalgebra of $\ca{S}$ thanks to the Stone-Weierstrass theorem.
We define the dense subalgebra $\ca{A}_\HKT(\ca{M}_{L^2_k})_\fin$ by the subalgebra generated by 
$$\cup_N \beta_N^\infty\bra{\ca{S}_\fin\grotimes^\alg Cl_{\tau,\scr{S}}(U_{L^2_k,N}\times \widetilde{M})},$$
where $Cl_{\tau,\scr{S}}$ stands for the set of Clifford algebra-valued Schwartz class functions.
\end{dfn}

Let us introduce an alternative description of $\ca{A}(\ca{M}_{L^2_k})$ using Proposition \ref{properties of GWY algebra} $(3)$. Let $\beta_N^M:\ca{S}_\vep\grotimes Cl_\tau(U_{L^2_k,N}\times \widetilde{M})\to \ca{S}_\vep\grotimes Cl_\tau(U_{L^2_k,M}\times \widetilde{M})$ be the homomorphism given by 
$$\beta_N^M:f\grotimes h\mapsto f(X\grotimes 1+1\grotimes C_{N+1}^M)\grotimes h$$
for $N\leq M$. These make a directed system and its limit
$$\varinjlim_N \ca{S}_\vep\grotimes Cl_\tau(U_{L^2_k,N}\times \widetilde{M})$$
is isomorphic to $\ca{A}(\ca{M}_{L^2_k})$ thanks to Proposition \ref{properties of GWY algebra} $(3)$.
The canonical homomorphisms are denoted by $\beta_N^\infty:\ca{S}_\vep\grotimes Cl_\tau(U_{L^2_k,N}\times \widetilde{M})
\to \ca{A}(\ca{M}_{L^2_k})$. 

This alternative description has the following advantages compared to the original one.

\begin{lem}
$(1)$ Thanks to the canonical isomorphism $Cl_\tau(U_{L^2_k,N}\times \widetilde{M})\cong Cl_\tau(U_{L^2_k,N})\grotimes Cl_\tau( \widetilde{M})$, we have a $*$-isomorphism
$$\ca{A(M}_{L^2_k})\cong \ca{A}(U_{L^2_k})\grotimes Cl_\tau(\widetilde{M}).$$
For the same reason, we have $\ca{A}_\HKT(\ca{M}_{L^2_k})\cong \ca{A}_\HKT(U_{L^2_k})\grotimes Cl_\tau(\widetilde{M}).$

$(2)$ By the zero-extension $\iota:\ca{S}_\vep\hookrightarrow \ca{S}$, the unbounded multiplier $X$ on $\ca{S}$ corresponds to the bounded multiplier $X$ on $\ca{S}_\vep$. Consequently, we have the following commutative diagram
$$\begin{CD}
\ca{S}_\vep\grotimes Cl_\tau(U_{L^2_k,N}\times \widetilde{M})
@>\beta_N^M>> \ca{S}_\vep\grotimes Cl_\tau(U_{L^2_k,M}\times \widetilde{M}) \\
@V\iota\grotimes \id  VV @VV\iota\grotimes \id  V \\
\ca{S}\grotimes Cl_\tau(U_{L^2_k,N}\times \widetilde{M})
@>\beta_N^M>> \ca{S}\grotimes Cl_\tau(U_{L^2_k,M}\times \widetilde{M}).
\end{CD}$$
Therefore, we have a $*$-homomorphism $\ca{A}(\ca{M}_{L^2_k})\to \ca{A}_\HKT(\ca{M}_{L^2_k})$ defined by the limit of 
$$\iota\grotimes \id:\ca{S}_\vep\grotimes Cl_\tau(U_{L^2_k,N}\times \widetilde{M})\to \ca{S}_\vep\grotimes Cl_\tau(U_{L^2_k,M}\times \widetilde{M}).$$
This homomorphism preserves the tensor product decomposition given in $(1)$. The constructed homomorphism is denoted by $\iota\grotimes \id$.
\end{lem}

These two $C^*$-algebras admit $LT_{L^2_k}=U_{L^2_k}\times (T\times\Pi_T)$-actions which look like left translations. 

\begin{lem}
$(1)$ The isometric $LT_{L^2_k}$-action on $\ca{M}_{L^2_k}$ induces a continuous action denoted by ``$\lt$'' on $\ca{A}(\ca{M}_{L^2_k})$ thanks to Proposition \ref{YusCstar algebra admits a group action}. It is given by $\lt_g\circ \beta_x(f)=\beta_{g\cdot x}(f)$.

$(2)$ The dense subgroup $LT_\fin:= U_\fin \times (T\times\Pi_T)$ acts on $\ca{A}_\HKT(\ca{M}_{L^2_k})$ by the following: For $g\in  U_{L^2_k,M} \times (T\times\Pi_T)$, we define $\lt_g\in \Aut(\ca{A}(U_{L^2_k,M}))$ by left translation; For $g\in LT_{\fin}$ and $a\in \ca{S}\grotimes Cl_\tau(U_{L^2_k,N}\times\widetilde{M})$, we define $\lt_g[\beta_N^\infty(a)]:=\beta_{N'}^\infty(\lt_g(\beta_N^{N'}(a)))$, where $N'\geq N$ is chosen so that $g\in  U_{L^2_k,N'} \times (T\times\Pi_T)$. This action extends to $LT_{L^2_k}$ as a continuous action.

$(3)$ The action given above can be written as the tensor product of two actions on $\ca{A}_\HKT(U_{L^2_k})\grotimes Cl_\tau(\widetilde{M})$, that is to say, for $g=(u,\gamma)\in U_{L^2_k}\times (T\times\Pi_T)$, $a \in \ca{A}_\HKT(U_{L^2_k})$ and $f\in Cl_\tau(\widetilde{M})$,
$$\lt_g(a\grotimes f)=\lt_u(a)\grotimes \lt_{\gamma}(f).$$

$(4)$ The $*$-homomorphism $\iota\grotimes \id:\ca{A}(\ca{M}_{L^2_k})\to\ca{A}_\HKT(\ca{M}_{L^2_k})$ is $LT_{L^2_k}$-equivariant. Consequently, the $LT_{L^2_k}$-action on $\ca{A(M}_{L^2_k})$ is written as the tensor product of the actions of $U_{L^2_k}$ and $T\times\Pi_T$.
\end{lem}
\begin{proof}
$(1)$ is automatic from Proposition \ref{YusCstar algebra admits a group action}. For $(2)$, see \cite[Theorem 4.18]{T4}. $(3)$ is obvious by definition.
$(4)$ is also obvious by definition. Thanks to the continuity of the group actions proved in $(2)$, the $*$-homomorphism is $LT_{L^2_k}$-equivariant.
\end{proof}

Therefore, we can deal with the $U_{L^2_k}$-part and the $\widetilde{M}$-part separately. We concentrate on the former one until Definition \ref{dfn index element for the present paper for proper LT space}.

We re-define a Hilbert space substituting for the ``$L^2$-space of the line bundle $\ca{L}$ on $U_{L^2_k}$''. In \cite{T4}, we used the infinitely many copies of a Gaussian in order to define the ``$L^2$-space'' by the inductive limit. In the present paper, however, we use the infinite sequence of different Gaussians which tends to the Dirac $\delta$-function. This is because we would like to define the $\ca{A}(U_{L^2_k})$-module structure on a Hilbert $\ca{S}_\vep$-module in a natural way.

Since the restriction of $\ca{L}$ to $U_{L^2_k,N}$ is topologically trivial, we can fix a trivialization and we describe $L^2(U_{L^2_k,N},\ca{L})$ as the set of scalar-valued functions with a $\tau$-twisted $U_{L^2_k,N}$-action.

\begin{dfn}\label{def of L2LTL}
Let $l>4$. 

$(1)$ We define a dense subspace $L^2(U_{L^2_k,N},\ca{L})_\fin$ of $L^2(U_{L^2_k,N},\ca{L})$ by the set of
$$\text{polynomial } \times \frac{(N!)^{l/2}}{\pi^{N/2}}e^{-\frac{1}{2}\sum_{n=1}^N n^l(x_n^2+y_n^2)}.$$

$(2)$ We define an isometric embedding 
\begin{align*}
J_N:L^2(U_{L^2_k,N},\ca{L})_\fin\to &L^2(U_{L^2_k,N+1},\ca{L})_\fin \\
&\cong L^2(U_{L^2_k,N},\ca{L})_\fin\grotimes L^2(\bb{R}^2,\ca{L})_\fin
\end{align*}
by the following:
$$\phi\mapsto \phi\grotimes \sqrt{\frac{(N+1)^l}{\pi}}e^{-\frac{(N+1)^l}{2}(x_{N+1}^2+y_{N+1}^2)}.$$
Note that the Gaussian $\sqrt{\frac{(N+1)^l}{\pi}}e^{-\frac{(N+1)^l}{2}(x_{N+1}^2+y_{N+1}^2)}$ is an $L^2$-unit vector, and hence $J_N$ is an isometric embedding. We denote the composition of these maps by
$$J_N^M:=J_{M-1}\circ J_{M-2}\circ\cdots \circ J_N:L^2(U_{L^2_k,N},\ca{L})_\fin\to L^2(U_{L^2_k,M},\ca{L})_\fin.$$
These make a directed system.

$(3)$ The {\it algebraic} inductive limit of this system is denoted by 
$$\ud{L^2(U_{L^2_k},\ca{L})_\fin}.$$
We define a Hilbert space $\ud{L^2(U_{L^2_k},\ca{L})}$ by the completion of $\ud{L^2(U_{L^2_k},\ca{L})_\fin}$.
\end{dfn}
\begin{rmks}\label{rmk ugly Hilbert space}
$(1)$ $\ud{L^2(U_{L^2_k},\ca{L})}$ can be given by the {\it Hilbert space} inductive limit of the directed system
$$\cdots \xrightarrow{J_{N-1}} L^2(U_{L^2_k,N},\ca{L})\xrightarrow{J_N} L^2(U_{L^2_k,N+1},\ca{L})\xrightarrow{J_{N+1}} \cdots ,$$
where $J_N:L^2(U_{L^2_k,N},\ca{L})\to L^2(U_{L^2_k,N+1},\ca{L})$ is the canonical extension of $J_N:L^2(U_{L^2_k,N},\ca{L})_\fin\to L^2(U_{L^2_k,N+1},\ca{L})_\fin$. 
The canonical homomorphism is denoted by $J_N^\infty:L^2(U_{L^2_k,N},\ca{L})\hookrightarrow \ud{L^2(U_{L^2_k},\ca{L})}$.

$(2)$ We will construct an unbounded Kasparov module using this Hilbert space, which depends on $l$. However, the resulting $\bb{KK}_{LT}^\tau$-element is independent of $l$ as proved in Proposition \ref{prop KK element is independent of l}.

$(3)$ Incidentally, an element of $\ud{L^2(U_{L^2_k},\ca{L})}$ is an ``asymptotically Dirac $\delta$-function''. We will explain the reason why we should adopt such a strange Hilbert space in Section \ref{section unsolved}.
\end{rmks}

\begin{nota}
We use the following symbols about the ``Gaussians'':
\begin{align*}
\vac_N^N(x_N,y_N) &:=\sqrt{\frac{N^l}{\pi}}e^{-\frac{N^l}{2}(x_{N}^2+y_{N}^2)},\\
\vac_N^M &:=J_N^M(\vac_N^N)=
\vac_N^N\grotimes \vac_{N+1}^{N+1}\grotimes \cdots \grotimes \vac_M^M,\\
\vac_N^\infty&:=J_{M}^{\infty}(\vac_N^M)=\vac_N^N\grotimes \vac_{N+1}^{N+1}\grotimes \cdots,\\
\vac&:= \vac_1^\infty.
\end{align*}
\end{nota}

The Hilbert space $L^2(U_{L^2_k,N},\ca{L})$ admits a $\tau$-twisted continuous $U_{L^2_k,N}$-action $L$ defined by the following:
For $\phi\in L^2(U_{L^2_k,N},\ca{L})$ and $(g,z)\in U_{L^2_k,N}^\tau$, we define
$$[L_{({g},z)}\phi]({x}):=
z\phi({x}-{g})\tau({g},{x}).$$
Recall that $\tau(g,x)$ is given by
$$\tau({g},{x})
=e^{\sqrt{-1}\sum_{n=1}^Nn^{1-2k}(a_ny_n-b_nx_n)}$$
for ${x}=(x_1,y_1,\cdots,x_N,y_N),{g}=(a_1,b_1,\cdots,a_N,b_N)\in U_{L^2_k,N}$.

We need to prove that the action $L$ extends to $U_{L^2_m}$ for $m\geq k+l$. We begin with an elementary exercise of the Lebesgue integration.

\begin{lem}\label{Gaussian lemma}
On $\bb{R}^{2N}$, consider the Gaussian function $\vac(x):=\pi^{-N/2}e^{-{\|x\|^2}/{2}}$ and an element $a\in\bb{R}^{2N}$. For any $\epsilon>0$, there exists $\Delta>0$ such that $\|a\|<\Delta$ implies that 
$$\int_{\bb{R}^{2N}}|\vac(x-a)-\vac(x)|^2dx<\epsilon.$$
$\Delta$ can be chosen independently of $N$.
\end{lem}
\begin{proof}
This is clear from the dominated convergence theorem, but we describe the complete proof which will be useful in the following discussions.

We may assume that $a$ is of the form $(a_1,0,0,\cdots,0)$ for $a_1=\|a\|$, because the Gaussian is rotation invariant. 
Then, the integral we are estimating becomes
$$\int_{\bb{R}}\pi^{-1/2}\bbbra{e^{-\frac{(x-a)^2}{2}}-e^{-\frac{x^2}{2}}}^2dx.$$
We can choose a large number $K>0$ such that 
$$\int_{|x|\geq K}\pi^{-1/2}e^{-x^2}dx<\frac{\epsilon}{4}.$$
We may assume that $a_1<K$, and hence we have inequalities
$$\int_{x\geq 2K}\pi^{-1/2}\bbbra{e^{-\frac{(x-a_1)^2}{2}}-e^{-\frac{x^2}{2}}}^2dx
\leq \int_{x \geq2K}\pi^{-1/2}e^{-(x-a_1)^2}dx
<\frac{\epsilon}{4}; \text{ and}$$
$$\int_{x\leq -2K}\pi^{-1/2}\bbbra{e^{-\frac{(x-a_1)^2}{2}}-e^{-\frac{x^2}{2}}}^2dx
\leq \int_{x \leq -2K}\pi^{-1/2}e^{-x^2}dx<\frac{\epsilon}{4}.$$

Finally, choose $\Delta$ small enough so that the following holds:
$$4K\cdot \pi^{-1/2}\bbbra{e^{-\frac{(x-\Delta)^2}{2}}-e^{-\frac{x^2}{2}}}^2<\frac{\epsilon}{2}$$
for all $-2K\leq x\leq 2K$. This is possible because the Gaussian is uniformly continuous. If $|a|<\Delta$, we have an inequality$$4K\cdot \pi^{-1/2}\bbbra{e^{-\frac{(x-a)^2}{2}}-e^{-\frac{x^2}{2}}}^2<\frac{\epsilon}{2}$$
for all $-2K\leq x\leq 2K$, and we obtain the result.
\end{proof}

\begin{pro}[See also Proposition 4.7 of \cite{T4}]\label{actions L and R are well defined}
Let $m\geq k+l$.
The Hilbert space $\ud{L^2(U_{L^2_k},\ca{L})}$ admits a continuous $U_{L^2_m}^\tau$-action $L$
such that the restriction of it to $U_{L^2_m,N}^\tau$ on the image of $L^2(U_{L^2_k,N},\ca{L})$ in $L^2(U_{L^2_k},\ca{L})$ coincides with the left regular representation: $[L_{({g},z)}\phi]({x})=z\phi({x}-{g})\tau({g},{x}).$
\end{pro}
\begin{proof}
We have defined the $U_{L^2_m,N}$-action $L$ on $L^2(U_{L^2_k,N},\ca{L})$ for each $N$. These are compatible in the following sense: Let $i_{N,N+1}:U_{L^2_m,N}\hookrightarrow U_{L^2_m,N+1}$ be the canonical embedding; For ${g}\in U_{L^2_m,N}$ and $\phi\in L^2(U_{L^2_k,N},\ca{L})$,
\begin{center}
$L_{(i_{N,N+1}({g}),z)}J_N(\phi)=J_N(L_{({g},z)}\phi)$.
\end{center}
This means that $U_\fin$ acts on $\ud{L^2(U_{L^2_k},\ca{L})_\fin}$. Since this action is unitary, it extends to a $U_\fin$-action on $\ud{L^2(U_{L^2_k},\ca{L})}$. We would like to extend this action to the whole group $U_{L^2_m}$ by $L_{({g},z)}\phi:=\lim L_{({g_i},z)}\phi$ for $\phi\in \ud{L^2(U_{L^2_k},\ca{L})}$ and $g=\lim g_i$ in $U_{L^2_m}$.

For this aim, it suffices to check that the $U_\fin$-action is continuous in the $L^2_m$-topology. Concretely, we need to prove the following: For any $\phi\in \ud{L^2(U_{L^2_k},\ca{L})}$ and a net $\{\bra{g_i,z_i}\}\subseteq U_\fin^\tau$ converging to $\bra{g,z}\in U_\fin^\tau$ in the $L^2_m$-topology, the net $\{L_{({g_i},z_i)}\phi\}$ converges to $L_{(g,z)}\phi$ in norm.
We would like to find, for a positive real number $\epsilon>0$, a large element $i_0$ such that $i\geq i_0$ implies that $\|L_{({g_i},z)}\phi- L_{(g,z)}\phi\|<3\epsilon$. Since $\ud{L^2(U_{L^2_k},\ca{L})_\fin}$ is dense in $\ud{L^2(U_{L^2_k},\ca{L})}$, we can find $\phi'\in \ud{L^2(U_{L^2_k},\ca{L})_\fin}$ satisfying $\|\phi-\phi'\|<\epsilon$. Then, we have an inequality
\begin{align*}
\|L_{(g_i,z_i)}\phi- L_{(g,z)}\phi\|
&\leq \|L_{(g_i,z_i)}\phi- L_{(g_i,z_i)}\phi'\|+\|L_{(g_i,z_i)}\phi'- L_{(g,z)}\phi'\|+\|L_{(g,z)}\phi'- L_{(g,z)}\phi\|\\
&<2\epsilon+\|L_{(g_i,z_i)}\phi'- L_{(g,z)}\phi'\|.
\end{align*}
Thus, we may assume that $\phi \in \ud{L^2(U_{L^2_k},\ca{L})_\fin}$. By definition, $\phi$ is of the form
$\phi_0\grotimes \vac_{N}^\infty,$
for some $\phi_0\in L^2(U_{L^2_k,N-1},\ca{L})$ for sufficiently large $N$.

We may focus on $\vac_N^\infty$ and we may assume that $z=z_i=1$ for each $i$. This is because
$$L_{({g}_i,z_i)}(\phi_0\grotimes \vac_{N}^\infty)
=L_{((a_{i,1},b_{i,1},\cdots,a_{i,N-1},b_{i,N-1}),z_i))}\phi_0\grotimes L_{((a_{i,N},b_{i,N},\cdots),1)}(\vac_N^\infty)$$
for ${g_i}=(a_{i,1},b_{i,1},\cdots)\in U_\fin$ and ${g}=(a_{1},b_{1},\cdots)\in U_\fin$; The first component converges to $L_{((a_{1},b_{1},\cdots,a_{N-1},b_{N-1}),z))}\phi_0$ in norm as $i\to \infty$, thanks to a standard argument of Lebesgue integration. This also means that we may ignore the $z$-component.
Moreover, by a standard argument of group actions, we may assume that the limit $g$ is the identity element.

Combining the arguments so far, we notice that we need to prove that
$\|L_{({g_i},1)}\vac-\vac\|\to 0$. More concretely, we prove that
$$\int_{U_{L^2_k,{M_i}}}\left|
\vac_1^{M_i}({x}-{g_i})\tau({g},{x})-\vac_1^{M_i}({x})\right|^2dx_1dy_1\cdots dx_{M_i}dy_{M_i}\to 0 $$
as $i\to \infty$, where ${M_i}$ is chosen so that $g_i\in U_{L^2_k,M_i}$.

This integral converges to $0$ by the following argument.
First, we notice that
{\small\begin{align*}
&\int_{U_{L^2_k,M_i}}\left|
\vac_1^{M_i}(x-g_i)\tau(g_i,x)-\vac_1^{M_i}(x)\right|^2dx_1dy_1\cdots dx_{M_i}dy_{M_i}\\
&\ \ \ =\int_{U_{L^2_k,{M_i}}}\left|
\vac_1^{M_i}(x-g_i)\tau(g_i,x)-\vac_1^{M_i}(x)\tau(g_i,x)
+\vac_1^{M_i}(x)\tau(g_i,x)
-\vac_1^{M_i}(x)\right|^2dx_1dy_1\cdots dx_{M_i}dy_{M_i}\\
&\ \ \ \leq 2\int_{U_{L^2_k,{M_i}}}\left|
\vac_1^{M_i}(x-g_i)-\vac_1^{M_i}(x)\right|^2dx_1dy_1\cdots dx_{M_i}dy_{M_i} \\
&\ \ \ \ \ \ 
+2\int_{U_{L^2_k,{M_i}}}
\vac_1^{M_i}(x)^2\left|\tau(g_i,x)-1\right|^2dx_1dy_1\cdots dx_{M_i}dy_{M_i}.
\end{align*}}

For the former integral, by the change of variables given later,
\begin{align*}
&\int_{U_{L^2_k,{M_i}}}\left|
\vac_1^{M_i}({x}-{g_i})-\vac_1^{M_i}({x})\right|^2dx_1dy_1\cdots dx_{M_i}dy_{M_i}\\
&\ \ \ =\int_{U_{L^2_k,{M_i}}}\frac{({M_i}!)^l}{\pi^{N}}\left|
e^{-\frac{1}{2}\sum_{j=1}^{M_i}j^l[(x_j-a_{i,j})^2+(y_j-b_{i,j})^2]}-e^{-\frac{1}{2}\sum_{j=1}^{M_i}j^l[x_j^2+y_j^2]}\right|^2dx_1dy_1\cdots dx_{M_i}dy_{M_i} \\
&\ \ \ =\int_{U_{L^2_k,{M_i}}}\frac{1}{\pi^{N}}\left|
e^{-\frac{1}{2}\sum_{j=1}^{M_i}[(X_j-A_{i,j})^2+(Y_j-B_{i,j})^2]}-e^{-\frac{1}{2}\sum_{j=1}^{M_i}[X_j^2+Y_j^2]}\right|^2dX_1dY_1\cdots dX_{M_i}dY_{M_i},
\end{align*}
where $X_j:=\sqrt{j^l}x_j$ and $Y_j:=\sqrt{j^l}y_j$, and we introduced the notations $A_{i,j}:=\sqrt{j^l}a_{i,j}$ and $B_{i,j}:=\sqrt{j^l}b_{i,j}$. Then, thanks to Lemma \ref{Gaussian lemma}, this integral is arbitrary small when the norm $\bbra{\sum_j(A_{i,j}^2+B_{i,j}^2)}^{1/2}$ is sufficiently small. This quantity is the $L^2_{k+l}$-norm of ${g_i}$, and it is not greater than the $L^2_m$-norm of $g_i$ since $m\geq k+l$.

For the latter integral, 
we follow the same story of Lemma \ref{Gaussian lemma}. 
Let $\epsilon>0$.
Note that the function $x\mapsto \left|\tau({g_i},x)-1\right|^2$ is bounded. We can find $K>0$ such that 
$$\int_{\sum_{j=1}^{M_i}(x_j^2+y_j^2)\geq K}
\vac_1^{M_i}({x})^2\left|\tau({g_i},{x})-1\right|^2dx_1dy_1\cdots dx_{M_i}dy_{M_i}<\frac{\epsilon}{2}.$$

Thus, it suffices to prove that $\left|\tau({g_i},x)-1\right|<\frac{\epsilon}{2}$ on $\bbra{{x}\ \middle|\ \sum_{j=1}^{M_i}(x_j^2+y_j^2)\leq K}$. For this aim, it suffices to check that $\sum_{j=1}^{M_i}j^{1-2k}(a_{i,j}y_j-b_{i,j}x_j)$ is uniformly small there. This quantity is bounded above by
$$\sqrt{\sum_{j=1}^{M_i}(x_j^2+y_j^2)}\sqrt{\sum_{j=1}^{M_i}j^{2-4k}(a_{i,j}^2+b_{i,j}^2)}
$$
thanks to the Cauchy-Schwarz inequality. $\sqrt{\sum_{j=1}^Mj^{2-4k}(a_{i,j}^2+b_{i,j}^2)}$ is the $L^2_{1-k}$-norm and it is not greater than the $L^2_m$-norm of ${g}_i$.
\end{proof}

In order to define a substitute for the index element associated to the line bundle $\ca{L}$ as an unbounded $U_{L^2_m}$-equivariant Kasparov $(\ca{A}(U_{L^2_k}),\ca{S}_\vep)$-module, we set
\begin{align*}
\ca{H}&:=\ud{L^2(U_{L^2_k},\ca{L})}\grotimes S_U^*\grotimes S_U; \text{ and }\\
\ca{H}_\fin&:= \ud{L^2(U_{L^2_k},\ca{L})_\fin}\grotimes^\alg S_{U,\fin}^*\grotimes^\alg S_{U,\fin}
\end{align*}
and we consider the Hilbert $\ca{S}_\vep$-module
$\ca{S}_\vep\grotimes \ca{H}$.
The next task is to define a left module structure
$$\pi:\ca{A}(\ca{M}_{L^2_k})\to \bb{L}_{\ca{S}_\vep}(\ca{S}_\vep\grotimes \ca{H}).$$

We will construct it with the help of $\ca{A}_\HKT(U_{L^2_k})$ and the $*$-homomorphism $\iota\grotimes \id:\ca{A}(U_{L^2_k})\to \ca{A}_\HKT(U_{L^2_k})$.
As essentially proved in \cite{T4}, $\ca{A}_\HKT(U_{L^2_k})$ admits a $U_{L^2_m}$-equivariant $*$-homomorphism $\pi:\ca{A}_\HKT(U_{L^2_k})\to\bb{L}_{\ca{S}}(\ca{S}\grotimes \ca{H})$ as outlined in the following proposition.

Before that, we notice that we can define the ``multiplication operators by $\{x_n,y_n\}_{n\in\bb{N}}$'' on $\ud{L^2(U_{L^2_k},\ca{L})_\fin}$. Strictly speaking, for $\phi\in J_N^\infty(L^2(U_{L^2_k,N},\ca{L}))\subseteq \ud{L^2(U_{L^2_k},\ca{L})_\fin}$, $x_n\phi$ is defined by $J_{N'}^\infty(x_nJ_N^{N'}(\psi))$, where $N'$ is chosen so that $n\leq N'$.

\begin{pro}
Let $\pi(C_N^\infty)$ be the infinite sum of unbounded operators on $\ud{L^2(U_{L^2_k},\ca{L})}\grotimes S_U^*\grotimes S_U$
$$\pi(C_N^\infty):=\sum_{k=N}^\infty \bra{x_k\grotimes c^*(e_k)+y_k\grotimes c^*(f_k)}.$$
This infinite sum strongly converges to an unbounded self-adjoint operator on $\ud{L^2(U_{L^2_k},\ca{L})}\grotimes S_U^*\grotimes S_U$.
\end{pro}
\begin{proof}
It is proved by the same argument of \cite[Proposition 4.24]{T4}. The necessary change is only the following: In the previous paper, we used the property that 
$$
\text{``the }L^2\text{-norms of }n^{-l}x_n\sqrt{\frac{1}{\pi}}e^{-\frac{1}{2}(x_n^2+y_n^2)}\ \text{ and }\ n^{-l}y_n\sqrt{\frac{1}{\pi}}e^{-\frac{1}{2}(x_n^2+y_n^2)}\ \text{ are small'',}
$$
but, in the present paper, we need to use the property that
$$
\text{``the }L^2\text{-norms of }x_n\sqrt{\frac{n^l}{\pi}}e^{-\frac{n^l}{2}(x_n^2+y_n^2)}\ \text{ and }\ y_n\sqrt{\frac{n^l}{\pi}}e^{-\frac{n^l}{2}(x_n^2+y_n^2)}\ \text{ are small''}
$$
instead. This property is guaranteed because the Gaussian $\sqrt{\frac{n^l}{\pi}}e^{-\frac{n^l}{2}(x_n^2+y_n^2)}$ is more localized as $n$ becomes bigger.
We leave the details to the reader.
\end{proof}

We have a $*$-representation of $Cl_\tau(U_{L^2_k,N})$ on $L^2(U_{L^2_k,N})\grotimes S_{U_N}^*\grotimes S_{U_N}$ by the combination of the multiplication of scalar-valued functions and the Clifford multiplication $c^*$. We denote it by
$$\mu_N:Cl_\tau(U_{L^2_k,N})\to \bb{L}_{\bb{C}}(L^2(U_{L^2_k,N})\grotimes S_{U_N}^*\grotimes S_{U_N}).$$

\begin{dfn-pro}\label{dfn-pro rep of AHKT}
$(1)$ For every $N\in \bb{N}$, we define a $*$-homomorphism $\pi_N:\ca{S}\grotimes Cl_\tau(U_{L^2_k,N})\to \bb{L}_{\ca{S}}(\ca{S}\grotimes \ca{H})$ by
$$\pi_N(f\grotimes h):=f(X\grotimes \id+\id\grotimes \pi(C_{N+1}^\infty))\grotimes \mu_N(h)$$
with the canonical isomorphism 
$$\ca{H}\cong \bbbra{\ud{L^2(U_{L^2_k}\ominus U_{L^2_k,N},\ca{L})}\grotimes S^*_{U\ominus U_N}\grotimes S_{U\ominus U_N}}\grotimes \bbbra{L^2(U_{L^2_k,N},\ca{L})\grotimes S^*_{U_N}\grotimes S_{U_N}}.$$
Then, $\pi_N$'s satisfy the following commutative diagram:
$$\xymatrix{
\ca{S}\grotimes Cl_\tau(U_{L^2_k,N}) \ar^{\beta_{N}^{M}}[rr] \ar_{\pi_N}[rd] && \ca{S}\grotimes Cl_\tau(U_{L^2_k,M}) \ar^{\pi_M}[ld]\\
& \bb{L}_{\ca{S}}(\ca{S}\grotimes \ca{H})&}$$
and hence it defines a $*$-homomorphism
$$\pi:\ca{A}_\HKT(U_{L^2_k})\to \bb{L}_{\ca{S}}(\ca{S}\grotimes \ca{H}).$$

$(2)$ $\pi\Bigl(\iota\grotimes\id(\ca{A}(U_{L^2_k}))\Bigr)(\ca{S}\grotimes \ca{H})\subseteq \ca{S}_\vep\grotimes \ca{H}$. 
Therefore,  we can regard $\pi\circ \iota \grotimes \id$ as a map $\ca{A}(U_{L^2_k})\to \bb{L}_{\ca{S}_\vep}(\ca{S}_\vep\grotimes \ca{H})$.
We use the same symbol $\pi$ to denote $\pi\circ \iota \grotimes \id$, in order to simplify the notation.

$(3)$ These homomorphisms are $U_{L^2_m}$-equivariant in the following sense: For $a\in \ca{A}_\HKT(U_{L^2_k})$, $\phi\in \ca{S}\grotimes \ca{H}$ and $g\in U_{L^2_m}$, we have $L_{(g,z)}[\pi(a)\phi]=\pi(\lt_g(a))(L_{(g,z)}\phi)$, and similarly for $\ca{A}(U_{L^2_k})$. 
\end{dfn-pro}
\begin{proof}
$(1)$ and $(3)$ are obvious from the definitions and the computation for $U_{L^2_k,\fin}$. See \cite[Definition-Theorem 4.26 and Lemma 5.1]{T4}.

We check $(2)$ in detail. First, we notice that $\ca{S}\cong C_0(\bb{R}_{\geq 0},\Cl_+(t\bb{R}))\subseteq C_0(\bb{R}_{\geq 0})\grotimes\Cl_+(\bb{R})$, just like Remark \ref{Cstar algebra S and Cl} $(1)$. Thus,
$$\ca{S}\grotimes \ca{H}\cong C_0([0,\infty),\ca{H}\grotimes \Cl_+(t\bb{R}));$$
$$\ca{S}_\vep\grotimes \ca{H}\cong C_0([0,\vep),\ca{H}\grotimes \Cl_+(t\bb{R})).$$
We denote the base of $\bb{R}=T_t\bb{R}_{\geq 0}$ for $t>0$ by $e_1$ in the following.

It is sufficient to prove that $\pi\circ \iota\grotimes\id(\beta_N^\infty(f\grotimes a))(g\grotimes \phi)\in \ca{S}_\vep\grotimes\ca{H}$ for $f\in \ca{S}_\vep$, $g\in \ca{S}$, $a\in Cl_\tau(U_{L^2_k,N})$  and $\phi\in J_M^\infty(L^2(U_{L^2_k,M},\ca{L}))\grotimes S_{U_M}^*\grotimes S_{U_M}$, for the following reason: $\ca{S}_\vep\grotimes\ca{H}$ is closed in $\ca{S}\grotimes \ca{H}$; $\cup_N \beta_N^\infty(\ca{S}_\vep\grotimes^\alg Cl_\tau(U_{L^2_k,N}))$ is dense in $\ca{A}(U_{L^2_k})$; $\ca{S}_\vep\grotimes^\alg\ca{H}_\fin$ is dense in $\ca{S}_\vep\grotimes\ca{H}$; and all the operations are continuous.
We may assume that $N\leq M$ and $f$ is smooth. By Lemma \ref{Lemma C is bounded on dom(D) to H} $(2)$ proved later, we can prove that
$$\pi\circ \iota\grotimes\id(\beta_N^\infty(f\grotimes a))(g\grotimes \phi)=\lim_{N'\to\infty}f(X\grotimes\id+\id\grotimes C_{N+1}^{N'})\grotimes\mu_N(a)(g\grotimes \phi),$$
and hence it is sufficient to prove that $f(X\grotimes\id+\id\grotimes C_{N+1}^{N'})\grotimes\mu_N(a)(g\grotimes \phi)\in \ca{S}_\vep\grotimes \ca{H}$ for all $N'$. We may assume that $\phi$ is of the form $\phi_1\grotimes\phi_2\grotimes \vac_{M+1}^\infty$, where $\phi_1\in L^2(U_{L^2_k,N},\ca{L})\grotimes S_{U_N}^*\grotimes S_{U_N}$ and $\phi_2\in L^2(U_{L^2_k,M}\ominus U_{L^2_k,N},\ca{L})\grotimes S_{U_M\ominus U_N}^*\grotimes S_{U_M\ominus U_N}$ and $N'\geq M$. Then,
\begin{align*}
f(X\grotimes\id+\id\grotimes C_{N+1}^{N'})\grotimes\mu_N(a)(g\grotimes \phi) =\pm f(X\grotimes\id+\id\grotimes C_{N+1}^{N'})(g\grotimes \phi_2\grotimes \vac^{N'}_{M+1})\grotimes \mu_N(a)\phi_1\grotimes \vac_{N'}^\infty.
\end{align*}
In order to compute $f(X\grotimes\id+\id\grotimes C_{N+1}^{N'})(g\grotimes \phi_2\grotimes \vac^{N'}_{M+1})$, we use the set-theoretical description. We regard this element as a section on $[0,\infty)\times (U_{L^2_k,N'}\ominus U_{L^2_k,N})$.
Suppose that $f(t)=f_0(t^2)+tf_1(t^2)\grotimes e_1$ for $f_0,f_1\in C_0([0,\vep^2))$ (it is possible since $f$ is smooth). Then,
\begin{align*}
&f(X\grotimes\id+\id\grotimes C_{N+1}^{N'})(g\grotimes \phi_2\grotimes \vac^{N'}_{M+1})(t,x) \\
&=f_0(t^2+\|x\|^2)g(t)\phi_2\grotimes \vac^{N'}_{M+1}(x)+tf_1(t^2+\|x\|^2)g(t)\phi_2\grotimes \vac^{N'}_{M+1}(x) \\
&\ \ \ +(-1)^{\partial g} f_1(t^2+\|x\|^2)g(t)c^*(x)\phi_2\grotimes \vac^{N'}_{M+1}(x).
\end{align*}
This function vanishes outside $[0,\vep)\times (U_{L^2_k,N'}\ominus U_{L^2_k,N})$, and thus $f(X\grotimes\id+\id\grotimes C_{N+1}^{N'})(g\grotimes \phi_2\grotimes \vac^{N'}_{M+1})\in \ca{S}_\vep\grotimes L^2(U_{L^2_k,M}\ominus U_{L^2_k,N},\ca{L})\grotimes S_{U_M\ominus U_N}^*\grotimes S_{U_M\ominus U_N}$. Consequently, $\pi\circ \iota\grotimes\id(\beta_N^\infty(f\grotimes a))(g\grotimes \phi)\in \ca{S}_\vep\grotimes\ca{H}$.
\end{proof}

Let us define ``the Dirac operator $\Dirac$ twisted by $\ca{L}$'' acting on $\ca{H}$.
We define two operators ${}^R\Dirac$ and ${}^L\Dirac$, and then we define $\Dirac$ by a linear combination of them following \cite[Definition 5.2]{T4}. See also \cite{Kos,Was,FHTII,Mei13} for details on algebraic Dirac operators.

\begin{dfn}[Definition 5.2 of \cite{T4}]\label{def of L' and R'}
$(1)$ We define two linear maps from $\Lie(U_\fin)$ to the set of unbounded operators on $\ud{L^2(U_{L^2_k},\ca{L})_\fin}$ as follows: For $\phi\in \ud{L^2(U_{L^2_k},\ca{L})_\fin}$,
$$dR'({e_n})\phi:= \frac{\partial \phi}{\partial x_n}+in^ly_n\phi,\text{ and }\ 
dR'({f_n})\phi:= \frac{\partial \phi}{\partial y_n}-in^lx_n\phi,
$$
$$
dL'({e_n})\phi:= -\frac{\partial \phi}{\partial x_n}+in^ly_n\phi, \text{ and }\ 
dL'({f_n})\phi:= -\frac{\partial \phi}{\partial y_n}-in^lx_n\phi.$$
The strict definition of $\frac{\partial \phi}{\partial x_n}$ is the following: For $\phi =J_N^\infty(\psi)\in J_N^\infty [L^2(U_{L^2_k,N},\ca{L})]$, $\frac{\partial \phi}{\partial x_n}:=J_{N'}^\infty(\frac{\partial J^{N'}_N(\psi)}{\partial x_n})$, where $N'$ is chosen to be greater than $N$ and $n$. We linearly extend the maps $dL'$ and $dR'$ to $\Lie(U_\fin)\grotimes \bb{C}$.

$(2)$ For $\phi\in \ca{H}_\fin$, we define
\begin{align*}
{}^L\Dirac\phi&:=\sum_nn^{-l/4}\bbbra{dL'(z_n)\grotimes \gamma^*(\overline{z_n})\grotimes \id+dL'(\overline{z_n})\grotimes \gamma^*(z_n)\grotimes \id}\phi, \\
{}^R\Dirac\phi&:=\sum_nn^{-l/4}\bbbra{dR'(z_n)\grotimes \id\grotimes \gamma(\overline{z_n})+dR'(\overline{z_n})\grotimes \id\grotimes \gamma(z_n)}\phi, \\
\Dirac\phi&:=\frac{1}{\sqrt{2}}{}^R\Dirac\phi+\frac{i}{\sqrt{2}}{}^L\Dirac\phi.
\end{align*}
\end{dfn}

\begin{rmks}

$(1)$ The summand  $dR'(z_n)\grotimes \id\grotimes \gamma(\overline{z_n})+dR'(\overline{z_n})\grotimes \id\grotimes \gamma(z_n)$ of ${}^R\Dirac$ can be rewritten as $dR'(e_n)\grotimes \id\grotimes \gamma(e_n)+dR'(f_n)\grotimes \id\grotimes \gamma(f_n)$, and similarly for ${}^L\Dirac$. Thus, they should be regarded as Dirac operators.

$(2)$ By the two Clifford multiplications 
$$c(v):= \frac{1}{\sqrt{2}}\bra{\id\grotimes\gamma(v)-i\gamma^*(v)\grotimes \id} \text{ and }\ c^*(v):=\frac{\sqrt{-1}}{\sqrt{2}}\bra{\id\grotimes\gamma(v)+i\gamma^*(v)\grotimes \id}$$
 for $v\in \Lie(U_{L^2_k})$, we can rewrite the summand of $\Dirac$
$$\frac{\sqrt{-1}}{\sqrt{2}}\bbbra{dL'(z_n)\grotimes \gamma^*(\overline{z_n})\grotimes \id+dL'(\overline{z_n})\grotimes \gamma^*(z_n)\grotimes \id}+\frac{1}{\sqrt{2}}\bbbra{dR'(z_n)\grotimes \id\grotimes \gamma(\overline{z_n})+dR'(\overline{z_n})\grotimes \id\grotimes \gamma(z_n)}$$
can be rewritten as, by a direct computation,
$$n^{-l/4}\bbra{\frac{\partial}{\partial x_n}\grotimes c(e_n)+\frac{\partial}{\partial y_n}\grotimes c(f_n)-n^lx_n\grotimes c^*(Je_n)-n^ly_n\grotimes c^*(Jf_n)}.$$
This operator resembles the Bott-Dirac operator studied in \cite{HK}. In fact, we can compute the spectrum of the square of this operator just like the Bott-Dirac case.

$(3)$ Each of ${}^L\Dirac\phi$, ${}^R\Dirac\phi$ and $\Dirac\phi$ is a finite sum for every $\phi\in\ca{H}_\fin$. This is because the element ``$\vac$'' is killed by $dR'(z_n)$ and $dL'(\overline{z_n})$ for all $n$, and ${\bf 1}^*_f\grotimes {\bf 1}_f$ is killed by $\gamma^*(\overline{z_n})\grotimes \id_{S_U}$ and $ \id_{S_U^*}\grotimes \gamma(z_n)$.
$\ca{H}_\fin$ is the common core of these operators.

$(4)$ ${}^R\Dirac$ is essentially self-adjoint and ${}^L\Dirac$ is essentially skew-adjoint.\footnote{Note that $[\gamma(z_n)]^*= -\gamma(\overline{z_n})$, $[\gamma^*(z_n)]^*= \gamma^*(\overline{z_n})$, $dL'(z_n)^*=-dL'(\overline{z_n})$ and $dR'(z_n)^*=-dR'(\overline{z_n})$.} They anti-commute with one another. Consequently, $\Dirac$ is essentially self-adjoint. The self-adjoint extension of $\Dirac$ is denoted by the same symbol.
\end{rmks}

Since our operator ``has a liner potential'', any element of the ``Sobolev space defined by $\Dirac$'' satisfies the following  estimate on the ``decay rate''. 

\begin{lem}\label{Lemma C is bounded on dom(D) to H}
$(1)$ The operator $C:\ca{H}_\fin\to \ca{H}$ extends to a bounded operator from $\dom(\Dirac)$ to $\ca{H}$.

$(2)$ $C_1^N$ also extends to a bounded operator from $\dom(\Dirac)$ to $\ca{H}$, and $\{C_1^N\}$ converges to $ C$ in $\bb{L}_{\bb{C}}(\dom(\Dirac),\ca{H})$.
\end{lem}
\begin{proof}
We prove both statements at the same time.

We first take complete orthonomal systems (CONS for short) of $\ca{H}$ and $\dom(\Dirac)$. For $\overrightarrow{\alpha}=(\alpha_1,\alpha_2,\cdots)$, $\overrightarrow{\beta}=(\beta_1,\beta_2,\cdots)$, $\overrightarrow{\xi}=(\xi_1,\xi_2,\cdots,\xi_M)$, and  $\overrightarrow{\eta}=(\eta_1,\eta_2,\cdots,\eta_N)$ such that $\alpha_i,\beta_i\in \bb{Z}_{\geq 0}$, $\alpha_i=\beta_i=0$ except for finitely many $i$'s, $\xi_i,\eta_i\in \bb{N}$, $N,M\in \bb{Z}_{\geq 0}$, $0<\xi_1<\xi_2<\cdots<\xi_M$, and $0<\eta_1<\eta_2<\cdots<\eta_N$, we put
\begin{align*}
\widetilde{\phi}_{\overrightarrow{\alpha},\overrightarrow{\beta},\overrightarrow{\xi},\overrightarrow{\eta}}
:=
\bra{dR'(\overline{z_1})^{\alpha_1}dL'(z_1)^{\beta_1}dR'(\overline{z_2})^{\alpha_2}dL'(z_2)^{\beta_2}\cdots }\vac\grotimes [z_{\xi_1}\wedge \cdots \wedge z_{\xi_M}]\grotimes [\overline{z_{\eta_1}}\wedge \cdots \wedge \overline{z_{\eta_N}}],
\end{align*}
and we put $\phi_{\overrightarrow{\alpha},\overrightarrow{\beta},\overrightarrow{\xi},\overrightarrow{\eta}}:=\|\widetilde{\phi}_{\overrightarrow{\alpha},\overrightarrow{\beta},\overrightarrow{\xi},\overrightarrow{\eta}}\|^{-1}\widetilde{\phi}_{\overrightarrow{\alpha},\overrightarrow{\beta},\overrightarrow{\xi},\overrightarrow{\eta}}$. Then, $\{\phi_{\overrightarrow{\alpha},\overrightarrow{\beta},\overrightarrow{\xi},\overrightarrow{\eta}}\}$ is a CONS of $\ca{H}$. 
We put $\lambda_{\overrightarrow{\alpha},\overrightarrow{\beta},\overrightarrow{\xi},\overrightarrow{\eta}}:=
\sum(4m^{l/2})^{\alpha_m}+\sum(4m^{l/2})^{\beta_m}+4\sum\xi_m^{l/2}+4\sum\eta_m^{l/2}$.
Since
\begin{align*}
\Dirac^2\phi_{\overrightarrow{\alpha},\overrightarrow{\beta},\overrightarrow{\xi},\overrightarrow{\eta}}&=\bra{\sum(4n^{l/2})^{\alpha_n}+\sum(4n^{l/2})^{\beta_n}+4\sum\xi_n^{l/2}+4\sum\eta_n^{l/2}}\phi_{\overrightarrow{\alpha},\overrightarrow{\beta},\overrightarrow{\xi},\overrightarrow{\eta}}\\
&=
\lambda_{\overrightarrow{\alpha},\overrightarrow{\beta},\overrightarrow{\xi},\overrightarrow{\eta}}\phi_{\overrightarrow{\alpha},\overrightarrow{\beta},\overrightarrow{\xi},\overrightarrow{\eta}}
\end{align*}
the following is a CONS of $\dom(\Dirac)$ with respect to the graph norm:
$$\bbra{
\bra{\sqrt{1+\lambda_{\overrightarrow{\alpha},\overrightarrow{\beta},\overrightarrow{\xi},\overrightarrow{\eta}}}}^{-1}\phi_{\overrightarrow{\alpha},\overrightarrow{\beta},\overrightarrow{\xi},\overrightarrow{\eta}}}.$$

In order to prove the statement, we will prove that each operator $x_n\grotimes c^*(e_n)+y_n\grotimes c^*(f_n)=z_n\grotimes c^*(\overline{z_n})+\overline{z_n}\grotimes c^*(z_n)$ is bounded, and that the infinite sum $\sum_n\|z_n\grotimes c^*(\overline{z_n})+\overline{z_n}\grotimes c^*(z_n)\|$ is finite. We consider $(z_n\grotimes c^*(\overline{z_n})+\overline{z_n}\grotimes c^*(z_n))\phi$ for
$$\phi=\sum_{\overrightarrow{\alpha},\overrightarrow{\beta},\overrightarrow{\xi},\overrightarrow{\eta}}c_{\overrightarrow{\alpha},\overrightarrow{\beta},\overrightarrow{\xi},\overrightarrow{\eta}}\sqrt{1+\lambda_{\overrightarrow{\alpha},\overrightarrow{\beta},\overrightarrow{\xi},\overrightarrow{\eta}}}^{-1}\phi_{\overrightarrow{\alpha},\overrightarrow{\beta},\overrightarrow{\xi},\overrightarrow{\eta}}\in \dom(\Dirac).$$
Note that the graph norm of $\phi$ is $\sum_{\overrightarrow{\alpha},\overrightarrow{\beta},\overrightarrow{\xi},\overrightarrow{\eta}}|c_{\overrightarrow{\alpha},\overrightarrow{\beta},\overrightarrow{\xi},\overrightarrow{\eta}}|^2$.

By a simple computation, we notice the following formulas:
$$z_n\phi=\frac{1}{2n^l}\bbra{dR'(z_n)+dL'(z_n)}\phi \ \text{  and  }\ \overline{z_n}\phi=-\frac{1}{2n^l}\bbra{dR'(\overline{z_n})+dL'(\overline{z_n})}\phi.$$
Thus, by the commutation relations on $dR'$s and $dL'$s,
$$z_ndR'(z_n)^{\alpha_n}dL'(\overline{z_n})^{\beta_n}\vac
=\frac{1}{2n^l}dR'(z_n)^{\alpha_n}dL'(\overline{z_n})^{\beta_n+1}\vac
-ndR'(z_n)^{\alpha_n-1}dL'(\overline{z_n})^{\beta_n}\vac,$$
$$\overline{z_n}dR'(z_n)^{\alpha_n}dL'(\overline{z_n})^{\beta_n}\vac
=-\frac{1}{2n^l}dR'(z_n)^{\alpha_n+1}dL'(\overline{z_n})^{\beta_n}\vac
-ndR'(z_n)^{\alpha_n}dL'(\overline{z_n})^{\beta_n-1}\vac.$$
By $\|\widetilde{\phi}_{\overrightarrow{\alpha},\overrightarrow{\beta},\overrightarrow{\xi},\overrightarrow{\eta}}\|^2=\overrightarrow{\alpha}!\overrightarrow{\beta}! \prod(2n^l)^{\alpha_n+\beta_n}$ (where $\overrightarrow{\gamma}!:=\gamma_1!\gamma_2!\cdots$ for a multi-index $\overrightarrow{\gamma}=(\gamma_{1},\gamma_2,\cdots)$),
$$z_n\phi_{\overrightarrow{\alpha},\overrightarrow{\beta},\overrightarrow{\xi},\overrightarrow{\eta}}
=\sqrt{\frac{\beta_n+1}{2n^l}}\phi_{\overrightarrow{\alpha},\overrightarrow{\beta}+e_n,\overrightarrow{\xi},\overrightarrow{\eta}}
-\frac{n}{\sqrt{2n^l\alpha_n}}\phi_{\overrightarrow{\alpha}-e_n,\overrightarrow{\beta},\overrightarrow{\xi},\overrightarrow{\eta}},$$
$$\overline{z_n}\phi_{\overrightarrow{\alpha},\overrightarrow{\beta},\overrightarrow{\xi},\overrightarrow{\eta}}
=-\sqrt{\frac{\alpha_n+1}{2n^l}}\phi_{\overrightarrow{\alpha}+e_n,\overrightarrow{\beta},\overrightarrow{\xi},\overrightarrow{\eta}}
-\frac{n}{\sqrt{2n^l\alpha_n}}\phi_{\overrightarrow{\alpha},\overrightarrow{\beta}-e_n,\overrightarrow{\xi},\overrightarrow{\eta}},$$
where, for a multi-index $\overrightarrow{\gamma}=(\cdots,\gamma_{n-1},\gamma_n,\gamma_{n+1},\cdots)$, we denote $(\cdots,\gamma_{n-1},\gamma_n\pm 1,\gamma_{n+1},\cdots)$ by $\overrightarrow{\gamma}\pm e_n$. If $\alpha_n=0$, we put $\frac{n}{\sqrt{2n^l\alpha_n}}\phi_{\overrightarrow{\alpha}-e_n,\overrightarrow{\beta},\overrightarrow{\xi},\overrightarrow{\eta}}:=0$, and similarly for $\frac{n}{\sqrt{2n^l\beta_n}}\phi_{\overrightarrow{\alpha},\overrightarrow{\beta}-e_n,\overrightarrow{\xi},\overrightarrow{\eta}}$.

Moreover, by definition,
\begin{align*}
c^*(\overline{z_n})
\phi_{\overrightarrow{\alpha},\overrightarrow{\beta},\overrightarrow{\xi},\overrightarrow{\eta}}
&=\bra{z_n\rfloor\circ \epsilon_{S^*}\grotimes \id+\id\grotimes \sqrt{-1}\overline{z_n}\wedge}\phi_{\overrightarrow{\alpha},\overrightarrow{\beta},\overrightarrow{\xi},\overrightarrow{\eta}} \\
&=:(-1)^M\bbra{\phi_{\overrightarrow{\alpha},\overrightarrow{\beta},\overrightarrow{\xi}\setminus\{n\},\overrightarrow{\eta}}+\sqrt{-1}\phi_{\overrightarrow{\alpha}+e_n,\overrightarrow{\beta},\overrightarrow{\xi},\overrightarrow{\eta}\cup\{n\}}}.
\end{align*}
Similarly,
\begin{align*}
c^*({z_n})
\phi_{\overrightarrow{\alpha},\overrightarrow{\beta},\overrightarrow{\xi},\overrightarrow{\eta}}
&=\bra{z_n\wedge\circ \epsilon_{S^*}\grotimes \id-\id\grotimes \sqrt{-1}\overline{z_n}\rfloor}\phi_{\overrightarrow{\alpha},\overrightarrow{\beta},\overrightarrow{\xi},\overrightarrow{\eta}} \\
&=:(-1)^M\bbra{\phi_{\overrightarrow{\alpha},\overrightarrow{\beta},\overrightarrow{\xi}\cup\{n\},\overrightarrow{\eta}}-\sqrt{-1}\phi_{\overrightarrow{\alpha},\overrightarrow{\beta},\overrightarrow{\xi},\overrightarrow{\eta}\setminus\{n\}}}.
\end{align*}

Therefore, $\bbra{z_n\grotimes c^*(\overline{z_n})+\overline{z_n}\grotimes c^*(z_n)}\phi_{\overrightarrow{\alpha},\overrightarrow{\beta},\overrightarrow{\xi},\overrightarrow{\eta}}$ is given by
\begin{align*}
&(-1)^M\bra{\sqrt{\frac{\beta_n+1}{2n^l}}\phi_{\overrightarrow{\alpha},\overrightarrow{\beta}+e_n,\overrightarrow{\xi}\setminus\{n\},\overrightarrow{\eta}}
-\frac{n}{\sqrt{2n^l\alpha_n}}\phi_{\overrightarrow{\alpha}-e_n,\overrightarrow{\beta},\overrightarrow{\xi}\setminus\{n\},\overrightarrow{\eta}}}\\
&+(-1)^M\bra{\sqrt{\frac{\beta_n+1}{2n^l}}\phi_{\overrightarrow{\alpha},\overrightarrow{\beta}+e_n,\overrightarrow{\xi},\overrightarrow{\eta}\cup\{n\}}
-\frac{n}{\sqrt{2n^l\alpha_n}}\phi_{\overrightarrow{\alpha}-e_n,\overrightarrow{\beta},\overrightarrow{\xi},\overrightarrow{\eta}\cup\{n\}}}\\
&+(-1)^M\bra{-\sqrt{\frac{\alpha_n+1}{2n^l}}\phi_{\overrightarrow{\alpha}+e_n,\overrightarrow{\beta},\overrightarrow{\xi}\cup\{n\},\overrightarrow{\eta}}
-\frac{n}{\sqrt{2n^l\alpha_n}}\phi_{\overrightarrow{\alpha},\overrightarrow{\beta}-e_n,\overrightarrow{\xi}\cup\{n\},\overrightarrow{\eta}}}\\
&+(-1)^M\bra{-\sqrt{\frac{\alpha_n+1}{2n^l}}\phi_{\overrightarrow{\alpha}+e_n,\overrightarrow{\beta},\overrightarrow{\xi},\overrightarrow{\eta}\setminus\{n\}}
-\frac{n}{\sqrt{2n^l\alpha_n}}\phi_{\overrightarrow{\alpha},\overrightarrow{\beta}-e_n,\overrightarrow{\xi},\overrightarrow{\eta}\setminus\{n\}}}.
\end{align*}
Thus, $\bbra{z_n\grotimes c^*(\overline{z_n})+\overline{z_n}\grotimes c^*(z_n)}\phi$ is given by{\small
\begin{align*}
&\sum_{\overrightarrow{\alpha},\overrightarrow{\beta},\overrightarrow{\xi},\overrightarrow{\eta}}(-1)^M\frac{c_{\overrightarrow{\alpha},\overrightarrow{\beta},\overrightarrow{\xi},\overrightarrow{\eta}}}{\sqrt{1+\lambda_{\overrightarrow{\alpha},\overrightarrow{\beta},\overrightarrow{\xi},\overrightarrow{\eta}}}}\bra{\sqrt{\frac{\beta_n+1}{2n^l}}\phi_{\overrightarrow{\alpha},\overrightarrow{\beta}+e_n,\overrightarrow{\xi}\setminus\{n\},\overrightarrow{\eta}}
-\frac{n}{\sqrt{2n^l\alpha_n}}\phi_{\overrightarrow{\alpha}-e_n,\overrightarrow{\beta},\overrightarrow{\xi}\setminus\{n\},\overrightarrow{\eta}}}\\
&\ \ \ +\sum_{\overrightarrow{\alpha},\overrightarrow{\beta},\overrightarrow{\xi},\overrightarrow{\eta}}(-1)^M\frac{c_{\overrightarrow{\alpha},\overrightarrow{\beta},\overrightarrow{\xi},\overrightarrow{\eta}}}{\sqrt{1+\lambda_{\overrightarrow{\alpha},\overrightarrow{\beta},\overrightarrow{\xi},\overrightarrow{\eta}}}}\bra{\sqrt{\frac{\beta_n+1}{2n^l}}\phi_{\overrightarrow{\alpha},\overrightarrow{\beta}+e_n,\overrightarrow{\xi},\overrightarrow{\eta}\cup\{n\}}
-\frac{n}{\sqrt{2n^l\alpha_n}}\phi_{\overrightarrow{\alpha}-e_n,\overrightarrow{\beta},\overrightarrow{\xi},\overrightarrow{\eta}\cup\{n\}}}\\
&\ \ \ +\sum_{\overrightarrow{\alpha},\overrightarrow{\beta},\overrightarrow{\xi},\overrightarrow{\eta}}(-1)^M\frac{c_{\overrightarrow{\alpha},\overrightarrow{\beta},\overrightarrow{\xi},\overrightarrow{\eta}}}{\sqrt{1+\lambda_{\overrightarrow{\alpha},\overrightarrow{\beta},\overrightarrow{\xi},\overrightarrow{\eta}}}}\bra{-\sqrt{\frac{\alpha_n+1}{2n^l}}\phi_{\overrightarrow{\alpha}+e_n,\overrightarrow{\beta},\overrightarrow{\xi}\cup\{n\},\overrightarrow{\eta}}
-\frac{n}{\sqrt{2n^l\alpha_n}}\phi_{\overrightarrow{\alpha},\overrightarrow{\beta}-e_n,\overrightarrow{\xi}\cup\{n\},\overrightarrow{\eta}}}\\
&\ \ \ +\sum_{\overrightarrow{\alpha},\overrightarrow{\beta},\overrightarrow{\xi},\overrightarrow{\eta}}(-1)^M\frac{c_{\overrightarrow{\alpha},\overrightarrow{\beta},\overrightarrow{\xi},\overrightarrow{\eta}}}{\sqrt{1+\lambda_{\overrightarrow{\alpha},\overrightarrow{\beta},\overrightarrow{\xi},\overrightarrow{\eta}}}}\bra{-\sqrt{\frac{\alpha_n+1}{2n^l}}\phi_{\overrightarrow{\alpha}+e_n,\overrightarrow{\beta},\overrightarrow{\xi},\overrightarrow{\eta}\setminus\{n\}}
-\frac{n}{\sqrt{2n^l\alpha_n}}\phi_{\overrightarrow{\alpha},\overrightarrow{\beta}-e_n,\overrightarrow{\xi},\overrightarrow{\eta}\setminus\{n\}}}\\
\end{align*}
\begin{align*}
&=-\sum_{\overrightarrow{\alpha},\overrightarrow{\beta},\overrightarrow{\xi},\overrightarrow{\eta}}(-1)^M\bra{\frac{c_{\overrightarrow{\alpha},\overrightarrow{\beta}-e_n,\overrightarrow{\xi}\cup\{n\},\overrightarrow{\eta}}}{\sqrt{1+\lambda_{\overrightarrow{\alpha},\overrightarrow{\beta}-e_n,\overrightarrow{\xi}\cup\{n\},\overrightarrow{\eta}}}}\sqrt{\frac{\beta_n}{2n^l}}
-\frac{c_{\overrightarrow{\alpha}+e_n,\overrightarrow{\beta},\overrightarrow{\xi}\cup\{n\},\overrightarrow{\eta}}}{\sqrt{1+\lambda_{\overrightarrow{\alpha}+e_n,\overrightarrow{\beta},\overrightarrow{\xi}\cup\{n\},\overrightarrow{\eta}}}}\frac{n}{\sqrt{2n^l(\alpha_n+1)}}}\phi_{\overrightarrow{\alpha},\overrightarrow{\beta},\overrightarrow{\xi},\overrightarrow{\eta}}\\
&\ \ \ +\sum_{\overrightarrow{\alpha},\overrightarrow{\beta},\overrightarrow{\xi},\overrightarrow{\eta}}(-1)^M\bra{\frac{c_{\overrightarrow{\alpha},\overrightarrow{\beta}-e_n,\overrightarrow{\xi},\overrightarrow{\eta}\setminus\{n\}}}{\sqrt{1+\lambda_{\overrightarrow{\alpha},\overrightarrow{\beta}-e_n,\overrightarrow{\xi},\overrightarrow{\eta}\setminus\{n\}}}}\sqrt{\frac{\beta_n}{2n^l}}
-\frac{c_{\overrightarrow{\alpha}+e_n,\overrightarrow{\beta},\overrightarrow{\xi},\overrightarrow{\eta}\setminus\{n\}}}{\sqrt{1+\lambda_{\overrightarrow{\alpha}+e_n,\overrightarrow{\beta},\overrightarrow{\xi},\overrightarrow{\eta}\setminus\{n\}}}}\frac{n}{\sqrt{2n^l(\alpha_n+1)}}}\phi_{\overrightarrow{\alpha},\overrightarrow{\beta},\overrightarrow{\xi},\overrightarrow{\eta}}\\
&\ \ \ -
\sum_{\overrightarrow{\alpha},\overrightarrow{\beta},\overrightarrow{\xi},\overrightarrow{\eta}}(-1)^M\bra{-\frac{c_{\overrightarrow{\alpha}-e_n,\overrightarrow{\beta},\overrightarrow{\xi}\cup\{n\},\overrightarrow{\eta}}}{\sqrt{1+\lambda_{\overrightarrow{\alpha}-e_n,\overrightarrow{\beta},\overrightarrow{\xi}\cup\{n\},\overrightarrow{\eta}}}}\sqrt{\frac{\alpha_n}{2n^l}}
-\frac{c_{\overrightarrow{\alpha},\overrightarrow{\beta}+e_n,\overrightarrow{\xi}\cup\{n\},\overrightarrow{\eta}}}{\sqrt{1+\lambda_{\overrightarrow{\alpha},\overrightarrow{\beta}+e_n,\overrightarrow{\xi}\cup\{n\},\overrightarrow{\eta}}}}\frac{n}{\sqrt{2n^l(\beta_n+1)}}}\phi_{\overrightarrow{\alpha},\overrightarrow{\beta},\overrightarrow{\xi},\overrightarrow{\eta}}\\
&\ \ \ +\sum_{\overrightarrow{\alpha},\overrightarrow{\beta},\overrightarrow{\xi},\overrightarrow{\eta}}(-1)^M\bra{-\frac{c_{\overrightarrow{\alpha}-e_n,\overrightarrow{\beta},\overrightarrow{\xi},\overrightarrow{\eta}\setminus\{n\}}}{\sqrt{1+\lambda_{\overrightarrow{\alpha}-e_n,\overrightarrow{\beta},\overrightarrow{\xi},\overrightarrow{\eta}\setminus\{n\}}}}\sqrt{\frac{\alpha_n}{2n^l}}
-\frac{c_{\overrightarrow{\alpha},\overrightarrow{\beta}+e_n,\overrightarrow{\xi},\overrightarrow{\eta}\setminus\{n\}}}{\sqrt{1+\lambda_{\overrightarrow{\alpha},\overrightarrow{\beta}+e_n,\overrightarrow{\xi},\overrightarrow{\eta}\setminus\{n\}}}}\frac{n}{\sqrt{2n^l(\beta_n+1)}}}\phi_{\overrightarrow{\alpha},\overrightarrow{\beta},\overrightarrow{\xi},\overrightarrow{\eta}},
\end{align*}}
where $c_{\overrightarrow{\alpha}-e_n,\overrightarrow{\beta},\cdots}:=0$ if $\alpha_n=0$, and $c_{\overrightarrow{\alpha},\overrightarrow{\beta}-e_n,\cdots}:=0$ if $\beta_n=0$.

Let $\star_{\overrightarrow{\alpha},\overrightarrow{\beta},\overrightarrow{\xi},\overrightarrow{\eta}}$ be{\small
$$\frac{\left|c_{\overrightarrow{\alpha},\overrightarrow{\beta}-e_n,\overrightarrow{\xi},\overrightarrow{\eta}\setminus\{n\}}\right|^2}{1+\lambda_{\overrightarrow{\alpha},\overrightarrow{\beta}-e_n,\overrightarrow{\xi},\overrightarrow{\eta}\setminus\{n\}}}\frac{\beta_n}{2n^l}
+\frac{\left|c_{\overrightarrow{\alpha}+e_n,\overrightarrow{\beta},\overrightarrow{\xi},\overrightarrow{\eta}\setminus\{n\}}\right|^2}{1+\lambda_{\overrightarrow{\alpha}+e_n,\overrightarrow{\beta},\overrightarrow{\xi},\overrightarrow{\eta}\setminus\{n\}}}\frac{n^2}{2n^l(\alpha_n+1)}$$
$$+\frac{\left|c_{\overrightarrow{\alpha}-e_n,\overrightarrow{\beta},\overrightarrow{\xi}\cup\{n\},\overrightarrow{\eta}}\right|^2}{1+\lambda_{\overrightarrow{\alpha}-e_n,\overrightarrow{\beta},\overrightarrow{\xi}\cup\{n\},\overrightarrow{\eta}}}\frac{\alpha_n}{2n^l}
+\frac{\left|c_{\overrightarrow{\alpha},\overrightarrow{\beta}+e_n,\overrightarrow{\xi}\cup\{n\},\overrightarrow{\eta}}\right|^2}{1+\lambda_{\overrightarrow{\alpha},\overrightarrow{\beta}+e_n,\overrightarrow{\xi}\cup\{n\},\overrightarrow{\eta}}}\frac{n^2}{2n^l(\beta_n+1)}$$
$$+\frac{\left|c_{\overrightarrow{\alpha}-e_n,\overrightarrow{\beta},\overrightarrow{\xi}\cup\{n\},\overrightarrow{\eta}}\right|^2}{1+\lambda_{\overrightarrow{\alpha}-e_n,\overrightarrow{\beta},\overrightarrow{\xi}\cup\{n\},\overrightarrow{\eta}}}\frac{\alpha_n}{2n^l}
+\frac{\left|c_{\overrightarrow{\alpha},\overrightarrow{\beta}+e_n,\overrightarrow{\xi}\cup\{n\},\overrightarrow{\eta}}\right|^2}{{1+\lambda_{\overrightarrow{\alpha},\overrightarrow{\beta}+e_n,\overrightarrow{\xi}\cup\{n\},\overrightarrow{\eta}}}}\frac{n^2}{2n^l(\beta_n+1)}$$
$$+\frac{\left|c_{\overrightarrow{\alpha}-e_n,\overrightarrow{\beta},\overrightarrow{\xi},\overrightarrow{\eta}\setminus\{n\}}\right|^2}{1+\lambda_{\overrightarrow{\alpha}-e_n,\overrightarrow{\beta},\overrightarrow{\xi},\overrightarrow{\eta}\setminus\{n\}}}\frac{\alpha_n}{2n^l}
+\frac{\left|c_{\overrightarrow{\alpha},\overrightarrow{\beta}+e_n,\overrightarrow{\xi},\overrightarrow{\eta}\setminus\{n\}}\right|^2}{1+\lambda_{\overrightarrow{\alpha},\overrightarrow{\beta}+e_n,\overrightarrow{\xi},\overrightarrow{\eta}\setminus\{n\}}}\frac{n^2}{2n^l(\beta_n+1)}.$$}
Then, $\|\bbra{z_n\grotimes c^*(\overline{z_n})+\overline{z_n}\grotimes c^*(z_n)}\phi\|^2$ is bounded by 
$$8\sum_{\overrightarrow{\alpha},\overrightarrow{\beta},\overrightarrow{\xi},\overrightarrow{\eta}}\star_{\overrightarrow{\alpha},\overrightarrow{\beta},\overrightarrow{\xi},\overrightarrow{\eta}}.$$
Since all of $\frac{\beta_n}{1+\lambda_{\overrightarrow{\alpha},\overrightarrow{\beta}-e_n,\overrightarrow{\xi},\overrightarrow{\eta}\setminus\{n\}}}$, $\frac{1}{(1+\lambda_{\overrightarrow{\alpha}+e_n,\overrightarrow{\beta},\overrightarrow{\xi},\overrightarrow{\eta}\setminus\{n\}})(\alpha_n+1)}$, $\cdots$ are bounded above by a constant (say $K$) which is independent of $\overrightarrow{\alpha},\overrightarrow{\beta},\overrightarrow{\xi},\overrightarrow{\eta}$ and $n$, $\star_{\overrightarrow{\alpha},\overrightarrow{\beta},\overrightarrow{\xi},\overrightarrow{\eta}}$ is bounded above by
$$K\bra{\left|c_{\overrightarrow{\alpha},\overrightarrow{\beta}-e_n,\overrightarrow{\xi},\overrightarrow{\eta}\setminus\{n\}}\right|^2
+\left|c_{\overrightarrow{\alpha}+e_n,\overrightarrow{\beta},\overrightarrow{\xi},\overrightarrow{\eta}\setminus\{n\}}\right|^2+\left|c_{\overrightarrow{\alpha}-e_n,\overrightarrow{\beta},\overrightarrow{\xi}\cup\{n\},\overrightarrow{\eta}}\right|^2
+\left|c_{\overrightarrow{\alpha},\overrightarrow{\beta}+e_n,\overrightarrow{\xi}\cup\{n\},\overrightarrow{\eta}}\right|^2}n^{2-l}$$
$$+K\bra{\left|c_{\overrightarrow{\alpha}-e_n,\overrightarrow{\beta},\overrightarrow{\xi}\cup\{n\},\overrightarrow{\eta}}\right|^2
+\left|c_{\overrightarrow{\alpha},\overrightarrow{\beta}+e_n,\overrightarrow{\xi}\cup\{n\},\overrightarrow{\eta}}\right|^2
+\left|c_{\overrightarrow{\alpha}-e_n,\overrightarrow{\beta},\overrightarrow{\xi},\overrightarrow{\eta}\setminus\{n\}}\right|^2
+\left|c_{\overrightarrow{\alpha},\overrightarrow{\beta}+e_n,\overrightarrow{\xi},\overrightarrow{\eta}\setminus\{n\}}\right|^2}n^{2-l}.$$
Therefore, the operator norm of $z_n\grotimes c^*(\overline{z_n})+\overline{z_n}\grotimes c^*(z_n)$ is bounded by $64Kn^{2-l}$. Since $l>4$, the infinite sum $\sum_n \|z_n\grotimes c^*(\overline{z_n})+\overline{z_n}\grotimes c^*(z_n)\|\leq \sum_n64n^{2-l}$ is finite.
\end{proof}

With this preliminaries, we can prove the following. Since we have proved almost the same result in \cite[Section 5.1]{T4}, we just outline the proof here for the convenience of the reader, and in order to prepare the proof of the next result.

\begin{lem}\label{index element for HKT algebra}
The triple $(\ca{S}\grotimes \ca{H},\pi,\id\grotimes \Dirac)$ defines an unbounded $U_{L^2_m}$-equivariant Kasparov \\
$(\ca{A}_\HKT(U_{L^2_k}),\ca{S})$-module. 
\end{lem}
\begin{proof}
We need to check the following: $(A)$ The pair is actually an unbounded Kasparov module; and $(B)$ This Kasparov module is equivariant.

$(A)$ can be divided into the following steps: $(0)$ $\ca{S}\grotimes \ca{H}$ is a countably generated Hilbert $\ca{S}$-module, and $\id\grotimes \Dirac$ is a densely defined, regular, odd, and essentially self-adjoint operator; $(1)$ The set of $a\in \ca{A}_\HKT(U_{L^2_k})$ preserving $\dom(\Dirac)$ and satisfying $[a,\id\grotimes \Dirac]\in \bb{L}_{\ca{S}}(\ca{S}\grotimes \ca{H})$ is dense; and $(2)$ $a(\id\grotimes\id+\id\grotimes \Dirac^2)^{-1}$ belongs to $\bb{K}_{\ca{S}}(\ca{S}\grotimes \ca{H})$ for any $a\in \ca{A}_\HKT(U_{L^2_k})$.

$(0)$ is clear from the fact that $\ca{H}$ is a separable Hilbert space, and the arguments about $\Dirac$ so far.

$(1)$ is proved like \cite[Proposition 5.10]{T4}.
We give a slightly different and more concrete proof.\footnote{In \cite{T4}, we have used $\{e^{-X^2},Xe^{-X^2}\}$ as a generating set of $\ca{S}$. Then, $\pi(\beta(e^{-X^2}))$ is given by ``$e^{-C^2}$''. However, it is difficult to say that it is given by $\sum_n (n!)^{-1}C^n$, because the infinite sum should be in the sense of uniform convergence on compact sets, if we use similar arguments for finite-dimensional spaces. Instead, we use alternative generators in this proof, for the following reasons: The values at these generators are simple to define; We would like to give as many computations on $\ca{A}(\ca{X})$ as possible.} We prove that $\pi\bra{\ca{A}_\HKT(U_{L^2_k})_\fin}$ preserves $\dom(\id\grotimes\Dirac)$ and that the commutator is bounded.
Let us check the simplest case: $\beta_0^\infty(f)\in \ca{A}_\HKT(U_{L^2_k})_\fin$, and the general cases are left to the reader.
We notice that
$$\pi(\beta_0^\infty(f))\phi=f(X\grotimes\id+\id\grotimes \pi(C_{1}^\infty))\phi=\lim_{N\to\infty}f(X\grotimes\id+\id\grotimes \pi(C_1^{N}))\phi$$
for $\phi\in\ca{S}_\fin\grotimes^\alg \ca{H}_\fin$ and $f\in \ca{S}$, thanks to \cite[Lemma 5.8]{T4}.  Since $\ca{S}_\fin\grotimes^\alg \ca{H}_\fin$ is dense in $\ca{S}\grotimes \ca{H}$, and since $\|\pi(\beta_0^\infty(f))\|$ is bounded by $\|f\|$, we have $\pi(\beta_0^\infty(f))\psi=\lim_{N\to\infty}f(X\grotimes\id+\id\grotimes \pi(C^{N}_1))\grotimes \psi$ for arbitrary $\psi\in \ca{S}\grotimes \ca{H}$.
By the fact that $\id\grotimes\Dirac$ is a closed operator, it is sufficient to prove the following: $(a)$ $f(X\grotimes\id+\id\grotimes \pi(C_1^N))\phi\in \dom(\id\grotimes \Dirac)$ for every $N$ and $(b)$ $\{(\id\grotimes \Dirac)\circ  f(X\grotimes\id+\id\grotimes \pi(C_1^N))(\phi)\}_{N\in\bb{N}}$ converges. We will also prove that the commutator is bounded at the same time.
It is enough to prove these properties for $f=f_\ev,f_\od$.

We put $f=f_\ev$ in this paragraph. $(a)$ is obvious because $ f(X\grotimes\id+\id\grotimes \pi(C_1^N))$ is essentially a multiplication operator by a smooth function vanishing at infinity on a finite-dimensional space.
For $(b)$, we compute $(\id\grotimes \Dirac)\circ  f(X\grotimes\id+\id\grotimes \pi(C_1^N))(\phi)$. Thanks to the formula on the graded commutator $[A,B^{-1}]=-(-1)^{\partial A\partial B}B^{-1}[A,B]B^{-1}$, we have
\begin{align*}
&(\id\grotimes \Dirac)\circ  f_\ev(X\grotimes\id+\id\grotimes \pi(C_1^N))(\phi) \\
&\ \ \ =[\id\grotimes \Dirac ,f_\ev(X\grotimes\id+\id\grotimes \pi(C_1^N))]\phi +f_\ev(X\grotimes\id+\id\grotimes \pi(C_1^N))[\id\grotimes \Dirac (\phi)] \\
&\ \ \ =-f_\ev(X\grotimes\id+\id\grotimes \pi(C_1^N))\circ[\id\grotimes \Dirac,1+X^2\grotimes\id+\id\grotimes (C_1^N)^2]\circ f_\ev(X\grotimes\id+\id\grotimes \pi(C_1^N))(\phi) \\
&\ \ \ \ \ \ +f_\ev(X\grotimes\id+\id\grotimes \pi(C_1^N))[\id\grotimes \Dirac (\phi)] \\
&\ \ \ =-f_\ev(X\grotimes\id+\id\grotimes \pi(C_1^N))\circ\bbbra{\id\grotimes\bbra{\sum_{n=1}^Nn^{-\frac{l}{4}}[2x_n\grotimes c(e_n)+2y_n\grotimes c(f_n)]}}\circ f_\ev(X\grotimes\id+\id\grotimes \pi(C_1^N))(\phi)\\
&\ \ \ \ \ \ +f_\ev(X\grotimes\id+\id\grotimes \pi(C_1^N))[\id\grotimes \Dirac (\phi)]\\
&\ \ \ =:-\fra{C}_{l,N}\circ f_\ev(X\grotimes\id+\id\grotimes \pi(C_1^N))^2(\phi)+f_\ev(X\grotimes\id+\id\grotimes \pi(C_1^N))[\id\grotimes \Dirac (\phi)],
\end{align*}
where we put
$$\fra{C}_{l,N}:=\id\grotimes\bbra{\sum_{n=1}^Nn^{-\frac{l}{4}}[2x_n\grotimes c(e_n)+2y_n\grotimes c(f_n)]}.$$
By the same argument of the previous lemma, $\fra{C}_{l,N}$ converges in $\bb{L}_{\ca{S}}(\dom(\id\grotimes \Dirac),\ca{S}\grotimes \ca{H})$. 
The limit is denoted by $\fra{C}_{l}$.
Let us prove that $-\fra{C}_{l,N}\circ f_\ev(X\grotimes\id+\id\grotimes \pi(C_1^N))^2(\phi)+f_\ev(X\grotimes\id+\id\grotimes \pi(C_1^N))[\id\grotimes \Dirac (\phi)]$ converges to $-\fra{C}_{l}\circ f_\ev(X\grotimes\id+\id\grotimes \pi(C_1^\infty))^2(\phi)+f_\ev(X\grotimes\id+\id\grotimes \pi(C_1^\infty))[\id\grotimes \Dirac (\phi)]$. 
In fact,
\begin{align*}
&-\fra{C}_{l,N}\circ f_\ev(X\grotimes\id+\id\grotimes \pi(C_1^N))^2(\phi)+f_\ev(X\grotimes\id+\id\grotimes \pi(C_1^N))[\id\grotimes \Dirac (\phi)]\\
&\ \ \ +\fra{C}_{l}\circ f_\ev(X\grotimes\id+\id\grotimes \pi(C_1^\infty))^2(\phi)-f_\ev(X\grotimes\id+\id\grotimes \pi(C_1^\infty))[\id\grotimes \Dirac (\phi)]\\
&= f_\ev(X\grotimes\id+\id\grotimes \pi(C_1^N))^2\circ (\fra{C}_{l}-\fra{C}_{l,N})\phi\\
&\ \ \ +\bbra{f_\ev(X\grotimes\id+\id\grotimes \pi(C_1^\infty))^2-f_\ev(X\grotimes\id+\id\grotimes \pi(C_1^N))^2}\circ \fra{C}_{l}\phi\\
&\ \ \ \ \ \ +\bbra{f_\ev(X\grotimes\id+\id\grotimes \pi(C_1^N))-f_\ev(X\grotimes\id+\id\grotimes \pi(C_1^\infty))}[\id\grotimes \Dirac (\phi)].
\end{align*}
Since $\| f_\ev(X\grotimes\id+\id\grotimes \pi(C_1^N))^2\|\leq 1$, the third line converges to $0$. The fourth and fifth lines, converge to $0$ thanks to the strong convergence $f(X\grotimes\id+\id\grotimes \pi(C^{N}_1))\to \pi(\beta_0^\infty(f))$ on $\ca{S}\grotimes \ca{H}$. This means $\pi(\beta_0^\infty(f_\ev))\phi\in\dom(\id\grotimes \Dirac)$.
Moreover, by the above computation, the commutator $[\id\grotimes\Dirac,\pi(\beta_0^\infty(f_\ev))]$ is given by $-\fra{C}_l f_\ev(X\grotimes\id+\id\grotimes \pi(C_1^\infty))^2$, which is bounded. In fact, the operator $\id\grotimes x_n\grotimes c(e_n)\circ f_\ev(X\grotimes\id+\id\grotimes\pi(C_1^\infty))$ is bounded and its operator norm is not greater than $1$, and $l/4$ is greater than $1$.

For $f_\od=Xf_\ev$, one can prove $(a)$ and $(b)$ by the following formulas:
$$[\id\grotimes \Dirac,\pi(\beta_0^\infty(Xf_\ev))]=[\id\grotimes \Dirac,X\grotimes \id+\id\grotimes C]\pi(\beta_0^\infty(f_\ev))-(X\grotimes \id+\id\grotimes C)[\id\grotimes \Dirac,\pi(\beta_0^\infty(f_\ev))].$$
$$[\id\grotimes \Dirac,X\grotimes \id+\id\grotimes C]
= \id\grotimes\id\grotimes \sum_nn^{-l/4}\{c(e_n)c^*(e_n)+c(f_n)c^*(f_n)\}.$$
We leave the details to the reader.

 $(2)$ follows from the same argument of the proof of \cite[Proposition 5.11]{T4}. We have computed the spectrum of $\Dirac^2$ in Lemma \ref{Lemma C is bounded on dom(D) to H}. By this computation, the number of combinations of multi-indices $\overrightarrow{\alpha}$, $\overrightarrow{\beta}$, $\overrightarrow{\xi}$, and  $\overrightarrow{\eta}$ such that the corresponding eigenvalue is less than $K$ is finite, for every $K\in \bb{R}$. Thus, $(1+\Dirac^2)^{-1}$ is a compact operator on $\ca{H}$. Therefore, $\pi(a)\circ \bbra{\id\grotimes (1+\Dirac^2)^{-1}}$ is $\ca{S}$-compact. 

For $(B)$, we need to check the following: The $U_{L^2_m}$-action preserves $\dom(\id\grotimes \Dirac)$; $g(\id\grotimes \Dirac)-\id\grotimes \Dirac$ is bounded for each $g\in U_{L^2_m}$; and the map $g\mapsto g(\id\grotimes \Dirac)-\id\grotimes \Dirac$ is continuous.

In order to prove them, we compute the difference $g(\Dirac)-\Dirac$. For this aim, we compute $L_{(g,1)}\circ dL'(v)\circ L_{(-g,1)}$ and $L_{(g,1)}\circ dR'(v)\circ L_{(-g,1)}$ for $v=e_n,f_n$.
First,  we suppose $g\in U_{L^2_m,N}$ and we compute the difference on $\phi\in L^2(U_{L^2_k,N},\ca{L})$ for $n\leq N$. For $g=(a_1,b_1,\cdots,a_N,b_N)$,
\begin{align*}
&L_{(g,1)}\circ dL'(e_n)\circ L_{(-g,1)}\phi(x) \\
&\ \ \ =dL'(e_n)\circ L_{(-g,1)}\phi(x-g)\tau(g,x) \\
&\ \ \ =\bbra{-\frac{\partial L_{(-g,1)}\phi}{\partial x_n}(x-g)+in^l(y_n-b_n)L_{(-g,1)}\phi(x-g))}\tau(g,x)\\
&\ \ \ =\bbra{
-\frac{\partial \phi}{\partial x_n}(x)\tau(-g,x-g)-in^{1-2k}b_n\phi(x)\tau(-g,x-g)+in^l(y_n-b_n)\phi(x)\tau(-g,x-g)}\tau(g,x) \\
&\ \ \ =dL'(e_n)\phi(x)+i\bra{-n^{1-2k}-n^l}b_n\phi(x).
\end{align*}

By similar computations, we obtain the following formulas:
\begin{align*}
L_{(g,1)}\circ dL'(f_n)\circ L_{(-g,1)} &= dL'(f_n)+i\bra{n^{1-2k}+n^l}a_n; \\
L_{(g,1)}\circ dR'(e_n)\circ L_{(-g,1)} &= dR'(f_n)+i\bra{n^{1-2k}-n^l}b_n; \\
L_{(g,1)}\circ dR'(f_n)\circ L_{(-g,1)} &= dR'(f_n)+i\bra{-n^{1-2k}+n^l}a_n.
\end{align*}

Let $Jg$ be $\sum_{j=1}^N(-b_je_j+a_jf_j)\in\Lie(U_{L^2_m,N})$. Then,
\begin{align*}
g(\Dirac)-\Dirac
&= \frac{1}{\sqrt{2}}\id\grotimes\id \grotimes\gamma\bra{\sum_{n\leq N} n^{-l/4}\bbbra{-i\bra{-n^{1-2k}+n^l}b_ne_n+i\bra{-n^{1-2k}+n^l}a_nf_n}} \\
&\ \ \ +\frac{i}{\sqrt{2}}\id\grotimes\gamma^*\bra{\sum_{n\leq N} n^{-l/4}\bbbra{ -i\bra{n^{1-2k}+n^l}b_ne_n+i\bra{n^{1-2k}+n^l}a_nf_n}}\grotimes\id\\
&= \frac{i}{\sqrt{2}}\id\grotimes \id\grotimes \gamma\bra{\bra{-|d|^{1-2k-l/4}+|d|^{3l/4}}Jg}+\frac{-1}{\sqrt{2}}\id\grotimes \gamma^*\bra{\bra{|d|^{1-2k-l/4}+|d|^{3l/4}}Jg}\grotimes \id.
\end{align*}
Since $g\in U_\fin$, the above is a finite sum and it is bounded. The operator norm of this difference is bounded above by the $L^2_m$- norm of $Jg$ because $m\geq k+l\geq 1-2k-l/4,3l/4$.

More generally, if $g\in U_{L^2_m}$, $L_{(g,z)}$ preserves $\dom(\Dirac)$, and $g(\Dirac)-\Dirac$ is bounded, because both of $\bra{|d|^{1-2k-l/4}+|d|^{3l/4}}Jg$ and $\bra{-|d|^{1-2k-l/4}+|d|^{3l/4}}Jg$ belong to $\Lie(U_{L^2_{m}})$ .

Since all the maps 
\begin{center}
$U_{L^2_m}\ni g\mapsto \bra{|d|^{1-2k-l/4}+|d|^{3l/4}}Jg\in \Lie(U_{L^2_{k}})$,

$U_{L^2_m}\ni g\mapsto \bra{-|d|^{1-2k-l/4}+|d|^{3l/4}}Jg\in \Lie(U_{L^2_{k}})$,

$\Lie(U_{L^2_m})\ni v\mapsto \id\grotimes \gamma(v)\grotimes \id \in\bb{L}(\ca{H})$
\end{center}
are continuous, the map $g\mapsto g(\id\grotimes \Dirac)-\id\grotimes \Dirac$ is continuous.
\end{proof}

\begin{rmk}
We have chosen strange Hilbert space and operator in order to consider the $L^2_k$-topology for arbitrary $k> 1/2$ and make $\ca{A}(U_{L^2_k})$ act there. If $k$ were $1/2$ {\it and $l$ were $0$}, ${}^R\Dirac$ is actually equivariant and it is the setting dealt in \cite{T4}.
\end{rmk}

This lemma gives an element of $KK_{U_{L^2_m}}^\tau(\ca{A}_\HKT(U_{L^2_k}),\ca{S})$. By the pullback via $\iota\grotimes \id$, and thanks to Definition-Proposition \ref{dfn-pro rep of AHKT} $(2)$, it is possible to obtain a substitute for the index element at the {\it $KK$-theory level}. However, we explicitly describe the {\it unbounded Kasparov module} which is naturally defined by the pullback of the above index element via $\iota\grotimes \id$ for the following reasons. First, the Kasparov module we are constructing is an ingredient of the index element of the whole manifold $\ca{M}_{L^2_k}$ and it is defined by the tensor product of two $KK$-elements. Unbounded Kasparov modules have an advantage when dealing with exterior tensor products. Second, in our opinion, the $C^*$-algebra $\ca{A(X)}$ constructed by \cite{Yu} should be studied much more. It has an advantage comparing with that constructed in \cite{HKT}, in that the former looks more geometrical. On the other hand, the former has a big disadvantage, in that the definition is abstract. We believe the following constructions can be useful to study the $\ca{A(X)}$ overcoming this disadvantage.

\begin{dfn-pro}\label{index element}
The triple $(\ca{S}_\vep\grotimes \ca{H},\pi,\id\grotimes \Dirac)$ defines an unbounded $U_{L^2_m}$-equivariant Kasparov $(\ca{A}(U_{L^2_k}),\ca{S}_\vep)$-module. The corresponding $KK$-element is denoted by $[\widetilde{\Dirac}]$ and called the {\bf index element}.
\end{dfn-pro}

The non-trivial thing is only the following: {\it there is a dense subalgebra of $\ca{A}(U_{L^2_k})$ consisting of ``$C^1$-elements''}, where we say an element $a$ of $\ca{A}(U_{L^2_k})$ is a {\bf $C^1$-element} if it preserves $\dom(\id\grotimes \Dirac)$ and $[\id\grotimes\Dirac,\pi(a)]$ is bounded. It is highly non-trivial because $\ca{A}_\HKT(U_{L^2_k})_\fin\cap \iota\grotimes\id(\ca{A}(U_{L^2_k}))$ is empty. Notice that the set of all $C^1$-elements is a subalgebra.

We divide the proof into the following steps.

\begin{lem}\label{lemma X+C has dense range}
$X\grotimes\id+\id\grotimes C$ has dense range. 
\end{lem}
\begin{proof}
We first notice that $C$ has dense range. We have essentially proved it as \cite[Proposition 4.24]{T4}. We need to modify the definition of the Clifford operator and to replace the  phrase ``$C_1^N\pm\sqrt{-1}\id$ has dense range'' with the  phrase ``$C_1^N$ has dense range''. It is correct because $C_1^N$ is self-adjoint, it is injective, and $\ran(C_1^N)^\perp=\ker(C_1^N)=0$ (note that we are working on ``$L^2$-spaces'' not on ``$C_0$-spaces'').

We prove that arbitrary element of $\ca{S}\grotimes \ca{H}$ can be approximated by elements in the range of $X\grotimes\id+\id\grotimes C$. 
It is sufficient to prove that $g\grotimes \phi$ can be approximated for $\phi\in \ca{H}$ and $g\in \ca{S}$ satisfying the following: $g$ is compactly supported and $\|g\|=1$. 
Let $\epsilon>0$. Since $C$ has dense range, there exists $\psi\in\ca{H}$ satisfying $\|C\psi-\phi\|<\epsilon/2$. Thus, if we can find $A\in \ca{S}\grotimes \ca{H}$ such that $\|(X\grotimes\id+\id\grotimes C)A-g\grotimes C\psi\|<\epsilon/2$, we have $\|(X\grotimes\id+\id\grotimes C)A-g\grotimes \phi\|<\epsilon$. Therefore, we may assume that $\phi=C\psi$, from the beginning. In addition, we suppose that $\|\psi\|\leq 1$ and $\|\phi\|\leq 1$.

Let us find such an approximate solution $A$ by the approximate spectral decomposition of the operator $X$. Pick up a positive real number $\Delta<\epsilon/16$ such that $|s-t|<\Delta$ implies $|g(t)-g(s)|<\epsilon/16$. This is possible because $g$ is compactly supported.
We define a bump function $\rho_0$ around $0$ by
$$\rho(t):=\begin{cases}
1+\frac{t}{\Delta} & t\in [-\Delta,0] \\
1-\frac{t}{\Delta} & t\in [0,\Delta] \\
0 & \text{otherwise}
\end{cases}$$
and we define a bump function at $x$ by $\rho_x(t):=\rho(t-x)$. Let $x_n:=n\Delta$ for $n\in\bb{Z}$. We notice the following properties:

$(1_\rho)$ $\|X\rho_x-x\rho_x\|\leq \Delta$;

$(2_\rho)$ $\{\rho_{x_n}\}_{n\in\bb{Z}}$ gives a continuous partition of unity; and

$(3_\rho)$ For all $t\in\bb{R}$, $\#\bbra{n\in\bb{Z}_{\geq0}\mid \rho_{x_n}(t)\neq 0}$ is at most $2$.

Let us consider
$$A:=g(0)\rho_0\grotimes \psi+\sum_{n\in\bb{Z}_{> 0}} \bra{\rho_{x_n}\grotimes \frac{g(x_n)x_n+g(-x_n)C}{x_n^2+C^2}\phi+\rho_{-x_n}\grotimes \frac{-g(-x_n)x_n+g(-x_n)C}{x_n^2+C^2}\phi}$$
Note that $A$ is a finite sum because $g$ is compactly supported.
We verify the inequality $\|(X\grotimes \id+\id\grotimes C)A-g\grotimes \phi\|<\epsilon$. Since the grading homomorphism on $\ca{S}$ exchanges $\rho_{x_n}$ and $\rho_{-x_n}$,
{\small
\begin{align*}
&(X\grotimes \id+\id\grotimes C)A-g\grotimes \phi\\
&=Xg(0)\rho_0\grotimes \psi+g(0)\rho_0\grotimes C\psi +
\sum_{n>0} \bra{X\rho_{x_n}\grotimes \frac{g(x_n)x_n+g(-x_n)C}{x_n^2+C^2}\phi+\rho_{-x_n}\grotimes \frac{g(x_n)x_nC+g(-x_n)C^2}{x_n^2+C^2}\phi} \\
&\ \ \ +\sum_{n>0}\bra{X\rho_{-x_n}\grotimes \frac{-g(-x_n)x_n+g(x_n)C}{x_n^2+C^2}\phi+\rho_{x_n}\grotimes \frac{-g(-x_n)x_nC+g(x_n)C^2}{x_n^2+C^2}\phi}-\sum_{n\in \bb{Z}}\rho_{x_n}g\grotimes \phi \\
%
%
%
%
&=Xg(0)\rho_0\grotimes \psi+g(0)\rho_0\grotimes C\psi-\rho_0g\grotimes \phi \\
&\ \ \ +\sum_{n>0} \bra{
X\rho_{x_n}\grotimes \frac{g(-x_n)C}{x_n^2+C^2}\phi
+\rho_{x_n}\grotimes \frac{-g(-x_n)x_nC}{x_n^2+C^2}\phi
+\rho_{-x_n}\grotimes \frac{g(x_n)x_nC}{x_n^2+C^2}\phi
+X\rho_{-x_n}\grotimes \frac{g(x_n)C}{x_n^2+C^2}\phi}\\
&\ \ \ \ \ \ +\sum_{n>0} \bra{
X\rho_{x_n}\grotimes \frac{g(x_n)x_n}{x_n^2+C^2}\phi
+\rho_{x_n}\grotimes \frac{g(x_n)C^2}{x_n^2+C^2}\phi
-\rho_{x_n}g\grotimes\phi} \\
&\ \ \ \ \ \ \ \ \ +\sum_{n>0} \bra{
X\rho_{-x_n}\grotimes \frac{-g(-x_n)x_n}{x_n^2+C^2}\phi
+\rho_{-x_n}\grotimes \frac{g(-x_n)C^2}{x_n^2+C^2}\phi
-\rho_{-x_n}g\grotimes\phi}\\
&=Xg(0)\rho_0\grotimes \psi+g(0)\rho_0\grotimes \phi-\rho_0g\grotimes \phi \\
&\ \ \ +\sum_{n>0} \bra{
(X\rho_{x_n}-x_n\rho_{x_n})\grotimes \frac{g(-x_n)C}{x_n^2+C^2}\phi
+(x_n\rho_{-x_n}+X\rho_{-x_n})\grotimes \frac{g(x_n)C}{x_n^2+C^2}\phi}\\
&\ \ \ \ \ \ +\sum_{n>0} \bra{
(X\rho_{x_n}-x_n\rho_{x_n})\grotimes \frac{g(x_n)x_n}{x_n^2+C^2}\phi
+g(x_n)\rho_{x_n}\grotimes \phi
-\rho_{x_n}g\grotimes\phi} \\
&\ \ \ \ \ \ \ \ \ +\sum_{n>0} \bra{
(X\rho_{-x_n}+x_n\rho_{x_n})\grotimes \frac{-g(-x_n)x_n}{x_n^2+C^2}\phi
+g(-x_n)\rho_{-x_n}\grotimes \phi
-\rho_{-x_n}g\grotimes\phi}.
\end{align*}}

The norm of the first line is less than $\epsilon/4$ because $\|g(0)X\rho_0\grotimes \psi\|<|g(0)|\cdot \Delta\cdot \|\psi\|\leq \epsilon/16$ and $\|g(0)\rho_0-\rho_0g\|\leq \epsilon/16 $. 

The norm of each summand of the second line is less than $\epsilon/8$ because $\|X\rho_{\pm x_n}\mp x_n\rho_{\pm x_n}\|<\Delta<\epsilon/16$, $|g(x_n)|\leq 1$ and $\|\frac{C}{x_n^2+C^2}\phi\|= \|\frac{C^2}{x_n^2+C^2}\psi\|\leq 1$ for $n\in\bb{Z}_{\neq 0}$. Thanks to the property $(3_\rho)$, we notice the norm of the sum is less than $\epsilon/4$.

The norm of each summand of the third line is less than $\epsilon/8$ because 
$\|g(x_n)\rho_{x_n}\grotimes \phi-\rho_{x_n}g\grotimes\phi\|
=\|g(x_n)\rho_{x_n}-\rho_{x_n}g\|\cdot \|\phi\|< \epsilon/16,$
and 
\begin{align*}
\left\|\frac{x_n}{x_n^2+C^2}\phi\right\|&=\left\|\frac{x_nC}{x_n^2+C^2}\psi\right\| 
=\left\|\frac{C/x_n}{1+\bra{C/x_n}^2}\psi\right\| 
\leq \|\psi\|\leq 1,
\end{align*}
where $\frac{{C}/{x_n}}{1+\bra{{C}/{x_n}}^2}$ is obtained by the operator calculus $F\bra{{C}/{x_n}}$ for $F(t)=\frac{t}{1+t^2}$. The norm of the sum is less than $\epsilon/4$ for the same reason of the second line.

The fourth line can be dealt with the same way. We omit the details.
\end{proof}

We have introduced the infinite sum $\fra{C}_l=\id\grotimes\bbra{\sum_{n=1}^\infty n^{-\frac{l}{4}}[2x_n\grotimes c(e_n)+2y_n\grotimes c(f_n)]}$. Thanks to the above lemma, we can define an operator $\fra{C}_l\{X\grotimes\id+\id\grotimes C\}^{-1}$ on the range of $X\grotimes\id+\id\grotimes C$ which is dense in $\ca{S}\grotimes \ca{H}$. It extends to a bounded operator on $\ca{S}\grotimes \ca{H}$ by the following lemma.

\begin{lem}
${\fra{C}_l}\{X\grotimes\id+\id\grotimes C\}^{-1}$ is bounded on $\ran\bra{X\grotimes\id+\id\grotimes C}$.
\end{lem}
\begin{proof}
It suffices to prove that ${x_n\grotimes c(e_n)}\{X\grotimes\id+\id\grotimes C\}^{-1}$ is bounded and its norm is at most $1$. For any $g\grotimes \phi\in\ca{S}\grotimes \ca{H}$ such that $Xg\in \ca{S}$ and $C\phi\in\ca{H}$, we have
\begin{align*}
&\|\{\id\grotimes x_n\grotimes c(e_n)\}(X\grotimes\id+\id\grotimes C)^{-1}(X\grotimes\id+\id\grotimes C)(g\grotimes \phi)\|^2 \\
&\ \ \ =\|\{\id\grotimes x_n\grotimes c(e_n)\}(g\grotimes \phi)\|^2 \\
&\ \ \ =\|\inpr{\{\id\grotimes x_n\grotimes c(e_n)\}(g\grotimes \phi)}{\{\id\grotimes x_n\grotimes c(e_n)\}(g\grotimes \phi)}{\ca{S}}\|_{\ca{S}} \\
&\ \ \ =\|g^*g\|\inpr{x_n\grotimes c(e_n)(\phi)}{x_n\grotimes c(e_n)(\phi)}{\ca{H}} \\
&\ \ \ =\|g^*g\|\inpr{x_n^2\grotimes \id(\phi)}{\phi}{\ca{H}}.
\end{align*}
On the other hand,
\begin{align*}
&\|(X\grotimes\id+\id\grotimes C)(g\grotimes \phi)\|^2 \\
&\ \ \ =\|\inpr{(X\grotimes\id+\id\grotimes C)(g\grotimes \phi)}{(X\grotimes\id+\id\grotimes C)(g\grotimes \phi)}{\ca{S}}\| \\
&\ \ \ =\|(Xg)^*Xg\inpr{\phi}{\phi}{}+g^*g\inpr{C \phi}{C \phi}{}\|\\
&\ \ \ \geq \|g^*g\|\inpr{C \phi}{C \phi}{},
\end{align*}
where we have used the definition of the norm of $\ca{S}$ to prove the last inequality: The norm is defined by the maximum value, and hence the norm of the sum of two non-negative elements is not less than the norm of each summand. Noticing that
$$\inpr{C \phi}{C \phi}{}=\sum_{m}(\inpr{x_m^2\grotimes \id(\phi)}{\phi}{}+\inpr{y_m^2\grotimes \id(\phi)}{\phi}{})\geq \inpr{x_n^2\grotimes \id(\phi)}{\phi}{\ca{H}},$$
we have the inequality 
$$\|\{\id\grotimes x_n\grotimes c^*(e_n)\}\{X\grotimes\id+\id\grotimes C\}^{-1}(X\grotimes\id+\id\grotimes C)(g\grotimes \phi)\|^2\leq \|(X\grotimes\id+\id\grotimes C)(g\grotimes \phi)\|^2.$$
Since the subspace spanned by $\bbra{ g\grotimes \phi\in\ca{S}\grotimes \ca{H}\,\middle|\, \text{ such that } Xg\in \ca{S}\text{ and }C\phi\in\ca{H}}$ is dense, we find that ${\fra{C}_l}\{X\grotimes\id+\id\grotimes C\}^{-1}$ is bounded on $\ran\bra{X\grotimes\id+\id\grotimes C}$.
\end{proof}

We have defined $\pi$ in an abstract and algebraic way. We need the following analytic and quantitative property for even elements of $\ca{S}$. The even part of a $\bb{Z}_2$-graded algebra $A$ is denoted by $A_0$. The following proof reminds us of the chain rule.

\begin{lem}\label{lemma C1 functions are mapped to C1 elements}
If $g\in \ca{S}_0$ is the $C^1$-limit of a sequence of $(\ca{S}_\fin)_0$, $\pi(\beta_0^\infty(g))$ preserves $\dom(\id\grotimes\Dirac)$ and the commutator with $\id\grotimes\Dirac$ is bounded.
\end{lem}
\begin{proof}
As a preliminary, we study the commutator $[\pi(a),\id\grotimes\Dirac]$ more, for $a\in \ca{A}_\HKT(U_{L^2_k})_\fin$.
An element $f$ of $(\ca{S}_\fin)_0$ can be written as $f(t)=p(f_\ev(t))=p\bra{\frac{1}{1+t^2}}$ for some polynomials $p$. This is because $f$ is given by a linear combination of $f_\od^\alpha\cdot f_\ev^\beta$ for $\beta\in \bb{N}$ and $\alpha\in 2\bb{N}$, and $f_\od^2=f_\ev-f_\ev^2$. Let us compute the commutator $[\id\grotimes\Dirac,\pi(\beta_0^\infty(f))]$. First, we compute it for the monomial $p(X)=X^m$.
By a standard algebraic technique on commutators, 
{\small
\begin{align*}
&\bbbra{\id\grotimes \Dirac,\pi(\beta_0^\infty(p(f_\ev)))} \\
&\ \ \ =\sum_{\alpha=0}^{m-1} \bra{\frac{1}{1+X^2\grotimes \id+\id\grotimes C^2}}^{\alpha+1}\bra{-2\sum_n\id\grotimes n^{-l/4}\bra{x_n\grotimes c(e_n)+y_n\grotimes c(f_n)}}\bra{\frac{1}{1+X^2\grotimes \id+\id\grotimes C^2}}^{m-\alpha}\\
&\ \ \ =\bra{\fra{C}_l(X\grotimes\id+\id\grotimes C)^{-1}}\cdot (-2m)(X\grotimes\id+\id\grotimes C)\bra{\frac{1}{1+X^2\grotimes \id+\id\grotimes C^2}}^{m+1}.
\end{align*}}
Note that $\fra{C}_l$ commutes with $\bra{\frac{1}{1+X^2\grotimes \id+\id\grotimes C^2}}^{\alpha+1}$.
Since $[p\circ f_\ev]'(t)=-2mt(1+t^2)^{-m-1}$, we notice that 
\begin{align*}
\bbbra{\id\grotimes \Dirac,\pi(\beta_0^\infty(f))} 
&= \fra{C}_l\{X\grotimes\id+\id\grotimes C\}^{-1}f'(X\grotimes\id+\id\grotimes C) \\
&={\fra{C}_l}\{X\grotimes\id+\id\grotimes C\}^{-1}\pi(\beta_0^\infty(f')).
\end{align*}
Therefore, for general polynomial $p$ and $f=p\circ f_\ev$,
$$\bbbra{\id\grotimes \Dirac,\pi(\beta_0^\infty(p\circ f_\ev))} ={\fra{C}_l}\{X\grotimes\id+\id\grotimes C\}^{-1}\pi(\beta_0^\infty(f'))$$
for $f\in (\ca{S}_\fin)_0$.

Let us prove the statement. By the assumption, there exists a sequence $\{g_m\}\subseteq (\ca{S}_\fin)_0$ such that $g_m\to g$ in the $C^1$-topology. We need to prove that $\pi(\beta_0^\infty(g))$ is a $C^1$-element. We have proved that $\pi(\beta_0^\infty(g_m))\phi\in \dom(\id\grotimes\Dirac)$ for $\phi\in\dom(\id\grotimes\Dirac)$ in $(1)$ of $(A)$ of the proof of Lemma \ref{index element for HKT algebra}. By the above preliminary,
\begin{align*}
\id\grotimes\Dirac\circ\pi(\beta_0^\infty(g_m))\phi 
&=[\id\grotimes\Dirac,\pi(\beta_0^\infty(g_m))]\phi +\pi(\beta_0^\infty(g_m))\id\grotimes\Dirac \phi \\
&={\fra{C}_l}\{X\grotimes\id+\id\grotimes C\}^{-1}\pi(\beta_0^\infty(g_m'))\phi +\pi(\beta_0^\infty(g_m))\id\grotimes\Dirac \phi \\
&\to {\fra{C}_l}\{X\grotimes\id+\id\grotimes C\}^{-1}\pi(\beta_0^\infty(g'))\phi  +\pi(\beta_0^\infty(g))\id\grotimes\Dirac \phi
\end{align*}
as $m\to \infty$. It means that $\pi(\beta_0^\infty(g))\phi\in \dom(\id\grotimes\Dirac)$ and $\id\grotimes\Dirac \pi(\beta_0^\infty(g))\phi=\lim_m \id\grotimes\Dirac \pi(\beta_0^\infty(g_m))\phi$. Moreover, $[\id\grotimes\Dirac, \pi(\beta_0^\infty(g))]={\fra{C}_l}\{X\grotimes\id+\id\grotimes C\}^{-1}\pi(\beta_0^\infty(g'))$ is a bounded operator. Thus $\pi(\beta_0^\infty(g))$ is a $C^1$-element.
\end{proof}

The above lemma gives an analytic condition to judge whether an element of $\ca{A}_\HKT(U_{L^2_k})$ is of $C^1$ or not. Next, we need to know the $C^1$-closure of $\beta_0^\infty((\ca{S}_\fin)_0)$.

\begin{lem}\label{lemma Sfineven is dense in C1epsilon}
An even function $f\in C^1(\bb{R})$ such that $f\in\ca{S}$ and $(X^2+1)f'\in\ca{S}$, is the $C^1$-limit of a sequence of $(\ca{S}_\fin)_0$. In particular, any element of $\iota((\ca{S}_\vep)_0)\cap C^1(\bb{R})$ is the $C^1$-limit of a sequence of $(\ca{S}_\fin)_0$. \end{lem}
\begin{proof}
Since $f_\od^2=f_\ev-f_\ev^2$, 
$$\ca{S}_\fin=f_\ev\bb{C}[f_\ev]\oplus f_\od\bb{C}[f_\ev]$$
as a vector space. 
Since $(X^2+1)f'\in\ca{S}$, there is a sequence of polynomials $\{p_n\}$ such that $\{f_\od p_n(f_\ev)\}$ uniformly converges to $(X^2+1)f'$. Thus, for any $\epsilon$, there exists $n_0$ such that $n\geq n_0$ implies that $|(t^2+1)f'(t)-f_\od(t)p_n(f_\ev(t))|<\epsilon/\pi$ on $\bb{R}$, and hence
$$\left|f'(t)-\frac{1}{t^2+1}f_\od(t) p_n(f_\ev(t))\right|<\frac{\epsilon}{\pi(t^2+1)}.$$
It means that
\begin{align*}
\left|f(t)-\int_{-\infty}^t\frac{1}{s^2+1}f_\od(s) p_n(f_\ev(s))ds
\right|&=
\left|\int_{-\infty}^t\bra{f'(s)-\frac{1}{s^2+1}f_\od(s) p_n(f_\ev(s))}ds\right| \\
&\leq \int_{-\infty}^t\left|f'(s)-\frac{1}{s^2+1}f_\od(s) p_n(f_\ev(s))\right|ds\\
&<\int_{-\infty}^t\frac{\epsilon}{\pi(s^2+1)}ds\\
&=\frac{\epsilon}{\pi}\bra{\arctan(t)+\frac{\pi}{2}}<\epsilon.
\end{align*}
Since $[f_\ev^\alpha]'=-2\alpha f_\od[f_\ev(t) ]^\alpha$, we have
$$\int_{-\infty}^t\frac{1}{s^2+1}f_\od(s) (f_\ev(s))^\beta ds
=\frac{-1}{2(\beta+1)}[f_\ev(t)]^{\beta+1}.$$
Hence, the function $t\mapsto \int_{-\infty}^t\frac{1}{s^2+1}f_\od(s) p_n(f_\ev(s)) ds$ belongs to $\ca{S}_\fin$. This sequence converges to $f$ in the $C^1$-topology.
\end{proof}

\begin{lem}\label{lemma Sfinodd is dense in C1epsilon}
Any element of $(\ca{S}_\vep)_1\cap C^2(\bb{R})$ is the $C^1$-limit of a sequence of $f_\od(\ca{S}_\fin)_0$. 
\end{lem}
\begin{proof}
Let $f\in (\ca{S}_\vep)_1\cap C^2(\bb{R})$. Then, $t^{-1}f(t)$ is defined on $\bb{R}$ and of $C^1$ (the value of this function at $0$ is defined by $f'(0)$). Hence, $f_\od^{-1}\cdot f(t):=\frac{1+t^2}{t}f(t)$ is of $C^1$. Note that it is an even element. Moreover, $(X^2+1)[f_\od^{-1}\cdot f]'\in \ca{S}_\vep\subseteq \ca{S}$. Thanks to the previous lemma, we have a sequence $\{f_n\}\subseteq (\ca{S}_\fin)_0$ converging to $f_\od^{-1}\cdot f$ in the $C^1$-topology. Then, the sequence $\{f_\od f_n\}$ converges to $f_\od\cdot f_\od^{-1}\cdot f=f$ in the $C^1$-topology. 
\end{proof}
\underline{\it Proof of Definition-Proposition \ref{index element}} :
The necessary change from the proof of Lemma \ref{index element for HKT algebra} is just the following: {\it The set of $C^1$-elements of $\ca{A}(U_{L^2_k})$ is dense.} Combining Lemma \ref{lemma X+C has dense range}--\ref{lemma Sfinodd is dense in C1epsilon}, we notice that all the elements of $\cup_N\beta_N^\infty([\ca{S}_\vep\cap C^2(\bb{R})]\grotimes^\alg Cl_{\tau,\scr{S}}(U_{L^2_k,N}))$ are of $C^1$. The other conditions are obvious.
\begin{flushright} $\square$ \end{flushright}

Let us define the substitute for the index element on the whole manifold $\ca{M}_{L^2_k}$. 

Recall the isomorphism $\ca{A(M}_{L^2_k})\cong \ca{A}(U_{L^2_k})\grotimes Cl_\tau(\widetilde{M})$.
We prepare the index element for the $\widetilde{M}$-direction. Let $\ca{L}|_{\widetilde{M}}$ be the restriction of $\ca{L}$ to $\widetilde{M}$, which is a $\tau$-twisted $T\times\Pi_T$-equivariant line bundle. It admits a $(T\times\Pi_T)^\tau$-invariant connection, by the averaging procedure using a cut-off function with respect to the $T\times\Pi_T$-action on $\widetilde{M}$. We have assumed that $(S_M,\gamma_M)$ is a $T$-equivariant Spinor bundle over $M$. We define $(S_{\widetilde{M}},\gamma_{\widetilde{M}})$ by the lift of this bundle to $\widetilde{M}$, and similarly for $(S_{\widetilde{M}}^*,\gamma^*_{\widetilde{M}})$.
Then, we define a $(T\times\Pi_T)^\tau$-equivariant Dirac operator $D_{\widetilde{M}}$ by
$$D_{\widetilde{M}}:=\sum_n c_{\widetilde{M}}(v_n)\circ \nabla_{v_n}^{\ca{L}|_{\widetilde{M}}\grotimes S_{\widetilde{M}}^*\grotimes S_{\widetilde{M}}},$$
where $c_{\widetilde{M}}(v):=\frac{1}{\sqrt{2}}\bra{\id\grotimes\gamma_{\widetilde{M}}(v)-\sqrt{-1}\gamma^*_{\widetilde{M}}(v)\grotimes \id}$ and $\{v_n\}$ is an orthonormal base of the tangent space. 
Moreover, $H:=L^2(\widetilde{M},\ca{L}|_{\widetilde{M}}\grotimes S_{\widetilde{M}}^*\grotimes S_{\widetilde{M}})$ admits a $*$-representation $\mu$ of $Cl_\tau(\widetilde{M})$ given by the Clifford multiplication $c^*_{\widetilde{M}}(v):=\frac{\sqrt{-1}}{\sqrt{2}}\bra{\id\grotimes\gamma_{\widetilde{M}}(v)+\sqrt{-1}\gamma_{\widetilde{M}}^*(v)\grotimes \id}$ for $v\in T\widetilde{M}$. We denote the $KK$-element corresponding to $(H,\mu,D_{\widetilde{M}})$ by $[\widetilde{D_{\widetilde{M}}}]\in KK_{T\times \Pi_T}^\tau(Cl_\tau(\widetilde{M}),\bb{C})$ (the reformulated index element for even-dimensional $Spin^c$-manifold). 
 
We define $\ca{D}:=\Dirac\grotimes \id+\id\grotimes D_{\widetilde{M}}$ on $\ca{H}\grotimes H$.
  
\begin{dfn}\label{dfn index element for the present paper for proper LT space}
The index element $[\widetilde{\ca{D}}]\in KK_{LT_{L^2_m}}^\tau(\ca{A(M}_{L^2_k}),\ca{S}_\vep)$ is the corresponding $KK$-element to the unbounded $\tau$-twisted $LT_{L^2_m}$-equivariant  Kasparov $(\ca{A(M}_{L^2_k}),\ca{S}_\vep)$-module
$$\bra{\ca{S}_\vep\grotimes \ca{H}\grotimes H,\, \pi\grotimes \mu,\, \id\grotimes \ca{D}}.$$
\end{dfn}

Recall that our construction has an artificial parameter $l$ and $m$. Although the conditions that $l>4$ and $m\geq l+k$ are the keys of our quantitative arguments, we hope that our object is independent of them. As we expect, at the $\bb{KK}_{LT}^\tau$-theory level, the index element is independent of them.

We have fixed $l$ and $m$ so far. From now on, we compare two Kasparov modules defined using two different parameters $(l,m)$ and $(l',m')$. In order to distinguish them, $\star_l$ denotes an object $\star$ which is defined with $(l,m)$, for example, $\ca{H}_l$, $\Dirac_l$, $\ca{D}_l$, $\pi_l$, $\vac_l$, $dL'_l(z_n)$ and so on. 
The objects defined independently $l$ is denoted without the subscript $l$, for example $S_U$ and $z_n$ depend only on $k$.

\begin{pro}\label{prop KK element is independent of l}
For $l,l'>4$, the equivariant unbounded Kasparov module constructed with $(l,m)$ is homotopic to that constructed with $(l',m')$, as $LT_{L^2_{m''}}$-equivariant Kasparov modules for $m''\geq \max\{m,m'\}$, that is to say, $i_{m'',m}^*[\widetilde{\ca{D}}_l]=i_{m'',m'}^*[\widetilde{\ca{D}}_{l'}]$
in $KK_{LT_{L^2_{m''}}}^\tau(\ca{A}(\ca{M}_{L^2_k}),\ca{S}_\vep)$.
\end{pro}
\begin{rmk}
By this result, the corresponding element of $\bb{KK}_{LT}^\tau(\ca{A(M}_{L^2_k}),\ca{S}_\vep)$ to $[\widetilde{D}_l]$ is independent of $l$. The resulting element is denoted by $[\widetilde{\ca{D}}]$.
\end{rmk}
\begin{proof}
It suffices to deal with the $U_{L^2_k}$-part. For $\overrightarrow{\alpha}=(\alpha_1,\alpha_2,\cdots)$, $\overrightarrow{\beta}=(\beta_1,\beta_2,\cdots)$, $\overrightarrow{\xi}=(\xi_1,\xi_2,\cdots,\xi_M)$, and  $\overrightarrow{\eta}=(\eta_1,\eta_2,\cdots,\eta_N)$ such that $\alpha_i,\beta_i\in \bb{Z}_{\geq 0}$, $\alpha_i=\beta_i=0$ except for finitely many $i$'s, $\xi_i,\eta_i\in \bb{N}$, $N,M\in \bb{Z}_{\geq 0}$, $0<\xi_1<\xi_2<\cdots<\xi_M$, and $0<\eta_1<\eta_2<\cdots<\eta_N$, we put
{\small
\begin{align*}
&\widetilde{\phi}_{\overrightarrow{\alpha},\overrightarrow{\beta},\overrightarrow{\xi},\overrightarrow{\eta}}(s)\\
&:=
\bra{dR'_{(1-s)l+sl'}(\overline{z_1})^{\alpha_1}dL'_{(1-s)l+sl'}(z_1)^{\beta_1}\cdots }\vac_{(1-s)l+sl'}\grotimes [z_{\xi_1}\wedge \cdots \wedge z_{\xi_M}]\grotimes [\overline{z_{\eta_1}}\wedge \cdots \wedge \overline{z_{\eta_N}}]\in \ca{S}_\vep\grotimes \ca{H}_{(1-s)l+sl'}.
\end{align*}}
We construct a homotopy, that is to say, an $U_{L^2_{m''}}$-equivariant Kasparov $(\ca{A}(U_{L^2_k}),\ca{S}_\vep\grotimes C(I))$-module $(E,\sigma,\cancel{D})$. The $U_{L^2_{m''}}$-equivariant Hilbert $\ca{S}_\vep\grotimes C(I)$-module $E$ is given as follows:
\begin{itemize}
\item Let $E_\fin$ be the pre-Hilbert $\ca{S}_\vep\grotimes C(I)$-module given by the $C(I)$-linear span of the \\
$\coprod_{s\in [0,1]} \ca{S}_\vep\grotimes \ca{H}_{(1-s)l+sl'}$-valued section
\begin{align*}
s\mapsto f\grotimes \psi_{\overrightarrow{\alpha},\overrightarrow{\beta},\overrightarrow{\xi},\overrightarrow{\eta}}(s)
\end{align*}
for $f\in\ca{S}_\vep$ and multi-indices $\overrightarrow{\alpha},\overrightarrow{\beta},\overrightarrow{\xi},\overrightarrow{\eta}$ satisfying the above condition.
\item The pre-Hilbert module structure is given by the following formulas: For $\psi,\psi_1,\psi_2\in E_\fin$ and $f\grotimes F\in \ca{S}_\vep\grotimes C(I)$,
$$[\psi\cdot f\grotimes F](s):= [\psi (s)\cdot f]\cdot F(s)\in \ca{S}_\vep\grotimes \ca{H}_{(1-s)l+sl'}$$
$$\inpr{\psi_1}{\psi_2}{E}(s):=\inpr{\psi_1(s)}{\psi_2(s)}{\ca{S}_\vep\grotimes \ca{H}_{(1-s)l+sl'}}.$$
\item The Hilbert module $E$ is the completion of $E_\fin$ with respect to the above inner product.
\item The $U_{L^2_{m''}}^\tau$-action is given by the pointwise action $\{(g,z)\cdot(\psi)\}(s):=L_{(g,z)}(\psi(s))$.
\end{itemize}
We need to prove that the inner product of two elements $\psi_1,\psi_2\in E_\fin$ is indeed an element of $\ca{S}_\vep\grotimes C(I)$. For this aim, it suffices to prove that the function
$$s\mapsto \inpr{f\grotimes \widetilde{\phi}_{\overrightarrow{\alpha},\overrightarrow{\beta},\overrightarrow{\xi},\overrightarrow{\eta}}(s)
}{f'\grotimes \widetilde{\phi}_{\overrightarrow{\alpha'},\overrightarrow{\beta'},\overrightarrow{\xi'},\overrightarrow{\eta'}}(s)}{}$$
is continuous. If $\overrightarrow{\alpha}=\overrightarrow{\alpha'}$, $\overrightarrow{\beta}=\overrightarrow{\beta'}$, $\overrightarrow{\xi}=\overrightarrow{\xi'}$, and $\overrightarrow{\eta}=\overrightarrow{\eta'}$, that function is given by
$$s\mapsto f^*f'\overrightarrow{\alpha}!\overrightarrow{\beta}!\prod_{m=1}^\infty (2m^{(1-s)l+sl'})^{\alpha_m+\beta_m},$$
and if not, the function is identically zero.
Both functions are continuous (note that $\alpha_m=0$ and $\beta_m=0$ except for finitely many $m$'s, and hence the above infinite product is in fact a finite product of continuous functions).

We define $\sigma:\ca{A}(U_{L^2_k})\to \bb{L}_{\ca{S}_\vep\grotimes C(I)}(E)$ by $\sigma(a)(\psi)(s):= \pi_{(1-s)l+sl'}(a)(\psi(s))$, and we define $\cancel{D}$ by $\cancel{D}(\psi)(s):=\Dirac_{(1-s)l+sl'}[\psi(s)]$.
Then, the triple $(E,\sigma,\cancel{D})$ is obviously a $U_{L^2_{m''}}$-equivariant unbounded Kasparov $(\ca{A}(U_{L^2_k}),\ca{S}_\vep\grotimes C(I))$-module.

Let us prove that the evaluation of $(E,\sigma,\cancel{D})$ at $s=0$ gives $(\ca{S}_\vep\grotimes \ca{H}_l,\pi_l,\ca{D}_l)$ and similarly for $s=1$. 
We prove it for arbitrary $s\in I$.
What we need to prove is that $\ev_s(E)$ is isomorphic to $\ca{S}_\vep\grotimes \ca{H}_{(1-s)l+sl'}$. For this aim, it suffices to prove that the evaluation at $s$ is surjective. 
It suffices to find an isometric embedding as a Banach space $j_s:\ca{S}_\vep\grotimes \ca{H}\hookrightarrow E$ satisfying $j_s(v)(s)=v$. We define it as follows: For $f\in \ca{S}_\vep$ and $ \psi_{\overrightarrow{\alpha},\overrightarrow{\beta},\overrightarrow{\xi},\overrightarrow{\eta}}(s)\in\ca{H}_{(1-s)l+sl'}$, we define $ j_s(f\grotimes \psi_{\overrightarrow{\alpha},\overrightarrow{\beta},\overrightarrow{\xi},\overrightarrow{\eta}}(s))$ by
$$t\mapsto \sqrt{\prod_m (2m)^{[(t-s)l+(s-t)l'](\alpha_m+\beta_m)}}f\grotimes \psi_{\overrightarrow{\alpha},\overrightarrow{\beta},\overrightarrow{\xi},\overrightarrow{\eta}}(t).$$
Since the function appearing as the coefficient is continuous, the above function belongs to $E$. In the remainder of this proof, we prove that $j_s$ is isometric.

If $\overrightarrow{\alpha}=\overrightarrow{\alpha'}$, $\overrightarrow{\beta}=\overrightarrow{\beta'}$, $\overrightarrow{\xi}=\overrightarrow{\xi'}$, and $\overrightarrow{\eta}=\overrightarrow{\eta'}$, the inner product of $ j_s(f\grotimes \psi_{\overrightarrow{\alpha},\overrightarrow{\beta},\overrightarrow{\xi},\overrightarrow{\eta}}(s))$ and $ j_s(f'\grotimes \psi_{\overrightarrow{\alpha'},\overrightarrow{\beta'},\overrightarrow{\xi'},\overrightarrow{\eta'}}(s))$ is given by the function on $t$
\begin{align*}
&\inpr{ j_s(f\grotimes \psi_{\overrightarrow{\alpha},\overrightarrow{\beta},\overrightarrow{\xi},\overrightarrow{\eta}}(s))}{ j_s(f'\grotimes \psi_{\overrightarrow{\alpha},\overrightarrow{\beta},\overrightarrow{\xi},\overrightarrow{\eta}}(s))}{\ca{S}_\vep\grotimes C(I)}(t) \\
&=f^*f'\overrightarrow{\alpha}!\overrightarrow{\beta}!\prod_m (2m)^{[(t-s)l+(s-t)l'](\alpha_m+\beta_m)}\prod_{m} (2m)^{[(1-t)l+tl'](\alpha_m+\beta_m)} \\
&= f^*f'\overrightarrow{\alpha}!\overrightarrow{\beta}!\prod_{m} (2m)^{[(1-s)l+sl'](\alpha_m+\beta_m)}\\
&=\inpr{f\grotimes  \psi_{\overrightarrow{\alpha},\overrightarrow{\beta},\overrightarrow{\xi},\overrightarrow{\eta}}(s)}{f'\grotimes  \psi_{\overrightarrow{\alpha},\overrightarrow{\beta},\overrightarrow{\xi},\overrightarrow{\eta}}(s)}{\ca{S}_\vep}.
\end{align*}
Taking the $C^*$-norm of both sides, we find that $j_s$ is isometric.\footnote{The $C^*$-norm of the right hand side is the maximum norm of the function $u\mapsto \overline{f(u)}f'(u)\inpr{\psi_{\overrightarrow{\alpha},\overrightarrow{\beta},\overrightarrow{\xi},\overrightarrow{\eta}}(s)}{ \psi_{\overrightarrow{\alpha},\overrightarrow{\beta},\overrightarrow{\xi},\overrightarrow{\eta}}(s)}{}$, and that of the left hand side is the maximum norm of the function $(u,t)\mapsto \inpr{ j_s(f\grotimes \psi_{\overrightarrow{\alpha},\overrightarrow{\beta},\overrightarrow{\xi},\overrightarrow{\eta}}(s))}{ j_s(f'\grotimes \psi_{\overrightarrow{\alpha},\overrightarrow{\beta},\overrightarrow{\xi},\overrightarrow{\eta}}(s))}{\ca{S}_\vep\grotimes C(I)}(t,u)$, where $u$ is the variable for $\ca{S}_\vep$ and $t$ is that for $C(I)$. Since it is constant with respect to $t$, $j_s$ is isometric.}
Consequently, $j_s$ extends to $\ca{S}_\vep\grotimes \ca{H}_{(1-s)l+sl'}$, and it satisfies $\ev_s(j_s(v))=v$ for all $v\in \ca{S}_\vep\grotimes \ca{H}_{(1-s)l+sl'}$.

If not, both sides are zero, and we finish the proof.
\end{proof}

In the following, we again fix $l$ and $m$, and we do not study this kind of problems.

\subsection{The Poincar\'e duality homomorphism for proper $LT$-spaces}\label{section PD LT spaces}

The task of this subsection is to apply the Poincar\'e duality homomorphism to the index element constructed in the previous subsection.

We consider the restriction of the local Bott element $[\widetilde{\Theta_{\ca{M}_{L^2_k},2}}]$ to $\ca{M}_{L^2_m}$ and we denote it by the same symbol: 
$$[\widetilde{\Theta_{\ca{M}_{L^2_k},2}}]\in \ca{R}KK_{LT_{L^2_m}}(\ca{M}_{L^2_m};\ca{S}_\vep\grotimes \scr{C}(\ca{M}_{L^2_m}),\ca{A}(\ca{M}_{L^2_k})\grotimes \scr{C}(\ca{M}_{L^2_m})).$$
Let $[\ca{L}]\in \ca{R}KK_{LT_{L^2_m}}^\tau(\ca{M}_{L^2_m};\scr{C}(\ca{M}_{L^2_m}),\scr{C}(\ca{M}_{L^2_m}))$ be the $\ca{R}KK$-element defined by 
$$\bra{\{\ca{L}_x\}_{x\in \ca{M}_{L^2_m}},\{1\}_{x\in \ca{M}_{L^2_m}},\{0\}_{x\in \ca{M}_{L^2_m}}}.$$
The goal of this subsection is to prove that 
$\PD([\widetilde{\ca{D}}])=\sigma_{\ca{S}_\vep}([\ca{L}])$.
For this aim, it is convenient to divide the problem into the $U_{L^2_m}$-part and the $\widetilde{M}$-part. Thus, we rewrite the local Bott element by the tensor product of those of $U_{L^2_m}$ and $\widetilde{M}$. We begin with a technical lemma.

\begin{lem}\label{Lem technical lemma on S and coprod}
We define the following $*$-homomorphisms: $\Delta:\ca{S}\to \ca{S}\grotimes \ca{S}$ by $\Delta(f):=f(X\grotimes \id+\id\grotimes X)$ and $\ev_0:\ca{S}\to \bb{C}$ by $\ev_0(f)=f(0)$. The same formulas define homomorphisms on $\ca{S}_\vep$.

$(1)$ $(\id\grotimes \ev_0)\circ \Delta=\id$. The same formula holds for $\ca{S}_\vep$.

$(2)$ Let $A$ and $B$ be $\bb{Z}_2$-graded $C^*$-algebras equipped with odd unbounded multipliers with compact resolvent $D_A$ and $D_B$. Suppose that $\Delta(f)=\sum f_i^1\grotimes f_i^2$. Then, we have $f(D_A\grotimes \id+\id\grotimes D_B)=\sum_i f_i^1(D_A)\grotimes f_2^i(D_B)$ on $A\grotimes B$,
\end{lem}
\begin{proof}
$(1)$ It is mentioned in \cite[P. 14]{GH} as an exercise. For $\ca{S}_\vep$, one can prove it by considering the commutative diagram
$$\begin{CD}
\ca{S}_\vep @>\Delta >> \ca{S}_\vep\grotimes \ca{S}_\vep \\
@V\iota VV @VV\iota \grotimes \iota V \\
\ca{S} @>\Delta >> \ca{S}\grotimes \ca{S}.
\end{CD}$$

$(2)$ The homomorphism $\ca{S}\ni f\mapsto f(D_A)\in A$ is denoted by $\beta_A$, and similarly for $B$. With these notations, the statement can be written as $\beta_A\grotimes \beta_B\circ \Delta=\beta_{A\grotimes B}$, where $D_{A\grotimes B}:=D_A\grotimes \id+\id\grotimes D_B$.
Since both sides are $*$-homomorphisms, it suffices to prove it for generators. We choose $f_e(t):=e^{-t^2}$ and $f_o(t)=te^{-t^2}$ as generators. Thanks to
\begin{center}
$\Delta(f_e)=f_e\grotimes f_e$ and $\Delta(f_o)=f_o\grotimes f_e+f_e\grotimes f_o$,
\end{center}
and thanks to computations on functional calculus in \cite[Appendix A.4]{HKT},
we have
$$\beta_A\grotimes \beta_B\circ \Delta(f_e)=f_e(D_A)\grotimes f_e(D_B)=\beta_{A\grotimes B}(f_e);$$
$$
\beta_A\grotimes \beta_B\circ \Delta(f_o)
=f_o(D_A)\grotimes f_e(D_B)+f_e(D_A)\grotimes f_o(D_B)=\beta_{A\grotimes B}(f_o).
$$
\end{proof}

The local Bott homomorphism for a Hilbert manifold $X$ at $x$ is denoted by $\beta^X_x$ when it is necessary to emphasize the manifold to consider it.

\begin{lem}\label{lemma local Bott element can be factorized}
Let $[\widetilde{\Theta_{\ca{M}_{L^2_k},2}}']$ be the $\ca{M}_{L^2_m}\rtimes LT_{L^2_m}$-equivariant Kasparov $(\ca{S}_\vep\grotimes \scr{C}(\ca{M}_{L^2_m}),\ca{A}(\ca{M}_{L^2_k})\grotimes \scr{C}(\ca{M}_{L^2_m}))$-module defined by the exterior tensor product of the following Kasparov modules:
$$\bra{\{\ca{A}(U_{L^2_k})\}_{g\in U_{L^2_m}},\bbra{\beta_g^{U_{L^2_k}}}_{g\in U_{L^2_m}},\{0\}_{g\in U_{L^2_m}}}\text{ and}$$
$$\bra{\{Cl_\tau(V_x)\}_{x\in \widetilde{M}},\{1\}_{x\in \widetilde{M}},\bbra{\vep^{-1}\Theta_x}_{x\in \widetilde{M}}},$$
where $V_x$ is the $\vep$-neighborhood of $x$ in $\widetilde{M}$. Then, $[\widetilde{\Theta_{\ca{M}_{L^2_k},2}}']$ is homotopic to $[\widetilde{\Theta_{\ca{M}_{L^2_k},2}}]$. 
\end{lem}
\begin{proof}

We define $[\widetilde{\Theta_{\ca{M}_{L^2_k},2}}'']\in \ca{R}KK_{LT_{L^2_m}}(\ca{M}_{LT_{L^2_m}};\ca{S}_\vep\grotimes\ca{S}_\vep\grotimes \scr{C}(\ca{M}_{L^2_m}),\ca{A}(U_{L^2_m})\grotimes\ca{A}(\widetilde{M})\grotimes \scr{C}(\ca{M}_{L^2_m}))$ by the exterior tensor product of the following Kasparov modules:
$$\bra{\{\ca{A}(U_{L^2_k})\}_{g\in U_{L^2_k}},\bbra{\beta_g^{U_{L^2_k}}}_{g\in U_{L^2_k}},\{0\}_{g\in U_{L^2_k}}}\text{ and}$$
$$\bra{\{\ca{S}_\vep\grotimes Cl_\tau(V_x)\}_{x\in \widetilde{M}},\{1_{\ca{S}_\vep}\grotimes \id\}_{x\in \widetilde{M}},\bbra{1_{\ca{S}_\vep}\grotimes \vep^{-1}\Theta_x}_{x\in \widetilde{M}}},$$
where $1_{\ca{S}_\vep}\grotimes \id(f)(g\grotimes \phi):=fg\grotimes \phi$ for $f,g\in \ca{S}_\vep$ and $\phi\in Cl_\tau(\widetilde{M})$.

We define $\ev_0\grotimes \id_{Cl_\tau(\widetilde{M})}:\ca{A}(\widetilde{M})\ni f\grotimes h\mapsto f(0)h\in Cl_\tau(\widetilde{M})$. It gives a $KK$-element $[\ev_0\grotimes \id]\in KK_{T\times\Pi_T}(\ca{A}(\widetilde{M}), Cl_\tau(\widetilde{M}))$.  It suffices to prove the following formulas in the $\ca{R}KK$-group:
\renewcommand{\labelenumi}{(\arabic{enumi})}
\begin{enumerate}
\item $[\widetilde{\Theta_{\ca{M}_{L^2_k},2}}']=[\Delta]\grotimes [\widetilde{\Theta_{\ca{M}_{L^2_k},2}}'']\grotimes
[\ev_0\grotimes \id_{Cl_\tau(\widetilde{M})}]$.
\item $[\widetilde{\Theta_{\ca{M}_{L^2_k},2}}]=[\Delta]\grotimes [\widetilde{\Theta_{\ca{M}_{L^2_k},2}}'']\grotimes
[\ev_0\grotimes \id_{Cl_\tau(\widetilde{M})}]$.
\end{enumerate}

$(1)$ By Lemma \ref{lemma reformulated Theta 2 is in fact Theta 1},
$[\widetilde{\Theta_{\widetilde{M},2}}]\grotimes [\ev_0\grotimes \id ]=[\ev_0]\grotimes_\bb{C}[ \Theta_{\widetilde{M},1}].$
Thus, $[\widetilde{\Theta_{\ca{M}_{L^2_k},2}}'']\grotimes [\ev_0\grotimes \id_{Cl_\tau(\widetilde{M})}]$ is given by
\begin{align*}
&[\widetilde{\Theta_{\ca{M}_{L^2_k},2}}'']\grotimes [\ev_0\grotimes \id_{Cl_\tau(\widetilde{M})}]=[\widetilde{\Theta_{U_{L^2_k},2}}]\grotimes_\bb{C} \bra{[\ev_0]\grotimes_\bb{C}[ \Theta_{\widetilde{M},1}]} \\
&\ \ \ =\bra{\{\ca{A}(U_{L^2_k})\grotimes Cl_\tau(V_x)\}_{(g,x)\in U_{L^2_k}\times \widetilde{M}},
\bbra{\beta_g^{U_{L^2_k}}\grotimes \ev_0\grotimes \id_{Cl_\tau(\widetilde{M})}}_{(g,x)\in U_{L^2_k}\times \widetilde{M}}
,\{\id\grotimes \vep^{-1}\Theta_x\}_{(g,x)\in U_{L^2_k}\times \widetilde{M}}}.
\end{align*}

Let us compute the triple Kasparov product. First, we notice that $\beta_g^{U_{L^2_k}}\grotimes \ev_0\grotimes \id_{Cl_\tau(\widetilde{M})}\circ \Delta(f)$ is given by $\beta_g^{U_{L^2_k}}(f)\grotimes \id_{Cl_\tau(\widetilde{M})}$. In fact, for $\Delta(f)=\sum_if_i^1\grotimes f_i^2$,
\begin{align*}
\beta_g^{U_{L^2_k}}\grotimes \ev_0\grotimes \id_{Cl_\tau(\widetilde{M})}\circ \Delta(f) 
&=\sum_i\beta_g^{U_{L^2_k}}(f_i^1)\grotimes \ev_0(f_i^2) \id_{Cl_\tau(\widetilde{M})}\\
&=\beta_g^{U_{L^2_k}}\bra{\sum_i f_i^1\ev_0(f_i^2)}\grotimes \id_{Cl_\tau(\widetilde{M})}\\
&= \beta_g^{U_{L^2_k}}(f)\grotimes \id_{Cl_\tau(\widetilde{M})},
\end{align*}
where we have used the formula $(\id\grotimes \ev_0)\circ \Delta=\id_{\ca{S}_\vep}$ at the second equality.
Therefore, the triple Kasparov product is given by $ [\widetilde{\Theta_{\ca{M}_{L^2_k},2}}']$.

$(2)$ By the proof of Proposition \ref{Prop reformulated local Bott element}, $\bra{\{\ca{S}_\vep\grotimes Cl_\tau(V_x)\}_{x\in \widetilde{M}},\{1_{\ca{S}_\vep}\grotimes \id\}_{x\in \widetilde{M}},\bbra{1_{\ca{S}_\vep}\grotimes \vep^{-1}\Theta_x}_{x\in \widetilde{M}}}$ is homotopic to $\bra{\{\ca{A}(V_x)\}_{x\in \widetilde{M}},\{\beta^{\widetilde{M}}_x\}_{x\in \widetilde{M}},\bbra{0}_{x\in \widetilde{M}}}$. Thus, the Kasparov product 
$$\bra{\{\ca{S}_\vep\grotimes Cl_\tau(V_x)\}_{x\in \widetilde{M}},\{1_{\ca{S}_\vep}\grotimes \id\}_{x\in \widetilde{M}},\bbra{1_{\ca{S}_\vep}\grotimes \vep^{-1}\Theta_x}_{x\in \widetilde{M}}}\grotimes [\ev_0\grotimes \id ]$$
 is represented by
$$\bra{\bbra{Cl_\tau(V_x)}_{x\in \widetilde{M}},
\bbra{(\ev_0\grotimes \id)\circ \beta_x^{\widetilde{M}}}_{x\in \widetilde{M}},
\{0\}_{x\in \widetilde{M}}}.$$
It is easy to see that $(\ev_0\grotimes \id)\circ \beta_x^{\widetilde{M}}:\ca{S}_\vep\to Cl_\tau(V_x)$ is given by $f\mapsto [y\mapsto f(C_x(y))]$. We denote it by $\beta_x^{\widetilde{M},0}$.

Thus, the triple Kasparov product $[\Delta]\grotimes [\widetilde{\Theta_{\ca{M}_{L^2_k},2}}'']\grotimes [\ev_0\grotimes \id ]$ is given by
$$\bra{\bbra{\ca{A}(U_{L^2_k})\grotimes Cl_\tau(V_x)}_{(g,x)\in U_{L^2_m}\times\widetilde{M}},
\bbra{\bra{\beta_g^{U_{L^2_k}}\grotimes \beta_x^{\widetilde{M},0}}\circ \Delta}_{(g,x)\in U_{L^2_m}\times\widetilde{M}},
\{0\}_{(g,x)\in U_{L^2_m}\times\widetilde{M}}}.$$
We would like to ``apply'' Lemma \ref{Lem technical lemma on S and coprod} $(2)$ to $\beta_g^{U_{L^2_k}}\grotimes \beta_x^{\widetilde{M},0}\circ \Delta:\ca{S}_\vep\to \ca{S}_\vep\grotimes \ca{S}_\vep\to \ca{A}(U_{L^2_k})\grotimes Cl_\tau(V_x)$. For this aim, we extend this homomorphism to $\ca{S}$ by embedding $Cl_\tau(V_x)$ into another $C^*$-algebra. Let $W_x$ be the $2\vep$-ball of $x$ equipped with the new Riemannian metric $ds_{\rm New}^2$ given by (the old one is denoted by $ds_{\rm Old}^2$)
$$ds_{\rm New}^2(y):= \rho (r(x,y))ds_{\rm Old}^2(y),$$
where $\rho$ is a smooth function $:[0,2\vep)\to \bb{R}_{> 0}$ such that $\rho'(s)\geq 0$ on $s\in [0,2\vep)$, $\rho(s)=1$ on $s\in[0,\vep]$ and $\lim_{s\to 2\vep}\int_{0}^s\sqrt{\rho(s)}ds=\infty$. Then, the new metric is complete, and the Clifford operator on $W_x$ with respect to the new metric is with compact resolvent. Thus, we can define a $*$-homomorphism $\ca{S}\to Cl_\tau(W_x)$.
Since the $\vep$-ball at $x$ in $W_x$ is isometric to $V_x$, we have the following commutative diagram:
$$\xymatrix{
\ca{S}_\vep \ar^-{\Delta}[r] \ar_-{\iota}[d] &
\ca{S}_\vep\grotimes \ca{S}_\vep \ar^-{\beta_g^{U_{L^2_k}}\grotimes \beta_x^{\widetilde{M},0}}[rr] \ar_-{\iota\grotimes \iota}[d]&&
\ca{A}(U_{L^2_k})\grotimes Cl_\tau(V_x) \ar_-{(\iota\grotimes \id)\grotimes \iota}[d]\\
\ca{S} \ar_-{\Delta}[r] &
\ca{S}\grotimes \ca{S} \ar_-{\beta_g^{U_{L^2_k}}\grotimes \beta_x^{W_x,0}}[rr] &&
\ca{A}_\HKT(U_{L^2_k})\grotimes Cl_\tau(W_x),
}$$
where each zero-extension is denoted by $\iota$. Thanks to the following commutative diagram verified by Lemma \ref{Lem technical lemma on S and coprod} $(2)$ and the fact that $C^{U_{L^2_k}}_g\grotimes \id+\id\grotimes C_x^{W_x}= C_{(g,x)}^{\ca{M}_{L^2_k}}$ on $V_x$, we finish the proof:
$$
\xymatrix{
\ca{S} \ar[r]^-{\Delta}\ar_\id[d] & 
\ca{S}\grotimes \ca{S} \ar^-{\beta_g^{U_{L^2_k}}\grotimes \beta_x^{W_x,0}}[rr] &&
\ca{A}_\HKT(U_{L^2_k})\grotimes Cl_\tau(W_x)\ar^{\cong}[d]\\
\ca{S} \ar^-{ \beta^{U_{L^2_k}\times W_x}_{(g,x)}}[rrr]& && 
\ca{A}_\HKT(U_{L^2_k}\times W_x) \\
\ca{S}_\vep\ar^\iota[u]\ar_\id[d] \ar_-{ \beta^{\ca{M}_{L^2_k}}_{(g,x)}}[rrr] &&& 
\ca{A}(U_{L^2_k}\times V_x)\ar_{(\iota\grotimes \id)\grotimes \iota}[u] \ar^{ \id\grotimes \iota}[d]\\
\ca{S}_\vep\ar_-{ \beta^{\ca{M}_{L^2_k}}_{(g,x)}}[rrr] &&& \ca{A}(\ca{M}_{L^2_k}).
}$$

\end{proof}

\begin{rmk}
Since  the Kasparov module $\bra{\{Cl_\tau(V_x)\}_{x\in \widetilde{M}},\{1\}_{x\in \widetilde{M}},\bbra{\vep^{-1}\Theta_x}_{x\in \widetilde{M}}}$ represents\\
 $[\Theta_{X,2}]\grotimes \sigma_{C_0(X)}^1\bra{[S^*]}\grotimes\sigma_{X,C_0(X)}^2\bra{\fgt[S]}$, we write 
the above result as 
$$[\widetilde{\Theta_{\ca{M}_{L^2_k},2}}]=[\widetilde{\Theta_{U_{L^2_k},2}}]\grotimes_\bb{C}\bra{[\Theta_{\widetilde{M},2}]\grotimes \sigma_{C_0(\widetilde{M})}^1\bra{[S^*]}\grotimes\sigma_{\widetilde{M},C_0(\widetilde{M})}^2\bra{\fgt[S]}}.$$
\end{rmk}

Note that the bounded transformation $D\mapsto \frac{D}{\sqrt{1+D^2}}$ commutes with the group action $D\mapsto g(D)=L_{(g,z)}\circ D\circ L_{(-g,z^{-1})}$. By Corollary \ref{Cor positivity of the commutator implies the homotopy invariance}, $\sigma_{U_{L^2_m},\scr{C}(U_{L^2_m})}([\widetilde{\Dirac}])$ is represented by
$$\bra{
\{\ca{S}_\vep\grotimes \ca{H}\}_{g\in U_{L^2_m}},\{\pi\}_{g\in U_{L^2_m}},\{g(\Dirac)\}_{g\in U_{L^2_m}}}.$$

With the preparation so far, we can prove the main theorem of this section.

\begin{thm}\label{Poincare duality for LT manifold}
$\PD([\widetilde{\ca{D}}])=\sigma_{\ca{S}_\vep}([\ca{L}])$.
\end{thm}
\begin{proof}
Thanks to $[\widetilde{\ca{D}}]=[\widetilde{\Dirac}]\grotimes_{\bb{C}}[\widetilde{D_{\widetilde{M}}}]$ and Lemma \ref{lemma local Bott element can be factorized}, we can divide the problem into the $U_{L^2_m}$-part and the $\widetilde{M}$-part. It suffices to check that
$${[\widetilde{\Theta_{U_{L^2_k},2}}]\grotimes_{\ca{A}(U_{L^2_k})}[\widetilde{\Dirac}]}=
\sigma_{\ca{S}_\vep}\bra{[\ca{L}|_{U_{L^2_k}}]},\text{ and}$$
$$\bra{[\Theta_{\widetilde{M},2}]\grotimes \sigma_{C_0(\widetilde{M})}^1\bra{[S^*]}\grotimes\sigma_{\widetilde{M},C_0(\widetilde{M})}^2\bra{\fgt[S]}}\grotimes_{Cl_\tau(\widetilde{M})}\bra{[\widetilde{D_{\widetilde{M}}}]}=[\ca{L}|_{\widetilde{M}}].$$

For the second one, see Proposition \ref{prop index theorem for Spinc} and Example \ref{ex computation of PD}.

Let us prove the first one. Since we would like to discuss it in the unbounded picture, we rewrite the statement using $\ca{A}_\HKT(U_{L^2_k})$.
We introduce and recall several notations:
\begin{itemize}
\item The $KK$-element corresponding to the Kasparov $(\ca{A}_\HKT(U_{L^2_k}),\ca{S})$-module constructed in Lemma \ref{index element for HKT algebra} is denoted by $[\widetilde{\Dirac}_\HKT]\in KK_{U_{L^2_m}}(\ca{A}_\HKT(U_{L^2_k}),\ca{S})$ in this proof.
\item The canonical embedding $\ca{A}({U_{L^2_k}})\to \ca{A}_\HKT(U_{L^2_k})$ is denoted by $\iota\grotimes \id$.
\item The canonical embedding $\ca{S}_\vep \to \ca{S}$ is denoted by $j$, in order to distinguish it from $\iota$.
\end{itemize}
Then, it is obvious that $[\widetilde{\Dirac}]\grotimes [j]=[\iota\grotimes \id]\grotimes [\widetilde{\Dirac}_\HKT]$. Since $[j]$ is invertible in $KK$-group, we have $[\widetilde{\Dirac}]=[\iota\grotimes \id]\grotimes [\widetilde{\Dirac}_\HKT] \grotimes [j]^{-1}$. Associated to it, we consider 
$$[\widetilde{\Theta}_{U_{L^2_k},2,\HKT}]:= [j]^{-1}\grotimes [\widetilde{\Theta_{U_{L^2_k},2}}]\grotimes [\iota\grotimes \id]\in \ca{R}KK_{U_{L^2_m}}(U_{L^2_m};\ca{S}\grotimes \scr{C}(U_{L^2_m}),\ca{A}_{\HKT}(U_{L^2_k})\grotimes \scr{C}(U_{L^2_m})).$$
Then, we have
$$[\widetilde{\Theta_{U_{L^2_k},2}}]\grotimes [\widetilde{\Dirac}]
=[j]\grotimes [\widetilde{\Theta}_{U_{L^2_k},2,\HKT}]\grotimes[\widetilde{\Dirac}_\HKT]\grotimes  [j]^{-1}.$$
Thus, it suffices to prove that $[\widetilde{\Theta}_{U_{L^2_k},2,\HKT}]\grotimes[\widetilde{\Dirac}_\HKT]=\sigma_{\ca{S}}([\ca{L}|_{U_{L^2_m}}])$.

Note that $[\widetilde{\Theta}_{U_{L^2_k},2,\HKT}]$ is represented by
$$\bra{\{\ca{A}_\HKT(U_{L^2_k})\}_{g\in U_{L^2_m}},\bbra{\beta_g}_{g\in U_{L^2_m}},\{0\}_{g\in U_{L^2_m}}}.$$
In fact, $[j]\grotimes \bra{\{\ca{A}_\HKT(U_{L^2_k})\}_{g\in U_{L^2_m}},\bbra{\beta_g}_{g\in U_{L^2_m}},\{0\}_{g\in U_{L^2_m}}}$ is clearly $[\widetilde{\Theta_{U_{L^2_k},2}}]\grotimes [\iota\grotimes \id]$. Thus, the Kasparov product $[\widetilde{\Theta}_{U_{L^2_k},2,\HKT}]\grotimes\ [\widetilde{\Dirac}_\HKT]$ is given by the field of Kasparov modules
$$\bra{\{\ca{S}\grotimes \ca{H}\}_{ g\in U_{L^2_m}},\{\pi\circ\beta_{g}\}_{ g\in U_{L^2_m}},\{g(\id\grotimes\Dirac)\}_{ g\in U_{L^2_m}}}.$$

We prove this $\ca{R}KK$-element is $\sigma_{\ca{S}}\bra{[\ca{L}|_{U_{L^2_m}}]}$. This is equivalent to the following by Lemma \ref{Lemma computation of KK(S,S)}:
$$[b_\pm]\grotimes \bra{\{\ca{S}_\vep\grotimes \ca{H}\}_{ g\in U_{L^2_m}},\{\pi\circ\beta_{g}\}_{ g\in U_{L^2_m}},\{g(\id\grotimes\Dirac)\}_{ g\in U_{L^2_m}}}\grotimes [d_\pm]=
[\ca{L}|_{U_{L^2_m}}],$$
$$[b_\pm]\grotimes \bra{\{\ca{S}_\vep\grotimes \ca{H}\}_{ g\in U_{L^2_m}},\{\pi\circ\beta_{g}\}_{ g\in U_{L^2_m}},\{g(\id\grotimes\Dirac)\}_{ g\in U_{L^2_m}}}\grotimes [d_\mp]=0$$
(double signs are in the same order). 

We prove only $[b_+]\grotimes \bra{\cdots }\grotimes [d_+]=
[\ca{L}|_{U_{L^2_m}}]$. By a direct computation,
$$\bra{\cdots }\grotimes [d_+]=\bra{\{L^2(\bb{R})_{\rm gr}\grotimes \ca{H}\}_{ g\in U_{L^2_m}},\{f\mapsto f(X\grotimes \id+\id\grotimes C_g)\}_{ g\in U_{L^2_m}},\{\epsilon d\grotimes \id+\id\grotimes g(\Dirac)\}_{ g\in U_{L^2_m}}},$$
where $X\phi(t):=t\phi(t)$ for $\phi \in L^2(\bb{R})_{\rm gr}$.
Next, $\ca{S}\grotimes_{\ca{S}} \bbra{L^2(\bb{R})_{\rm gr}\grotimes \ca{H}}$ is isomorphic to $L^2(\bb{R})_{\rm gr}\grotimes \ca{H}$, by $f\grotimes \phi\grotimes \psi\mapsto f(X\grotimes \id+\id\grotimes C_g)\phi\grotimes \psi$. Under this isomorphism, $X\grotimes_{\ca{S}} \id_{L^2(\bb{R})_{\rm gr}\grotimes \ca{H}}$ is given by $X\grotimes \id+\id\grotimes C_g$. Noticing it, we find that the triple Kasparov product is given by 
$$\bra{\bbra{L^2(\bb{R})_{\rm gr}\grotimes \ca{H}}_{g\in U_{L^2_m}},
\bbra{1}_{g\in U_{L^2_m}},
\bbra{(\epsilon d+X)\grotimes \id+\id\grotimes g(\Dirac)}_{g\in U_{L^2_m}}}.$$
In fact, it is obvious that $\{(\epsilon d+X)\grotimes \id+\id\grotimes g(\Dirac)\}$ is an $(\epsilon d\grotimes \id+\id\grotimes g(\Dirac))$-connection for $\ca{S}$. Moreover, 
$$[X\grotimes \id+\id\grotimes C_g,(\epsilon d+X)\grotimes \id+\id\grotimes g(\Dirac)]=(2X^2+\epsilon)\grotimes \id+g\bra{\id\grotimes \sum_nn^{-l/4}\{c(e_n)c^*(e_n)+c(f_n)c^*(f_n)\}}.$$
The first term is positive modulo bounded, and the second term is bounded. The essential point of this computation is, roughly speaking, that the ``potential term'' of $\Dirac$ commutes with $C$, which is because the vector $x_nJe_n+y_nJf_n$ (the potential of the $n$-th summand of $\Dirac$) is orthogonal to $x_ne_n+y_nf_n$ (the $n$-th summand of $C$)  at each point. This formal argument is justified as follows. The positivity condition to be a Kasparov product is, strictly speaking, the following: The quadratic form
{\small
$$\psi\grotimes \phi\mapsto \inpr{(X\grotimes\id+\id\grotimes C_g)(\psi\grotimes \phi)}{\id\grotimes g(\Dirac) (\psi\grotimes \phi)}{}+\inpr{\id\grotimes g(\Dirac) (\psi\grotimes \phi)}{(X\grotimes\id+\id\grotimes C_g)(\psi\grotimes \phi)}{}$$}
is positive modulo bounded. Since $C_g=g\circ C\circ g^{-1}$, $g(\Dirac)=g\circ \Dirac\circ g^{-1}$ and the $U_{L^2_m}$-action is isometric, it is given by
{\small
$$\lim_{N\to \infty}\bbra{\inpr{(X\grotimes\id+\id\grotimes C_0^N)(\psi\grotimes g(\phi))}{\id\grotimes \Dirac (\psi\grotimes g\cdot\phi)}{}+\inpr{\id\grotimes \Dirac (\psi\grotimes g\cdot\phi)}{(X\grotimes\id+\id\grotimes C_0^N)(\psi\grotimes g\cdot\phi)}{}},$$
}
thanks to the definition of $C$. In this expression,  we can do the above formal argument.

The Kasparov module $\bra{\bbra{L^2(\bb{R})_{\rm gr}\grotimes \ca{H}}_{g\in U_{L^2_m}},
\bbra{1}_{g\in U_{L^2_m}},
\bbra{(\epsilon d+X)\grotimes \id+\id\grotimes g(\Dirac)}_{g\in U_{L^2_m}}}$ is homotopic to 
$$\bra{\bbra{\ker(\epsilon d+X)\grotimes \ker(g(\Dirac))
}_{g\in U_{L^2_m}},
\bbra{1}_{g\in U_{L^2_m}},
\bbra{0}_{g\in U_{L^2_m}}}$$
by the family version of the argument to prove that $KK(\bb{C},\bb{C})\cong \bb{Z}$. Since $\ker(\epsilon d+X)\grotimes \ker(g(\Dirac))\cong \bb{C}$ and the $U_{L^2_m}^\tau$-action is at level $1$, it is $[\ca{L}|_{U_{L^2_m}}]$.

\end{proof}

\section{Topological assembly map for proper $LT$-spaces}\label{section top ass map}

The aim of this section is to construct an infinite-dimensional version of the topological assembly map for proper $LT$-spaces, and to compute it. The main result of this section is the following: {\it The value of the topological assembly map at the value of the Poincar\'e duality of the index element, coincides with the analytic index constructed in \cite{Thesis}}. This is an analogous result of Proposition \ref{prop index theorem for Spinc}. In addition, as a concluding remark, we will explain what we should do after the present paper.

\subsection{Crossed products and descent homomorphisms for proper $LT$-spaces}\label{section top ass map descent}

We have explained the description of crossed products and the descent homomorphism for $\ca{R}KK_G^\tau$-theory in terms of fields of Hilbert modules in Section \ref{section index theorem top ass map}. Imitating it, we define substitutes for the descent homomorphisms for proper $LT$-spaces. 
For this aim, we need to introduce a substitute for ``$L^2(U_{L^2_k},\ca{L}^{\otimes q})$'' for each $q\in \bb{Z}$.

\begin{dfn}\label{L2LT}
$(1)$ For $q\in\bb{Z}$, we define a new $U_{L^2_m}^\tau$-action $R$ at level $-q$ on $\ud{L^2(U_{L^2_k},\ca{L})}$ by 
$$R_{(g,z)}\phi(x):=z^{-q} \phi(x+g)\tau(g,x)^q$$
for $g\in U_{\fin}$ and $\phi \in \ud{L^2(U_{L^2_k},\ca{L})_\fin}$. This representation extends to $U_{L^2_m}^\tau$ as a continuous homomorphism.
This $-q\tau$-twisted $U_{L^2_m}$-representation space is denoted by  $\ud{L^2(U_{L^2_k},\ca{L}^{\otimes q})}$. 

$(2)$ The $-q\tau$-twisted $T\times\Pi_T$-action on $L^2(T\times \Pi_T,\ca{L}^{\otimes q})$ induced by the right regular representation of $(T\times \Pi_T)^\tau$, is denoted by $R$.

$(3)$ The Hilbert space $\ud{L^2(LT_{L^2_k},\ca{L}^{\otimes q})}$ is defined by the tensor product
$$\ud{L^2(LT_{L^2_k},\ca{L}^{\otimes q})}:=\ud{L^2(U_{L^2_k},\ca{L}^{\otimes q})}\grotimes L^2(T\times \Pi_T,\ca{L}^{\otimes q}),$$
and it is equipped with a $-q\tau$-twisted representation $R$ of  $LT_{L^2_m}^\tau=U_{L^2_m}^\tau\boxtimes_{U(1)} (T\times\Pi_T)^\tau$ defined by $R_{(g_1,g_2,z)}\phi_1\grotimes \phi_2:=R_{(g_1,z)}\phi_1\grotimes R_{(g_2,1)}\phi_2$ for $(g_1,g_2,z)\in LT_{L^2_m}^\tau$, $\phi_1\in \ud{L^2(U_{L^2_k},\ca{L}^{\otimes q})}$ and $\phi_2\in L^2(T\times \Pi_T,\ca{L}^{\otimes q})$.
\end{dfn}


\begin{dfn-pro}\label{descent and crossed product for LT theory}
$(1)$ Let $\scr{A}=(\{A_x\}_{x\in \ca{M}_{L^2_m}},\Gamma_{\scr{A}})$ be an $\ca{M}_{L^2_m}\rtimes LT_{L^2_m}$-equivariant u.s.c. field of $C^*$-algebras. 
{\it Then, the set of invariant sections
$$C\bra{\widetilde{M}\times_{ T\times \Pi_T}\bbra{\bb{K}\bra{\ud{L^2(LT_{L^2_k},\ca{L}^{\otimes q})}}\grotimes A_x}_{x\in \widetilde{M}}}$$
for the $T\times\Pi_T$-action defined by the restriction of the $LT_{L^2_m}$-action $\Ad R\grotimes \alpha^{\scr{A}}$, is an $({M}/T)\rtimes\{e\}$-equivariant $C^*$-algebra}.
We regard the constructed $C^*$-algebra as the ``{\bf $q\tau$-twisted crossed product of $\scr{A}$ by $LT_{L^2_m}$}'' and we denote it by $\ud{\scr{A}\rtimes_{q\tau} LT_{L^2_m}}$.

$(2)$ Let $\scr{B}=(\{B_x\}_{x\in \ca{M}_{L^2_m}},\Gamma_{\scr{B}})$ be an $\ca{M}_{L^2_m}\rtimes LT_{L^2_m}$-equivariant u.s.c. field of $C^*$-algebras, and let $\scr{E}=(\{E_x\}_{x\in \ca{M}_{L^2_m}},\Gamma_{\scr{E}})$ be a $p\tau$-twisted $\ca{M}_{L^2_m}\rtimes LT_{L^2_m}$-equivariant u.s.c. field of Hilbert $\scr{B}$-modules.
{\it Then, the set of invariant sections
$$C\bra{\widetilde{M}\times_{ T\times \Pi_T}\bbra{\bb{K}\bra{\ud{L^2(LT_{L^2_k},\ca{L}^{\otimes q})},\ud{L^2(LT_{L^2_k},\ca{L}^{\otimes (q-p)})}}\grotimes E_x}_{x\in \widetilde{M}}}$$
for the $T\times\Pi_T$-action defined by the restriction of the $LT_{L^2_m}$-action $\Ad R\grotimes \alpha^{\scr{E}}$ (this action is untwisted. See the exposition before Proposition \ref{twisted cp and fpa and cp and fpm}), is a Hilbert $\ud{\scr{B}\rtimes_{q\tau} LT_{L^2_m}}$-module.}
We regard the constructed Hilbert module as the ``{\bf $q\tau$-twisted crossed product of $\scr{E}$ by $LT_{L^2_m}$}'', and we denote it by $\ud{\scr{E}\rtimes_{q\tau} LT_{L^2_m}}$.

$(3)$ Let $\bra{\scr{E},\pi,F}$ be a $p\tau$-twisted $\ca{M}_{L^2_m}\rtimes LT_{L^2_m}$-equivariant Kasparov $(\scr{A,B})$-module. We define the $*$-homomorphism 
$$\id\grotimes \pi:\ud{\scr{A}\rtimes_{(q-p)\tau} LT_{L^2_m}}\to \scr{L}_{\ud{\scr{B}\rtimes_{q\tau} LT_{L^2_m}}}(\ud{\scr{E}\rtimes_{q\tau} LT_{L^2_m}})$$
by $\id\grotimes \pi_x(k_1\grotimes a_x)(k_2\grotimes e_x):=k_1\circ k_2\grotimes \pi_x(a_x)(e_x)$ and the operator
$$\id\grotimes F\in \bb{L}_{\ud{\scr{B}\rtimes_{q\tau} LT_{L^2_m}}}(\ud{\scr{E}\rtimes_{q\tau} LT_{L^2_m}})$$
by $\id\grotimes F_x(k_2\grotimes e_x):=k_2\grotimes F_x(e_x)$, for $k_2\in \bb{K}\bra{\ud{L^2(LT_{L^2_k},\ca{L}^{\otimes q})},\ud{L^2(LT_{L^2_k},\ca{L}^{\otimes (q-p)})}}$, $k_1\in \bb{K}\bra{\ud{L^2(LT_{L^2_k},\ca{L}^{\otimes q})}}$, $a_x\in A_x$ and $e_x\in E_x$.
{\it Then, the triple 
$$\bra{\ud{\scr{E}\rtimes_{q\tau} LT_{L^2_m}},\{\id\grotimes \pi_x\}_{x\in \ca{M}_{L^2_m}},\{\id\grotimes F_x\}_{x\in \ca{M}_{L^2_m}}}$$
is an $({M}/T)\rtimes\{e\}$-equivariant Kasparov $(\ud{\scr{A}\rtimes_{(q-p)\tau} LT_{L^2_m}},\ud{\scr{B}\rtimes_{q\tau} LT_{L^2_m}})$-module.} It is denoted by $\ud{j^{q\tau}_{LT_{L^2_m}}}\bra{\scr{E},\pi,F}$ and the correspondence $(\scr{E},\pi,F)\mapsto \ud{j^{q\tau}_{LT_{L^2_m}}}\bra{\scr{E},\pi,F}$ is called the {\bf partial descent homomorphism}.
\end{dfn-pro}
\begin{proof}
$(1)$ This is obvious from the fact that 
\begin{align*}
&C\bra{\widetilde{M}\times_{ T\times \Pi_T}\bbra{\bb{K}\bra{\ud{L^2(LT_{L^2_k},\ca{L}^{\otimes q})}}\grotimes A_x}_{x\in \widetilde{M}}}\\
&\cong C\bra{\widetilde{M}\times_{ T\times \Pi_T}\bbra{\bb{K}\bra{L^2(T\times\Pi_T,\ca{L}^{\otimes q})}\grotimes A_x}_{x\in \widetilde{M}}}\grotimes \bb{K}\bra{\ud{L^2(LT_{L^2_k},\ca{L}^{\otimes q})}} \\
&\cong C_0(\widetilde{M})\rtimes_{q\tau}(T\times\Pi_T)\grotimes \bb{K}\bra{\ud{L^2(LT_{L^2_k},\ca{L}^{\otimes q})}}.
\end{align*}

$(2)$ One can prove it in the same way.

$(3)$ For $a\in \Gamma_{\scr{A}}$, $k\in
\bb{K}\bra{\ud{L^2(LT_{L^2_k},\ca{L}^{\otimes (q-p)})}}$, $e\in \Gamma_{\scr{E}}$ and $k'\in \bb{K}\bra{\ud{L^2(LT_{L^2_k},\ca{L}^{\otimes q})},\ud{L^2(LT_{L^2_k},\ca{L}^{\otimes (q-p)})}}$,
\begin{align*}
[\id\grotimes \pi_x(k\grotimes a_x),\id\grotimes F_x](k'\grotimes e)&=k\circ k'\grotimes [\pi_x(a_x),F_x](e_x) \\
&=(k\circ\star) \grotimes [\pi_x(a_x),F_x](k'\grotimes e_x),
\end{align*}
where $k\circ\star$ means the operator $k'\mapsto k\circ k'$.
The second component $x\mapsto [\pi_x(a_x),F_x]$ is an element of $\Gamma_{\scr{K}(\scr{E})}$ because $(\scr{E},\{\pi_x\},\{F_x\})$ is a Kasparov $(\scr{A,B})$-module. The first one 
$$k\circ\star:\bb{K}\bra{\ud{L^2(LT_{L^2_k},\ca{L}^{\otimes q})},\ud{L^2(LT_{L^2_k},\ca{L}^{\otimes (q-p)})}}\to \bb{K}\bra{\ud{L^2(LT_{L^2_k},\ca{L}^{\otimes q})},\ud{L^2(LT_{L^2_k},\ca{L}^{\otimes (q-p)})}}$$ 
is $\bb{K}\bra{\ud{L^2(LT_{L^2_k},\ca{L}^{\otimes q})}}$-compact as proved in the next paragraph. Consequently, the section $x\mapsto [\id\grotimes \pi_x(k\grotimes a),\id\grotimes F_x]$ is a $\ud{\scr{B}\rtimes LT_{L^2_m}}$-compact operator.

What we need to prove is that $k\circ\star$ is a $\bb{K}\bra{\ud{L^2(LT_{L^2_k},\ca{L}^{\otimes q})}}$-compact operator on\\
$\bb{K}\bra{\ud{L^2(LT_{L^2_k},\ca{L}^{\otimes q})},\ud{L^2(LT_{L^2_k},\ca{L}^{\otimes (q-p)})}}$. For this aim, we may assume that $k$ is a single Schatten form $\phi\grotimes\psi^*$ as an operator on the Hilbert space. Then, the operator $k\circ \star$ can be written as the composition of the following: For a unit vector $\lambda\in \ud{L^2(LT_{L^2_k},\ca{L}^{\otimes q})}$,
$$\xymatrix{
\bb{K}\bra{\ud{L^2(LT_{L^2_k},\ca{L}^{\otimes q})},\ud{L^2(LT_{L^2_k},\ca{L}^{\otimes (q-p)})}} \ar^-{(\psi\grotimes \lambda^*)^*}[d] \\
\bb{K}\bra{\ud{L^2(LT_{L^2_k},\ca{L}^{\otimes q})}} \ar^-{\phi\grotimes \lambda^*}[d] \\
\bb{K}\bra{\ud{L^2(LT_{L^2_k},\ca{L}^{\otimes q})},\ud{L^2(LT_{L^2_k},\ca{L}^{\otimes (q-p)})}} .}$$
This is a single Schatten form on $\bb{K}\bra{\ud{L^2(LT_{L^2_k},\ca{L}^{\otimes q})},\ud{L^2(LT_{L^2_k},\ca{L}^{\otimes (q-p)})}}$ as the Hilbert $\bb{K}\bra{\ud{L^2(LT_{L^2_k},\ca{L}^{\otimes q})}}$-module. In particular, it is compact.

For the same reason, the section
$$x\mapsto \id\grotimes\pi_x(k\grotimes a_x)(1-\id\grotimes F_x^2)=(k\circ\star) \grotimes \pi_x(a_x)(1-F_x^2)$$
defines a $\ud{\scr{B}\rtimes_{q\tau} LT_{L^2_m}}$-compact operator.
\end{proof}

\begin{rmks}
$(1)$ Since we have assumed that the $LT_{L^2_m}$-action on $\ca{M}_{L^2_m}$ is co-compact, the crossed product is not $C_0(\cdots)$ but $C(\cdots)$. If the action were just ``co-locally compact'', the crossed product would become $C_0(\cdots)$ which is the set of continuous sections whose norms vanish at infinity in the quotient space $\ca{M}_{L^2_m}/LT_{L^2_m}$. Note that the norm function is defined on the quotient space because the norm is invariant under the $LT_{L^2_m}$-action.

$(2)$ $\ud{L^2(LT_{L^2_k},\ca{L}^{\otimes q})}$ and $\ud{L^2(LT_{L^2_k},\ca{L}^{\otimes (q-p)})}$ are trivially graded, and hence signs coming from the gradings do not appear.
\end{rmks}

We must prove that $\ud{j^{p\tau}_{LT_{L^2_m}}}$ makes sense at the level of $KK$-theory.

\begin{pro}\label{jLT is well defined}
Let $\scr{A}$ and $\scr{B}$ be $\ca{M}_{L^2_m}\rtimes LT_{L^2_m}$-equivariant u.s.c. fields of $C^*$-algebras. Then, the correspondence $\ud{j^{q\tau}_{LT_{L^2_m}}}$ is homotopy invariant and hence it defines a homomorphism 
$$\ca{R}KK_{LT_{L^2_m}}^{p\tau}(\ca{M}_{L^2_m};\scr{A},\scr{B})\to 
\ca{R}KK\bra{M/T;\ud{\scr{A}\rtimes_{(q-p)\tau} LT_{L^2_m}},\ud{\scr{B}\rtimes_{q\tau} LT_{L^2_m}}}.$$
\end{pro}

\begin{proof}
Let $\bra{\widetilde{\scr{E}},\{\widetilde{\pi}_x\}_{x\in \ca{M}_{L^2_m}},\{\widetilde{F}_x\}_{x\in \ca{M}_{L^2_m}}}$ be a homotopy between $\bra{\scr{E}_1,\pi_1,F_1}$ and $\bra{\scr{E}_2,\pi_2,F_2}$, which is a $p\tau$-twisted $\ca{M}_{L^2_m}\rtimes LT_{L^2_m}$-equivariant Kasparov $(\scr{A},\scr{B}\grotimes C(I))$-module. The Kasparov \\
$\bra{\ud{\scr{A}\rtimes_{(q-p)\tau} LT_{L^2_m}},\ud{\scr{B}\rtimes_{q\tau} LT_{L^2_m}}\grotimes C(I)}$-module 
$$\bra{\ud{\widetilde{\scr{E}}\rtimes_{q\tau} LT_{L^2_m}},\{\id\grotimes\widetilde{\pi}_x\},\{\id\grotimes\widetilde{F}_x\}}$$
gives a homotopy between $\ud{j_{LT_{L^2_m}}^{q\tau}}\bra{\scr{E}_1,\pi_1,F_1}$ and $\ud{j_{LT_{L^2_m}}^{q\tau}}\bra{\scr{E}_2,\pi_2,F_2}$. 
\end{proof}

Let us introduce an infinite-dimensional version of $[c_X]$, imitating Lemma \ref{lemma operator description of cX}.

Take a cut-off function $\fra{c}_{\widetilde{M}}:\widetilde{M}\to\bb{R}_{\geq 0}$ with respect to the $T\times \Pi_T$-action. It defines a function $\fra{c}_{\widetilde{M},x}:T\times \Pi_T\to \bb{R}_{\geq 0}$ by $\fra{c}_{\widetilde{M},x}(g):=\fra{c}_{\widetilde{M}}(g\cdot x)$. Note that $T\times\Pi_T$ is unimodular and the modular function does not appear here.

The Hilbert space $\ud{L^2(U_{L^2_k})}$ has a unit vector ``$\vac$''. We regard it as the ``square root of a cut-off function on $U_{L^2_k}$''. Recall that a cut-off function on a locally compact group with respect to left translation is just a non-negative-valued $L^1$-unit function.
The rank one projection onto the one-dimensional subspace $\bb{C}\vac$ is denoted by $P_\vac$.

\begin{dfn}
Under the identification 
$$\ud{\scr{C}(\ca{M}_{L^2_m})\rtimes LT_{L^2_m}}
\cong C\bra{\widetilde{M}\times_{T\times\Pi_T}\bbra{\bb{K}\bra{\ud{L^2(U_{L^2_k})}}\grotimes \bb{K}(L^2(T\times\Pi_T)}},$$
the {\bf Mishchenko line bundle} $\ud{\bbbra{c_{\ca{M}_{L^2_m}}}}\in KK\bra{\bb{C},\ud{\scr{C}(\ca{M}_{L^2_m})\rtimes LT_{L^2_m}}}$ is defined by the equivariant family of rank one projections
$$P:x\mapsto P_x:= P_\vac\grotimes P_{\sqrt{\fra{c}_{\widetilde{M},x}}}.$$
\end{dfn}

\begin{pro}\label{Mishchenko line byndle for LTmanifold}
$\ud{\bbbra{c_{\ca{M}_{L^2_m}}}}$ is represented by the following Kasparov module:
$$\bra{C_0\bra{\widetilde{M}\times_{T\times\Pi_T}\ud{L^2\bra{LT_{L^2_k}}}^*},1,0},$$
where the Hilbert module structure is given as follows: For equivariant sections $f,f_1,f_2: \widetilde{M}\to\ud{L^2(LT_{L^2_k})}^*$ and $b: \widetilde{M}\to\bb{K}\bra{\ud{L^2(LT_{L^2_k})}}$,
\begin{itemize}
\item $(f\cdot b)(x):=f(x)\circ b(x)$; and 
\item $\inpr{f_1}{f_2}{\ud{\scr{C}(\ca{M}_{L^2_m})\rtimes LT_{L^2_m}}}(x):=f_1(x)^*\circ f_2(x)$, where $f_1(x)^*$ is an element of $\bbbra{\ud{L^2\bra{LT_{L^2_k}}}^*}^*\cong \ud{L^2(LT_{L^2_k})}$ by the Riesz representation theorem.
\end{itemize}
\end{pro}
\begin{proof}
We construct an isomorphism 
$$\Phi:P\cdot\bbbra{ C_0\bra{\ca{M}_{L^2_m}\times_{LT_{L^2_m}}\bb{K}\bra{\ud{L^2(LT_{L^2_k})}}}} \to C_0\bra{\widetilde{M}\times_{T\times\Pi_T}\ud{L^2\bra{LT_{L^2_k}}}^*}$$
by the formula
$$\Phi(P\cdot f)(x):=[\vac\grotimes \sqrt{\fra{c}_{\widetilde{M},x}}]^*\circ f(x),$$
where $[\vac\grotimes \sqrt{\fra{c}_{\widetilde{M},x}}]^*$ is the Riesz dual of $\vac\grotimes \sqrt{\fra{c}_{\widetilde{M},x}}$.
We need to verify the following: $(1)$ $\Phi$ is well-defined; $(2)$ $\Phi$ is isometric right module homomorphism; 
and $(3)$ $\Phi$ is surjective.

$(1)$ We need to check that $\Phi(P\cdot a)$ is $T\times\Pi_T$-equivariant, that is to say,
$ \Phi(P\cdot a)(g^{-1}x)\circ R_g^{-1} = \Phi(P\cdot a)(x).$
By definition, $a$ satisfies the equivariant condition
$$ R_g\circ a(g^{-1}x)\circ R_g^{-1} = a(x).$$
Since $\Phi(P\cdot a)(x)=[\vac\grotimes \sqrt{\fra{c}_{\widetilde{M},x}}]^*\circ a(x)$, it is 
sufficient to verify that  $[\vac\grotimes \sqrt{\fra{c}_{\widetilde{M},x}}]^*$ is $T\times\Pi_T$-invariant.
Since ``$\vac$'' is $T\times\Pi_T$-invariant, it suffices to deal with the $\sqrt{\fra{c}_{\widetilde{M},x}}$-part.
We can prove the invariance as follows:
\begin{align*}
R_g\bra{\sqrt{\fra{c}_{\widetilde{M},g^{-1}\cdot x}}}(h)
&=\sqrt{\fra{c}_{\widetilde{M},g^{-1}\cdot x}}(hg) \\
&= \sqrt{\fra{c}_{\widetilde{M}}(hgg^{-1}\cdot x)} \\
&= \sqrt{\fra{c}_{\widetilde{M}}(hx)} \\
&= \sqrt{\fra{c}_{\widetilde{M},x}}(h).
\end{align*}

$(2)$ $\Phi$ is clearly a module homomorphism.
We check that it is isometric. For \\
$a_1,a_2\in C_0\bra{\ca{M}_{L^2_m}\times_{LT_{L^2_m}}\bb{K}\bra{\ud{L^2(LT_{L^2_k})}}}$, since $P=P^*=P^2$,
\begin{align*}
\inpr{P\cdot a_1}{P\cdot a_2}{\ud{\scr{C}(\ca{M}_{L^2_m})\rtimes LT_{L^2_m}}}(x)
&= [P\cdot a_1(x)]^*\circ P\cdot a_2(x) \\
&= a_1(x)^* \circ P\circ a_2(x)\\
&= a_1(x)^* \circ [\vac\grotimes \sqrt{\fra{c}_{\widetilde{M},x}}] \circ [\vac\grotimes \sqrt{\fra{c}_{\widetilde{M},x}}]^* \circ a_2(x) \\
&= \bra{[\vac\grotimes \sqrt{\fra{c}_{\widetilde{M},x}}]^* \circ a_1(x) }^*\circ \bra{[\vac\grotimes \sqrt{\fra{c}_{\widetilde{M},x}}]^* \circ a_2(x)} \\
&=\inpr{\Phi(P\cdot a_1)}{\Phi(P\cdot a_2)}{}(x).
\end{align*}

$(3)$ For $\phi\in C\bra{\widetilde{M}\times_{T\times\Pi_T}\ud{L^2\bra{LT_{L^2_m}}}^*}$, let us consider the compact operator-valued function $a$ defined by 
$$x\mapsto [\vac\grotimes \sqrt{\fra{c}_{\widetilde{M},x}}]\circ \phi(x).$$
It gives an equivariant section thanks to the same argument for $(1)$, and clearly $\Phi(P\cdot a)=\phi$.
\end{proof}

In the remainder of this subsection, we prove several properties of the crossed products and the descent homomorphism for proper $LT$-spaces, in order to show how appropriate our constructions are. Since these results will not be used in the following, the reader can skip them. For the same reason, we will often skip the detailed proof.

We begin with an infinite-dimensional version of the following fundamental property: The crossed product by a semi-direct product group is obtained by the iterated crossed product. We are interested in this property for the following reason:
$LT_{L^2_m}$ has the canonical decomposition $U_{LT_{L^2_m}}\times T\times \Pi_T$, and this decomposition is partially inherited to the central extension $LT_{L^2_m}^\tau= U_{LT_{L^2_m}}^\tau\boxtimes (T\times \Pi_T)^\tau$; It has an alternative description 
$(U_{L^2_m}^\tau\times\Pi_T)\rtimes T,$ where the $T$-action is defined by
$$\Ad_t((u,z),n):=((u,z\kappa^\tau_{-n}(t)),n)$$
for $(u,z)\in U_{L^2_m}^\tau$, $n\in \Pi_T$ and $t\in T$. This factorization is valid also for $LG$ for a compact Lie group $G$. Although $LG$ does not have a subgroup corresponding to  $U_{L^2_m}$, it does have a subgroup corresponding to $U_{L^2_m}^\tau\times\Pi_T$. Thus, the following observation can be useful even for the case of $LG$.

We can define the concept of (twisted) crossed products by $U_{L^2_m}$ and $U_{L^2_m}\times\Pi_T$ as follows.

\begin{dfn}
$(1)$ For $q\in \bb{Z}$, we define $\ud{L^2(  U_{L^2_k}\times\Pi_T,\ca{L}^{\otimes q})}:=\ud{L^2(U_{L^2_k},\ca{L}^{\otimes q})}\grotimes L^2(\Pi_T,\ca{L}^{\otimes q})$, and we define a $T$-action on it by the following: For $\phi\in \ud{L^2(U_{L^2_k},\ca{L}^{\otimes q})}$, $\psi\in C_c(\Pi_T,\ca{L}^{\otimes q})$, $t\in T$ and $n\in \Pi_T$,
$$t\cdot[\phi\grotimes \psi](n):=[\kappa^\tau_{n}(t)]^q\phi\grotimes \psi(n).$$
It admits the ``right regular representation'' $R: U_{L^2_m}^\tau\times\Pi_T\to \Aut\bra{\ud{L^2(  U_{L^2_k}\times\Pi_T,\ca{L}^{\otimes q})}}$, which is compatible with the above $T$-action.

$(2)$ For an $\ca{M}_{L^2_m}\rtimes LT_{L^2_m}$-equivariant u.s.c. field of $C^*$-algebras $\scr{A}$, we define the {\bf $q\tau$-twisted crossed product of $\scr{A}$ by $  U_{L^2_m}\times\Pi_T$} by
$$\ud{\scr{A}\rtimes_{q\tau}   (U_{L^2_m}\times\Pi_T)}:=C\bra{\widetilde{M}\times_{\Pi_T}\bbra{\bb{K}\bra{\ud{L^2(  U_{L^2_k}\times\Pi_T,\ca{L}^{\otimes q})}}\grotimes A_x}_{x\in \widetilde{M}}}.$$

$(3)$ In the same setting, we define a substitute for the {\bf $q\tau$-twisted crossed product of $\scr{A}$ by $  U_{L^2_m}$}
$$\ud{\scr{A}\rtimes_{q\tau}   U_{L^2_m}}:=C_0\bra{\widetilde{M},\bbra{\bb{K}\bra{\ud{L^2(  U_{L^2_k},\ca{L}^{\otimes q})}}\grotimes A_x}_{x\in \widetilde{M}}}$$
similarly. 
\end{dfn}
\begin{rmk}
Since $\widetilde{M}$ is non-compact, we need to use $C_0\bra{\cdots}$ instead of $C\bra{\cdots}$ in $(3)$. 
\end{rmk}

It is obvious that $\ud{\scr{A}\rtimes_{q\tau}   (U_{L^2_m}\times\Pi_T)}$ is an $M\rtimes T$-equivariant $C^*$-algebra and $\ud{\scr{A}\rtimes_{q\tau}  U_{L^2_m}}$ is an $\widetilde{M}\rtimes (T\times\Pi_T)$-equivariant $C^*$-algebra.
We will prove that these algebras are related in a natural way which is analogous to the following fundamental result. 

\begin{lem}[{See \cite[Proposition 3.11]{Wil}}]\label{Lemma on iterated twisted crossed product}
Let $G=N\rtimes H$ be the semi-direct product of two locally compact amenable groups $N$ and $H$, and let $A$ be a $G$-$C^*$-algebra. 
Suppose that $N$ admits a $U(1)$-central extension $N^\tau$:
$1\to U(1)\xrightarrow{i} N^\tau \xrightarrow{p} N\to 1$, 
and suppose that the $H$-action on $N$ lifts to $N^\tau$. Then, the semi-direct product $N\rtimes H$ admits a $U(1)$-central extension $N^\tau\rtimes H$ by the natural homomorphisms:
$1\to U(1)\xrightarrow{i\rtimes 1} N^\tau\rtimes H \xrightarrow{p\rtimes \id} N\rtimes H\to 1$. Then, we have a $*$-isomorphism 
$$A\rtimes_{k\tau} G\cong (A\rtimes_{k\tau} N)\rtimes H.$$
\end{lem}

\begin{pro}\label{Prop iterated crossed product}
For an $\ca{M}_{L^2_m}\rtimes LT_{L^2_m}$-equivariant u.s.c. field of $C^*$-algebras $\scr{A}$ and $q\in\bb{Z}$,
$$\ud{\scr{A}\rtimes_{q\tau}( U_{L^2_m}\times \Pi_T)}\cong \ud{\scr{A}\rtimes_{q\tau} U_{L^2_m}}\rtimes_{q\tau} \Pi_T$$
as $M\rtimes T$-equivariant $C^*$-algebras,
and
$$\ud{\scr{A}\rtimes_{q\tau} LT_{L^2_m}}\cong \ud{\scr{A}\rtimes_{q\tau}( U_{L^2_m}\times \Pi_T)}\rtimes T\cong \ud{\scr{A}\rtimes_{q\tau} U_{L^2_m}}\rtimes_{q\tau} (T\times \Pi_T)$$
as $(M/T)\rtimes\{e\}$-$C^*$-algebras.
\end{pro}
\begin{proof}
We prove only $\ud{\scr{A}\rtimes_{q\tau} LT_{L^2_m}}\cong \ud{\scr{A}\rtimes_{q\tau} U_{L^2_m}}\rtimes_{q\tau} (T\times \Pi_T)$, and the others are left to the reader.
Recall that $\ud{\scr{A}\rtimes_{q\tau}   U_{L^2_m}}$ is defined by 
$$C_0\bra{\widetilde{M}, \bbra{\bb{K}\bra{\ud{L^2(  U_{L^2_k},\ca{L}^{\otimes q})}}\grotimes A_x}_{x\in \widetilde{M}}}.$$ 
Since the $T\times \Pi_T$-action on $\bb{K}\bra{\ud{L^2(  U_{L^2_k},\ca{L}^{\otimes q})}}$ is trivial, $\ud{\scr{A}\rtimes_{q\tau}   U_{L^2_m}}$ is isomorphic to
$$C_0\bra{\widetilde{M},\bbra{A_x}_{x\in \widetilde{M}}}\grotimes \bb{K}\bra{\ud{L^2(  U_{L^2_k},\ca{L}^{\otimes q})}}$$
as $\widetilde{M}\rtimes (T\times\Pi_T)$-$C^*$-algebras. 
Thus, the $q\tau$-twisted crossed product of this $C^*$-algebra by $T\times\Pi_T$ is isomorphic to
\begin{align*}
&C_0\bra{\widetilde{M},\bbra{A_x}_{x\in \widetilde{M}}}\rtimes_{q\tau}(T\times\Pi_T)
 \grotimes \bb{K}\bra{\ud{L^2(  U_{L^2_k},\ca{L}^{\otimes q})}} \\
&\ \ \  \cong C\bra{\widetilde{M}\times_{T\times\Pi_T}\bbra{\bb{K}(L^2(T\times\Pi_T,\ca{L}^{\otimes q}))\grotimes A_x}_{x\in \widetilde{M}}} \grotimes \bb{K}\bra{\ud{L^2(  U_{L^2_k},\ca{L}^{\otimes q})}}\\
&\ \ \  \cong C\bra{\widetilde{M}\times_{T\times \Pi_T}\bbra{\bb{K}(\ud{L^2(LT_{L^2_k},\ca{L}^{\otimes q})})\grotimes A_x}_{x\in \widetilde{M}}}\\
&\ \ \ = \ud{\scr{A}\rtimes_{q\tau} LT_{L^2_m}}.
\end{align*}
\end{proof}
\begin{rmks}
$(1)$ $\ud{\scr{A}\rtimes_{q\tau}( U_{L^2_m}\times \Pi_T)}$ carries information of the central extension of $T\times\Pi_T$ as the $T$-action on $\bbra{\bb{K}\bra{\ud{L^2(U_{L^2_k}\times \Pi_T,\ca{L}^{\otimes q})}}\grotimes A_x}_{x\in\widetilde{M}}$. However, $\ud{\scr{A}\rtimes_{q\tau} U_{L^2_m}}$ does not know this kind of information. This is the reason why two kinds of crossed products $\rtimes$ and $\rtimes_{q\tau}$ appear in the statement.

$(2)$ The isomorphism $\ud{\scr{A}\rtimes_{q\tau}( U_{L^2_m}\times \Pi_T)}\rtimes T\cong \ud{\scr{A}\rtimes_{q\tau} U_{L^2_m}}\rtimes_{q\tau} (T\times \Pi_T)$ can be proved by the isomorphism $\ud{\scr{A}\rtimes_{q\tau}( U_{L^2_m}\times \Pi_T)}\cong \ud{\scr{A}\rtimes_{q\tau} U_{L^2_m}}\rtimes_{q\tau} \Pi_T$ and Lemma \ref{Lemma on iterated twisted crossed product}. 
\end{rmks}

In a finite-dimensional setting, a crossed product is a noncommutative analogue of the quotient space. A fixed-point algebra is a more straightforward substitute. It is especially useful for a proper free action. Since the $U_{L^2_m}\times \Pi_T$-action on $\ca{M}_{L^2_m}$ is proper and free, it is worth trying to construct fixed-point algebras for the $U_{L^2_m}\times \Pi_T$-action and those for the $U_{L^2_m}$-action.

\begin{dfn}
$(1)$ For an $\ca{M}_{L^2_m}\rtimes LT_{L^2_m}$-equivariant u.s.c. field of $C^*$-algebras $\scr{A}=\bra{\{A_x\}_{x\in \ca{M}_{L^2_m}},\Gamma_{\scr{A}}}$, we define the {\bf fixed-point algebra} $\ud{\scr{A}^{  U_{L^2_m}}}$ by $C_0(\widetilde{M},\{A_x\}_{x\in \widetilde{M}})$.

$(2)$ In the same setting, we define $\ud{\scr{A}^{  U_{L^2_m}\times\Pi_T}}$ by $C(\widetilde{M}\times_{T\times \Pi_T}\{A_x\}_{x\in \widetilde{M}})$.
\end{dfn}

\begin{lem}
$(1)$ $\ud{\scr{A}^{  U_{L^2_m}\times\Pi_T}}$ is an $M\rtimes T$-equivariant $C^*$-algebra.

$(2)$ $\ud{\scr{A}^{  U_{L^2_m}}}$ is an $\widetilde{M}\rtimes(T\times\Pi_T)$-equivariant $C^*$-algebra.

$(3)$ The generalized fixed-point algebra of $\ud{\scr{A}^{  U_{L^2_m}}}$ with respect to the $\Pi_T$-action $\bra{\ud{\scr{A}^{  U_{L^2_m}}}}^{\Pi_T}$, is isomorphic to $\ud{\scr{A}^{  U_{L^2_m}\times\Pi_T}}$ as $M\rtimes T$-equivariant $C^*$-algebras.

$(4)$ $\ud{\scr{A}^{  U_{L^2_m}}}$ is $\ca{R}KK_{T\times\Pi_T}(\widetilde{M};-,-)$-equivalent to $\ud{\scr{A}\rtimes{  U_{L^2_m}}}$, and similarly for $\ud{\scr{A}^{  U_{L^2_m}\times\Pi_T}}$ and \\
$\ud{\scr{A}\rtimes \bra{  U_{L^2_m}\times\Pi_T}}$.

\end{lem}
\begin{proof}
We leave the proofs of the first three statements to the reader.

For $(4)$, we must specify the $\ca{R}KK$-elements giving the equivalences. We will do that only for $\ud{\scr{A}^{  U_{L^2_m}}}$. We define 
\begin{center}
$[\scr{I}_{\scr{A},U_{L^2_m}}]\in \ca{R}KK_{T\times\Pi_T}(\widetilde{M};\ud{\scr{A}^{  U_{L^2_m}}},\ud{\scr{A}\rtimes{  U_{L^2_m}}})$ and 

$[\scr{J}_{\scr{A},U_{L^2_m}}]\in \ca{R}KK_{T\times\Pi_T}(\widetilde{M};\ud{\scr{A}\rtimes{  U_{L^2_m}}},\ud{\scr{A}^{  U_{L^2_m}}})$ 
\end{center}
by
$$[\scr{I}_{\scr{A},U_{L^2_m}}]:=\bra{C_0(\widetilde{M},\{A_x\grotimes \ud{L^2(U_{L^2_k})^*}\}_{x\in\widetilde{M}}),\{\id\grotimes1\}_{x\in\widetilde{M}},0}\text{ and}$$
$$[\scr{J}_{\scr{A},U_{L^2_m}}]:=\bra{C_0(\widetilde{M},\{A_x\grotimes \ud{L^2(U_{L^2_k})}\}_{x\in\widetilde{M}}),\{\id\grotimes\Op\}_{x\in\widetilde{M}},0},$$
where $\id\grotimes1(a)(a'\grotimes \phi):=aa'\grotimes \phi$ and $\id\grotimes\Op(a\grotimes k)(a'\grotimes \psi):=aa'\grotimes k\psi$ for $a,a'\in A_x$, $\phi\in \ud{L^2(U_{L^2_k})^*}$, $\psi \in \ud{L^2(U_{L^2_k})}$ and $k\in \bb{K}(\ud{L^2(U_{L^2_k})})$. It is clear that $[\scr{I}_{\scr{A},U_{L^2_m}}]$ and $[\scr{J}_{\scr{A},U_{L^2_m}}]$ are mutually inverse, thanks to the isomorphisms $V\grotimes_{\bb{C}}V^*\cong \bb{K}(V)$ and $V^*\grotimes_{\bb{K}(V)}V\cong \bb{C}$ for a Hilbert space $V$.\footnote{Note that $V\grotimes_{\bb{C}}V^*$ is not a Hilbert space but a Hilbert $\bb{K}(V)$-module. The correspondence $\phi\grotimes \psi\mapsto[\lambda\mapsto \phi\innpro{\psi}{\lambda}{}]$ gives an isometric isomorphism. On the other hand, $V^*\grotimes_{\bb{K}(V)}V$ is isomorphic to $\bb{C}$ by the correspondence $\psi\grotimes \phi\mapsto \innpro{\psi}{\phi}{}$. This is isomorphic because the tensor product is taken over $\bb{K}(V)$.}
\end{proof}

\begin{ex}
As everyone expects, we have isomorphisms $\scr{C}\bra{\ca{M}_{L^2_m}}^{  U_{L^2_m}}\cong C_0(\widetilde{M})$
and\\
$\scr{C}\bra{\ca{M}_{L^2_m}}^{  U_{L^2_m}\times\Pi_T}\cong C(M)$.
\end{ex}

With the fixed-point algebra construction, we can define another kind of descent homomorphism following \cite[Theorem 3.4]{Kas88}. This construction looks quite natural, and it is perhaps much more acceptable than $\ud{j_{LT_{L^2_m}}}$.

\begin{dfn}
$(1)$ Let $\scr{B}=\bra{\{B_x\}_{x\in \ca{M}_{L^2_m}},\Gamma_{\scr{B}}}$ be an $\ca{M}_{L^2_m}\rtimes LT_{L^2_m}$-equivariant u.s.c. field of $C^*$-algebras.
For an $\ca{M}_{L^2_m}\rtimes LT_{L^2_m}$-equivariant Kasparov $(\scr{A,B})$-module $(\scr{E},\pi,F)$, we define the {\bf fixed-point module} $\ud{\scr{E}^{  U_{L^2_m}\times\Pi_T}}$ by 
$C(\widetilde{M}\times_{\Pi_T}\{E_x\}_{x\in \widetilde{M}}).$
The restriction of $\{\pi_x:A_x\to \bb{L}_{B_x}(E_x)\}_{x\in \ca{M}_{L^2_m}}$ to the fixed-point algebra $\ud{\scr{A}^{  U_{L^2_m}\times\Pi_T}}$ is denoted by $\pi^{  U_{L^2_m}\times\Pi_T}:\ud{\scr{A}^{  U_{L^2_m}\times\Pi_T}}\to \bb{L}_{\ud{\scr{B}^{  U_{L^2_m}\times\Pi_T}}}\bra{\ud{\scr{E}^{  U_{L^2_m}\times\Pi_T}}}$, and the restriction of $\{F_x\}_{x\in \ca{M}_{L^2_m}}$ to $\ud{\scr{E}^{  U_{L^2_m}\times\Pi_T}}$ is denoted by $F^{  U_{L^2_m}\times\Pi_T}$.
We denote 
$$\bra{\ud{E^{  U_{L^2_m}\times\Pi_T}},\pi^{  U_{L^2_m}\times\Pi_T},F^{  U_{L^2_m}\times\Pi_T}}$$
by $\ud{\lambda^{U_{L^2_m}\times\Pi_T}}(\scr{E},\pi,F)$.

$(2)$ Similarly, we define $\ud{\lambda^{U_{L^2_m}}}(\scr{E},\pi,F)$.
\end{dfn}

We define two homomorphisms $\ud{j_{U_{L^2_m}\times \Pi_T}}$ and $\ud{j_{U_{L^2_m}}}$ in the same way of $\ud{j_{LT_{L^2_m}}}$. Five ``decent homomorphisms'' are related to each other as follows.

\begin{pro}\label{Pro two descent homs are connected}
$(1)$ $\ud{\lambda^{U_{L^2_m}\times\Pi_T}}(\scr{E},\pi,F)$ is an $M\rtimes T$-equivariant Kasparov $\bra{\ud{\scr{A}^{  U_{L^2_m}\times\Pi_T}},\ud{\scr{B}^{  U_{L^2_m}\times\Pi_T}}}$-module. Similarly, $\ud{\lambda^{U_{L^2_m}}}(\scr{E},\pi,F)$ is an $\widetilde{M}\rtimes (T\times\Pi_T)$-equivariant Kasparov $\bra{\ud{\scr{A}^{  U_{L^2_m}}},\ud{\scr{B}^{  U_{L^2_m}}}}$-module.

$(2)$ Both constructions $\ud{\lambda^{U_{L^2_m}\times\Pi_T}}$ and $\ud{\lambda^{U_{L^2_m}}}$ are homotopy invariant. Thus, they define homomorphisms
$$\ud{\lambda^{U_{L^2_m}\times\Pi_T}}:\ca{R}KK_{LT_{L^2_m}}(\ca{M}_{L^2_m};\scr{A},\scr{B})\to \ca{R}KK_{T}(M;\ud{\scr{A}^{  U_{L^2_m}\times\Pi_T}},\ud{\scr{B}^{  U_{L^2_m}\times\Pi_T}}) \text{ and}$$
$$\ud{\lambda^{U_{L^2_m}}}:\ca{R}KK_{LT_{L^2_m}}(\ca{M}_{L^2_m};\scr{A},\scr{B})\to \ca{R}KK_{T\times\Pi_T}(\widetilde{M};\ud{\scr{A}^{  U_{L^2_m}}},\ud{\scr{B}^{  U_{L^2_m}}}).$$

$(3)$ $\ud{\lambda^{U_{L^2_m}\times\Pi_T}}=\lambda^{\Pi_T}\circ \ud{\lambda^{U_{L^2_m}}}$.

$(4)$ Under the $\ca{R}KK$-equivalence of crossed products and fixed-point algebras, $j$'s correspond to $\lambda$'s, that is to say, the following two diagrams commute:
$$
\xymatrix{
\ca{R}KK_{LT_{L^2_m}}(\ca{M}_{L^2_m};\scr{A},\scr{B}) \ar^-{\ud{j_{U_{L^2_m}\times\Pi_T}}}[rr] \ar_-{\ud{\lambda^{U_{L^2_m}\times\Pi_T}}}[drr]&&\ca{R}KK_{T}(M;\ud{\scr{A}\rtimes( U_{L^2_m}\times\Pi_T)},\ud{\scr{B}\rtimes( U_{L^2_m}\times\Pi_T)}) \ar^-\cong[d] \\
&&
\ca{R}KK_{T}(M;\ud{\scr{A}^{  U_{L^2_m}\times\Pi_T}},\ud{\scr{B}^{  U_{L^2_m}\times\Pi_T}}), 
}$$
$$
\xymatrix{
\ca{R}KK_{LT_{L^2_m}}(\ca{M}_{L^2_m};\scr{A},\scr{B}) \ar^-{\ud{j_{U_{L^2_m}}}}[rr] \ar_-{\ud{\lambda^{U_{L^2_m}}}}[drr]&&\ca{R}KK_{T\times\Pi_T}(M;\ud{\scr{A}\rtimes U_{L^2_m}},\ud{\scr{B}\rtimes U_{L^2_m}}) \ar^-\cong[d] \\
&&
\ca{R}KK_{T\times\Pi_T}(M;\ud{\scr{A}^{  U_{L^2_m}}},\ud{\scr{B}^{  U_{L^2_m}}}).
}$$
\end{pro}
\begin{proof}
$(1)$ and $(3)$ are clear. For $(2)$, consider the parallel construction of Proposition \ref{jLT is well defined} $(2)$.

$(4)$ We prove only the latter one. Let $(\scr{E},\pi,F)\in \ca{R}KK_{LT_{L^2_m}}(\ca{M}_{L^2_m};\scr{A},\scr{B})$. Then,
$$\ud{j_{U_{L^2_m}}}(\scr{E},\pi,F)=\bra{
C_0(\widetilde{M},\{E_x\grotimes \bb{K}(\ud{L^2(U_{L^2_k})})\}_{x\in \widetilde{M}}),\{\pi_x\grotimes \id\}_{x\in \widetilde{M}},\{F_x\grotimes \id\}_{x\in \widetilde{M}}}\text{ and}$$
$$\ud{\lambda^{U_{L^2_m}}}(\scr{E},\pi,F)=\bra{
C_0(\widetilde{M},\{E_x\}_{x\in \widetilde{M}}),\{\pi_x\}_{x\in \widetilde{M}},\{F_x\}_{x\in \widetilde{M}}}.$$
We may assume the non-degeneracy $A_xE_x=E_x$ for every $x$. Thus, the Kasparov module 
$$\bra{
C_0(\widetilde{M},\{E_x\}_{x\in \widetilde{M}}),\{\pi_x\}_{x\in \widetilde{M}},\{F_x\}_{x\in \widetilde{M}}}$$
represents the Kasparov product $[\scr{I}_{\scr{A},U_{L^2_m}}]\grotimes \ud{j_{U_{L^2_m}}}(\scr{E},\pi,F) \grotimes [\scr{J}_{\scr{B},U_{L^2_m}}]$, because $V^*\grotimes_{\bb{K}(V)} \bb{K}(V)\grotimes_{\bb{K}(V)} V\cong \bb{C}$ for a Hilbert space $V$.
\end{proof}

A parallel property of $(3)$ holds for $j$'s.

\begin{pro}
At the level of Kasparov modules,
$\ud{j_{LT_{L^2_m}}^{p\tau}}=
j_T\circ \ud{j_{U_{L^2_m}\times\Pi_T}^{p\tau}}=
j_T\circ j_{\Pi_T}^{p\tau}\circ \ud{j_{U_{L^2_m}}^{p\tau}}.$
\end{pro}
\begin{rmks}
$(1)$ This is clear from the same argument of Proposition \ref{Prop iterated crossed product}.

$(2)$ The equality $\ud{j_{U_{L^2_m}\times\Pi_T}}=j_{\Pi_T}\circ \ud{j_{U_{L^2_m}}}$ can be proved as a corollary of Proposition \ref{Pro two descent homs are connected}.
\end{rmks}

As the final result of this subsection, we prove that the descent map preserves the Kasparov product for proper $LT$-spaces. 

\begin{pro}
Let $\scr{A,A}_1,\scr{B}$ be $\ca{M}_{L^2_m}\rtimes LT_{L^2_m}$-equivariant locally separable u.s.c. fields of $C^*$-algebras. 
Then, the following diagram commutes 
$$\begin{CD}
\ca{R}KK_{LT_{L^2_m}}^{q_1\tau}(\ca{M}_{L^2_m};\scr{A},\scr{A}_1)\times \ca{R}KK_{LT_{L^2_m}}^{q_2\tau}(\ca{M}_{L^2_m};\scr{A}_1,\scr{B}) @>\grotimes>>
\ca{R}KK_{LT_{L^2_m}}^{(q_1+q_2)\tau}(\ca{M}_{L^2_m};\scr{A},\scr{B}) \\
@V\ud{j_{LT_{L^2_m}}^{(p-q_2)\tau}}\times \ud{j_{LT_{L^2_m}}^{p\tau}}VV @VV\ud{j_{LT_{L^2_m}}^{p\tau}}V \\
\ca{R}KK(M/T; \star_1,\star_2)\times \ca{R}KK(M/T; \star_2,\star_3) @>\grotimes>>
\ca{R}KK(M/T; \star_1,\star_3),
\end{CD}$$
where the $\star$'s stand for the following:
$$\star_1:=\ud{\scr{A}\rtimes_{(p-q_1-q_2)\tau} LT_{L^2_m}},$$
$$\star_2:=\ud{\scr{A}_1\rtimes_{(p-q_2)\tau} LT_{L^2_m}},$$
$$\star_3:=\ud{\scr{B}\rtimes_{p\tau} LT_{L^2_m}}.$$

The same is true for $\ud{j^{p\tau}_{U_{L^2_m}}}$, $\ud{j^{p\tau}_{U_{L^2_m}\times\Pi_T}}$, $\ud{\lambda^{U_{L^2_m}}}$ and $\ud{\lambda^{U_{L^2_m}\times\Pi_T}}$.
\end{pro}
\begin{proof}
We deal with only the most important one $\ud{j^{p\tau}_{LT_{L^2_m}}}$. 
Let $\bra{\scr{E}_1,\pi_1,F_1}$ be a $q_1\tau$-twisted $\ca{M}_{L^2_m}\rtimes LT_{L^2_m}$-equivariant Kasparov $(\scr{A,A}_1)$-module,  and let $\bra{\scr{E}_2,\pi_2,F_2}$ be a $q_2\tau$-twisted $\ca{M}_{L^2_m}\rtimes LT_{L^2_m}$-equivariant Kasparov $(\scr{A_1,B})$-module. Let $(\scr{E},\pi,F)$ be a representative of the Kasparov product of these two Kasparov modules, where $\scr{E}=\scr{E}_1\grotimes \scr{E}_2$. We need to prove that
$$\ud{j^{p\tau}_{LT_{L^2_m}}}(\scr{E},\pi,F)=\ud{j^{(p-q_2)\tau}_{LT_{L^2_m}}}\bra{\scr{E}_1,\pi_1,F_1}\grotimes \ud{j^{p\tau}_{LT_{L^2_m}}}\bra{\scr{E}_2,\pi_2,F_2}.$$

Let us prove that the bimodules are isomorphic. Since
$$
\ud{\scr{E}_1\rtimes_{(p-q_2)\tau} LT_{L^2_m}} 
= C\bra{\widetilde{M}\times_{T\times\Pi_T} \bbra{E_{1,x}\grotimes \bb{K}(\ud{L^2(LT_{L^2_k},\ca{L}^{\otimes (p-q_2)})},\ud{L^2(LT_{L^2_k},\ca{L}^{\otimes (p-q_1-q_2)})})}_{x\in\widetilde{M}}};$$
$$
\ud{\scr{E}_2\rtimes_{p\tau} LT_{L^2_m}} 
=C\bra{\widetilde{M}\times_{T\times\Pi_T} \bbra{E_{2,x}\grotimes \bb{K}(\ud{L^2(LT_{L^2_k},\ca{L}^{\otimes p})},\ud{L^2(LT_{L^2_k},\ca{L}^{\otimes (p-q_2)})})}_{x\in\widetilde{M}}}; 
$$
$$
\ud{\scr{A}_1\rtimes_{(p-q_2)\tau} LT_{L^2_m}} 
=C\bra{\widetilde{M}\times_{T\times\Pi_T} \bbra{A_{1,x}\grotimes \bb{K}(\ud{L^2(LT_{L^2_k},\ca{L}^{\otimes (p-q_2)})})}_{x\in\widetilde{M}}};
$$
$$
\ud{\scr{E}\rtimes_{p\tau} LT_{L^2_m}} 
= C\bra{\widetilde{M}\times_{T\times\Pi_T} \bbra{E_{x}\grotimes \bb{K}(\ud{L^2(LT_{L^2_k},\ca{L}^{\otimes p})},\ud{L^2(LT_{L^2_k},\ca{L}^{\otimes (p-q_1-q_2)})})}_{x\in\widetilde{M}}},
$$
we have a natural isomorphism 
$$\ud{\scr{E}\rtimes_{p\tau} LT_{L^2_m}} \cong \ud{\scr{E}_1\rtimes_{(p-q_2)\tau} LT_{L^2_m}} \grotimes_{\ud{\scr{A}_1\rtimes_{(p-q_2)\tau} LT_{L^2_m}} }\ud{\scr{E}_2\rtimes_{p\tau} LT_{L^2_m}} $$
as $(\ud{\scr{A}\rtimes_{(p-q_1-q_2)\tau} LT_{L^2_m}},\ud{\scr{B}\rtimes_{p\tau} LT_{L^2_m}})$-bimodules, thanks to the isomorphism $\bb{K}(V_2,V_1)\grotimes_{\bb{K}(V_2)} \bb{K}(V_3,V_2)\cong \bb{K}(V_3,V_1)$ for Hilbert spaces $V_1,V_2, V_3$

The conditions on $F$ can be easily proved, and we leave the details to the reader.
\end{proof}


\subsection{The ``descent of the Dirac element''}\label{section top ass map Dirac}

The final tool to define the topological assembly map is the ``descent of the reformulated Dirac element''. In this subsection, we define an infinite-dimensional version of it, by imitating Proposition \ref{prop descent of Dirac element for twisted version}.

In this formula, a Hilbert $\bb{C}\rtimes_\tau G$-module $\bb{C}\rtimes_{-\tau} G$ appears. We begin with the $LT_{L^2_m}$-version of it.

\begin{dfn}
$(1)$ We define a Hilbert $\ud{\bb{C}\rtimes_\tau LT_{L^2_m}}$-module $\ud{\bb{C}\rtimes_{-\tau} LT_{L^2_m}}$ by the following:
\begin{itemize}
\item We define $C^*$-algebras $\ud{\bb{C}\rtimes_{-\tau}U_{L^2_m}}$ and $\ud{\bb{C}\rtimes_{-\tau}LT_{L^2_m}}$ in the same way of Definition \ref{definition of twisted group Cstar algebra of LT} by replacing $\bb{K}\bra{\ud{L^2(\bb{R}^\infty)^*}}$ and $\bb{C}\rtimes_\tau (T\times \Pi_T)$ with $\bb{K}\bra{\ud{L^2(\bb{R}^\infty)}}$ and $\bb{C}\rtimes_{-\tau} (T\times \Pi_T)$.
When we regard an element $\phi$ of $\bb{C}\rtimes_{-\tau}U_{L^2_m}$ as an element of $\bb{K}\bra{\ud{L^2(\bb{R}^\infty)}}$, we denote it by $\Op(\phi)$. 
\item For a function $b$ on $(T\times\Pi_T)^\tau$, we define $b^\vee(g):=b(g^{-1})$. This operation exchanges ${\bb{C}\rtimes_{\tau}(T\times\Pi_T)}$ and ${\bb{C}\rtimes_{-\tau}(T\times\Pi_T)}$.
\item Similarly, we define $\phi^\vee:=\Op^{-1}({}^t\Op(\phi))$ for $\phi \in \ud{\bb{C}\rtimes_{-\tau}U_{L^2_m}}$ or $\ud{\bb{C}\rtimes_{\tau}U_{L^2_m}}$. This correspondence exchanges $\ud{\bb{C}\rtimes_{\tau}U_{L^2_m}}$ and $\ud{\bb{C}\rtimes_{-\tau}U_{L^2_m}}$.
\item We define a Hilbert $\ud{\bb{C}\rtimes_\tau LT_{L^2_m}}$-module structure on $\ud{\bb{C}\rtimes_{-\tau} LT_{L^2_m}}$ by the following: For $\phi \grotimes \psi, \phi_1 \grotimes \psi_1, \phi_2 \grotimes \psi_2\in \ud{\bb{C}\rtimes_{-\tau} U_{L^2_m}}\grotimes [\bb{C}\rtimes_{-\tau}(T\times\Pi_T)]= \ud{\bb{C}\rtimes_{-\tau} LT_{L^2_m}}$ and $ a \grotimes b\in \ud{\bb{C}\rtimes_{\tau} U_{L^2_m}}\grotimes [\bb{C}\rtimes_{\tau}(T\times\Pi_T)]= \ud{\bb{C}\rtimes_\tau LT_{L^2_m}}$,
$$\phi \grotimes\psi\cdot  a \grotimes b:=  a^\vee * \phi \grotimes b^\vee* \psi;$$
$$\inpr{\phi_1 \grotimes \psi_1}{\phi_2 \grotimes \psi_2}{}:=(\phi_2* \phi_1^*)^\vee\grotimes (\psi_2*\psi_1^*)^\vee.$$
\end{itemize}

$(2)$ We define a dense subspace $\ud{(\bb{C}\rtimes_{-\tau}LT_{L^2_m})_\fin}$ just like Definition \ref{definition of twisted group Cstar algebra of LT} $(3)$. It has a pre-Hilbert $\ud{(\bb{C}\rtimes_{\tau}LT_{L^2_m})_\fin}$-module structure.
\end{dfn}

The above Hilbert module is equipped with an $LT_{L^2_m}^\tau$-action denoted by ``$\rt$''. See also Definition \ref{dfn of roup Cstar algebra of LT}.

\begin{dfn}
$(1)$ We define an $LT_{L^2_m}^\tau$-action ``$\rt$'' on the Hilbert $\ud{(\bb{C}\rtimes_{\tau}LT_{L^2_m})}$-module $\ud{\bb{C}\rtimes_{-\tau}LT_{L^2_m}}$ by the tensor product of the following actions:
\begin{itemize}
\item For $\phi \in \ud{\bb{C}\rtimes_{-\tau}U_{L^2_m}}$ and $u\in U_{L^2_m}^\tau$, we define $\Op(\rt_u\phi):=\Op(\phi)\circ \rho_{u^{-1}}$, and
\item For $\psi\in C_c(T\times\Pi_T,-\tau)\subseteq \bb{C}\rtimes_{-\tau}(T\times\Pi_T)$ and $\gamma\in (T\times\Pi_T)^\tau$, we define $\rt_\gamma\psi(x):=\psi(x\gamma)$.
\end{itemize}

$(2)$ The infinitesimal version of ``$\rt$'' is denoted by $d\rt_X$ for $X\in \Lie(LT_\fin)$. It is an operator defined on $(\bb{C}\rtimes_{-\tau} LT_{L^2_m})_\fin$. Its extension is also denoted by the same symbol.
\end{dfn}

With this Hilbert module, we define a substitute for ``the descent of the Dirac element''. We denote the exterior tensor product of the Spinor bundles $S_U$ and $S_{\widetilde{M}}$ by $S_{\ca{M}}$. Take a cut-off function $\fra{c}:\widetilde{M}\to \bb{R}_{\geq 0}$ with respect to the $T\times\Pi_T$-action. We refer to \cite[Section 6.2.1]{Thesis} for the ``fiber'' of the following Kasparov module. 
See Remark \ref{rmk exposition of bimodule structure on decent family descr} $(2)$ for the strict definitions of the following symbolic formulas.

\begin{dfn}\label{dfn descent of Dirac element}
$(1)$ We define a Hilbert $\ud{\bb{C}\rtimes_\tau LT_{L^2_m}}$-module
$$\ud{L^2(\ca{M}_{L^2_k},S_{\ca{M}})\rtimes_\tau LT_{L^2_m}}=
L^2\bra{\widetilde{M}\times_{T\times\Pi_T}\bbra{S_U\grotimes S_{\widetilde{M},x}\grotimes\ud{L^2\bra{LT_{L^2_k},\ca{L}}}\grotimes \ud{\bb{C}\rtimes_{-\tau} LT_{L^2_m}}}_{x\in \widetilde{M}}}$$
by the completion of $C\bra{\widetilde{M}\times_{T\times\Pi_T}\bbra{S_U\grotimes S_{\widetilde{M},x}\grotimes\ud{L^2\bra{LT_{L^2_k},\ca{L}}}\grotimes \ud{\bb{C}\rtimes_{-\tau} LT_{L^2_m}}}_{x\in \widetilde{M}}}$
by the following operations: For $\phi,\phi_1,\phi_2:\widetilde{M}\to S_U\grotimes S_{\widetilde{M}}\grotimes \ud{L^2\bra{LT_{L^2_k},\ca{L}}}$ and $\psi,\psi_1,\psi_2:\widetilde{M}\to\ud{\bb{C}\rtimes_{-\tau} LT_{L^2_m}}$ 
such that $k_1=\phi_1\grotimes \psi_1,k_2=\phi_2\grotimes \psi_2\in \ud{L^2(\ca{M}_{L^2_k},S_{\ca{M}})\rtimes_\tau LT_{L^2_m}}$, and
 $b\in \ud{\bb{C}\rtimes_\tau LT_{L^2_m}}$,
\begin{itemize}
\item $[\phi\grotimes \psi\cdot b](x):=\phi(x)\grotimes[\psi(x)\cdot b]$; and
\item $\inpr{\phi_1\grotimes \psi_1}{\phi_2\grotimes \psi_2}{\ud{\bb{C}\rtimes_\tau LT_{L^2_m}}}:=\int_{\widetilde{M}}\inpr{\phi_1(x)}{\phi_2(x)}{}\inpr{\psi_1(x)}{\psi_2(x)}{\ud{\bb{C}\rtimes_\tau LT_{L^2_m}}}\fra{c}(x)dx$.
\end{itemize}

$(2)$ We define a left $\ud{\scr{C}\bra{\ca{M}_{L^2_m}}\rtimes_\tau LT_{L^2_m}}$-module structure 
$$\ud{\pi\rtimes_{\tau} \lt }: \ud{\scr{C}\bra{\ca{M}_{L^2_m}}\rtimes_\tau LT_{L^2_m}}\to \bb{L}_{\ud{\bb{C}\rtimes_\tau LT_{L^2_m}}}\bra{\ud{L^2(\ca{M}_{L^2_k},S_{\ca{M}})\rtimes_\tau LT_{L^2_m}}}$$
as follows: For $a\in \ud{\scr{C}\bra{\ca{M}_{L^2_m}}\rtimes_\tau LT_{L^2_m}}$, $\phi:\widetilde{M}\to S_U\grotimes S_{\widetilde{M}}\grotimes \ud{L^2\bra{LT_{L^2_k},\ca{L}}}$, $\psi:\widetilde{M}\to\ud{\bb{C}\rtimes_{-\tau} LT_{L^2_m}}$ 
such that $k=\phi\grotimes \psi\in \ud{L^2(\ca{M}_{L^2_k},S_{\ca{M}})\rtimes_\tau LT_{L^2_m}}$, 
$$[\ud{\pi\rtimes_\tau \lt }(a)(\phi\grotimes \psi)](x):=[\pi_x(a(x))(\phi(x))]\grotimes\psi(x).$$

$(3)$ By using the identification
\begin{align*}
&\ud{L^2(\ca{M}_{L^2_k},S_{\ca{M}})\rtimes_\tau LT_{L^2_m}}\\
&\ \ \ \cong
L^2\bra{\widetilde{M}\times_{T\times\Pi_T}\bbra{ S_{\widetilde{M},x}\grotimes L^2\bra{T\times\Pi_T,\ca{L}}\grotimes [\bb{C}\rtimes_{-\tau} (T\times\Pi_T)]}_{x\in \widetilde{M}}}
\widehat{\bigotimes} \ \ud{L^2\bra{U_{L^2_k},\ca{L}}}\grotimes \ud{\bb{C}\rtimes_{-\tau} U_{L^2_m}}\grotimes
S_U,
\end{align*}
we define an operator $\widetilde{D}$ by
\begin{align*}
\widetilde{D}&:=\bra{\sum_ic(e_i)\nabla^{S_{\widetilde{M}}}_{e_i}\grotimes \id_{ \ud{\bb{C}\rtimes_{-\tau} (T\times\Pi_T)}}}\widehat{\bigotimes}\ \id\\
&\ \ \  +\id\ \widehat{\bigotimes}\bra{\sum_n\bbra{\bra{n^{-2l}dR'(z_n)\grotimes \id+\id\grotimes \sqrt{n}d\rt_{z_n}}\grotimes \gamma\bra{\overline{z_n}}+\bra{n^{-2l}dR'(\overline{z_n})\grotimes \id+\id\grotimes \sqrt{n}d\rt_{\overline{z_n}}\grotimes \gamma\bra{z_n}}}}.
\end{align*}
We denote $\sum_ic(e_i)\nabla^{S_{\widetilde{M}}}_{e_i}\grotimes \id_{ \ud{\bb{C}\rtimes_{-\tau} (T\times\Pi_T)}}$ by $\widetilde{D_{\rm base}}$.

$(4)$ The $KK$-element corresponding to the unbounded Kasparov $(\ud{\scr{C}(\ca{M}_{L^2_m})\rtimes_\tau LT_{L^2_m}},\ud{\bb{C}\rtimes_\tau LT_{L^2_m}} )$-module
$$\bra{\ud{L^2(\ca{M}_{L^2_k},S_{\ca{M}})\rtimes_\tau LT_{L^2_m}},\ud{\pi\rtimes_\tau \lt },\widetilde{D}}$$
is denoted by
$$\ud{j_{LT_{L^2_m}}^\tau\bra{\fgt[S]\grotimes \bbbra{\widetilde{d_{\ca{M}_{L^2_m}}}}}}\in KK(\ud{\scr{C}(\ca{M}_{L^2_m})\rtimes_\tau LT_{L^2_m}},\ud{\bb{C}\rtimes_\tau LT_{L^2_m}} )$$
(we will prove that the above triple is actually an unbounded Kasparov module later).
We regard it as the ``{\bf descent of the reformulated Dirac element}''.
\end{dfn}

\begin{rmks}\label{Remark on the descent of Dirac}
$(1)$ The ``fiber'' of this $KK$-element is almost the same with $[\widetilde{\Dirac}_R]$ discussed in \cite[Theorem 6.8]{Thesis}. 

$(2)$ The tensor product between the $\widetilde{M}$-part and the $U_{L^2_m}$-part is denoted by $\widehat{\bigotimes}$, and others are denoted by $\grotimes$.

$(3)$ We have defined neither $j_{LT_{L^2_m}}^\tau$ nor $\fgt[S]$. Although we have defined a $KK$-element ``$[\widetilde{d_{\ca{M}_{L^2_m}}}]$'' in \cite{T4}, we have not proved any relationships between ``$[\widetilde{d_{\ca{M}_{L^2_m}}}]$'' and ``$\ud{j_{LT_{L^2_m}}^\tau\bra{\fgt[S]\grotimes \bbbra{\widetilde{d_{\ca{M}_{L^2_m}}}}}}$''. We will give a comment on this issue in Section \ref{section unsolved}.
\end{rmks}

\begin{thm}\label{jLTDirac}
The triple $\bra{\ud{L^2(\ca{M}_{L^2_k},S_{\ca{M}})\rtimes_\tau LT_{L^2_m}},\ud{\pi\rtimes_\tau \lt },\widetilde{D}}$ is an unbounded Kasparov \\
$\bra{\ud{\scr{C}(\ca{M}_{L^2_m})\rtimes_\tau LT_{L^2_m}}, \ud{\bb{C}\rtimes_\tau LT_{L^2_m}}}$-module.
\end{thm}
\begin{proof}
We need to prove the following properties: $(1)$ The operator $\widetilde{D}$ is well-defined and essentially self-adjoint; $(2)$ For ``smooth'' $a \in \ud{\scr{C}(\ca{M}_{L^2_m})\rtimes_\tau LT_{L^2_m}}$, the commutator $\bbbra{\ud{\pi\rtimes_{\tau}\lt }(a),\widetilde{D}}$ is bounded; $(3)$ For any $a \in \ud{\scr{C}(\ca{M}_{L^2_m})\rtimes_\tau LT_{L^2_m}}$, $\ud{\pi\rtimes_{\tau}\lt }(a)\bra{1+\widetilde{D}^2}^{-1}$ is $\ud{\bb{C}\rtimes_\tau LT_{L^2_m}}$-compact. We refer to \cite[Theorem 6.8]{Thesis} for several estimates of this proof.

$(1)$ Let us consider the dense subspace
{\small
$$C^\infty\bra{\widetilde{M}\times_{T\times\Pi_T}\bbra{S_{\widetilde{M}}\grotimes L^2(T\times\Pi_T,\ca{L})\grotimes [\bb{C}\rtimes_{-\tau} (T\times\Pi_T)]}}\widehat{\bigotimes}^\alg \ud{L^2\bra{U_{L^2_k},\ca{L}}_\fin}\grotimes^\alg \ud{[\bb{C}\rtimes_{-\tau} U_{L^2_m}]_\fin}\grotimes^\alg S_{U,\fin}$$}
of $\ud{L^2(\ca{M}_{L^2_k},S_{\ca{M}})\rtimes_\tau LT_{L^2_m}}$. Note that the $(T\times\Pi_T)$-action on $\ud{L^2\bra{U_{L^2_k},\ca{L}}_\fin}\grotimes \ud{[\bb{C}\rtimes_{-\tau} U_{L^2_m}]_\fin}\grotimes S_{U,\fin}$ is trivial, and hence the above is actually a subspace of $\ud{L^2(\ca{M}_{L^2_k},S_{\ca{M}})\rtimes_\tau LT_{L^2_m}}$. On this subspace, $\widetilde{D}$ acts as
\begin{align*}
&\widetilde{D_{\rm base}}\widehat{\bigotimes}\id+
\id\widehat{\bigotimes} \sum_n\bra{n^{-2l}dR'(z_n)\grotimes \id\grotimes \gamma\bra{\overline{z_n}}+n^{-2l}dR'(\overline{z_n})\grotimes \id\grotimes \gamma\bra{z_n}}\\
&\ \ \ +\id\widehat{\bigotimes}\sum_n\bbra{\id\grotimes \sqrt{n}d\rt_{z_n}\grotimes \gamma\bra{\overline{z_n}}+\id\grotimes \sqrt{n}d\rt_{\overline{z_n}}\grotimes \gamma\bra{z_n}}\\
&=:\widetilde{D_{\rm base}}\widehat{\bigotimes}\id+\id\widehat{\bigotimes} \widetilde{D_2}+\id\widehat{\bigotimes}\widetilde{D_3}.
\end{align*}
$\widetilde{D_{\rm base}}$ is well-defined and essentially self-adjoint by a finite-dimensional argument.
$\widetilde{D_2}$ is well-defined and essentially self-adjoint just like Definition-Theorem \ref{index element}. For $\widetilde{D_3}$, see  \cite[Lemma 6.9]{Thesis}.

$(2)$ The ``smooth algebra'' is the dense subspace 
$$C^\infty\bra{\widetilde{M}\times_{T\times\Pi_T} \bb{F}\bra{C_c^\infty(T\times\Pi_T,\ca{L})}}\widehat{\bigotimes}^\alg\bb{F}\bra{\ud{L^2\bra{U_{L^2_k},\ca{L}}_\fin}},$$
where $\bb{F}\bra{V'}$ for a Hilbert space $V$ and its dense subspace $V'\subseteq V$ is the set of finite rank operators on $V$ preserving $V'$.
Let us verify that the commutator of 
$$a\widehat{\bigotimes} k\in C^\infty\bra{\widetilde{M}\times_{T\times \Pi_T} \bb{F}\bra{C^\infty_c(T\times\Pi_T,\ca{L})}}\widehat{\bigotimes}^\alg\bb{F}\bra{\ud{L^2\bra{U_{L^2_k},\ca{L}}_\fin}}$$
and $\widetilde{D}$ is a bounded operator. Since $d\rt$ commutes with $k$, and $a $ commutes with the Clifford actions on $S_{\widetilde{M}}$ and $S_U$, we obtain
$$[a\widehat{\bigotimes} k,\widetilde{D}]=[a,\widetilde{D_{\rm base}}]\widehat{\bigotimes} k+ a\widehat{\bigotimes}[k,\widetilde{D_2}].$$

The first term is a bounded operator, thanks to the ordinary argument of the descent homomorphism for unbounded Kasparov modules. For the second one, we put $k=\sum \phi_i\grotimes\psi_i^*$ for $\phi_i,\psi_i\in \ud{L^2\bra{U_{L^2_k},\ca{L}}_\fin}$. Since $[k,\widetilde{D_2}]=\sum_i[\phi_i\grotimes\psi_i^*,\widetilde{D_2}]$, it is sufficient to prove that each commutator is bounded. Thus, we may assume that $k$ is a single Schatten form $k=\phi\grotimes\psi^*$.
For $\xi\grotimes s\grotimes b\in \ud{L^2\bra{U_{L^2_k},\ca{L}}_\fin}\grotimes  S_{U,\fin}\grotimes \ud{[\bb{C}\rtimes_{-\tau}U_{L^2_m}]_\fin}$,
\begin{align*}
&[k,\widetilde{D_2}](\xi\grotimes s\grotimes b)\\
&= k\sum_n \bbra{n^{-2l}dR'(z_n)(\xi)\grotimes \gamma\bra{\overline{z_n}}(s)+n^{-2l}dR'(\overline{z_n})(\xi)\grotimes \gamma\bra{z_n}(s)}\grotimes b\\
&\ \ \ -\sum_n\bbra{n^{-2l}dR'(z_n)\circ k(\xi)\grotimes \gamma\bra{\overline{z_n}}(s)+n^{-2l}dR'(\overline{z_n})\circ k(\xi)\grotimes \gamma\bra{z_n}(s)}\grotimes b \\
&=\sum_n \bbra{\phi\inpr{\psi}{n^{-2l}dR'(z_n)(\xi)}{}\grotimes \gamma\bra{\overline{z_n}}(s)+\phi\inpr{\psi}{n^{-2l}dR'(\overline{z_n})(\xi)}{}\grotimes \gamma\bra{z_n}(s)}\grotimes b\\
&\ \ \ -\sum_n  \bbra{n^{-2l}dR'(z_n)[\phi\inpr{\psi}{\xi}{}]\grotimes \gamma\bra{\overline{z_n}}(s)+n^{-2l}dR'(\overline{z_n})[\phi\inpr{\psi}{\xi}{}]\grotimes \gamma\bra{z_n}(s)} \grotimes b\\
&=\sum_n  \bbra{\phi\inpr{\{n^{-2l}dR'(z_n)\}^*(\psi)}{\xi}{}\grotimes \gamma\bra{\overline{z_n}}(s)+\phi\inpr{\{n^{-2l}dR'(\overline{z_n})\}^*(\psi)}{\xi}{}\grotimes \gamma\bra{z_n}(s)}\grotimes b\\
&\ \ \ -\sum_n \bbra{n^{-2l}dR'(z_n)(\phi)\inpr{\psi}{\xi}{}\grotimes \gamma\bra{\overline{z_n}}(s)+n^{-2l}dR'(\overline{z_n})(\phi)\inpr{\psi}{\xi}{}\grotimes \gamma\bra{z_n}(s)} \grotimes b\\
&=\sum_n 
\bra{
\bbra{
\bbbra{\phi\grotimes [\{n^{-2l}dR'(z_n)\}^*(\psi)]^*}\grotimes \gamma\bra{\overline{z_n}}
+\bbbra{\phi\grotimes [\{n^{-2l}dR'(\overline{z_n})\}^*(\psi)]^*}\grotimes \gamma\bra{z_n} }
\grotimes \id}\xi\grotimes s\grotimes b\\
&\ \ \ -\sum_n \bra{\bbra{\bbbra{n^{-2l}dR'(z_n)(\phi)\grotimes \psi^*}\grotimes \gamma\bra{\overline{z_n}}
+\bbbra{n^{-2l}dR'(\overline{z_n})(\phi)\grotimes \psi^*}\grotimes \gamma\bra{z_n}}\grotimes \id}\xi\grotimes s\grotimes b,
\end{align*}
where $\phi\grotimes [\{n^{-2l}dR'(z_n)\}^*(\psi)]^*$ stands for the Schatten form $\lambda\mapsto \phi\inpr{\{n^{-2l}dR'(z_n)\}^*(\psi)}{\lambda}{}$, and similarly for other terms of the last two lines.
Thus, the commutator can be divided into four parts:
$$\sum_n\phi\grotimes [\{n^{-2l}dR'(z_n)\}^*(\psi)]^*\grotimes\gamma\bra{\overline{z_n}} +\sum_n\phi\grotimes [\{n^{-2l}dR'(\overline{z_n})\}^*(\psi)]^*\grotimes\gamma\bra{z_n}$$
$$-\sum_nn^{-2l}dR'(z_n)(\phi)\grotimes \psi^* \grotimes\gamma\bra{\overline{z_n}} -\sum_nn^{-2l}dR'(\overline{z_n})(\phi)\grotimes \psi^* \grotimes\gamma\bra{z_n}.$$

Since $dR'(\overline{z_n})^*=-dR'(z_n)$ is an ``annihilator'', and since $\phi,\psi\in \ud{L^2(U_{L^2_k},\ca{L})_\fin}$, the second and third terms are finite sums of finite rank operators, which are obviously bounded. 

For the first and fourth terms, we prove that the infinite sums converge in operator norm. We deal with only the first one.
For this aim, we take an orthonormal base of $\ud{L^2(U_{L^2_k},\ca{L})}$. First, we put
\begin{align*}
\widetilde{\phi}_{\overrightarrow{\alpha},\overrightarrow{\beta}}
:=
\bra{dR'(\overline{z_1})^{\alpha_1}dL'(z_1)^{\beta_1}dR'(\overline{z_2})^{\alpha_2}dL'(z_2)^{\beta_2}\cdots }\vac,
\end{align*}
and we put $\phi_{\overrightarrow{\alpha},\overrightarrow{\beta}}:=\|\widetilde{\phi}_{\overrightarrow{\alpha},\overrightarrow{\beta}}\|^{-1}\widetilde{\phi}_{\overrightarrow{\alpha},\overrightarrow{\beta}}$.
Since $\|\widetilde{\phi}_{\overrightarrow{\alpha},\overrightarrow{\beta}}\|^2=\prod_n(2n^l)^{\alpha_n+\beta_n}\alpha_n!\beta_n!$, we obtain
$$dR'(\overline{z_n})\phi_{\overrightarrow{\alpha},\overrightarrow{\beta}}=
\sqrt{2n^l(\alpha_n+1)}\phi_{\overrightarrow{\alpha}+e_n,\overrightarrow{\beta}},$$
where $\overrightarrow{\alpha}+e_n:=(\cdots,\alpha_{n-1},\alpha_n+1,\alpha_{n+1},\cdots)$.

We have assumed that $\psi$ is a finite linear combination of $\phi_{\overrightarrow{\alpha},\overrightarrow{\beta}}$'s. Thus, we may assume that $\psi=\phi_{\overrightarrow{\alpha},\overrightarrow{\beta}}$ from the beginning, since a finite sum of bounded operators is again bounded. We have
$$\|\phi\grotimes [\{n^{-2l}dR'(z_n)\}^*(\phi_{\overrightarrow{\alpha},\overrightarrow{\beta}})]^*\grotimes\gamma\bra{\overline{z_n}}\|\leq 2n^{-3l/2}\|\phi\|\sqrt{\alpha_n+1}.$$
Since $\alpha_n=0$ except for finitely many $n$'s, the infinite sum $\sum_n\phi\grotimes [\{n^{-2l}dR'(z_n)\}^*(\psi)]^*\grotimes\gamma\bra{\overline{z_n}}$ converges in norm.

$(3)$ See also \cite[Lemma 6.12]{Thesis} for details of the following argument. We prove this property by the following two steps: For $a\in C(\widetilde{M}\times_{T\times\Pi_T}\bb{K}(L^2(T\times\Pi_T,\ca{L})))$ and $k\in \bb{K}\bra{ \ud{L^2\bra{U_{L^2_k},\ca{L}}}}$,
$(a)$ $a\widehat{\bigotimes} k\grotimes\id(1+[\widetilde{D_{\rm base}}\widehat{\bigotimes}\id+\id\widehat{\bigotimes}\widetilde{D_3}]^2)^{-1}$ is compact; $(b)$ the difference $a\widehat{\bigotimes} k(1+[\widetilde{D_{\rm base}}\widehat{\bigotimes}\id+\id\widehat{\bigotimes}\widetilde{D_3}]^2)^{-1}-a\widehat{\bigotimes} k(1+[\widetilde{D_{\rm base}}\widehat{\bigotimes}\id+\id\widehat{\bigotimes}\widetilde{D_2}+\id\widehat{\bigotimes}\widetilde{D_3}]^2)^{-1}$ is also compact. 

$(a)$ First, we recall how $\ud{\bb{C}\rtimes_\tau LT_{L^2_m}}$-compact operators look like. 
Note that the tensor product of a $\ud{\bb{C}\rtimes_\tau U_{L^2_m}}$-compact operator and a $\bb{C}\rtimes_\tau (T\times\Pi_T)$-compact operator is a $\ud{\bb{C}\rtimes_\tau LT_{L^2_m}}$-compact operator. 
Since $\ud{\bb{C}\rtimes_\tau U_{L^2_m}}=\bb{K}(\ud{L^2(\bb{R}^\infty)^*})$, a Hilbert $\ud{\bb{C}\rtimes_\tau U_{L^2_m}}$-module is isomorphic to $V\grotimes\ud{L^2(\bb{R}^\infty)}$ for some Hilbert space $V$. Thus, $\ud{\bb{C}\rtimes_\tau U_{L^2_m}}$-compact operator is $k\grotimes\id_{\ud{L^2(\bb{R}^\infty)}}$ for $k\in\bb{K}(V)$. For example, $\ud{\bb{C}\rtimes_{-\tau} U_{L^2_m}}=\bb{K}(\ud{L^2(\bb{R}^\infty)})$ can be regarded as
$\ud{L^2(\bb{R}^\infty)}\grotimes \ud{L^2(\bb{R}^\infty)^*}\cong \ud{L^2(\bb{R}^\infty)^*}\grotimes \ud{L^2(\bb{R}^\infty)}.$

With this observation, we consider the spectral decomposition of $\widetilde{D_3}^2$. 
By the identification $\ud{[\bb{C}\rtimes_{-\tau} U_{L^2_m}]_\fin}\grotimes^\alg S_{U,\fin} \cong\ud{L^2(\bb{R}^\infty)_\fin}\grotimes^\alg \ud{L^2(\bb{R}^\infty)_\fin^*}\grotimes^\alg S_{U,\fin}$, we rewrite $\widetilde{D_3}$ as 
$$\id_{\ud{L^2(U_{L^2_k},\ca{L})}}\grotimes \id_{\ud{L^2(\bb{R}^\infty)}}\grotimes D_{\alg}:=\id_{\ud{L^2(U_{L^2_k},\ca{L})}}\grotimes \id_{\ud{L^2(\bb{R}^\infty)}}\grotimes
\sum_n\bbra{\sqrt{n}d\rho^*_{z_n}\grotimes \gamma\bra{\overline{z_n}}+\sqrt{n}d\rho^*_{\overline{z_n}}\grotimes \gamma\bra{z_n}}.$$
The spectral decomposition of $D_\alg$ can be computed just like the proof of $(2)$ of $(A)$ of Lemma \ref{index element for HKT algebra}. The property we need is that $D_{\alg}^2$ has discrete spectrum bounded below with finite multiplicity, that is to say, $D_{\alg}^2=\sum_n\lambda_nP_n$, where $\lambda_n\geq 0$, $\#\bbra{\lambda_n\mid \lambda_n<K}$ is finite for any $K>0$, $P_n$ is the orthogonal projection onto a finite-dimensional subspace of $\ud{L^2(\bb{R}^\infty)_\fin^*}\grotimes^\alg S_{U,\fin}$, and $P_n$'s are mutually orthogonal.

Thanks to this spectral decomposition, and thanks to the fact that $\widetilde{D_{\rm base}}\widehat{\bigotimes}\id$ anti-commutes with $\id \widehat{\bigotimes}\widetilde{D_3}$, 
\begin{align*}
1+[\widetilde{D_{\rm base}}\widehat{\bigotimes}\id+\id \widehat{\bigotimes}\widetilde{D_3}]^2
&=1+\widetilde{D_{\rm base}}^2\widehat{\bigotimes}\id+\id\widehat{\bigotimes}\id_{\ud{L^2(U_{L^2_k},\ca{L})}}\grotimes \id_{\ud{L^2(\bb{R}^\infty)}}\grotimes\sum_n\lambda_nP_n \\
&=\sum_n\bbra{1+\lambda_n+\widetilde{D_{\rm base}}^2}\widehat{\bigotimes}\id_{\ud{L^2(U_{L^2_k},\ca{L})}}\grotimes \id_{\ud{L^2(\bb{R}^\infty)^*}}\grotimes P_n.
\end{align*}
Thus, the operator $a\widehat{\bigotimes} k(1+\widetilde{D_{\rm base}}^2\widehat{\bigotimes}\id+\id\widehat{\bigotimes} \widetilde{D_3}^2)^{-1}$ is rewritten as
$$\sum_na\bbra{1+\lambda_n+\widetilde{D_{\rm base}}^2}^{-1}\widehat{\bigotimes} k\grotimes \id_{\ud{L^2(\bb{R}^\infty)^*}}\grotimes P_n.
$$
Notice that $\widetilde{D_{\rm base}}$ is the operator for the descent of the index element of a Dirac operator ``$D_{\rm base}$'', and hence $a\bbra{1+\lambda_n+\widetilde{D_{\rm base}}^2}^{-1}$ is $\bb{C}\rtimes_\tau (T\times\Pi_T)$-compact. Moreover, since $k$ is a compact operator and since $P_n$ is a finite rank operator, $k\grotimes \id_{\ud{L^2(\bb{R}^\infty)}}\grotimes P_n$ is $\bb{C}\rtimes_\tau U_{L^2_m}$-compact. Thus, each summand is $\bb{C}\rtimes_\tau LT_{L^2_m}$-compact. Finally, thanks to the estimate
$$\left\|a\bbra{1+\lambda_n+\widetilde{D_{\rm base}}^2}^{-1}\widehat{\bigotimes} k\grotimes \id_{\ud{L^2(\bb{R}^\infty)}}\grotimes P_n\right\|\leq \|a\|\cdot \frac{1}{1+\lambda_n}\cdot \|k\|,$$
and the fact that $P_n$'s are mutually orthogonal,
the above infinite sum converges in norm because $\lambda_n\to \infty$ as $m\to \infty$. Therefore, it is a $\bb{C}\rtimes_\tau LT_{L^2_m}$-compact operator.

$(b)$ One can prove the statement by the argument of the proof of the latter half of \cite[Lemma 6.12]{Thesis}. We leave the details to the reader.
\end{proof}

\subsection{The topological assembly map for proper $LT$-spaces}\label{section top ass map top ass map}

In this subsection, we define an proper $LT$-spaces version of the topological assembly map, and we will prove that the value of it at $[\ca{L}]$ is given by the analytic index constructed in \cite{Thesis}. Combining this result and Theorem \ref{Poincare duality for LT manifold}, we will prove the main result of the present paper.

\begin{dfn}
We define a homomorphism substituting for the topological assembly map 
$$\ud{\nu_{LT_{L^2_m}}^\tau}:\ca{R}KK_{LT_{L^2_m}}^\tau(\ca{M}_{L^2_m};\scr{C}(\ca{M}_{L^2_m}), \scr{C}(\ca{M}_{L^2_m}))\to KK(\bb{C},\ud{\bb{C}\rtimes_\tau LT_{L^2_m}})$$
by 
$$\ud{\nu_{LT_{L^2_m}}^\tau}(y):=\ud{[c_{\ca{M}_{L^2_m}}]}\grotimes 
\fgt\bra{\ud{j^\tau_{LT_{L^2_m}}}(y)}
\grotimes \ud{
j_{LT_{L^2_m}}^\tau\bra{\fgt[S]\grotimes \bbbra{\widetilde{d_{\ca{M}_{L^2_m}}}}}}.$$
\end{dfn}

According to the index theorem explained in Section \ref{section index theorem Spinc 2n}, the value of the topological assembly map at the $\ca{R}KK$-element corresponding to a vector bundle $E$, coincides with the index of the Dirac operator twisted by $E$. In order to prove a parallel result for our case, we recall the definition of the analytic index from \cite[Definition 6.18]{Thesis}

\begin{dfn}
$(1)$ We define a pre-Hilbert $\ud{\bb{C}\rtimes_\tau LT_{L^2_m}}$-module structure on
$$C_c^\infty(\widetilde{M},\ca{L}_{\widetilde{M}}\grotimes S_{\widetilde{M}})\widehat{\bigotimes}^\alg \ud{(\bb{C}\rtimes_{-\tau}U_{L^2_m})_\fin}\grotimes S_{U,\fin}$$
by the following operations: For $f,f_1,f_2\in C_c^\infty(\widetilde{M},\ca{L}_{\widetilde{M}}\grotimes S_{\widetilde{M}})$, $\phi,\phi_1,\phi_2\in \ud{(\bb{C}\rtimes_{-\tau}U)_\fin}$, $s,s_1,s_2\in  S_{U,\fin}$, $b_1\widehat{\bigotimes} b_2\in [\bb{C}\rtimes_\tau (T\times\Pi_T)] \widehat{\bigotimes} \ud{\bb{C}\rtimes_\tau U_{L^2_m}}$ and $\gamma\in (T\times\Pi_T)^\tau$,
\begin{itemize}
\item $(f\widehat{\bigotimes} \phi\grotimes s)\cdot (b_1\widehat{\bigotimes} b_2):= \int_{(T\times\Pi_T)^\tau} \gamma(f)b_1(\gamma^{-1})d\gamma\widehat{\bigotimes}(b_2^\vee* \phi)\grotimes s$; and
\item $\inpr{f_1\widehat{\bigotimes}\phi_1\grotimes s_1}{f_2\widehat{\bigotimes}\phi_2\grotimes s_2}{\ud{\bb{C}\rtimes_\tau LT_{L^2_m}}}(\gamma):= 
\int_{\widetilde{M}}\inpr{f_1(x)}{\gamma.f_2(\gamma^{-1}.x)}{}dx[\phi_2^** \phi_1]^\vee\inpr{s_1}{s_2}{S_U}.$
\end{itemize}
The completion of this pre-Hilbert module is denoted by $\ca{E}_{\ca{L}}$.

$(2)$ On this Hilbert module, we define an unbounded operator $\ca{D}_{\ca{L}}$ by the following:
\begin{align*}
\ca{D}_{\ca{L}}&:=\sum_n c(e_i)\circ \nabla_{e_i}^{\ca{L}_{\widetilde{M}}\grotimes S_{\widetilde{M}}}\widehat{\bigotimes}\id_{\ud{(\bb{C}\rtimes_{-\tau}LT_{L^2_m})_\fin}\grotimes S_{U}} \\
&\ \ \ +\id_{L^2(\widetilde{M},\ca{L}_{\widetilde{M}}\grotimes S_{\widetilde{M}})}\widehat{\bigotimes}\sum_n\bra{\sqrt{n}d\rt_{z_n}\grotimes \gamma\bra{\overline{z_n}}+\sqrt{n}d\rt_{\overline{z_n}}\grotimes \gamma\bra{z_n}}.
\end{align*}
The self-adjoint extension of $\ca{D}_{\ca{L}}$ is also denoted by the same symbol.
The {\bf analytic index} $\ud{\ind(\ca{D}_{\ca{L}})}$ is defined by $[(\ca{E}_{\ca{L}},\ca{D}_{\ca{L}})]\in KK(\bb{C},\ud{\bb{C}\rtimes_\tau LT_{L^2_m}})$.
\end{dfn}
\begin{rmk}
The index is given by the exterior tensor product of the index of $D_{\rm base}$ and the $KK$-element given in \cite[Definition 6.18]{Thesis}.
\end{rmk}

\begin{thm}
$\ud{\nu_{LT_{L^2_m}}^\tau}([\ca{L}])=\ud{\ind(\ca{D}_{\ca{L}})}.$
\end{thm}
\begin{proof}
We prove the result by the following steps. $(1)$ We will compute the Kasparov product $\ud{\bbbra{c_{\ca{M}_{L^2_m}}}}\grotimes 
\fgt\bra{\ud{j^\tau_{LT_{L^2_m}}}\bra{[\ca{L}]}}$. Then, we will prove $\ud{\nu_{LT_{L^2_m}}^\tau}([\ca{L}])=\ud{\ind(\ca{D}_{\ca{L}})}$. For this aim, we will prove the following: $(2)$ The modules of both sides are isomorphic; $(3)$ $\ca{D}_{\ca{L}}$ satisfies the connection condition; $(4)$ It satisfies the positivity condition. In fact, $(4)$ is immediate from the fact that the operator of the $KK$-element given in $(1)$ is $0$.

$(1)$ Let us compute the Kasparov product $\ud{\bbbra{c_{\ca{M}_{L^2_m}}}}\grotimes 
\fgt\bra{\ud{j^\tau_{LT_{L^2_m}}}\bra{[\ca{L}]}}$. As proved in Proposition \ref{Mishchenko line byndle for LTmanifold}, $\ud{\bbbra{c_{\ca{M}_{L^2_m}}}}$ is represented by
$$\bra{C_0\bra{\widetilde{M}\times_{T\times \Pi_T} \ud{L^2(LT_{L^2_k})}^*},1,0}.$$
As defined in Definition-Proposition \ref{descent and crossed product for LT theory}, $\ud{j^\tau_{LT_{L^2_m}}}\bra{[\ca{L}]}$ is represented by
$$\bra{C_0\bra{\widetilde{M}\times_{T\times\Pi_T}\bbra{\bb{K}\bra{\ud{L^2\bra{LT_{L^2_k},\ca{L}}},\ud{L^2\bra{LT_{L^2_k}}}}\grotimes\ca{L}_{\widetilde{M}}}},\ud{\pi\rtimes_\tau \lt},0}.$$
Therefore, the Kasparov product we are computing is given by
$$\bra{C_0\bra{\widetilde{M}\times_{T\times\Pi_T}\bbra{\ud{L^2\bra{LT_{L^2_k},\ca{L}}^*}\grotimes\ca{L}_{\widetilde{M}}}},1,0}.$$

$(2)$ We would like to prove that the Kasparov product of the above and $\ud{
j_{LT_{L^2_m}}^\tau\bra{\fgt[S]\grotimes\bbbra{\widetilde{d_{\ca{M}_{L^2_m}}}}}}$ coincides with $\ud{\ind(\ca{D}_{\ca{L}})}$.
First, we prove that the modules are isomorphic.
{\small
\begin{align*}
&C_0\bra{\widetilde{M}\times_{T\times\Pi_T}[\ud{L^2\bra{LT_{L^2_k},\ca{L}}^*}\grotimes\ca{L}_{\widetilde{M}}]}\grotimes_{C_0\bra{\widetilde{M}\times_{T\times\Pi_T}\bb{K}\bra{\ud{L^2\bra{LT_{L^2_k},\ca{L}}}}}} \\
&\ \ \ \ \ \ L^2\bra{\widetilde{M}\times_{T\times \Pi_T} \bbra{\ud{L^2\bra{LT_{L^2_k},\ca{L}}}\grotimes \ud{\bb{C}\rtimes_{-\tau} LT_{L^2_m}}\grotimes [S_U\grotimes S_{\widetilde{M}}]
}} \\
&\cong L^2\bra{\widetilde{M}\times_{T\times\Pi_T}\bbra{\bra{\ud{L^2\bra{LT_{L^2_k},\ca{L}}^*}\grotimes\ca{L}_{\widetilde{M},x}}\grotimes_{\bb{K}\bra{\ud{L^2\bra{LT_{L^2_k},\ca{L}}}}} \bra{\ud{L^2\bra{LT_{L^2_k},\ca{L}}}\grotimes \ud{\bb{C}\rtimes_{-\tau} LT_{L^2_m}}\grotimes S_{\ca{M},x}
}}_{x\in \widetilde{M}}}\\
&\cong L^2\bra{\widetilde{M}\times_{T\times \Pi_T}\bbra{\ud{\bb{C}\rtimes_{-\tau} LT_{L^2_m}}\grotimes \ca{L}_{\widetilde{M}}\grotimes S_{\ca{M}}|_{\widetilde{M}}}},
\end{align*}}
thanks to $V^*\grotimes_{\bb{K}(V)}V\cong \bb{C}$ for a Hilbert space $V$.
The isomorphism is given by 
$$[\phi\grotimes l]\grotimes_{C_0\bra{\widetilde{M}\times_{T\times\Pi_T}\bb{K}\bra{\ud{L^2\bra{LT_{L^2_k},\ca{L}}}}}} [\psi\grotimes b\grotimes s]\mapsto \innpro{\phi}{\psi}{}b\grotimes l\grotimes s$$
for $\phi:\widetilde{M}\to \ud{L^2\bra{LT_{L^2_k},\ca{L}}^*}$, $l:\widetilde{M}\to\ca{L}_{\widetilde{M}}$, $\psi:\widetilde{M}\to \ud{L^2\bra{LT_{L^2_k},\ca{L}}}$, $b:\widetilde{M}\to\ud{\bb{C}\rtimes_{-\tau} LT_{L^2_m}}$ and $s:\widetilde{M}\to S_{\ca{M}}|_{\widetilde{M}}$ such that $\phi\grotimes l$ and $\psi\grotimes b\grotimes s$ are equivariant sections. 

For the next step, we  describe this isomorphism in detail.
We can factorize $\ud{L^2\bra{LT_{L^2_k},\ca{L}}}$, $\ud{L^2\bra{LT_{L^2_k},\ca{L}}}^*$ and $\ud{\bb{C}\rtimes_{-\tau} LT_{L^2_m}}$ into the $T\times\Pi_T$-part and the $U_{L^2_m}$-part. 
Since $T\times\Pi_T$ acts on the latter factor trivially, $\phi$ can be approximated by finite sums of functions of the form $x\mapsto \phi_1(x)\grotimes \phi_2$ for $\phi_1:\widetilde{M}\to L^2(T\times\Pi_T,\ca{L})^*$ and $\phi_2\in \ud{L^2\bra{U_{L^2_k},\ca{L}}}^*$. 
Similarly, $\psi$, $b$ and $s$ can be approximated by finite sums of functions of the following forms, respectively: $x\mapsto \psi_1(x)\grotimes \psi_2$, $x\mapsto b_1(x)\grotimes b_2$ and $x\mapsto s_1(x)\grotimes s_2$ for $\psi_1:\widetilde{M}\to L^2(T\times\Pi_T,\ca{L})$, $\psi_2\in \ud{L^2\bra{U_{L^2_k},\ca{L}}}$, $b_1:\widetilde{M}\to \bb{C}\rtimes_{-\tau} (T\times\Pi_T)$, $b_2\in \ud{\bb{C}\rtimes_{-\tau} U_{L^2_m}}$, $s_1:\widetilde{M}\to S_{\widetilde{M}}$ and $s_2\in S_U$. Using these factorizations, we identify the vector spaces as follows:

\begin{align*}
C_0\bra{\widetilde{M}\times_{T\times\Pi_T}\bbra{\ud{L^2\bra{LT_{L^2_k},\ca{L}}^*}\grotimes\ca{L}_{\widetilde{M}}}}\cong C_0\bra{\widetilde{M}\times_{T\times\Pi_T}\bbra{{L^2\bra{T\times\Pi_T,\ca{L}}^*}\grotimes\ca{L}_{\widetilde{M}}}}\widehat{\bigotimes} \ud{L^2\bra{U_{L^2_k},\ca{L}}^*};
\end{align*}
\begin{align*}
&L^2\bra{\widetilde{M}\times_{T\times \Pi_T} \bbra{\ud{L^2\bra{LT_{L^2_k},\ca{L}}}\grotimes \ud{\bb{C}\rtimes_{-\tau} LT_{L^2_m}}\grotimes  S_{\ca{M}}|_{\widetilde{M}}}} \\
&\ \ \ \cong 
L^2\bra{\widetilde{M}\times_{T\times \Pi_T} \bbra{{L^2\bra{T\times\Pi_T,\ca{L}}}\grotimes {\bb{C}\rtimes_{-\tau} (T\times\Pi_T)}\grotimes  S_{\widetilde{M}}}}
\widehat{\bigotimes} \ud{L^2\bra{U_{L^2_k},\ca{L}}}\grotimes \ud{\bb{C}\rtimes_{-\tau} U_{L^2_m}}\grotimes S_U;
\end{align*}
\begin{align*}
&L^2\bra{\widetilde{M}\times_{T\times \Pi_T}\bbra{\ud{\bb{C}\rtimes_{-\tau} LT_{L^2_m}}\grotimes \ca{L}_{\widetilde{M}}\grotimes S_{\ca{M}}|_{\widetilde{M}}}} \\
&\ \ \ \cong L^2\bra{\widetilde{M}\times_{T\times \Pi_T}\bbra{\bra{\bb{C}\rtimes_{-\tau} (T\times\Pi_T)}\grotimes \ca{L}_{\widetilde{M}}\grotimes S_{\widetilde{M}}}}
\widehat{\bigotimes} \ud{\bb{C}\rtimes_{-\tau} U_{L^2_m}}\grotimes  S_{U}.
\end{align*}
Under these identifications, the isomorphism $[\phi\grotimes l]\grotimes [\psi\grotimes b\grotimes s]\mapsto \innpro{\phi}{\psi}{}b\grotimes l\grotimes s$ can be described as 
$$\bra{\phi_1\grotimes l\widehat{\bigotimes} \phi_2}\grotimes
\bra{\psi_1\grotimes b_1\grotimes s_1\widehat{\bigotimes} \psi_2\grotimes b_2\grotimes s_2}\mapsto \bra{\innpro{\phi_1}{\psi_1}{}b_1\grotimes  l\grotimes s_1}\widehat{\bigotimes}\bra{ \innpro{\phi_2}{\phi_2}{} b_2\grotimes s_2}.$$
Note that $\innpro{\phi_2}{\phi_2}{} (b_2\grotimes s_2)$ is independent of $x\in \widetilde{M}$.

$(3)$ Let us check the connection condition. Let us consider
$$C_c^\infty\bra{\widetilde{M}\times_{T\times\Pi_T} \bbbra{L^2(T\times\Pi_T,\ca{L})^*\grotimes\ca{L}_{\widetilde{M}}}}\widehat{\bigotimes}^\alg \ud{L^2\bra{U_{L^2_k},\ca{L}}_\fin^*}.$$
It is a dense subspace of $C_0\bra{\widetilde{M}\times_{T\times\Pi_T}\bbra{\ud{L^2\bra{LT_{L^2_k},\ca{L}}^*}\grotimes\ca{L}_{\widetilde{M}}}}$.
We prove that the commutator
$$
\bbbra{\begin{pmatrix}
\ca{D}_{\ca{L}} & 0 \\ 0 & \widetilde{D} \end{pmatrix},
\begin{pmatrix}
0 & T_e \\ T_e^* & 0 \end{pmatrix}
}=\begin{pmatrix}
0 & \ca{D}_{\ca{L}}\circ T_e-T_e\circ \widetilde{D} \\ 
\widetilde{D}\circ T_e^*-T_e^*\circ\ca{D}_{\ca{L}} & 0 \end{pmatrix}
$$
is bounded for an element $e$ of the above dense subspace. 
We may assume that $e$ is a finite sum of elements of the form $\phi_1\grotimes l\widehat{\bigotimes} \phi_2$ for $\phi_1:\widetilde{M}\to L^2(T\times\Pi_T,\ca{L})^*$, $l:\widetilde{M}\to \ca{L}_{\widetilde{M}}$ and $\phi_2\in \ud{L^2\bra{U_{L^2_k},\ca{L}}_\fin^*}$ satisfying the equivariance condition $g\cdot [\phi_1\grotimes l(g^{-1}x)]=\phi_1\grotimes l(x)$ for $g\in T\times \Pi_T$.  We may assume $e=\phi_1\grotimes l\widehat{\bigotimes} \phi_2$ from the beginning, because a finite sum of bounded operators is again bounded. 
\begin{align*}
& \ca{D}_{\ca{L}}\circ T_{\phi_1\grotimes l\widehat{\bigotimes} \phi_2}\bra{\psi_1\grotimes b_1\grotimes s_1\widehat{\bigotimes} \psi_2\grotimes b_2\grotimes s_2}-T_{\phi_1\grotimes l\widehat{\bigotimes} \phi_2}\circ \widetilde{D}\bra{\psi_1\grotimes b_1\grotimes s_1\widehat{\bigotimes} \psi_2\grotimes b_2\grotimes s_2} \\
&=  \ca{D}_{\ca{L}}\innpro{\phi_1}{\psi_1}{}(b_1\grotimes l\grotimes s_1)\widehat{\bigotimes} \innpro{\phi_2}{\psi_2}{}(b_2\grotimes  s_2)-T_{\phi_1\grotimes l\widehat{\bigotimes} \phi_2}\bra{\sum_jc(e_j)\circ\nabla^{S_{\widetilde{M}}}_{e_j}\bra{\psi_1\grotimes b_1\grotimes s_1\widehat{\bigotimes} \psi_2\grotimes b_2\grotimes s_2}} \\
&\ \ \ -T_{\phi_1\grotimes l\widehat{\bigotimes} \phi_2}\bra{\bra{\psi_1\grotimes b_1\grotimes s_1}\widehat{\bigotimes}\sum_n\bbra{\bra{n^{-2l}dR'(z_n)\psi_2\grotimes b_2+\psi_2\grotimes \sqrt{n}d\rt_{z_n}b_2}\grotimes \gamma\bra{\overline{z_n}}(s_2)}} \\
&\ \ \ \ \ \ -T_{\phi_1\grotimes l\widehat{\bigotimes} \phi_2}\bra{\bra{\psi_1\grotimes b_1\grotimes s_1}\widehat{\bigotimes}\sum_n\bbra{\bra{n^{-2l}dR'(\overline{z_n})\psi_2\grotimes b_2+\psi_2\grotimes \sqrt{n}d\rt_{\overline{z_n}}b_2}\grotimes \gamma\bra{z_n}(s_2)}} \\
&=\bra{\sum_j c(e_j)\circ \nabla_{e_j}^{\ca{L}_{\widetilde{M}}\grotimes S_{\widetilde{M}}}\bbra{\innpro{\phi_1}{\psi_1}{}(b_1\grotimes l\grotimes s_1)}}\widehat{\bigotimes} \innpro{\phi_2}{\psi_2}{}(b_2\grotimes  s_2)\\
&\ \ \ +
\innpro{\phi_1}{\psi_1}{}(b_1\grotimes l\grotimes s_1)\widehat{\bigotimes} \innpro{\phi_2}{\psi_2}{}\sum_n\bra{\sqrt{n}d\rt_{z_n}b_2\grotimes \gamma\bra{\overline{z_n}}(s_2)+\sqrt{n}d\rt_{\overline{z_n}}b_2\grotimes \gamma\bra{z_n}(s_2)} \\
&\ \ \ \ \ \ -T_{\phi_1\grotimes l\widehat{\bigotimes} \phi_2}\bra{\sum_jc(e_j)\circ\nabla^{S_{\widetilde{M}}}_{e_j}\bra{\psi_1\grotimes b_1\grotimes s_1\widehat{\bigotimes} \psi_2\grotimes b_2\grotimes s_2}} \\
&\ \ \ \ \ \ \ \ \ -T_{\phi_1\grotimes l\widehat{\bigotimes} \phi_2}\bra{\bra{\psi_1\grotimes b_1\grotimes s_1}\widehat{\bigotimes}\sum_n\bbra{\bra{n^{-2l}dR'(z_n)\psi_2\grotimes b_2+\psi_2\grotimes \sqrt{n}d\rt_{z_n}b_2}\grotimes \gamma\bra{\overline{z_n}}(s)}} \\
&\ \ \ \ \ \ \ \ \ \ \ \ -T_{\phi_1\grotimes l\widehat{\bigotimes} \phi_2}\bra{\bra{\psi_1\grotimes b_1\grotimes s_1}\widehat{\bigotimes}\sum_n\bbra{\bra{n^{-2l}dR'(\overline{z_n})\psi_2\grotimes b_2+\psi_2\grotimes \sqrt{n}d\rt_{\overline{z_n}}b_2}\grotimes \gamma\bra{z_n}(s_2)}} \\
&=:\Delta_1+\Delta_2+\Delta_3+\Delta_4+\Delta_5.
\end{align*}
We prove that the unbounded terms are cancelled out. Let us begin with the computation of $\Delta_1$. The trivial connection on the trivial bundles $\widetilde{M}\times  \ud{L^2\bra{U_{L^2_k},\ca{L}}}$, $\widetilde{M}\times \ud{\bb{C}\rtimes_{-\tau} U_{L^2_m}}$ or $\widetilde{M}\times S_U$, is denoted by $\nabla$.
Thanks to the Leibniz rule,
\begin{align*}
\Delta_1&=\sum_j c(e_j)\bbra{ 
\innpro{\nabla_{e_j}\phi_1}{\psi_1}{}b_1\grotimes l\grotimes s_1
+\innpro{\phi_1}{\nabla_{e_j}\psi_1}{}b_1\grotimes l\grotimes s_1
}\widehat{\bigotimes} \innpro{\phi_2}{\psi_2}{}b_2\grotimes  s_2\\
&\ \ \ +\sum_j \innpro{\phi_1}{\psi_1}{}c(e_j)\bra{\nabla_{e_j}b_1\grotimes l\grotimes s_1+b_1\grotimes \nabla_{e_j}^{\ca{L}}l\grotimes s_1
+b_1\grotimes l\grotimes \nabla_{e_j}^{S_{\widetilde{M}}}s_1}
\widehat{\bigotimes} \innpro{\phi_2}{\psi_2}{}b_2\grotimes  s_2\\
&=\sum_j c(e_j)\bbra{ 
\innpro{\nabla_{e_j}\phi_1}{\psi_1}{}b_1\grotimes l\grotimes s_1
+\innpro{\phi_1}{\psi_1}{}b_1\grotimes \nabla_{e_j}^{\ca{L}}l\grotimes s_1
}\widehat{\bigotimes} \innpro{\phi_2}{\psi_2}{}b_2\grotimes  s_2\\
&\ \ \ +\sum_j \innpro{\phi_1}{\psi_1}{}c(e_j)\bra{\nabla_{e_j}b_1\grotimes l\grotimes s_1
+b_1\grotimes l\grotimes \nabla_{e_j}^{S_{\widetilde{M}}}s_1}
\widehat{\bigotimes} \innpro{\phi_2}{\psi_2}{}b_2\grotimes  s_2\\
&\ \ \ \ \ \ +\sum_j \innpro{\phi_1}{\nabla_{e_j}\psi_1}{}b_1\grotimes l\grotimes c(e_j)s_1\widehat{\bigotimes} \innpro{\phi_2}{\psi_2}{}b_2\grotimes  s_2.
\end{align*}
Note that the assignment $\psi_1\grotimes b_1\grotimes s_1\widehat{\bigotimes} \psi_2\grotimes b_2\grotimes s_2 \mapsto$ ``the first line of the above''
does not contain derivations. Therefore, it is a bounded operator (notice that a {\it finite} sum of bounded operators is bounded). The remainder terms are canceled by $\Delta_3$ thanks to the Leibniz rule.

Finally, $\Delta_2+\Delta_4+\Delta_5$ is {\small
\begin{align*}
&-T_{\phi_1\grotimes l\widehat{\bigotimes} \phi_2}\bra{[\psi_1\grotimes b_1\grotimes s_1]\widehat{\bigotimes}\sum_n\bbra{n^{-2l}dR'(z_n)\psi_2\grotimes b_2\grotimes \gamma\bra{\overline{z_n}}(s_2)+n^{-2l}dR'(\overline{z_n})\psi_2\grotimes b_2\grotimes \gamma\bra{z_n}(s_2)} }\\
&\ \ \ =-\innpro{\phi_1}{\psi_1}{}b_1 \grotimes l \grotimes s_1\widehat{\bigotimes} \sum_n\bra{
\innpro{\phi_2}{n^{-2l}dR'(z_n)\psi_2}{}b_2\grotimes \gamma\bra{\overline{z_n}}(s_2)+
\innpro{\phi_2}{n^{-2l}dR'(\overline{z_n})\psi_2}{}b_2\grotimes \gamma\bra{z_n}(s_2)}\\ 
&\ \ \ =-\innpro{\phi_1}{\psi_1}{}b_1 \grotimes l \grotimes s_1\widehat{\bigotimes} \sum_n\bra{
\innpro{n^{-2l}[dR'(z_n)]^*\phi_2}{\psi_2}{}b_2\grotimes \gamma\bra{\overline{z_n}}(s_2)+
\innpro{n^{-2l}[dR'(\overline{z_n})]^*\phi_2}{\psi_2}{}b_2\grotimes \gamma\bra{z_n}(s_2)}.
\end{align*}}
For the same reason of Theorem \ref{jLTDirac}, the correspondence
$${\psi_1\grotimes b_1\grotimes s_1\widehat{\bigotimes} \psi_2\grotimes b_2\grotimes s_2}\mapsto \Delta_2+\Delta_4+\Delta_5$$
is bounded.

One can prove that $\widetilde{D}\circ T_e^*-T_e^*\circ\ca{D}_{\ca{L}} $ is bounded in the same way of the above.
Therefore, $\ca{D}_{\ca{L}}$ satisfies the connection condition. 

$(4)$ Since $[0\grotimes \id,\ca{D}_{\ca{L}}]=0\geq0$, $\ca{D}_{\ca{L}}$ satisfies the positivity condition.
\end{proof}

Combining it and Theorem \ref{Poincare duality for LT manifold}, we obtain the following main result of the present paper.

\begin{cor}
$\ud{\nu_{LT}^\tau}(\PD[\widetilde{\ca{D}}])=\sigma_{\ca{S}_\vep}\bra{\ud{\ind(\ca{D}_{\ca{L}})}}.$
\end{cor}

\subsection{Next problems}\label{section unsolved}

As a concluding remark of the present paper, we explain what we should do in order to complete the index theory for proper $LT$-spaces. Then, we will give a comment to improve our theory as mentioned in Remarks \ref{rmk ugly Hilbert space}  $(3)$. 
Roughly speaking, this ``improvement'' is to replace $KK_{LT_{L^2_m}}(\ca{A}(\ca{M}_{L^2_k}),\ca{S}_\vep)$ with $KK_{LT_{L^2_k}}(\ca{A}(\ca{M}_{L^2_k}),\ca{S}_\vep)$. We will explain why we needed to use the former one in the present paper, and why we want to replace it.

\vspace{0.3cm}
Recall the big diagram in the proof of Proposition \ref{pro index thm big diagram}. The infinite-dimensional version of it should be of the following form. Dotted arrows have not been defined. Undefined symbols will be explained later as Conjecture \ref{Conjecture}. We omit the subscripts of $LT$ and $\ca{M}$ for simplicity.
$$\xymatrix{
KK_{LT}^{\tau}(\ca{A}(\ca{M}),\ca{S}_\vep) 
\ar@{.>}_-{\ud{\widetilde{j^\tau_{LT}}}}[d]
\ar@/^20pt/[rr]^{\PD}
&
KK_{LT}^{\tau}(\ca{A}(\ca{M}),\ca{A}(\ca{M})) 
\ar^-{-\grotimes\ud{[\widetilde{d_{\ca{M}}}']}}[l] 
\ar@{.>}^-{\ud{\widetilde{j^\tau_{LT}}}}[d] &
\ca{R}KK_{LT}^{\tau}(\ca{M};\ca{S}_\vep\grotimes \scr{C}(\ca{M}),\ca{S}_\vep\grotimes \scr{C}(\ca{M})) 
\ar@{.>}^-{\widetilde{\fgt}}[l]
 \ar^-{\ud{j^\tau_{LT}}}[d] \\
KK(\ca{S}_\vep\grotimes \star_1,\ud{\ca{S}_\vep\rtimes_\tau {LT}})  
\ar_-{\sigma_{\ca{S}_\vep}[\ud{c_{\ca{M}}}]\grotimes-}[d] &
KK(\ca{S}_\vep\grotimes \star_1,\ca{S}_\vep\grotimes \star_2) 
\ar[l]^-{-\grotimes \ud{j_{LT}^\tau(\fgt[S]\grotimes [\widetilde{d_{\ca{M}}}])}} 
\ar^-{\sigma_{\ca{S}_\vep}[\ud{c_{\ca{M}}}]\grotimes-}[d] &
\ca{R}KK(\ca{M}/LT;\ca{S}_\vep\grotimes \star_1,\ca{S}_\vep\grotimes \star_2) 
\ar^-{\fgt }[l]\\
KK(\ca{S}_\vep,\ud{\ca{S}_\vep\rtimes_\tau {LT}})  &
KK(\ca{S}_\vep,\ca{S}_\vep\grotimes \star_2),
\ar^-{-\grotimes \ud{\widetilde{j^\tau_{LT}}(\fgt[S]\grotimes [\widetilde{d_{\ca{M}}}])}}[l]
}$$
where $\star_1$ stands for $\ud{\scr{C}(\ca{M})\rtimes LT}$ and $\star_2$ stands for $\ud{\scr{C}(\ca{M})\rtimes_\tau LT}$.
$\ud{[\widetilde{d_{\ca{M}}}']}\in KK_{LT}(\ca{A}(\ca{M}),\ca{S}_\vep)$ is the reformulated Dirac element defined in \cite{T4} (we change the symbol from this paper).

\begin{conj}\label{Conjecture}
$(1)$ We can define a substitute for the descent homomorphism
$$\ud{\widetilde{j^{p\tau}_{LT,q}}}:KK_{LT}^{q\tau}(\ca{A}(\ca{M}),\ca{S}_\vep)\to 
KK(\ca{S}_\vep\grotimes \ud{\scr{C}(\ca{M})\rtimes_{(p-q)\tau} {LT}},\ud{\ca{S}_\vep\rtimes_{p\tau} {LT}})$$
and satisfies
$$\ud{j_{LT}^\tau(\fgt[S]\grotimes [\widetilde{d_{\ca{M}}}])}=\ud{\widetilde{j^{\tau}_{LT,0}}}(\ud{[\widetilde{d_{\ca{M}}}']}).$$

$(2)$ We can define a substitute for the descent homomorphism
$$\ud{\widetilde{j^{p\tau}_{LT,q}}}:KK_{LT}^{q\tau}(\ca{A}(\ca{M}),\ca{A}(\ca{M}))\to 
KK(\ca{S}_\vep\grotimes \ud{\scr{C}(\ca{M})\rtimes_{(p-q)\tau} {LT}},\ca{S}_\vep\grotimes \ud{\scr{C}(\ca{M})\rtimes_{p\tau} {LT}})$$
satisfying the following: For $x\in KK_{LT}^{q_1\tau}(\ca{A}(\ca{M}),\ca{A}(\ca{M}))$ and $y\in KK_{LT}^{q_2\tau}(\ca{A}(\ca{M}),\ca{S}_\vep)$,
$$j_{LT,q_1+q_2}^{p\tau}(x\grotimes y)= j_{LT,q_1}^{(p-q_2)\tau}(x)\grotimes j_{LT,q_2}^{p\tau}(y).$$
As a corollary, the top left corner of the above big diagram commutes.

$(3)$ We can define a homomorphism 
$$\widetilde{\fgt}:\ca{R}KK_{LT}^{q\tau}(X;\ca{S}_\vep\grotimes \scr{C}(\ca{M}),\ca{S}_\vep\grotimes \scr{C}(\ca{M})) \to 
KK_{LT}^{q\tau}(\ca{A}(\ca{M}),\ca{A}(\ca{M}))$$
and it satisfies the following commutative diagram
$$\begin{CD}
\ca{R}KK_{LT}^{q\tau}(\ca{M};\ca{S}_\vep\grotimes \scr{C}(\ca{M}),\ca{S}_\vep\grotimes \scr{C}(\ca{M}))  
@>\widetilde{\fgt}>>
KK_{LT}^{q\tau}(\ca{A}(\ca{M}),\ca{A}(\ca{M})) \\
@V\widetilde{j^{p\tau}_{LT}}VV @VV\ud{j^{p\tau}_{LT}}V \\
\ca{R}KK(\ca{M}/LT;\ca{S}_\vep\grotimes \star_1,\ca{S}_\vep\grotimes \star_2) 
@>\fgt>>
KK(\ca{S}_\vep\grotimes \star_1,\ca{S}_\vep\grotimes \star_2),
\end{CD}$$
where $\star_1:=\ud{\scr{C}(\ca{M})\rtimes_{(p-q)\tau} {LT}}$ and $\star_2:=\ud{\scr{C}(\ca{M})\rtimes_{p\tau} {LT}}$.
As a corollary, the top right corner of the above big diagram commutes.

$(4)$ The homomorphisms $\PD$ and $y\mapsto \widetilde{\fgt}(y)\grotimes\ud{[\widetilde{d_{\ca{M}}}']}$
are mutually inverse.
\end{conj}

As a corollary of this conjecture, a substitute for the analytic assembly map will be defined by $\ud{\mu_{LT}^\tau}(x):=\sigma_{\ca{S}_\vep}[\ud{c_{\ca{M}}}]\grotimes\ud{\widetilde{j^\tau_{LT}}}(x)$ for $x\in KK_{LT}(\ca{A}(\ca{M}),\ca{S}_\vep)$, and the index theorem type equality $\ud{\mu_{LT}^\tau}(x)=\ud{\nu_{LT}^\tau}(\PD(x))$ will hold. By the results of the present paper, we have $\ud{\mu_{LT}^\tau}([\widetilde{\ca{D}}])=\sigma_{\ca{S}_\vep}\bra{\ud{\ind(\ca{D}_{\ca{L}})}}$.
If the index theorem is completed, we will have various applications as we have explained in Section \ref{section introduction}.

\vspace{0.3cm}

Next, let us discuss the ``improvement'' we have in mind. We will explain the reason why we need it.
Unfortunately, we do not have concrete ideas on this problem. From now on, we write the subscripts $L^2_k$'s again, because it is essential. 

Seeing Definition \ref{dfn index element for the present paper for proper LT space}, everyone would wish the following.

\begin{conj}
There is a ``collect'' index element $[\widetilde{\ca{D}}]\in KK_{LT_{L^2_k}}(\ca{A}(\ca{M}_{L^2_k}),\ca{S}_\vep)$ such that 
$$i_{m,k}^*[\widetilde{\ca{D}}]\in KK_{LT_{L^2_m}}(\ca{A}(\ca{M}_{L^2_k}),\ca{S}_\vep)$$
is the index element defined in this paper, where $i_{m,k}:LT_{L^2_m}\to LT_{L^2_k}$ is the canonical embedding and $i_{m,k}^*$ is the restriction of the $LT_{L^2_k}$-action to $LT_{L^2_m}$.
\end{conj}

Probably, our direct and naive method does not work for this problem. In order to explain the reason of this, we explain the reason why we needed to use $KK_{LT_{L^2_m}}(\ca{A}(\ca{M}_{L^2_k}),\ca{S}_\vep)$.

Let us attempt to do the same thing of Section \ref{section PD Cstar algebra HKT Yu} for $l=0$ and $m=k$. Then, $\vac\in \ud{L^2(U_{L^2_k})}$ looks like the Gaussian
``$\prod_N \pi^{-1/2}e^{-(x_N^2+y_N^2)/2}$''. Thus, if we try to define $\pi:\ca{A}_\HKT(U_{L^2_k})\to \ca{L}_{\ca{S}}(\ca{S}\grotimes \ca{H})$ in the same way of the present paper, for $f_e(t):=e^{-rt^2}\in \ca{S}$ for arbitrary $r>0$, $\pi(\beta_0^\infty(f_e))\vac$ looks like
$$\pi^{-\infty/2}e^{-(\frac{1}{2}+r)\sum_N(x_N^2+y_N^2)}.$$
By the eigenfunction expansion, $\vac_N^N:=\pi^{-1/2}e^{-(\frac{1}{2}+r)(x_N^2+y_N^2)}$ can be written as 
$$\sum_{\alpha,\beta}c_{\alpha,\beta}\frac{dR'(\overline{z_N})^\alpha dL'({z_N})^\beta\vac_N^N}{\|dR'(\overline{z_N})^\alpha dL'({z_N})^\beta\vac_N^N\|},$$
where $dR'(\overline{z_N})$'s are defined for $l=0$.
Since $c_{\alpha,\beta}$ is independent of $N$ and $|c_{\alpha,\beta}|<1$ (this is because $\|\pi(\beta_0^\infty(f_e))\vac\|\leq \|f_e\|\|\vac\|=1$), the infinite tensor product
\begin{align*}
\pi(\beta_0^\infty(f_e))\vac 
&=\bigotimes_N \sum_{\alpha_N,\beta_N}c_{\alpha_N,\beta_N}\frac{dR'(\overline{z_N})^{\alpha_N}dL'({z_N})^{\beta_N}\vac_N^N}{\|dR'(\overline{z_N})^{\alpha_N}dL'({z_N})^{\beta_N}\vac_N^N\|} \\
&=\sum_{\alpha_1,\beta_1,\alpha_2,\beta_2,\cdots}\bra{\prod_Nc_{\alpha_N,\beta_N}}\frac{dR'(\overline{z_1})^{\alpha_1}dL'({z_1})^{\beta_1}dR'(\overline{z_2})^{\alpha_2}dL'({z_2})^{\beta_2}\cdots \vac}{\|dR'(\overline{z_1})^{\alpha_1}dL'({z_1})^{\beta_1}dR'(\overline{z_2})^{\alpha_2}dL'({z_2})^{\beta_2}\cdots \vac\|}
\end{align*}
must be zero. For example, the coefficient of ``$\vac$'' is ``$(c_{0,0})^\infty$'', which is zero because $|c_{0,0}|<1$.
This is the reason why we need to assume that $l$ is positive to define the $*$-homomorphism $\pi$.

This argument is rather algebraic. From the analytic or geometric point of view, the trouble of this observation is due to the fact that $\pi(\beta_0^\infty(f_e))$ is ``too localized to give an operator on the $L^2$-space''. Therefore, in \cite{T4}, a strange definition of $\pi$ is adopted, so that $\pi(a)$ looks like an ``asymptotically constant function''. Conversely, in the present paper, we adopted a strange Hilbert space so that each element of $\ca{H}$ looks like an ``asymptotically Dirac $\delta$-function''.
The cost we have paid is that the natural $U_{L^2_m}$-action on our Hilbert space $\ud{L^2(U_{L^2_k})}$ does not extend to $U_{L^2_k}$. This made us work in the strange setting: The Poincar\'e duality homomorphism is a homomorphism from the $KK$-group of  $\ca{A}(\ca{M}_{L^2_k})$ to $\ca{R}KK_{LT_{L^2_m}}(\ca{M}_{L^2_m};-,-)$-group.

As long as we regard proper $LT$-spaces as ILH-manifolds, the above problem is not essential. The statement in the form of Section \ref{subsection statement of the main result} looks natural. Moreover, this setting is convenient from the view point of $LT$-equivariant $\bb{KK}$-theory, Definition \ref{definition of LT KK theory}.

However, if one hopes to fix the Sobolev level and regard a proper $LT$-space as a fixed Hilbert manifold, our construction is not satisfying, and the above conjecture must be investigated.


\section*{Acknowledgments}
An essential part of Section \ref{section PD} was done during my stay in Pennsylvania State University. I am deeply grateful to professor Nigel Higson, Yiannis Loizides and Shintaro Nishikawa for discussions held there. Especially, the primitive idea of $\ca{R}KK$-theory for non-locally compact spaces is due to Shintaro Nishikawa. I again thank him for permission to introduce it before our joint paper \cite{NT}.
I am also grateful for the warm hospitality of Penn State during my stay. 

I am supported by JSPS KAKENHI Grant Number 18J00019.

Graduate School of Mathematical Sciences, the University of Tokyo, 3-8-1 Komaba Meguro-ku Tokyo, Japan

E-mail address: {\tt dtakata@ms.u-tokyo.ac.jp}

\end{document}